\documentclass[a4paper,openright,twoside,11pt]{report}

\usepackage[top=5.35cm,bottom=5.35cm,left=4cm,right=4cm]{geometry}

\usepackage{amsfonts}
\usepackage{amsmath}
\usepackage{amssymb}
\usepackage{amsthm}

\usepackage{indentfirst}
\usepackage[english,greek]{babel}

\newcommand{\en}{\latintext}

\theoremstyle{plain}
\newtheorem{thm}{Θεώρημα}[section]
\newtheorem{lem}[thm]{Λήμμα}

\newtheorem{cor}[thm]{Πόρισμα}
\newtheorem{defn}[thm]{Ορισμός}
\newtheorem{exmp}[thm]{Παράδειγμα}

\theoremstyle{remark}
\newtheorem*{rem}{Παρατήρηση}

\usepackage{booktabs}

\usepackage{multirow}

\usepackage{boldline}

\usepackage[section]{algorithm}
\usepackage{algorithmic}

\makeatletter
\newenvironment{breakablealgorithm}
  {
   \begin{center}
     \refstepcounter{algorithm}
     \hrule height.8pt depth0pt \kern2pt
     \renewcommand{\caption}[2][\relax]{
       {\raggedright\textbf{\fname@algorithm~\thealgorithm} ##2\par}%
       \ifx\relax##1\relax 
         \addcontentsline{loa}{algorithm}{\protect\numberline{\thealgorithm}##2}%
       \else 
         \addcontentsline{loa}{algorithm}{\protect\numberline{\thealgorithm}##1}%
       \fi
       \kern2pt\hrule\kern2pt
     }
  }{
     \kern2pt\hrule\relax
   \end{center}
  }
\makeatother

\makeatletter
\renewcommand{\ALG@name}{Αλγόριθμος}
\makeatother

\usepackage{subfig}
\usepackage{graphicx}

\usepackage{xcolor}

\usepackage{diagbox}

\usepackage{fancyhdr}

\fancypagestyle{toc}{%
    \fancyhf{}
    \rhead{\it Περιεχόμενα}
    \renewcommand{\headrulewidth}{0.4pt}
    \fancyfoot[C]{\thepage}
}
\fancypagestyle{main}{%
    \fancyhf{}
    \lhead{\textit{\chaptername\ \thechapter}}
    \rhead{\textit{\nouppercase{\rightmark}}}
    \fancyfoot[C]{\thepage}
    \renewcommand{\headrulewidth}{0.4pt}        
}
\fancypagestyle{syn}{%
    \fancyhf{}
    \rhead{\it Σύνοψη}
    \renewcommand{\headrulewidth}{0.4pt}
    \fancyfoot[C]{\thepage}
}
\fancypagestyle{bib}{%
    \fancyhf{}
    \rhead{\it Βιβλιογραφία}
    \renewcommand{\headrulewidth}{0.4pt}
    \fancyfoot[C]{\thepage}
}

\usepackage{tikz}
\usetikzlibrary{calc,arrows}

\usepackage[bookmarks=false]{hyperref}

\renewcommand{\headrulewidth}{0pt}

\newcommand{\uoiauthor}{Γρηγόριος Ταχυρίδης}

\newcommand{\uoititle}{\LARGE\textsc{Μέθοδοι Υποχώρων {\en Krylov} για την Επίλυση Γραμμικών Συστημάτων {\en Toeplitz}}}

\newcommand{\thesisdedication}{Στη μνήμη του\\ πατέρα μου}

\newcommand{\advisor}{Δημήτριος Νούτσος}
\newcommand{\rankadvisor}{Επιβλέπων, Ομότιμος Καθηγητής, Πανεπιστήμιο Ιωαννίνων, Ελλάδα}
\newcommand{\examinera}{Ευστράτιος Γαλλόπουλος}
\newcommand{\ranka}{Μέλος συμβουλευτικής επιτροπής, Καθηγητής, Πανεπιστήμιο Πατρών, Ελλάδα}
\newcommand{\examinerb}{Παρασκευάς Βασσάλος}
\newcommand{\rankb}{Μέλος συμβουλευτικής επιτροπής, Αναπληρωτής Καθηγητής, Οικονομικό Πανεπιστήμιο Αθηνών, Ελλάδα}
\newcommand{\examinerd}{Φωτεινή Καρακατσάνη}
\newcommand{\rankd}{Επίκουρη Καθηγήτρια, Πανεπιστήμιο Ιωαννίνων, Ελλάδα}
\newcommand{\examinere}{Μιχαήλ Τσατσόμοιρος}
\newcommand{\ranke}{Καθηγητής, {\en Washington State University, USA}}
\newcommand{\examinerc}{Μιχαήλ Βραχάτης}
\newcommand{\rankc}{Καθηγητής, Πανεπιστήμιο Πατρών, Ελλάδα}
\newcommand{\examinerf}{Παναγιώτης Ψαρράκος}
\newcommand{\rankf}{Καθηγητής, Εθνικό Μετσόβιο Πολυτεχνείο, Ελλάδα}

\begin{document}
\thispagestyle{empty}

\noindent\begin{tabular}{l c r}
{\includegraphics[width=1.3cm]{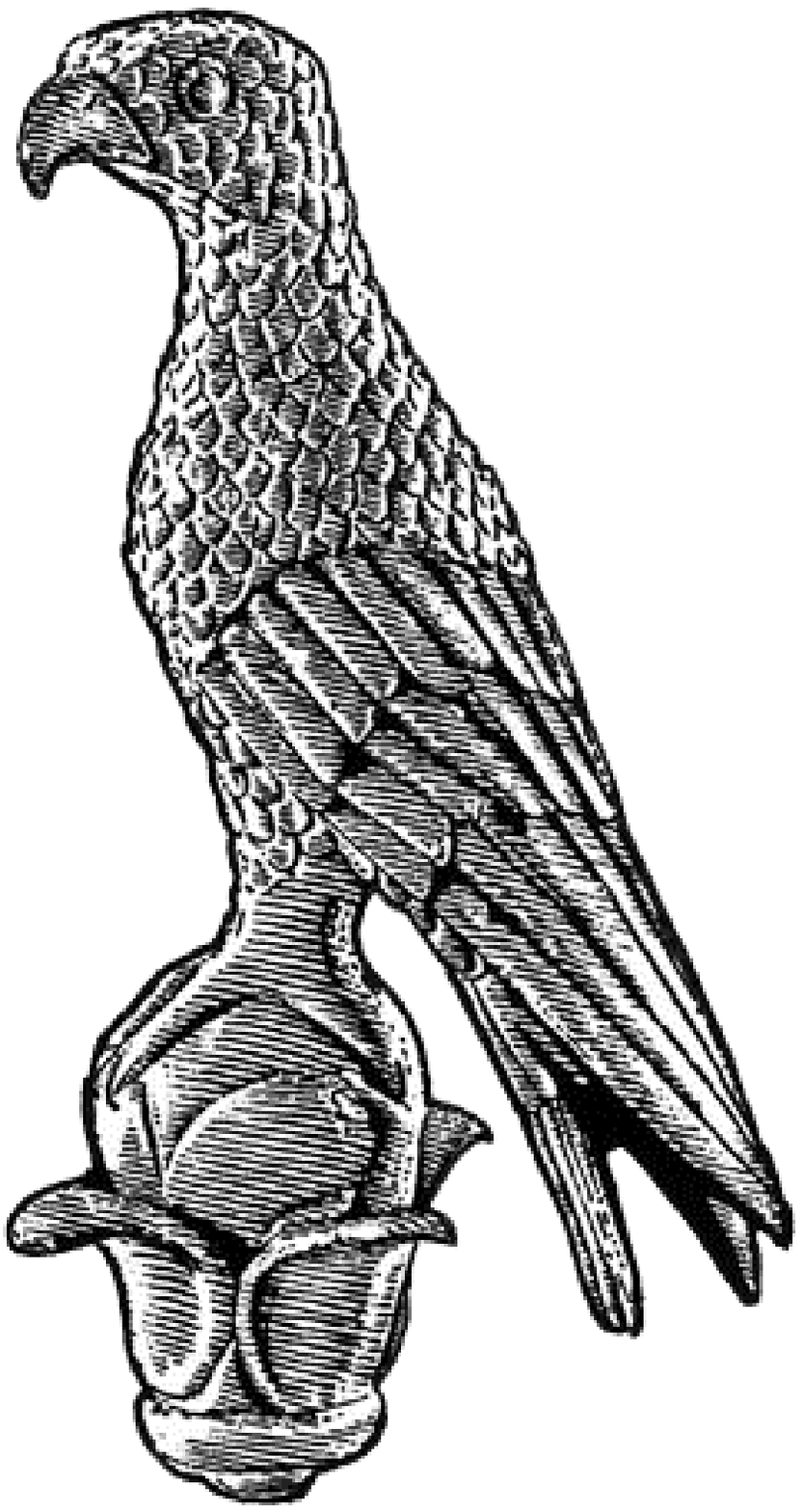}} & \raisebox{4.7\height}
{\bf \Large ΠΑΝΕΠΙΣΤΗΜΙΟ ΙΩΑΝΝΙΝΩΝ} & 
{\hfill\includegraphics[width=2.46cm]{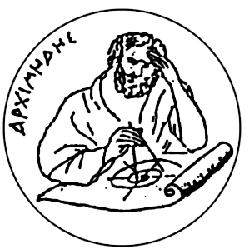}}\\[-1.3cm]
& {\textsc{\textbf{\Large ΤΜΗΜΑ ΜΑΘΗΜΑΤΙΚΩΝ}}}
\end{tabular}

\vspace*{\stretch{.3}}

\begin{center}
{\Large\bf\uoiauthor}
\end{center}

\vspace*{\stretch{.2}}

\noindent\rule{\linewidth}{0.5mm}
\begin{center}
\uoititle
\end{center}
\noindent\rule{\linewidth}{0.5mm}

\vspace*{\stretch{.3}}

\begin{center}
{\Large ΔΙΔΑΚΤΟΡΙΚΗ ΔΙΑΤΡΙΒΗ}
\end{center}

\vspace*{\fill}

\begin{center}
{\Large Ιωάννινα, 2022}
\end{center}

\newpage\thispagestyle{empty}\mbox{}\newpage

\thispagestyle{empty}

\vspace*{0.1\paperheight}

\begin{flushright}
{\em \thesisdedication}
\end{flushright}

\newpage\thispagestyle{empty}\mbox{}\newpage

\thispagestyle{empty}
\thispagestyle{empty}
\parindent0pt
\thispagestyle{empty}

Η παρούσα Διδακτορική Διατριβή εκπονήθηκε στο πλαίσιο των σπουδών για την
απόκτηση του Διδακτορικού Διπλώματος στα Μαθηματικά που απονέμει το Τμήμα Μαθηματικών του
Πανεπιστημίου Ιωαννίνων.

\vskip1cm

Εγκρίθηκε την 19/01/2022 από την εξεταστική επιτροπή: \bigskip

\begin{itemize}
\item \advisor\hspace{0.4pc}(\rankadvisor)
\item \examinera\hspace{0.4pc}(\ranka)
\item \examinerb\hspace{0.4pc}(\rankb)
\item \examinerc\hspace{0.4pc}(\rankc)
\item \examinerd\hspace{0.4pc}(\rankd)
\item \examinere\hspace{0.4pc}(\ranke)
\item \examinerf\hspace{0.4pc}(\rankf)
\end{itemize}


\vspace{0.3cm}

Η έγκριση της διδακτορικής διατριβής από το Τμήμα Μαθηματικών της Σχολής Θετικών Επιστημών του Πανεπιστημίου Ιωαννίνων δεν υποδηλώνει την αποδοχή γνωμών του συγγραφέα (Νόμος 5343/1932, άρθρο 202, παρ. 2 και Νόμος 1268/1982, άρθρο 50, παρ. 8).

\vspace{0.3cm}

\begin{center} ΥΠΕΥΘΥΝΗ ΔΗΛΩΣΗ
\end{center}

{\en ``}Δηλώνω υπεύθυνα ότι η παρούσα διατριβή εκπονήθηκε κάτω από τους
διεθνείς ηθικούς και ακαδημαϊκούς κανόνες δεοντολογίας και προστασίας της
πνευματικής ιδιοκτησίας. Σύμφωνα με τους κανόνες αυτούς, δεν έχω προβεί σε
ιδιοποίηση ξένου επιστημονικού έργου και έχω πλήρως αναφέρει τις πηγές που
χρησιμοποί\-ησα στην εργασία αυτή.{\en "}

\vspace{0.2cm}

\begin{center}
\uoiauthor
\end{center}
\parindent12pt

\newpage\thispagestyle{empty}\mbox{}\newpage 

\chapter*{Ευχαριστίες}
\thispagestyle{empty}

Θα ήθελα να ευχαριστήσω εκ βαθέων τον επιβλέποντα καθηγητή της διδακτορικής μου διατριβής, Δημήτριο Νούτσο, για την αφοσίωση, τη στήριξη, τη συνεισφορά και την αστείρευτη καθοδήγηση κατά τη διάρκεια της έρευνας και υλοποίησης της παρούσας διδακτορικής διατριβής. Θα ήθελα επίσης να ευχαριστήσω τα δύο μέλη της Συμβουλευτικής Επιτροπής, καθηγητές Ευστράτιο Γαλλόπουλο και Παρασκευά Βασσάλο, για τις χρήσιμες υποδείξεις τους και τη βοήθεια που μου παρείχαν. Θα ήθελα να αποδώσω ιδιαίτερες ευχαριστίες στον καθηγητή Απόστολο Χατζηδήμο για τη συνεργασία και τις επισημάνσεις του, καθώς επίσης και στη δόκτορα {\en Chaysri Thaniporn} για την αμφίδρομη αλληλεπίδραση. Τέλος, θα ήθελα να ευχαριστήσω το Τμήμα Μαθηματικών του Πανεπιστημίου Ιωαννίνων και ιδιαιτέρως τον Τομέα Εφαρμοσμένων και Υπολογιστικών Μαθηματικών για την παροχή χώρου εργασίας και ηλεκτρονικού υπολογιστή, καθώς και για τις ποικίλες διευκολύνσεις σε επιστημονικό και τεχνικό επίπεδο, που οδήγησαν στην ομαλότερη συνεργασία και διευκόλυνση της έρευνας.

\vspace{6pc}
Μέρος της έρευνας συγχρηματοδοτήθηκε από την Ελλάδα και την Ευρωπαϊκή `Ενωση (Ευρωπαϊκό Κοινωνικό Ταμείο) μέσω του Επιχειρησιακού Προγράμματος «Ανάπτυξη Ανθρώπινου Δυναμικού, Εκπαίδευση και Διά Βίου Μάθηση», στα πλαίσια του έργου με τίτλο ````\textlatin{Krylov subspace methods and Perron-Frobenius theory}'''' (\textlatin{MIS 5047643}).

\newpage\thispagestyle{empty}\mbox{}\newpage 
\newpage

\pagenumbering{roman}
\addcontentsline{toc}{chapter}{Περίληψη}
\chapter*{Περίληψη}

Στην παρούσα διδακτορική διατριβή γίνεται μελέτη της προρρύθμισης τετραγωνικών, μη-συμμετρικών και πραγματικών συστημάτων {\en Toeplitz}. Αποδεικνύονται θεωρητικά αποτελέσματα, τα οποία αποτελούν ικανές συνθήκες για την αποτελεσματικότητα των προτεινόμενων προρρυθμιστών και την ταχεία σύγκλιση στη λύση του συστήματος, με μεθόδους όπως η Προρρυθμισμένη Γενικευμένη μέθοδος Ελαχίστων Υπολοίπων {\en (PGMRES)} και η Προρρυθμισμένη μέθοδος Συζυγών Κλίσεων για το σύστημα των Κανονικών Εξισώσεων {\en (PCGN)}.

Στο πρώτο κεφάλαιο παραθέτουμε βασικές εισαγωγικές έννοιες, ορισμούς και θεωρητικά αποτελέσματα, τα οποία χρησιμοποιήσαμε για να αποδείξουμε τα θεωρητικά αποτελέσματα της διατριβής. Αυτά έχουν να κάνουν κυρίως με τη συσσώρευση του φάσματος, αλλά και των ιδιαζουσών τιμών, αφού αυτή αποτελεί κριτήριο για το πόσο αποτελεσματικός είναι κάποιος προρρυθμιστής.

Στο δεύτερο κεφάλαιο κατασκευάζουμε έναν ταινιωτό {\en Toeplitz} προρρυθμιστή, για συστήματα με καλή, αλλά και κακή κατάσταση. Η τεχνική προρρύθμισης βασίζεται στην εύρεση ενός κατάλληλου τριγωνομετρικού πολυωνύμου για την άρση των ριζών της γεννήτριας συνάρτησης (αν υπάρχουν), σε συνδυασμό με προσέγγιση από κάποιο άλλο τριγωνομετρικό πολυώνυμο. Αποδείχθηκε η συσσώρευση των ιδιοτιμών, καθώς και των ιδιαζουσών τιμών του προρρυθμισμένου συστήματος.

Στο τρίτο κεφάλαιο κατασκευάζουμε έναν κυκλοειδή {\en (circulant)} προρρυθμιστή για συστήματα {\en Toeplitz} με καλή κατάσταση, καθώς κι έναν ταινιωτό-επί-κυκλοειδή {\en (band-times-circulant)} προρρυθμιστή για συστήματα με κακή κατάσταση. Αποδεικνύονται αντίστοιχα θεωρητικά αποτελέσματα, ενώ γίνεται και σύγκριση με τον προρρυθμιστή του προηγούμενου κεφαλαίου στα αριθμητικά αποτελέσματα, της τελευταίας ενότητας.

Στο τέταρτο και τελευταίο κεφάλαιο της διατριβής μελετάμε συστήματα {\en Toeplitz}, των οποίων η γεννήτρια συνάρτηση υπάρχει, αλλά δεν είναι γνωστή εκ των προτέρων. Γίνεται κατάλληλη προσαρμογή των προρρυθμιστών που κατασκευάστηκαν στα προηγούμενα κεφάλαια. Με τεχνικές εκτίμησης της γεννήτριας συνάρτησης, των ριζών αυτής και της πολλαπλότητας των εν λόγω ριζών, κατασκευάζουμε αντίστοιχους προρρυθμιστές.


\chapter*{{\en Abstract}}
\setcounter{page}{3}
\addcontentsline{toc}{chapter}{{\en Abstract}}

{\en%
In this thesis we study the preconditioning of square, non-symmetric and real Toeplitz systems. We prove theoretical results, which constitute sufficient conditions for the efficiency of the proposed preconditioners and the fast convergence to the solution of the system, by the Preconditioned Generalized Minimal Residual method (PGMRES) as well as by the Preconditioned Conjugate Gradient method applied to the system of Normal Equations (PCGN).

As introduction, in the first chapter, we give the basic definitions and theorems/lemmas that we use to prove the theoretical results of the thesis. These are dealing with the clustering of the eigenvalues, as well as of the singular values, which is a criterion for the efficiency of the preconditioner.

In the second chapter we construct a band Toeplitz preconditioner for well-conditioned, as well as for ill-conditioned systems. The preconditioning technique is based on the elimination of the roots of the generating function (if there exist), by a trigonometric polynomial, and on a further approximation. The clustering of the eigenvalues and the singular values of the preconditioned system has been proven.

In the next chapter we construct a circulant preconditioner dealing with well-conditioned Toeplitz systems and a band-times-circulant preconditioner for ill-conditioned ones. We prove analogous theoretical results and we give a comparison with the preconditioner proposed previously at the numerical results of the last section.

In the fourth and last chapter of the thesis we study Toeplitz systems, having an unknown generating function. We adapt the preconditioners constructed at the previous chapters. After estimating the generating function, its roots and the multiplicities of them, we construct the corresponding preconditioners.
}

\newpage\thispagestyle{empty}\mbox{}\newpage 

\pagenumbering{Roman}
\pagestyle{toc}
\tableofcontents

\renewcommand{\headrulewidth}{0.4pt}


\clearpage
\pagenumbering{arabic}
\pagestyle{main}

\chapter{Εισαγωγή}

Στην παρούσα διδακτορική διατριβή ασχολούμαστε με την προρρύθμιση τετραγωνικών, μη-συμμετρικών και πραγματικών συστημάτων {\en Toeplitz} διάστασης $n$. Πίνακες {\en Toeplitz} εμφανίζονται σε πληθώρα εφαρμογών όπως στην επεξεργασία σήματος, επεξεργασία εικόνας \cite{Chan_Jin, Ng_book}, σε εφαρμογές που προκύπτουν από τη διακριτοποίηση διαφορικών και ολοκληρωτικών εξισώσεων \cite{bai2011sinc, bai2009preconditioned, sachs}, καθώς κι εφαρμογές που σχετίζονται με οικονομικά, μηχανική και πιθανότητες \cite{duffy,higham}. Σε πολλές εξ αυτών μπορεί να συναντήσουμε ακόμα πιο συγκεκριμένες μορφές πινάκων {\en Toeplitz}, όπως τριδιαγώνιους μη-συμμετρικούς πίνακες, που εμφανίζονται στην ανάλυση χρονοσειρών και στην επίλυση μερικών διαφορικών εξισώσεων (\cite{noschese2013tridiagonal} και αναφορές εντός αυτής). Θα μπορούσαμε να πούμε ότι η ταχεία επίλυση συστημάτων {\en Toeplitz} είναι επιτακτική ανάγκη και αποτελεί πρόκληση, διότι πολλές από τις παραπάνω εφαρμογές αναζητούν επίλυση σε πραγματικό χρόνο {\en (real time)}. 

Τα στοιχεία ενός πίνακα {\en Toeplitz} είναι οι συντελεστές του αναπτύγματος {\en Fourier} μιας συνάρτησης $f$, οι οποία καλείται γεννήτρια συνάρτηση του πίνακα. `Ετσι, αν $t_{jk}$ συμβολίζει το στοιχείο που βρίσκεται στη $j$-οστή γραμμή και $k$-οστή στήλη, ισχύει:
\begin{equation}\label{eq:Toeplitz entries}
t_{jk}=\frac{1}{2\pi}\int\limits_{-\pi}^{\pi}f(x)\mathrm{e}^{-\mathrm{i}(j-k)x}\mathrm{d}x,~1\leq j,k\leq n,~\mathrm{i}=\sqrt{-1}.
\end{equation}

Η γεννήτρια συνάρτηση του πίνακα {\en Toeplitz} παίζει σημαντικό ρόλο στην επιλογή του προρρυθμιστή. `Οταν αυτή δεν έχει ρίζες λαμβάνουμε συστήματα με καλή κατάσταση {\en (well-conditioned)}, δηλαδή συστήματα των οποίων ο δείκτης κατάστασης {\en (condition number)} είναι φραγμένος από σταθερά ανεξάρτητη της διάστασης $n$, ενώ όταν μηδενίζεται σε κάποια σημεία (ή και διαστήματα) του πεδίου ορισμού της λαμβάνουμε συστήματα με κακή κατάσταση {\en (ill-conditioned)}. Περισσότερες λεπτομέρειες θα δοθούν στα επόμενα κεφάλαια.

Από τη σχέση (\ref{eq:Toeplitz entries}), γίνεται κατανοητό ότι τα στοιχεία ενός πίνακα {\en Toeplitz} $T_n$, εξαρτώνται μόνο από τη διαφορά $j-k$ κι επομένως, ένας πίνακας {\en Toeplitz} έχει τα ίδια στοιχεία κατά μήκος των διαγωνίων του, δηλαδή είναι της μορφής:
\begin{equation*}
T_{n}=
\begin{pmatrix}
t_{0} & t_{-1} & t_{-2} & \cdots & t_{-(n-2)} & t_{-(n-1)} \\
t_{1} & t_{0} & t_{-1} & \cdots & t_{-(n-3)} & t_{-(n-2)} \\
t_{2} & t_{1} & t_{0} & \cdots & t_{-(n-4)} & t_{-(n-3)} \\
\vdots & \vdots & \vdots & \ddots & \vdots & \vdots \\
t_{n-2} & t_{n-3} & t_{n-4} & \cdots & t_{0} & t_{-1} \\
t_{n-1} & t_{n-2} & t_{n-3} & \cdots & t_{1} & t_{0}
\end{pmatrix}.
\end{equation*}

Σκοπός μας είναι η ταχεία και αποτελεσματική λύση ενός πραγματικού και μη συμμετρικού συστήματος {\en Toeplitz} της μορφής:
\begin{equation}\label{eq:Toeplitz system}
T_n(f)x=b,
\end{equation}
όπου $f=f_1+\mathrm{i}f_2$, με $f_1$ μια άρτια, $2\pi$-περιοδική συνάρτηση και $f_2$ περιττή κι επίσης $2\pi$-περιοδική. Τόσο η $f_1$, όσο και η $f_2$ ορίζονται στο διάστημα $\left(-\pi,\pi\right]$.

Για την αρτιότητα της διατριβής, δίνουμε τον ορισμό του δείκτη κατάστασης ενός αντιστρέψιμου πίνακα:
\begin{defn}
Δείκτης κατάστασης ενός αντιστρέψιμου πίνακα $A\in\mathbb{C}^{n,n}$, ως προς μία φυσική νόρμα $\Vert\cdot\Vert$, καλείται ο αριθμός $\kappa\left(A\right)=\Vert A\Vert\Vert A^{-1}\Vert$.
\end{defn}
`Οσο πιο κοντά στη μονάδα είναι ο δείκτης κατάστασης του πίνακα συντελεστών, τόσο καλύτερη κατάσταση μπορούμε να πούμε ότι έχει το αντίστοιχο σύστημα.

Στη βιβλιογραφία υπάρχουν γνωστοί αλγόριθμοι για την επίλυση συστημάτων {\en Toeplitz}, όπως ο αλγόριθμος του {\en N.~Levinson} \cite{levinson}, του {\en J.~Durbin} \cite{durbin} και του {\en W.~Trench} \cite{trench}. Αυτοί αν και αποτελούν βελτίωση των κλασικών άμεσων μεθόδων, όπως της Απαλοιφής {\en Gauss} (η οποία αν εφαρμοστεί αυτούσια χρειάζεται $\mathcal{O}\left(n^3\right)$ πράξεις), μπορούν να αντιμετωπίσουν ειδικές μορφές συστημάτων {\en Toeplitz} \cite{golub}, απαιτούν $\mathcal{O}\left(n^2\right)$ πράξεις και είναι αποτελεσματικοί για συστήματα που έχουν καλή κατάσταση, χωρίς όμως να γνωρίζουμε κάτι για συστήματα με κακή κατάσταση \cite{bunch2, bunch}. Αναφέρουμε ότι αναπτύχθηκαν και πιο γρήγορες άμεσες μέθοδοι, οι οποίες μας δίνουν τη λύση του συστήματος σε $\mathcal{O}\left(n\log^2{n}\right)$ πράξεις (\cite{chan_ng} και αναφορές εντός αυτής). Ωστόσο, οι επαναληπτικές μέθοδοι όπως η μέθοδος Συζυγών Κλίσεων {\en (CG)} \cite{stiefel, reid}, η Γενικευμένη μέθοδος Ελαχίστων Υπολοίπων {\en (GMRES)} \cite{saad1986gmres} και άλλες μέθοδοι υποχώρων {\en Krylov}, υπερτερούν των άμεσων αφού ο πολλαπλασιασμός πίνακα {\en Toeplitz} επί διάνυσμα μπορεί να γίνει σε $\mathcal{O}(n\log{n})$ πράξεις με χρήση του Ταχέως Μετασχηματισμού {\en Fourier (Fast Fourier Transform)}, γνωστό ως {\en FFT} \cite{Strang}. Η διαφορά (στο σύνολο των πράξεων) ανάμεσα στις άμεσες κι επαναληπτικές μεθόδους είναι πολύ μεγαλύτερη σε {\en block Toeplitz} συστήματα, όπου το κάθε {\en block} είναι πίνακας {\en Toeplitz}. Πιο συγκεκριμένα, με αλγορίθμους τύπου {\en Levinson} χρειάζονται $\mathcal{O}\left(N^2M^3\right)$ πράξεις, αντί $\mathcal{O}\left(NM\log{NM}\right)$ που χρειάζονται οι επαναληπτικές μέθοδοι \cite{NSV_2006}, όπου $M$ είναι η διάσταση του κάθε {\en block} και $N$ είναι το πλήθος των {\en blocks} σε κάθε γραμμή του πίνακα.

Αν και οι επαναληπτικές μέθοδοι χρειάζονται $\mathcal{O}(n\log{n})$ πράξεις σε κάθε επανάληψη, δεν είναι ιδιαίτερα αποτελεσματικές για συστήματα με κακή κατάσταση, με την έννοια ότι χρειάζονται πολλές επαναλήψεις (ίσως όσες και η διάσταση του συστήματος). Αυτό το φαινόμενο μπορεί να εξαλειφθεί με την τεχνική της προρρύθμισης. Ουσιαστικά προσπαθούμε να βρούμε έναν αντιστρέψιμο πίνακα ο οποίος όχι μόνο κατασκευάζεται γρήγορα, αλλά δίνει και τη λύση του αντίστοιχου συστήματος επίσης γρήγορα. `Οσον αφορά στα συστήματα {\en Toeplitz}, θέλουμε/απαιτούμε η κατασκευή και επίλυση του προρρυθμισμένου συστήματος να γίνεται το πολύ σε $\mathcal{O}(n\log{n})$ πράξεις, διότι όπως προαναφέραμε για τον πολλαπλασιασμό πίνακα {\en Toeplitz} επί διάνυσμα απαιτούνται $\mathcal{O}(n\log{n})$ πράξεις. Θα θέλαμε να σχολιάσουμε ότι για να γίνει χρήση του {\en FFT}, θα πρέπει αρχικά να εμβαπτίσουμε τον πίνακα $T_n$, σε μια συγκεκριμένη κυκλοειδή {\en (circulant)} μορφή διπλάσιας διάστασης. Περισσότερες λεπτομέρειες μπορούν να βρεθούν στο κλασικό βιβλίο του {\en M. Ng} \cite{Ng_book}.

\section{Βασική θεωρία}
Η προρρύθμιση συμμετρικών και θετικά ορισμένων συστημάτων {\en Toeplitz} έχει μελετηθεί εκτενώς από πολλούς ερευνητές και είναι πλέον ευρέως γνωστό ότι αυτά μπορούν να επιλυθούν αποτελεσματικά με την Προρρυθμισμένη μέθοδο Συζυγών Κλίσεων {\en (PCG)}. Οι προρρυθμιστές που εισήχθησαν τα τελευταία χρόνια είναι ιδιαιτέρως αποτελεσματικοί και κάποιοι από αυτούς, όπως για παράδειγμα αυτός στην \cite{NV_2008}, οδηγούν στην υπεργραμμική {\en (superlinear)} σύγκλιση της {\en PCG}, παρέχοντας τη λύση του συστήματος σε λίγες επαναλήψεις, ανεξάρτητες της διάστασης $n$.

Ειδικότερα, το 1991 ο {\en R.~Chan} \cite{chan1991toeplitz} εισήγαγε έναν ταινιωτό {\en Toeplitz} προρρυθμιστή για συμμετρικά συστήματα {\en Toeplitz} με κακή κατάσταση. Η γεννήτρια συνάρτηση του προτεινόμενου προρρυθμιστή ήταν ένα τριγωνομετρικό πολυώνυμο $g$, το οποίο είχε τις ίδιες ρίζες, καθώς επίσης και την ίδια πολλαπλότητα ριζών, με τη γεννήτρια συνάρτηση (του αρχικού συστήματος) $f$. Αυτή η τεχνική οδηγεί στην άρση της κακής κατάστασης του συστήματος, αφού η $\frac{f}{g}$ λαμβάνει μόνο θετικές τιμές. Το 1994 οι {\en R.~Chan} και {\en P.~Tang} \cite{Chan_Tang} εισήγαγαν έναν ταινιωτό {\en Toeplitz} προρρυθμιστή, ο οποίος προέκυψε από το τριγωνομετρικό πολυώνυμο $g$, το οποίο ελαχιστοποιούσε το μέγιστο σχετικό σφάλμα $\Vert\frac{f-g}{f}\Vert_\infty$. Το 1997 ο {\en S.~Serra-Capizzano} \cite{Serra} πρότεινε εναλλακτικούς τρόπους για την ελαχιστοποίηση του $\Vert\frac{f-g}{f}\Vert_\infty$ και εισήγαγε με τη σειρά του έναν ταινιωτό {\en Toeplitz} προρρυθμιστή, η κατασκευή του οποίου συνδυάζει την άρση της κακής κατάστασης (με κάποιο τριγωνομετρικό πολυώνυμο $z_k$) και την προσέγγιση της συνάρτησης $\frac{f}{z_k}$, με παρεμβολή στα σημεία {\en Chebyshev} πρώτου είδους \cite{rivlin1981introduction} ή με βέλτιστη ομοιόμορφη προσέγγιση λαμβάνοντας ως κόμβους τα ίδια σημεία. Το 2002 οι {\en D.~Noutsos} και {\en P.~Vassalos} \cite{NV_2002} πρότειναν έναν προρρυθμιστή, ο οποίος είναι το αποτέλεσμα της άρσης των ρίζών της $f$, με το $z_k$, και της ρητής προσέγγισης της συνάρτησης $\sqrt{\frac{f}{z_k}}$. Οι ίδιοι συγγραφείς το 2008 \cite{NV_2008} πρότειναν έναν προρρυθμιστή, ο οποίος κατασκευάζεται ως το γινόμενο ενός τριγωνομετρικού πολυωνύμου $g$, που αίρει τις ρίζες της μη-αρνητικής γεννήτριας συνάρτησης $f$, και πίνακες που ανήκουν σε τριγωνομετρική άλγεβρα \cite{Ng_book} και αντιστοιχούν στη θετική συνάρτηση $\frac{f}{g}$. Το 2016 οι {\en D.~Noutsos, S.~Serra-Capizzano} και {\en P.~Vassalos} \cite{NSV_2016} πρότειναν έναν προρρυθμιστή ο οποίος ανήκει στην $\tau$-άλγεβρα \cite{Bini_Benedetto,Benedetto_Serra}, για συστήματα {\en Toeplitz} των οποίων η γεννήτρια συνάρτηση έχει ρίζες μη-ακέραιας τάξης. 

Σε πολλές εφαρμογές εμφανίζονται {\en Block Toeplitz} πίνακες με {\en Toeplitz Blocks (BTTB)} \cite{Chan_Jin, Ng_book, Jin}. Αυτοί αποτελούν μια γενίκευση της κλάσης των πινάκων {\en Toeplitz}. Οι {\en BTTB} πίνακες καλούνται δι-διάστατοι {\en (two-level)} πίνακες {\en Toeplitz} και συμβολίζονται ως $T_{nm}(f)$, όπου $m$ είναι η διάσταση του κάθε {\en block} και $n$ είναι το πλήθος των {\en blocks} σε κάθε γραμμή του πίνακα, που σημαίνει ότι $T_{nm}(f)\in\mathbb{R}^{nm\times nm}$. Η γεννήτρια συνάρτηση αυτών είναι η $2\pi$-περιοδική συνάρτηση δύο μεταβλητών $f=f(x,y):[-\pi,\pi]^2\rightarrow\mathbb{C}$. Κάθε στοιχείο $t_{rs}$ χαρακτηρίζεται από τις παραμέτρους $r$ και $s$, όπου $r$ συμβολίζει τη {\en block} διαγώνιο και $s$ τη διαγώνιο εντός των {\en blocks}. `Οπως και στη μονοδιάστατη περίπτωση οι τιμές του πίνακα {\en BTTB} είναι οι συντελεστές του αναπτύγματος {\en Fourier} της $f$: 

\begin{equation*}
t_{rs}=\frac{1}{4\pi^2}\int\limits_{-\pi}^{\pi}\int\limits_{-\pi}^{\pi}f(x,y)\mathrm{e}^{-\mathrm{i}(rx+sy)}\mathrm{d}x\mathrm{d}y.
\end{equation*}

Συμμετρικά και θετικά ορισμένα {\en BTTB} συστήματα, έχουν μελετηθεί επίσης από πολλούς ερευνητές. Αποτελεσματικοί {\en BTTB} προρρυθμιστές προτάθηκαν από τον {\en M.~Ng} \cite{Ng}, καθώς επίσης και από τους {\en D.~Noutsos, S.~Serra-Capizzano} και {\en P.~Vassalos} \cite{NSV_2005, NSV_2006, NSV_2006i}. Σημειώνουμε ότι οι τελευταίοι προρρυθμιστές κατασκευάστηκαν μετά από την προσέγγιση της γεννήτριας συνάρτησης, προσαρμόζοντας την ιδέα της \cite{Serra_1999} που αφορούσε στη μονοδιάστατη περίπτωση. Προρρυθμιστές που ανήκουν σε άλγεβρα πινάκων προτάθηκαν στην \cite{NV_2011} για συστήματα με καλή κατάσταση. Σχολιάζουμε ότι για συστήματα με κακή κατάσταση έχουν αποδειχθεί ορισμένα αρνητικά αποτελέσματα \cite{NSV_2003, NSV_2004, Serra_2002, ST_2000, ST_2003}. Εν ολίγοις, προρρυθμιστές από οποιαδήποτε τριγωνομετρική άλγεβρα πινάκων δεν είναι ικανοί να οδηγήσουν σε υπεργραμμική σύγκλιση, συστήματα με κακή κατάσταση από μόνοι τους, που σημαίνει ότι δεν είναι αποτελεσματικοί χωρίς κάποια τεχνική άρσης της κακής κατάστασης. Από την άλλη, ο συνδυασμός ταινιωτών {\en BTTB} πινάκων με πίνακες οι οποίοι ανήκουν σε κάποια άλγεβρα (πινάκων), είναι ιδιαιτέρως αποτελεσματικός όπως φαίνεται στην \cite{NV_2011}, αφού επιτεύχθηκε η συσσώρευση των ιδιοτιμών γύρω από το 1.

Παραπάνω παρουσιάσαμε συνοπτικά ποικίλες τεχνικές προρρύθμισης από τη βιβλιογραφία, συμπεραίνοντας ότι η συμμετρική περίπτωση συστημάτων {\en Toeplitz} μελετήθηκε εκτενώς. Ωστόσο, η περίπτωση μη-συμμετρικών συστημάτων {\en Toeplitz} χρήζει μελέτης, αφού υπάρχουν πολλά σημεία περαιτέρω ανάλυσης. Ο αναγνώστης μπορεί να βρει διάφορα θεωρητικά εποτελέσματα για τη συσσώρευση των ιδιοτιμών και ιδιαζουσών τιμών στις \cite{Parter, Tilli, Tilli1999, tyrt, Tyrtyshnikov1999}. Σε αυτές δίνονται γενικεύσεις ενός βασικού θεωρήματος για την ισοκατανομή των ιδιοτιμών. Αυτό είναι το θεώρημα ισοκατανομής {\en (equally distribution)} του {\en G.~Szeg{\" o}} \cite{Grenander}. Θα το δώσουμε, αφού αρχικά ορίσουμε το ουσιώδες άνω και κάτω φράγμα μιας συνάρτησης.

\begin{defn}
Ο μέγιστος αριθμός $m$ για τον οποίο ισχύει η ανισότητα $f\left(x\right)\geq m$, με εξαίρεση ένα σύνολο μέτρου {\en Lebesgue} μηδέν, λέγεται ουσιώδες κάτω φράγμα ({\en essential lower bound}) της συνάρτησης $f$.
\end{defn}

\begin{defn}
Ο ελάχιστος αριθμός $M$ για τον οποίο ισχύει η ανισότητα $f\left(x\right)\leq M$, με εξαίρεση ένα σύνολο μέτρου {\en Lebesgue} μηδέν, λέγεται ουσιώδες άνω φράγμα ({\en essential upper bound}) της συνάρτησης $f$.
\end{defn}

\begin{thm}[{\en G.~Szeg{\" o}}]
Ας είναι $f$ μία πραγματική συνάρτηση, ολοκληρώσιμη κατά {\en Lebesgue} και $\lambda_{1},\lambda_{2},\dots,\lambda_{n}$ οι ιδιοτιμές του πίνακα $T_n(f)$. Συμβολίζουμε με $m$ και $M$ το ουσιώδες κάτω και ουσιώδες άνω φράγμα της $f$, αντίστοιχα. Αν $F(\lambda)$ είναι μία οποιαδήποτε συνεχής συνάρτηση ορισμένη στο διάστημα $\left[m,M\right]$ ισχύει:
\begin{center}
$\displaystyle\lim\limits_{n\rightarrow\infty}\frac{1}{n}\sum\limits_{k=1}^nF(\lambda_{k})=\frac{1}{2\pi}\int\limits_{-\pi}^{\pi}F\left(f(x)\right)\mathrm{d}x$.
\end{center}
\end{thm}

Αν στο παραπάνω θεώρημα επιλέξουμε ως $F(x)=x$, προφανώς: $$\displaystyle\lim\limits_{n\rightarrow\infty}\frac{\lambda_{1}+\lambda_{2}+\dots+\lambda_{n}}{n}=\frac{1}{2\pi}\int\limits_{-\pi}^{\pi}f(x)\mathrm{d}x.$$ Αυτό σημαίνει ότι ο μέσος όρος των ιδιοτιμών του $T_n(f)$, συγκλίνει στο στοιχείο $t_{0}=\displaystyle\frac{1}{2\pi}\int\limits_{-\pi}^{\pi}f(x)\mathrm{d}x$. Γενικότερα, η κατάλληλη επιλογή της $F$ είναι αυτή που μπορεί να μας οδηγήσει στο επιθυμητό αποτέλεσμα (βλ. απόδειξη Θεωρήματος \ref{thm:eig_clustering}).

Καταλαβαίνουμε ότι ο τρόπος κατανομής/συγκέντρωσης των ιδιοτιμών παίζει σημαντικό ρόλο για την αποτελεσματικότητα μεθόδων {\en Krylov}. Είναι μάλιστα γνωστό ότι η αποτελεσματικότητα της μεθόδου {\en GMRES} εξαρτάται από τον τρόπο συσσώρευσης των ιδιοτιμών του πίνακα συντελεστών, ενώ αυτή της {\en CGN} από τη συσσώρευση των ιδιαζουσών τιμών \cite{trefethen}. Αναφερόμαστε σε αυτές τις δύο μεθόδους, διότι μπορούν να επιλύσουν ένα μη-συμμετρικό σύστημα, σε αντίθεση με τη {\en PCG}, η οποία ναι μεν επιλύει με αποτελεσματικότητα συμμετρικά και θετικά ορισμένα συστήματα, αλλά δε μπορεί να εφαρμοστεί σε μη-συμμετρικούς πίνακες. Παρακάτω ορίζουμε δύο είδη συσσώρευσης των ιδιοτιμών, όπως δόθηκαν από τον {\en E.~Tyrtyshnikov} στην \cite{tyrt}:

\begin{defn}
`Ενα σύνολο $\Phi\subset\mathbb{C}$ καλείται σύνολο γενικής συσσώρευσης {\en (general cluster)} των ιδιοτιμών μιας ακολουθίας πινάκων $\{A_n\}$, $A_n\in\mathbb{C}^{n\times n}$, αν και μόνο αν για κάθε $\varepsilon>0$:
\begin{equation*}
\lim_{n\rightarrow\infty}\frac{\gamma_n(\varepsilon)}{n}=0,
\end{equation*}
ενώ αν:
\begin{equation*}
\gamma_n(\varepsilon)\leq c(\varepsilon),
\end{equation*}
όπου $c(\varepsilon)$ είναι σταθερά ανεξάρτητη της διάστασης $n$, το $\Phi$ καλείται σύνολο κύριας συσσώρευσης {\en (proper cluster)} των ιδιοτιμών.
\end{defn}

Παραπάνω, με $\gamma_n(\varepsilon)$ συμβολίζουμε τον αριθμό των ιδιοτιμών που κυμαίνονται εκτός του πεδίου $\Phi_\varepsilon$ {\en (outliers)}, δηλαδή εκτός της $\varepsilon$-επέκτασης του $\Phi$ (ένωση του $\Phi$ με όλες τις $\varepsilon$-μπάλες, που έχουν ως κέντρο τα σημεία του $\Phi$).

\begin{rem}
Προσαρμόσαμε τον ορισμό από την \cite{tyrt} για μη-Ερμιτιανούς πίνακες, διότι οι ιδιοτιμές του μη-συμμετρικού πίνακα {\en Toeplitz} είναι μιγαδικές. Υπάρχει ανάλογος ορισμός για τη συσσώρευση των ιδιαζουσών τιμών \cite{Serra1}.
\end{rem}


Δίνουμε ένα θεώρημα, το οποίο θα φανεί χρήσιμο στην απόδειξη της συσσώρευσης των ιδιαζουσών τιμών του προρρυθμισμένου συστήματος, στο επόμενο κεφάλαιο. Αυτό είναι το Θεώρημα 4.5 της \cite{Tilli}, το οποίο δόθηκε από τον {\en P.~Tilli}:
\begin{thm}[{\en P.~Tilli}]\label{thm: Tilli}
`Εστω $f\in\mathrm{L}^2(Q,\mathbb{C}^{h\times k})$ και ας είναι $\lbrace T_n\rbrace$ ένα σύνολο {\en block} πινάκων {\en Toeplitz} με γεννήτρια συνάρτηση $f$. Τότε, το $[\sigma_{\operatorname{min}}(f),\sigma_{\operatorname{max}}(f)]$ αποτελεί διάστημα συσσώρευσης για τις ιδιάζουσες τιμές του $\lbrace T_n\rbrace$.
\end{thm}

Σημειώνουμε ότι $Q=(-\pi,\pi)$ και $f\in\mathrm{L}^2(Q,\mathbb{C}^{h\times k})$, σημαίνει ότι η $f$ είναι μια συνάρτηση πινάκων στο $\mathbb{C}^{h\times k}$ με οποιοδήποτε στοιχείο $f_{ij}\in\mathrm{L}^2(-\pi,\pi)$.

\begin{rem}
Η συσσώρευση στο παραπάνω θεώρημα χαρακτηρίζεται ως γενική, αφού στην \cite{Tilli} περιγράφεται ότι $\operatorname{o}(n)$ ιδιάζουσες τιμές του $T_n$ είναι μικρότερες από $\sigma_{\operatorname{min}}(f)$ και όλες τους μικρότερες από $\sigma_{\operatorname{max}}(f)$.
\end{rem}


Χρήσιμο επίσης είναι και το παρακάτω λήμμα (Λήμμα 2.1 της \cite{Tyrtyshnikov1999}), το οποίο δόθηκε από τους {\en E.~Tyrtyshnikov} και {\en N.~Zamarashkin}:
\begin{lem}[{\en E.~Tyrtyshnikov, N.~Zamarashkin}]\label{lem:tyrt_zamar}
Δοσμένων δύο ακολουθιών πινάκων $\{\mathcal{A}_n\}$ και $\{\mathcal{B}_n\}$, υποθέτουμε ότι $\Vert\mathcal{A}_n-\mathcal{B}_n\Vert_F^2=o(n)$ και $\Vert\mathcal{B}_n\Vert_2\leq M$ ομοιόμορφα ως προς το $n$. Τότε, το $\lbrace z:\vert z\vert\leq M\rbrace$ αποτελεί σύνολο γενικής συσσώρευσης των ιδιοτιμών του $\mathcal{A}_n$.
\end{lem}

\begin{rem}
Αν $\Vert\mathcal{A}_n-\mathcal{B}_n\Vert_F^2=\mathcal{O}(1)$ (και $\Vert\mathcal{B}_n\Vert_2\leq M$ ομοιόμορφα ως προς το $n$), εύκολα βλέπουμε, από την απόδειξη του αντίστοιχου λήμματος της \cite{Tyrtyshnikov1999}, ότι το $\lbrace z:\vert z\vert\leq M\rbrace$ αποτελεί σύνολο κύριας συσσώρευσης των ιδιοτιμών του $\{\mathcal{A}_n\}$.
\end{rem}

Είναι προφανές ότι η συσσώρευση των ιδιοτιμών και ιδιαζουσών τιμών αναφέρεται σε ακολουθία πινάκων. Ωστόσο, για λόγους απλούστευσης, στην πορεία της διατριβής όταν αναφερόμαστε στη συσσώρευση πίνακα, θα εννοούμε συσσώρευση της αντίστοιχης ακολουθίας πινάκων.

\newpage\thispagestyle{empty}\mbox{}\newpage 

\pagestyle{main}

\chapter{Ταινιωτοί Προρρυθμιστές}

Σε αυτό το κεφάλαιο θα μελετήσουμε την προρρύθμιση $n\times n$ πραγματικών και μη-συμμετρικών συστημάτων {\en Toeplitz}, χρησιμοποιώντας ως προρρυθμιστές ταινιωτούς πίνακες {\en Toeplitz}. Η γεννήτρια συνάρτηση που αντιστοιχεί σε αυτά είναι μιγαδική, της μορφής $f=f_1+{\rm i}f_2$, όπου $\mathrm{i}=\sqrt{-1}$. Σημειώνουμε επίσης ότι η $f_1$ είναι $2\pi$-περιοδική και άρτια συνάρτηση ενώ η $f_2$ είναι μεν $2\pi$-περιοδική, αλλά περιττή. Η λύση των προρρυθμισμένων συστημάτων λαμβάνεται με την Προρρυθμισμένη μέθοδο Γενικευμένων Ελαχίστων Υπολοίπων {\en (PGMRES)} \cite{Saad, Saad_Schultz} και την Προρρυθμισμένη μέθοδο Συζυγών Κλίσεων για το σύστημα των Κανονικών Εξισώσεων {\en (PCGN)}. Θα παρουσιάσουμε μια τεχνική προρρύθμισης η οποία συνδυάζει την άρση των ριζών της γεννήτριας συνάρτησης (αν υπάρχουν) και βέλτιστη ομοιόμορφη προσέγγιση ή παρεμβολή με τριγωνομετρικά πολυώνυμα. Ο προρρυθμιστής που προκύπτει είναι ο πίνακας $T_n(p)$, όπου $p=gq$, με $g$ να είναι το τριγωνομετρικό πολυώνυμο το οποίο έχει τις ίδιες ρίζες με τη γεννήτρια συνάρτηση $f$ και $q$ το τριγωνομετρικό πολυώνυμο που προκύπτει κατόπιν προσέγγισης της $\frac{f}{g}$. Με χρήση του προαναφερθέντος προρρυθμιστή επιτυγχάνεται συσσώρευση των ιδιοτιμών και ιδιαζουσών τιμών του προρρυθμισμένου συστήματος σε μια μικρή περιοχή μακριά από το 0 και σε ένα μικρό διάστημα κοντά στο 1, αντίστοιχα. Αυτό σημαίνει ότι μπορούμε να λάβουμε τη λύση του συστήματος με λίγες επαναλήψεις των μεθόδων που αναφέραμε \cite{trefethen}, γεγονός το οποίο επιβεβαιώνεται και στα διάφορα αριθμητικά παραδείγματα που δίνουμε στο τέλος του κεφαλαίου.

\section{Θεωρητικά αποτελέσματα}

Αρχικά θα δώσουμε τα θεωρητικά αποτελέσματα, που αφορούν στη σύγκλιση των μεθόδων {\en PGMRES} και {\en PCGN}, δηλαδή στη συσσώρευση του φάσματος των ιδιοτιμών και ιδιαζουσών τιμών του προρρυθμισμένου συστήματος. Όπως αναφέραμε, στόχος μας είναι να βρούμε έναν κατάλληλο ταινιωτό πίνακα {\en Toeplitz} $T_n(p)$, τον οποίο θα χρησιμοποιήσουμε ως προρρυθμιστή, για την αποτελεσματική επίλυση του συστήματος $T_n(f)x=b$. Για να επιτύχουμε τον στόχο μας, θα πρέπει να επιλέξουμε ένα τριγωνομετρικό πολυώνυμο $p=p_1+\mathrm{i}p_2$ έτσι ώστε $\operatorname{Re}\left(\frac{f}{p}\right)>0$ και το εύρος {\en (range)} της $\left|\frac{f}{p}\right|$, ορισμένο ως:

\begin{equation*}
\operatorname{range}\left(\left|\frac{f}{p}\right|\right)=\left[\min\limits_{-\pi\leq x\leq\pi}{\left|\frac{f(x)}{p(x)}\right|},\max\limits_{-\pi\leq x\leq\pi}{\left|\frac{f(x)}{p(x)}\right|}\right],
\end{equation*}
να είναι ένα θετικό διάστημα μακριά από το 0.

Ο τρόπος επιλογής του τριγωνομετρικού πολυωνύμου $p$, θα περιγραφεί λεπτομερώς στην επόμενη ενότητα, όπου δίνεται επίσης και ο τρόπος κατασκευής του προρρυθμιστή. Στη συνέχεια αυτής της ενότητας, θεωρούμε ότι το $p$ έχει ήδη βρεθεί και θα μελετήσουμε τις ιδιότητες που αφορούν στη συσσώρευση των ιδιοτιμών και ιδιαζουσών τιμών του προρρυθμισμένου πίνακα. `Ετσι, δίνουμε τα ακόλουθα θεωρήματα με τις αποδείξεις τους:

\begin{thm}\label{thm:gen_cluster}
`Εστω $f\in\mathrm{L}^2(-\pi,\pi)$ και $p$ ένα τριγωνομετρικό πολυώνυμο τέτοιο ώστε $\operatorname{Re}\left(\frac{f}{p}\right)>0$ και $0<\alpha=\operatorname*{ess~inf}\limits_{-\pi\leq x\leq\pi}{\left|\frac{f(x)}{p(x)}\right|}\leq\operatorname*{ess~sup}\limits_{-\pi\leq x\leq\pi}{\left|\frac{f(x)}{p(x)}\right|}=\beta<\infty$. Τότε, το διάστημα $\left[\alpha,\beta\right]$ αποτελεί ένα σύνολο γενικής συσσώρευσης για τις ιδιάζουσες τιμές του προρρυθμισμένου πίνακα $T_n^{-1}(p)T_n(f)$.
\end{thm}

\begin{proof}
Είναι ευρέως γνωστό ότι οι ιδιάζουσες τιμές ενός μη-συμμετρικού και πραγματικού πίνακα $A$ είναι οι τετραγωνικές ρίζες των ιδιοτιμών του $A^TA$. Επομένως, θα μελετήσουμε τη συμπεριφορά των ιδιοτιμών του πίνακα που αντιστοιχεί στις κανονικές εξισώσεις του προρρυθμισμένου συστήματος.

\begin{equation*}
\begin{split}
\left(T_n^{-1}(p)T_n(f)\right)^T T_n^{-1}(p)T_n(f)&=T_n^T(f)T_n^{-T}(p)T_n^{-1}(p)T_n(f)\\&=T_n(\bar{f})T_n^{-1}(\bar{p})T_n^{-1}(p)T_n(f).
\end{split}
\end{equation*}
Ο τελευταίος πίνακας είναι όμοιος με τον:

\begin{equation*}
A_n(f,p)=T_n^{-1}(p)T_n(f)T_n(\bar{f})T_n^{-1}(\bar{p}).
\end{equation*}
Για απλούστευση θα μελετήσουμε τη συμπεριφορά των ιδιοτιμών του $A_n(f,p)$. Επίσης, θα κάνουμε χρήση της παρακάτω σχέσης:

\begin{equation}\label{eq:low rank E}
T_n(f)=T_n(p)T_n\left(\frac{f}{p}\right)+E_n,
\end{equation}
όπου $E_n$ είναι ένας πίνακας που έχει βαθμίδα ίση με $d-1$, $d$ είναι το πλάτος ταινίας {\en (bandwidth)} του $T_n(p)$. Εφόσον ο $T_n(p)$ είναι ένας ταινιωτός πίνακας, είναι εύκολο να παρατηρήσουμε ότι οι πίνακες $T_n(f)$ και $T_n(p)T_n\left(\frac{f}{p}\right)$ διαφέρουν μόνο στις $\frac{d-1}{2}$ πρώτες και τελευταίες γραμμές, γεγονός το οποίο αποδεικνύει τη σχέση (\ref{eq:low rank E}). Ομοίως, ισχύει ότι:

\begin{equation*}
T_n(\bar{f})=T_n\left(\frac{\bar{f}}{\bar{p}}\right)T_n(\bar{p})+E_n^T.
\end{equation*}
Τώρα θα μελετήσουμε το φάσμα του $A_n(f,p)$.
\begin{equation}\label{eq:main proof}
\begin{split}
A_n(f,p)&=T_n^{-1}(p)T_n(f)T_n(\bar{f})T_n^{-1}(\bar{p})\\
&=T_n^{-1}(p)\left(T_n(p)T_n\left(\frac{f}{p}\right)+E_n\right)\\
&\phantom{A_n(f,p)}\cdot\left(T_n\left(\frac{\bar{f}}{\bar{p}}\right)T_n(\bar{p})+E_n^T\right)T_n^{-1}(\bar{p})\\
&=T_n\left(\frac{f}{p}\right)T_n\left(\frac{\bar{f}}{\bar{p}}\right)+T_n^{-1}(p)E_nT_n\left(\frac{\bar{f}}{\bar{p}}\right)\\
&\phantom{A_n(f,p)}+T_n\left(\frac{f}{p}\right)E_n^T T_n^{-1}(\bar{p})+T_n^{-1}(p)E_nE_n^T T_n^{-1}\left(\bar{p}\right)\\
&=T_n\left(\frac{f}{p}\right)T_n\left(\frac{\bar{f}}{\bar{p}}\right)+R_n
\end{split}
\end{equation}
όπου $R_n$ είναι πίνακας βαθμίδας ίσης το πολύ με $2d-2$.


Από το Θεώρημα \ref{thm: Tilli} (Θεώρημα 4.5 της \cite{Tilli}), λαμβάνουμε ότι οι ιδιάζουσες τιμές του $T_n\left(\frac{f}{p}\right)$ ή ισοδύναμα οι τετραγωνικές ρίζες των ιδιοτιμών του $T_n\left(\frac{f}{p}\right)T_n\left(\frac{\bar{f}}{\bar{p}}\right)$, συσσωρεύονται στο διάστημα $[\alpha,\beta]$. Η φύση της συσσώρευσης (στο $[\alpha,\beta]$) μέσω του Θεωρήματος \ref{thm: Tilli} είναι ότι $\operatorname{o}(n)$ ιδιάζουσες τιμές του $T_n\left(\frac{f}{p}\right)$ είναι μικρότερες από $\alpha$ και όλες τους μικρότερες από $\beta$. Από την ($\ref{eq:main proof}$) έχουμε ότι:

\begin{equation*}
A_n(f,p)=T_n\left(\frac{f}{p}\right)T_n\left(\frac{\bar{f}}{\bar{p}}\right)+R_n.
\end{equation*}
Επομένως, $\operatorname{o}(n)$ ιδιάζουσες τιμές του $T_n^{-1}(p)T_n(f)$ είναι μικρότερες από $\alpha$ και το πολύ $2d-2$ επιπλέον (ιδιάζουσες τιμές) κυμαίνονται εκτός του $[\alpha,\beta]$, το οποίο σημαίνει ότι ένας σταθερός αριθμός ιδιαζουσών τιμών μπορεί να έχουν τιμή μεγαλύτερη του $\beta$. Η συσσώρευση αυτού του είδους, έχει οριστεί από τον {\en E.~Tyrtyshnikov} ως γενική (συσσώρευση) \cite{tyrt}. Επομένως λαμβάνουμε τη γενική συσσώρευση των ιδιαζουσών τιμών του προρρυθμισμένου πίνακα $T_n^{-1}(p)T_n(f)$ και η απόδειξη ολοκληρώθηκε.
\end{proof}

Πρέπει να αναφέρουμε ότι το Θεώρημα 4.5 στην \cite{Tilli}, έχει δοθεί σε ένα γενικότερο πλαίσιο, για πίνακες {\en block Toeplitz}, όπου η $f$ είναι συνάρτηση πίνακας. Η περίπτωση που μας ενδιαφέρει, σχετίζεται με ένα απλούστερο πλαίσιο, όπου όλες οι υποθέσεις του εν λόγω θεωρήματος επίσης ισχύουν.

\begin{thm}\label{thm:eig_clustering}
`Εστω $f\in\mathrm{L}^1([-\pi,\pi])$ και $p$ ένα τριγωνομετρικό πολυώνυμο τέτοιο ώστε $\operatorname*{ess~inf}\limits_{-\pi\leq x\leq\pi}{\operatorname{Re}\left(\frac{f(x)}{p(x)}\right)}>0$. Τότε, οι ιδιοτιμές του προρρυθμισμένου πίνακα $T_n^{-1}(p)T_n(f)$ βρίσκονται εντός του ορθογωνίου ${\cal R}=\left[\alpha,\beta\right]\times\left[-\gamma,\gamma\right]$, όπου $\alpha=\operatorname*{ess~inf}\limits_{-\pi\leq x\leq\pi}{\operatorname{Re}\left(\frac{f(x)}{p(x)}\right)}>0$, $\beta=\operatorname*{ess~sup}\limits_{-\pi\leq x\leq\pi}{\operatorname{Re}\left(\frac{f(x)}{p(x)}\right)}$ και $\gamma=\operatorname*{ess~sup}\limits_{-\pi\leq x\leq\pi}{\operatorname{Im}\left(\frac{f(x)}{p(x)}\right)}$, εκτός ίσως από το πολύ $2d-2$ ιδιοτιμές, οι οποίες ανήκουν σε μια $\varepsilon$-επέκταση του ${\cal R}$. `Αρα, το ${\cal R}$ αποτελεί σύνολο κύριας συσσώρευσης των ιδιοτιμών, όπου $d$ συμβολίζει το πλάτος ταινίας του $T_n(p)$.
\end{thm}

\begin{proof}
Θα μελετήσουμε το εύρος φάσματος {\en (range)} του προρρυθμισμένου πίνακα $T_n(p)^{-1}T_n(f)$:
\begin{equation}\label{eq:rayleigh_eig}
\begin{split}
A&=\frac{x^H T_n(p)^{-1}T_n(f)x}{x^H x}=\frac{1}{2}\frac{x^H\left(T_n(p)^{-1}T_n(f)+T_n(\bar{f})T_n(\bar{p})^{-1}\right)x}{x^H x}\\
&\phantom{=}+\frac{1}{2}\mathrm{i}\frac{x^H\left(T_n(p)^{-1}T_n(f)-T_n(\bar{f})T_n(\bar{p})^{-1}\right)x}{\mathrm{i}x^H x}\\
& 
=\frac{1}{2}\frac{y^H\left(T_n(f)T_n(\bar{p})+T_n(p)T_n(\bar{f})\right)y}{y^H T_n(p)T_n(\bar{p})y}\\ &+\frac{1}{2}\mathrm{i}\frac{y^H\left(T_n(f)T_n(\bar{p})-T_n(p)T_n(\bar{f})\right)y}{\mathrm{i}y^H T_n(p)T_n(\bar{p})y}\\ 
& =\frac{1}{2}\frac{y^H T_n(f\bar{p}+p\bar{f})y+y^H R_1 y}{y^H T_n(\vert p\vert^2) y-y^H R_2y}+\frac{1}{2}\mathrm{i}\frac{y^H T_n(f\bar{p}-p\bar{f})y+y^H R_3 y}{\mathrm{i}y^H T_n(\vert p\vert^2) y-y^H R_2y},
\end{split}
\end{equation}
όπου $y=T_n(\bar{p})^{-1}x$ και $R_1$, $R_2$, $R_3$ είναι πίνακες χαμηλής βαθμίδας {\en (low-rank)}, ίσης με $d-1$. Αυτοί οι πίνακες έχουν μη-μηδενικά στοιχεία στις $\frac{d-1}{2}$ πρώτες και τελευταίες γραμμές και στήλες, λόγω της ταινιωτής δομής των $T_n(p)$ και $T_n(\bar{p})$.

Είναι προφανές ότι ο πίνακας $T_n(f)T_n(\bar{p})+T_n(p)T_n(\bar{f})$ είναι συμμετρικός, ενώ ο $\frac{1}{\mathrm{i}}\left(T_n(f)T_n(\bar{p})-T_n(p)T_n(\bar{f})\right)$ Ερμιτιανός. Διαγράφοντας τις $\frac{d-1}{2}$ πρώτες και τελευταίες γραμμές και στήλες του $T_n(f)T_n(\bar{p})+T_n(p)T_n(\bar{f})$, καθώς επίσης και του $T_n(f\bar{p}+p\bar{f})$, λαμβάνουμε ότι οι εναπομείναντες πίνακες ταυτίζονται. Το ίδιο ισχύει και για τους $\frac{1}{\mathrm{i}}\left(T_n(f)T_n(\bar{p})-T_n(p)T_n(\bar{f})\right)$ και $\frac{1}{\mathrm{i}}T_n(f\bar{p}-p\bar{f})$, όπως επίσης και για τους $T_n(p)T_n(\bar{p})$ και $T_n(\vert p\vert^2)$.

Χρησιμοποιούμε τη σχέση (\ref{eq:rayleigh_eig}), όπου το $x$ επιλέγεται από τον υπόχωρο $\cal V$ του $\mathbb{R}^n$, με $\operatorname{dim}{\cal V}=n-(d-1)$ και τέτοιο ώστε το $y=T_n(\bar{p})^{-1}x$ να έχει μηδενικά στις πρώτες και τελευταίες $\frac{d-1}{2}$ συνιστώσες. `Εχουμε:
\begin{equation*}
\begin{split}
&\operatorname*{range}\limits_{x\in{\cal V}}\left(\operatorname{Re}\left(T_n(p)^{-1}T_n(f)\right)\right)=\operatorname*{range}\limits_{x\in{\cal V}}\left(\operatorname{Re}\left(T_n\left(\frac{f}{p}\right)\right)\right)\\
&\subset\operatorname*{range}\limits_{x\in\mathbb{R}^{n}}\left(\operatorname{Re}\left(T_n\left(\frac{f}{p}\right)\right)\right)={\cal ER}\left(\operatorname{Re}\left(\frac{f}{p}\right)\right),
\end{split}
\end{equation*}
όπου με ${\cal ER}(h)$ συμβολίζουμε το ουσιώδες εύρος {\en (essential range)} της συνάρτησης $h$.

Αυτό σημαίνει ότι το πολύ $d-1$ ιδιοτιμές του συμμετρικού μέρους του πίνακα $T_n(p)^{-1}T_n(f)$ κυμαίνονται εκτός του ${\cal ER}\left(\operatorname{Re}\left(\frac{f}{p}\right)\right)$, το οποίο είναι το διάστημα:
\begin{equation*}
[\alpha,\beta]=\left[\operatorname*{ess~inf}\limits_{-\pi\leq x\leq\pi}{\operatorname{Re}\left(\frac{f(x)}{p(x)}\right)},\operatorname*{ess~sup}\limits_{-\pi\leq x\leq\pi}{\operatorname{Re}\left(\frac{f(x)}{p(x)}\right)}\right]\text{, όπου }\alpha>0.
\end{equation*}
Ομοίως, για το αντισυμμετρικό μέρος προκύπτει ότι το πολύ $d-1$ ιδιοτιμές κυμαίνονται εκτός του ${\cal ER}\left(\operatorname{Im}\left(\frac{f}{p}\right)\right)$, δηλαδή του διαστήματος:
\begin{equation*}
[-\gamma,\gamma]=\left[-\operatorname*{ess~sup}\limits_{-\pi\leq x\leq\pi}{\operatorname{Im}\left(\frac{f(x)}{p(x)}\right)},\operatorname*{ess~sup}\limits_{-\pi\leq x\leq\pi}{\operatorname{Im}\left(\frac{f(x)}{p(x)}\right)}\right].
\end{equation*}
Συμπερασματικά, το πολύ $2(d-1)$ ιδιοτιμές κυμαίνονται εκτός του ορθογωνίου ${\cal R}=\left[\alpha,\beta\right]\times\left[-\gamma,\gamma\right]$, το οποίο αποδεικνύει την κύρια συσσώρευση.

Το ερώτημα που γεννάται είναι: ````Πόσο μακριά από το σύνορο του $\cal R$, κυμαίνονται οι ιδιοτιμές?'''' Προκειμένου να δώσουμε απάντηση σε αυτό, θα πρέπει να εκτιμήσουμε το πηλίκο {\en Rayleigh} $\frac{x^H T_n(p)^{-1}T_n(f)x}{x^H x}$, θεωρώντας ότι το $x$ είναι ένα ιδιοδιάνυσμα του $T_n(p)^{-1}T_n(f)$, το οποίο αντιστοιχεί στην ιδιοτιμή $\lambda$. Οπότε,
\begin{equation*}
\lambda=\frac{x^H T_n(p)^{-1}T_n(f)x}{x^H x}\overset{y=T_n(\bar{p})^{-1}x}{=}\frac{y^H T_n(f)T_n(\bar{p})y}{y^H T_n(p)T_n(\bar{p})y}.
\end{equation*}

Είναι γνωστό ότι παρόλο που οι $T_n(f)T_n(\bar{p})$ και $T_n(p)T_n(\bar{p})$ δεν είναι πίνακες {\en Toeplitz}, όσο $n\rightarrow\infty$, συμπεριφέρονται ως τελεστές {\en Toeplitz} που γεννιούνται από τις $f\bar{p}$ και $p\bar{p}$, αντίστοιχα. Σε κάθε $x\in[-\pi,\pi]$ αντιστοιχεί μια ιδιοτιμή $f(x)\bar{p}(x)$ ή $p(x)\bar{p}(x)$ με αντίστοιχο άπειρο ιδιοδιάνυσμα $y=\frac{1}{\sqrt{n}}\left(1~\mathrm{e}^{\mathrm{i}x}~\mathrm{e}^{\mathrm{i}2x}~\cdots\right)^T$ \cite{Serra_1999, Tyrtyshnikov_1994, Dai}. Επομένως, για κάποιο αρκετά μεγάλο $n$, το $y$ είναι κοντά στο διάνυσμα $\frac{1}{\sqrt{n}}\left(1~\mathrm{e}^{\mathrm{i}x}~\mathrm{e}^{\mathrm{i}2x}~\cdots~\mathrm{e}^{\mathrm{i}(n-1)x}\right)^T$.

Χωρίζουμε το $y$ ως $y=\left[y_{1}^T\vert y_2^T\vert y_3^T\right]^T$, όπου $y_1, y_3\in\mathbb{R}^{\frac{d-1}{2}}$ και $y_2\in\mathbb{R}^{n-(d-1)}$. Είναι προφανές ότι $\|y_1\|=\mathcal{O}\left(\frac{1}{\sqrt{n}}\right)$ και $\|y_3\|=\mathcal{O}\left(\frac{1}{\sqrt{n}}\right)$. Εν συνεχεία, εφαρμόζουμε τη σχέση (\ref{eq:rayleigh_eig}) για το συγκεκριμένο $y$. Τότε, προφανώς $y^H R_1 y=\mathcal{O}\left(\frac{1}{n}\right)$, $y^H R_2 y=\mathcal{O}\left(\frac{1}{n}\right)$ και $y^H R_3 y=\mathcal{O}\left(\frac{1}{n}\right)$. Αυτό σημαίνει ότι κάθε ιδιοτιμή του $T_n(p)^{-1}T_n(f)$ ανήκει σε μια $\varepsilon-$επέκταση του ορθογωνίου που προαναφέραμε, με $\varepsilon=\mathcal{O}\left(\frac{1}{n}\right)$.
\end{proof}

Το Θεώρημα \ref{thm:gen_cluster} εγγυάται την ταχεία σύγκλιση της μεθόδου {\en PCGN}, αφού μας δίνει τη συσσώρευση των ιδιαζουσών τιμών, ενώ το Θεώρημα \ref{thm:eig_clustering} εγγυάται την ταχεία σύγκλιση της {\en PGMRES}, αφού μέσω του εν λόγω προρρυθμιστή επιτυγχάνεται η συσσώρευση των ιδιοτιμών.

\section{Κατασκευή του προρρυθμιστή}
\label{S:3}
Θα παρουσιάσουμε την κατασκευή δύο διαφορετικών μεταξύ τους προρρυθμιστών. Ο πρώτος θα αντιστοιχεί σε συστήματα με καλή κατάσταση {\en (well-conditioned)}, όπου η γεννήτρια συνάρτηση του πίνακα {\en Toeplitz} δεν έχει ρίζες, ενώ ο δεύτερος θα αφορά σε συστήματα, των οποίων η γεννήτρια συνάρτηση έχει ρίζες και χαρακτηρίζονται ως συστήματα με κακή κατάσταση {\en (ill-conditioned)}. Προφανώς, και στις δύο περιπτώσεις η γεννήτρια συνάρτηση του προρρυθμιστή θα είναι ένα τριγωνομετρικό πολυώνυμο $p=p_1+\mathrm{i}p_2$, με το $p_1$ να είναι άρτιο πολυώνυμο βαθμού $d_1$ και το $p_2$ περιττό, βαθμού $d_2$. Η κύρια διαφορά εντοπίζεται στο γεγονός ότι σε συστήματα με καλή κατάσταση, τα $p_1$ και $p_2$ προκύπτουν αποκλειστικά από την προσέγγιση της $f$, ενώ σε συστήματα με κακή κατάσταση είναι γινόμενα κατάλληλων τριγωνομετρικών πολυωνύμων, για τα οποία θα δώσουμε περισσότερες λεπτομέρειες στην υποενότητα \ref{Ss:32}.

\subsection{Συστήματα με καλή κατάσταση}\label{Ss:31}
Γενικά, η καλή κατάσταση ενός μη-συμμετρικού και πραγματικού συστήματος {\en Toeplitz} χαρακτηρίζεται από το γεγονός ότι η γεννήτρια συνάρτηση $f$ δεν έχει ρίζες. Ωστόσο, για τον προρρυθμιστή που κατασκευάζουμε σε αυτό το κεφάλαιο, θα θέλαμε να ικανοποιείται ένας ισχυρότερος περιορισμός. Πιο συγκεκριμένα, θα θέλαμε η $f_1$ να είναι θετική συνάρτηση και η $|f|$ φραγμένη, για να εξασφαλίσουμε ότι το ουσιώδες εύρος των ιδιοτιμών του προρρυθμισμένου συστήματος, θα είναι ένα συμπαγές υποσύνολο στο δεξιό ημιεπίπεδο, του μιγαδικού επιπέδου.

Ο ταινιωτός προρρυθμιστής, για συστήματα με καλή κατάσταση, θα κατασκευαστεί θεωρώντας κάποιου είδους προσέγγιση των συναρτήσεων που απαρτίζουν την $f$, δηλαδή των $f_1$ και $f_2$. Προτείνουμε δύο τύπους προσέγγισης:
\begin{enumerate}
\item Τη βέλτιστη ομοιόμορφη προσέγγιση της $f_1$ από ένα άρτιο τριγωνομετρικό πολυώνυμο και της $f_2$ από ένα περιττό, στο διάστημα $\left[-\pi,\pi\right]$.
\item Την τριγωνομετρική παρεμβολή των $f_1$ και $f_2$ στο ίδιο διάστημα.
\end{enumerate}

Για τη βέλτιστη ομοιόμορφη προσέγγιση εφαρμόζουμε τον αλγόριθμο εναλλαγής σημείων του {\en Remez}, παίρνοντας τους κόμβους σε ένα πλέγμα {\en (grid)} με $k$ σημεία {\en Chebyshev} πρώτου είδους, απεικονισμένα στο διάστημα $\left[0,\pi\right]$:

\begin{equation*}
x_j=\frac{\pi}{2}\left(\cos{\left(\frac{2(k-j)-1}{2k}\pi\right)}+1\right), ~j=1,2,\dots,k,
\end{equation*}
όπου $k>>d_1,d_2$. Ο λόγος που επιλέγουμε τους κόμβους προσέγγισης στο υποδιάστημα $[0,\pi]$, είναι ότι η $f_1$ είναι άρτια και η $f_2$ περιττή, στο $[-\pi,\pi]$.

Αναφέρουμε ότι ο αλγόριθμος εναλλαγής σημείων του {\en Remez} δεν επιβαρύνει το κόστος των μεθόδων {\en PGMRES} και {\en PCGN}, κατά τάξη μεγέθους, αφού εξαρτάται από τη μεταβλητή $k$, η οποία είναι σταθερή και ανεξάρτητη της διάστασης $n$. 

Το παρακάτω θεώρημα εγγυάται τη συσσώρευση των ιδιαζουσών τιμών του προρρυθμισμένου πίνακα, μετά από ομοιόμορφη προσέγγιση των $f_1$ και $f_2$ με χρήση του αλγορίθμου {\en Remez}.

\begin{thm}\label{thm:Me_cont}
`Εστω $f=f_1+\mathrm{i}f_2$, η γεννήτρια συνάρτηση του πίνακα $T_n(f)$, όπου $f_1$ και $f_2$ συνεχείς, $2\pi-$περιοδικές και φραγμένες συναρτήσεις. `Εστω επιπλέον ότι η $f_1$ είναι θετική και άρτια συνάρτηση, ενώ η $f_2$ περιττή και $p=p_1+\mathrm{i}p_2$, όπου $p_1$ είναι το άρτιο τριγωνομετρικό πολυώνυμο βέλτιστης ομοιόμορφης προσέγγισης της $f_1$, με μέγιστο σφάλμα $\epsilon_1$ και $p_2$ είναι το περιττό τριγωνομετρικό πολυώνυμο βέλτιστης ομοιόμορφης προσέγγισης της $f_2$, με μέγιστο σφάλμα $\epsilon_2$. Τότε, για κάποιο δεδομένο $\epsilon>0$, υπάρχουν $p_1$ και $p_2$, καταλλήλων βαθμών, τέτοια ώστε οι ιδιάζουσες τιμές του προρρυθμισμένου πίνακα $T^{-1}_n(p)T_n(f)$ να έχουν γενική συσσώρευση στο διάστημα $[1-M\epsilon,1+M\epsilon]$, όπου $M=\max\limits_{-\pi\leq x\leq\pi}\left(\frac{1}{p_1^2(x)+p_2^2(x)}\right)^{\frac{1}{2}}$.
\end{thm}

\begin{proof}
Από το Θεώρημα \ref{thm:gen_cluster} έχουμε ότι οι ιδιάζουσες τιμές του προρρυθμισμένου πίνακα $T_n^{-1}(p)T_n(f)$ έχουν γενική συσσώρευση στο εύρος της $\left|\frac{f}{p}\right|$. Η κατασκευή του $p$, μέσω βέλτιστης ομοιόμορφης προσέγγισης μας δίνει:

\begin{equation*}
f_1(x)=p_1(x)+e_1(x)\text{ και }f_2(x)=p_2(x)+e_2(x),
\end{equation*}
όπου $\Vert e_1\Vert_\infty=\epsilon_1$ και $\Vert e_2\Vert_\infty=\epsilon_2$, είναι τα αντίστοιχα σφάλματα της βέλτιστης ομοιόμορφης προσέγγισης. Επομένως, $\forall x\in[-\pi,\pi]$ ισχύει ότι:

\begin{equation*}
\frac{f}{p}=\frac{f_1+\mathrm{i}f_2}{p_1+\mathrm{i}p_2}=\frac{p_1+e_1+\mathrm{i}(p_2+e_2)}{p_1+\mathrm{i}p_2}.
\end{equation*}
Για κάθε $x\in[-\pi,\pi]$, ορίζουμε το διάνυσμα $e=(e_1~e_2)^T$, καθώς επίσης και το διάνυσμα των μεγίστων τιμών $\widehat{\epsilon}=(\epsilon_1~\epsilon_2)^T$. Για να μελετήσουμε τη συμπεριφορά της $\left\vert\frac{f}{p}\right\vert$ θεωρούμε τη συνάρτηση $\left\vert\frac{f}{p}\right\vert^2$:

\begin{equation*}
\left\vert\frac{f}{p}\right\vert^2=\frac{(p_1+e_1)^2+(p_2+e_2)^2}{p_1^2+p_2^2}=1+2\frac{p_1e_1+p_2e_2}{p_1^2+p_2^2}+\frac{e_1^2+e_2^2}{p_1^2+p_2^2}.
\end{equation*}
Για να λάβουμε τα κάτω και άνω φράγματα χρησιμοποιούμε την τριγωνική ανισότητα, καθώς και την ανισότητα {\en Cauchy-Schwarz}, με τον τρόπο που φαίνεται παρακάτω:

\begin{equation}\label{eq:lower}
\begin{split}
&\left\vert\frac{f}{p}\right\vert^2\geq 1-2\left\vert\frac{p_1e_1+p_2e_2}{p_1^2+p_2^2}\right\vert+\frac{e_1^2+e_2^2}{p_1^2+p_2^2}\\
&\phantom{\left\vert\frac{f}{p}\right\vert^2}\geq 1-2\frac{|p_1|\vert e_1\vert+\vert p_2\vert\vert e_2\vert}{p_1^2+p_2^2}+\frac{e_1^2+e_2^2}{p_1^2+p_2^2}\\
&\phantom{\left\vert\frac{f}{p}\right\vert^2}\geq
1-2\frac{\sqrt{p_1^2+p_2^2}}{p_1^2+p_2^2}\sqrt{e_1^2+e_2^2}+\frac{e_1^2+e_2^2}{p_1^2+p_2^2}\\
&\phantom{\left\vert\frac{f}{p}\right\vert^2}=
\left(1-\frac{1}{\sqrt{p_1^2+p_2^2}}\Vert{e}\Vert_2\right)^2.
\end{split}
\end{equation}
Επομένως,

\begin{equation*}
\left\vert\frac{f}{p}\right\vert\geq 1-\frac{1}{\sqrt{p_1^2+p_2^2}}\Vert{e}\Vert_2\geq 1-\max\limits_{-\pi\leq x\leq\pi}\frac{1}{\sqrt{p_1^2+p_2^2}}\Vert\widehat\epsilon\Vert_2=1-M\Vert\widehat\epsilon\Vert_2.
\end{equation*}

Σχετικά με το άνω φράγμα, έχουμε ότι ισχύει η αντίστροφη ανισότητα της (\ref{eq:lower}), βάζοντας $+$ αντί του $-$ ως πρόσημο στον δεύτερο όρο. `Ετσι έχουμε ότι:

\begin{equation*}
\left\vert\frac{f}{p}\right\vert^2\leq\left(1+\frac{1}{\sqrt{p_1^2+p_2^2}}\Vert{e}\Vert_2\right)^2,
\end{equation*}
και ακολούθως,

\begin{equation*}
\left\vert\frac{f}{p}\right\vert\leq 1+\frac{1}{\sqrt{p_1^2+p_2^2}}\Vert{e}\Vert_2\leq 1+\max\limits_{-\pi\leq x\leq\pi}\frac{1}{\sqrt{p_1^2+p_2^2}}\Vert\widehat\epsilon\Vert_2=1+M\Vert\widehat\epsilon\Vert_2.
\end{equation*}

Σημειώνουμε ότι η $p_1$ μπορεί να είναι θετική για μια κατάλληλη επιλογή του βαθμού $d_1$, επειδή η $f_1$ είναι θετική. Αυτό σημαίνει ότι:

\begin{equation*}
M=\max\limits_{-\pi\leq x\leq\pi}\left(\frac{1}{p_1^2(x)+p_2^2(x)}\right)^{\frac{1}{2}}<\infty.
\end{equation*}

Τα σφάλματα της βέλτιστης ομοιόμορφης προσέγγισης, $\epsilon_1$ και $\epsilon_2$, μικραίνουν όσο οι βαθμοί των πολυωνύμων αυξάνονται. Μπορούμε να επιλέξουμε τους βαθμούς $d_1$ και $d_2$, έτσι ώστε $\Vert\widehat\epsilon\Vert_2\leq\epsilon$. Αυτό σημαίνει ότι,
\begin{equation}\label{eq:disk}
1-M\epsilon\leq\left\vert\frac{f}{p}\right\vert\leq 1+M\epsilon,
\end{equation}
το οποίο μας δίνει τη συσσώρευση γύρω από το 1.
\end{proof}

Σχολιάζουμε ότι μέσω της ανισότητας (\ref{eq:disk}), έχουμε ότι η γραφική παράσταση της $\left|\frac{f}{p}\right|$ ανήκει στον δίσκο με κέντρο το σημείο $(1,0)$ και ακτίνα $M\epsilon$. Το Θεώρημα \ref{thm:eig_clustering} περιγράφει με ακριβή τρόπο (ορθογώνιο) τη θέση των ιδιοτιμών. Με ανάλογο χειρισμό της παραπάνω απόδειξης λαμβάνουμε ότι αυτό το ορθογώνιο βρίσκεται εντός του δίσκου με κέντρο το $(1,0)$ και ακτίνα $M\epsilon$.

Στο Θεώρημα \ref{thm:Me_cont} υποθέσαμε ότι η $f$ είναι συνεχής. Σε μια πιο γενική περίπτωση όπου η $f$ είναι συνεχής στο $(-\pi,\pi)$, αλλά παρουσιάζει ασυνέχεια στο σημείο $\pi$ (ειδικότερα στο $\pm\pi$), το οποίο σημαίνει ότι $f_2(\pi)\neq 0$, αλλάζουμε ελαφρώς τη διαδικασία προσέγγισης της $f_2$: Δε χρησιμοποιούμε, για την προσέγγιση, το διάστημα $[0,\pi]$, αλλά ένα υποδιάστημα $[0,c]$. Αφορμή γι αυτήν την αλλαγή στάθηκε το γεγονός ότι οποιοδήποτε περιττό τριγωνομετρικό πολυώνυμο έχει ρίζα στο $\pi$, ενώ η $f_2$ όχι, το οποίο αυξάνει τη τιμή του μέγιστου σφάλματος προσέγγισης κατά πολύ. Χρησιμοποιώντας το διάστημα $[0,c]$, $c<\pi$, επιτυγχάνουμε ένα μικρό σφάλμα προσέγγισης στο $[0,c]$, κι ένα μεγαλύτερο στην περιοχή του $\pi$. `Οπως θα φανεί και από τα αριθμητικά παραδείγματα στην επόμενη ενότητα, αυτό δεν επηρεάζει τη συσσώρευση των ιδιαζουσών τιμών και ιδιοτιμών. Η επιλογή του $c$ γίνεται εμπειρικά. Αναφέρουμε ότι στα αριθμητικά παραδείγματα επιλέξαμε $c=\frac{5\pi}{7}$.

`Οσον αφορά στην παρεμβολή της $f_1$ από κάποιο άρτιο τριγωνομετρικό πολυώνυμο βαθμού $d_1$, χρησιμοποιούμε τα αντίστοιχα $d_1+1$ σημεία {\en Chebyshev} στο $[0,\pi]$. Για την παρεμβολή της $f_2$ από κάποιο περιττό τριγωνομετρικό πολυώνυμο βαθμού $d_2$, χρησιμοποιούμε το σημείο 0 και $d_2$ σημεία {\en Chebyshev} στο $\left[0,c\right]$.

\begin{thm}\label{thm:Me}
`Εστω $f=f_1+\mathrm{i}f_2$, όπου $f_1>0$ είναι συνεχής, $2\pi-$περιοδική συνάρτηση και $f_2\in C((-\pi,\pi))$, $2\pi-$περιοδική με $f_2(\pi)\neq 0$. `Εστω επίσης $p=p_1+\mathrm{i}p_2$, όπου $p_1$ είναι το πολυώνυμο βέλτιστης ομοιόμορφης τριγωνομετρικής προσέγγισης της $f_1$ στο $[-\pi,\pi]$, με μέγιστο σφάλμα προσέγγισης $\epsilon_1$, και $p_2$ είναι το τριγωνομετρικό πολυώνυμο βέλτιστης ομοιόμορφης προσέγγισης της $f_2$ στο $[-c,c]\subset(-\pi,\pi)$, με μέγιστο σφάλμα προσέγγισης $\epsilon_2$, ενώ $\epsilon_2^{\prime}=\max\limits_{-\pi\leq x\leq\pi}\vert f_2(x)-p_2(x)\vert$. Τότε, για κάποιο δεδομένο $\epsilon>0$, υπάρχουν πολυώνυμα $p_1,p_2$, καταλλήλων βαθμών, τέτοια ώστε για αρκετά μεγάλη διάσταση $n$, $L$ ιδιάζουσες τιμές, όπου $L\geq\frac{c}{\pi}n$, να ανήκουν στο διάστημα $I_\epsilon=[1-M\epsilon,1+M\epsilon]$, όπου $M=\max\limits_{-\pi\leq x\leq\pi}\left(\frac{1}{p_1^2(x)+p_2^2(x)}\right)^{\frac{1}{2}}$. Οι υπόλοιπες ιδιάζουσες τιμές κυμαίνονται εκτός του $I_\epsilon$ και παρουσιάζουν γενική συσσώρευση στο $I_\epsilon'=[1-M\epsilon,1+M\epsilon']$, όπου $\epsilon'=\sqrt{\epsilon_1^2+{\epsilon_2^{\prime}}^2}$.
\end{thm}

\begin{proof}
Εύκολα παρατηρούμε ότι όλα τα βήματα της απόδειξης του Θεωρήματος \ref{thm:Me_cont} ισχύουν στο υποδιάστημα $[-c,c]$, αφού ο αλγόριθμος {\en Remez} συμπεριφέρεται καλά (παρουσιάζει μικρό σφάλμα) σε αυτό. `Ετσι, λαμβάνουμε την ανάλογη συσσώρευση  για την $\left\vert\frac{f}{p}\right\vert$:

\begin{equation*}
\left\vert\frac{f}{p}\right\vert\in[1-M\epsilon,1+M\epsilon],~x\in[-c,c].
\end{equation*}

`Εστω $D=\left\{ x\in[-\pi,\pi]:\left\vert\frac{f(x)}{p(x)}\right\vert\in[1-M\epsilon,1+M\epsilon]\right\}$. Προφανώς, $[-c,c]\subset D\subset[-\pi,\pi]$. Εφαρμόζουμε το τύπου-{\en Szeg{\" o}} θεώρημα ισοκατανομής των ιδιαζουσών τιμών \cite{tyrt, Tilli, Grenander, Parter, Avram}. Θεωρούμε τη συνεχή συνάρτηση $0\leq F_h\leq 1$ ως εξής:

\begin{equation*}
\begin{split}
&F_h(z)=1,~z\leq 1-M\epsilon-h,~z\geq 1+M\epsilon+h.\\
&F_h(z)=0,~z\in[1-M\epsilon,1+M\epsilon],
\end{split}
\end{equation*}
όπου $h$ είναι κάποιος μικρός θετικός αριθμός. Τότε,

\begin{align*}
&\limsup\limits_{n\rightarrow\infty}\frac{1}{n}\#\lbrace\sigma_j\leq 1-M\epsilon-h~\vee~\sigma_j\geq 1+M\epsilon+h\rbrace\\
&\leq\lim\limits_{n\rightarrow\infty}\frac{1}{n}\sum\limits_{j=1}^{n}F_h(\sigma_j)=\frac{1}{2\pi}\int\limits_{-\pi}^{\pi}F_h\left(\left\vert\frac{f(x)}{p(x)}\right\vert\right)\mathrm{d}x\\
&=\frac{1}{2\pi}\int\limits_{-\pi}^{-c}F_h\left(\left\vert\frac{f(x)}{p(x)}\right\vert\right)\mathrm{d}x+\frac{1}{2\pi}\int\limits_{c}^{\pi}F_h\left(\left\vert\frac{f(x)}{p(x)}\right\vert\right)\mathrm{d}x\\
&\leq\frac{1}{2\pi}\left(\int\limits_{-\pi}^{-c}1\mathrm{d}x+\int\limits_{c}^{\pi}1\mathrm{d}x\right)=\frac{2(\pi-c)}{2\pi}=\frac{\pi-c}{\pi},
\end{align*}
όπου $\#$ δηλώνει τον πληθικό αριθμό του συνόλου και $\vee$ τη λογική διάζευξη {\en (OR)}.

Η τευλευταία ανισότητα ισχύει διότι $[-c,c]\subset D$. Παίρνοντας το $h>0$ να τείνει προς το 0, προκύπτει ότι:

\begin{equation*}
\limsup\limits_{n\rightarrow\infty}\frac{1}{n}\#\lbrace\sigma_j\leq 1-M\epsilon~\vee~\sigma_j\geq 1+M\epsilon\rbrace\leq\frac{\pi-c}{\pi}.
\end{equation*}
Επομένως,

\begin{equation*}
\limsup\limits_{n\rightarrow\infty}\frac{1}{n}\#\lbrace\sigma_j\in[1-M\epsilon,1+M\epsilon]\rbrace\geq\frac{c}{\pi}.
\end{equation*}
Αυτό σημαίνει ότι για αρκετά μεγάλη διάσταση $n$, $L$ ιδιάζουσες τιμές ανήκουν στο διάστημα $I_{\epsilon}=[1-M\epsilon,1+M\epsilon]$, όπου $L\geq\frac{c}{\pi}n$. Θα δείξουμε ότι οι υπόλοιπες ιδιάζουσες τιμές παρουσιάζουν γενική συσσώρευση στο $I_{\epsilon'}=[1-M\epsilon,1+M\epsilon^{\prime}]$ και το κύριο σώμα αυτών, στο $[1+M\epsilon,1+M\epsilon^{\prime}]$, όπου $\epsilon'=\sqrt{\epsilon_1^2+{\epsilon_2^{\prime}}^2}$.

Παρατηρούμε ότι $p_2(\pi)=0$, ενώ $f_2(\pi)\neq 0$. Λόγω των υποθέσεων συνέχειας, μπορούμε εύκολα να συμπεράνουμε ότι $\vert f_2\vert>\vert p_2\vert$ και $\operatorname{sign}(f_2)=\operatorname{sign}(p_2)$ σε μια περιοχή του $\pi$. Σχολιάζουμε ότι επιλέγουμε την τιμή του $c$ αρκετά κοντά στο σημείο $\pi$, ώστε οι παραπάνω σχέσεις να ισχύουν. Για την ανάλυση, ορίζουμε τη συνάρτηση $s$ ως εξής:

\begin{equation*}
s(z)=\left\{
     \begin{array}{@{}r@{\thinspace}l}
       1 &, z\geq 0\\
       -1 &, z<0 \\
     \end{array}
   \right.
   \end{equation*}
Για κάθε $x\in[-\pi,\pi]$ ισχύει:
\begin{equation}\label{eq:se1}
\begin{array}{c}
-\epsilon_1\leq e_1(x)\leq\epsilon_1,~x\in[-\pi,\pi]\\
-\epsilon_2\leq e_2(x)\leq\epsilon_2\Leftrightarrow -\epsilon_2\leq s(p_2(x))e_2(x)\leq\epsilon_2,~x\in[-c,c].
\end{array}
\end{equation}
Μέσω της παραπάνω υπόθεσης λαμβάνουμε ότι στο διάστημα $[c,\pi]$, η συνάρτηση $e_2$ έχει το ίδιο πρόσημο με την $p_2$. Το ίδιο ισχύει επίσης στο $[-\pi,-c]$. Επομένως,
\begin{equation}\label{eq:se2}
0\leq s(p_2(x))e_2(x)\leq\epsilon_2^{\prime},~\vert x\vert\in[c,\pi].
\end{equation}
Συνδυάζοντας τις σχέσεις (\ref{eq:se1}) και (\ref{eq:se2}) λαμβάνουμε ότι:
\begin{equation}\label{eq:se3}
-\epsilon_2\leq s(p_2(x))e_2(x)\leq\epsilon_2^{\prime},~x\in[-\pi,\pi].
\end{equation}

Παρακάτω εκτιμούμε τα σφάλματα της $\left\vert\frac{f}{p}\right\vert^2$:

\begin{equation*}
\begin{split}
&\left\vert\frac{f}{p}\right\vert^2=\frac{(p_1+e_1)^2+(p_2+e_2)^2}{p_1^2+p_2^2}=\frac{(p_1+e_1)^2+(s(p_2)p_2+s(p_2)e_2)^2}{p_1^2+p_2^2}\\
&\phantom{\left\vert\frac{f}{p}\right\vert^2}=\frac{(p_1+e_1)^2+(\vert p_2\vert+s(p_2)e_2)^2}{p_1^2+p_2^2}.
\end{split}
\end{equation*}
Όπως και στην απόδειξη του Θεωρήματος \ref{thm:Me_cont} ορίζουμε τα διανύσματα $e=(e_1~e_2)^T$, $\widehat{\epsilon}=(\epsilon_1~\epsilon_2)^T$ και $\widehat{\epsilon}^{\prime}=(\epsilon_1~\epsilon_2^{\prime})^T$. Από τις (\ref{eq:se1}),(\ref{eq:se3}) και το γεγονός ότι $p_1>0$, λαμβάνουμε ότι:

\begin{equation*}
\frac{(p_1-\epsilon_1)^2+(\vert p_2\vert-\epsilon_2)^2}{p_1^2+p_2^2}\leq\left\vert\frac{f}{p}\right\vert^2\leq\frac{(p_1+\epsilon_1)^2+(\vert p_2\vert+\epsilon_2^{\prime})^2}{p_1^2+p_2^2}.
\end{equation*}
Το κάτω φράγμα μας δίνει:

\begin{equation*}
\begin{split}
&\frac{(p_1-\epsilon_1)^2+(\vert p_2\vert-\epsilon_2)^2}{p_1^2+p_2^2}=1-2\frac{p_1\epsilon_1+\vert p_2\vert\epsilon_2}{p_1^2+p_2^2}+\frac{\epsilon_1^2+\epsilon_2^2}{p_1^2+p_2^2}\\
&\phantom{\frac{(p_1-\epsilon_1)^2+(\vert p_2\vert-\epsilon_2)^2}{p_1^2+p_2^2}}\geq 1-2\frac{\sqrt{p_1^2+p_2^2}}{p_1^2+p_2^2}\Vert\widehat{\epsilon}\Vert_2+\frac{\Vert\widehat{\epsilon}\Vert_2^2}{p_1^2+p_2^2}\geq(1-M\Vert\widehat{\epsilon}\Vert_2)^2,
\end{split}
\end{equation*}
ενώ το άνω φράγμα:

\begin{equation*}
\frac{(p_1+\epsilon_1)^2+(\vert p_2\vert+\epsilon_2^{\prime})^2}{p_1^2+p_2^2}\leq(1+M\Vert\widehat{\epsilon}^{\prime}\Vert_2)^2.
\end{equation*}
Επομένως,

\begin{equation*}
1-M\epsilon\leq 1-M\Vert\widehat{\epsilon}\Vert_2\leq\left\vert\frac{f}{p}\right\vert\leq 1+M\Vert\widehat{\epsilon}^{\prime}\Vert_2=1+M\epsilon^{\prime},
\end{equation*}
και η απόδειξη ολοκληρώθηκε.
\end{proof}

Στο σημείο αυτό, θα πρέπει να σχολιάσουμε ότι η ουσιαστική διαφορά των δύο παραπάνω διαστημάτων είναι ότι το $\epsilon$ τείνει προς το 0, όσο οι βαθμοί των πολυωνύμων αυξάνονται, ενώ το $\epsilon'$ τείνει προς μια σταθερά μεγαλύτερη του 0. Πρακτικά, το $\epsilon$ μεγαλώνει όσο το $c$ επιλέγεται πιο κοντά στο σημείο $\pi$. Αυτό σημαίνει ότι λαμβάνουμε λιγότερες ιδιάζουσες τιμές εκτός του διαστήματος $[1-M\epsilon, 1+M\epsilon]$, αλλά αυτό γίνεται μεγαλύτερο. Αυτός είναι ο λόγος που επιλέγουμε τη σταθερά $c$ εμπειρικά.

\subsection{Συστήματα με κακή κατάσταση}\label{Ss:32}
Προχωρούμε με τον τρόπο κατασκευής του προρρυθμιστή για συστήματα με κακή κατάσταση, όπου σε αντίθεση με την περίπτωση που εξετάσαμε παραπάνω, η συνάρτηση $f_1$ έχει ρίζες. Εφόσον η $f_2$, ως περιττή συνάρτηση, έχει πάντα ρίζα στο 0, αρχικά θα εξετάσουμε την περίπτωση όπου η $f_1$ έχει επίσης ρίζα στο 0.

\subsubsection{Η $f$ έχει μοναδική ρίζα στο 0}\label{Sss:321}

Παρατηρούμε ότι η $f_1$ δεν αλλάζει πρόσημο στο $[-\pi,\pi]$, αφού είναι άρτια συνάρτηση. Συμβολίζουμε με $m_1$ και $m_2$ την πολλαπλότητα της ρίζας για την $f_1$ και $f_2$, αντίστοιχα. Θα μελετήσουμε την περίπτωση όπου $m_1$ είναι άρτιος ακέραιος, ενώ $m_2$ περιττός. Αρχικά υποθέτουμε ότι $m_1<m_2$, το οποίο σημαίνει ότι η πολλαπλότητα της ρίζας, της συνάρτησης $f$, είναι $m_1$. Τότε, ακολουθώντας την ευρέως γνωστή τεχνική που προτάθηκε από τον {\en R.~Chan} στην \cite{chan1991toeplitz}, κάνουμε άρση της κακής κατάστασης, διαιρώντας με το τριγωνομετρικό πολυώνυμο:

\begin{equation*}
g(x)=\left(2-2\cos{(x)}\right)^{\frac{m_1}{2}}.
\end{equation*}
Επομένως, λαμβάνουμε τη συνάρτηση:

\begin{equation*}
\widehat{f}=\frac{f}{g}=\frac{f_1+\mathrm{i}f_2}{g}=\frac{f_1}{g}+\mathrm{i}\frac{f_2}{g}=\widehat{f_1}+\mathrm{i}\widehat{f_2}.
\end{equation*}
Με αυτόν τον τρόπο, έχουμε ότι $\widehat{f_1}>0$ κι έτσι ερχόμαστε σε αντιστοιχία με την καλή κατάσταση. Αυτό σημαίνει ότι μπορούμε να προσεγγίσουμε την $\widehat{f}$, αντί της $f$, με τον τρόπο που προαναφέραμε. Ας είναι $q=q_1+\mathrm{i}q_2$ το πολυώνυμο με το οποίο προσεγγίζουμε την $\widehat{f}$. Διαιρώντας με αυτό, λαμβάνουμε:

\begin{equation*}
\frac{\widehat{f}}{q}=\frac{\frac{f}{g}}{q}=\frac{\frac{f_1+\mathrm{i}f_2}{g}}{q_1+\mathrm{i}q_2}=\frac{f_1+\mathrm{i}f_2}{gq_1+\mathrm{i}gq_2}=\frac{f}{p}.
\end{equation*}

Στη συνέχεια κατασκευάζουμε τον ταινιωτό πίνακα {\en Toeplitz} $T_n(p)$, τον οποίο χρησιμοποιούμε και ως προρρυθμιστή, όπου $p=gq_1+\mathrm{i}gq_2$. Το εύρος της $\frac{f}{p}$ αποτελεί ένα σύνολο συσσώρευσης γύρω από τη μονάδα και από το Θεώρημα \ref{thm:gen_cluster}, οι ιδιάζουσες τιμές του $T_n^{-1}(p)T_n(f)$ συσσωρεύονται στο εύρος της $\left|\frac{f}{p}\right|$ με την έννοια της γενικής συσσώρευσης.

Θα θέλαμε να σημειώσουμε ότι αν η $f$ είναι συνεχής συνάρτηση, τότε για τον προρρυθμισμένο πίνακα $T_n^{-1}(p)T_n(f)$ ισχύει το Θεώρημα \ref{thm:Me_cont}. Από την άλλη, αν η $f_2$ παρουσιάζει ασυνέχεια στο σημείο $\pi$, ισχύει το Θεώρημα \ref{thm:Me}. Σχολιάζουμε ότι η $\widehat{f}_2$ διατηρεί τη ρίζα στο 0, αλλά αυτό δεν έχει κανένα αρνητικό αντίκτυπο αφού είναι μια περιττή συνάρτηση.

Στην περίπτωση όπου $m_2<m_1$, η πολλαπλότητα της ρίζας της $f$ είναι ίση με $m_2$, δηλαδή εξαρτάται από το φανταστικό μέρος της συνάρτησης. Αν προσπαθήσουμε να προκαλέσουμε άρση της κακής κατάστασης, εφαρμόζοντας την τεχνική που περιγράψαμε παραπάνω, που σημαίνει να διαιρέσουμε με $\left(2-2\cos{(x)}\right)^{\frac{m_1}{2}}$, το φανταστικό μέρος του πηλίκου που προκύπτει θα τείνει προς το άπειρο στο σημείο 0. Από την άλλη, αν προσπαθήσουμε να μειώσουμε την επίδραση της ρίζας της $f_2$, διαιρώντας με $\mathrm{i}\left(\sin{(x)}\right)^{m_2}$, τότε το φανταστικό μέρος της $f$ μετατρέπεται στο πραγματικό μέρος του πηλίκου και με τη σειρά του, το πραγματικό μέρος της $f$ μετατρέπεται σε φανταστικό (του πηλίκου). Το πρόβλημα σε αυτήν την περίπτωση είναι ότι το νέο πραγματικό μέρος τείνει προς το άπειρο όταν $x\rightarrow\pm\pi$.

Για να αποτρέψουμε αυτό το φαινόμενο, όπου το πηλίκο της $f$ προς το τριγωνομετρικό πολυώνυμο απειρίζεται, διαιρούμε με έναν συνδυασμό των δύο συναρτήσεων και πιο συγκεκριμένα με το τριγωνομετρικό πολυώνυμο:

\begin{equation*}
g=g_1+\mathrm{i}g_2=\left(2-2\cos{(x)}\right)^{\frac{m_1}{2}}+\mathrm{i}\left(\sin{(x)}\right)^{m_2}.
\end{equation*}
Επομένως, λαμβάνουμε τη συνάρτηση:

\begin{equation*}
\widehat{f}=\frac{f_1+\mathrm{i}f_2}{g_1+\mathrm{i}g_2}=\frac{f_1g_1+f_2g_2}{g_1^2+g_2^2}+\mathrm{i}\frac{f_2g_1-f_1g_2}{g_1^2+g_2^2}=\widehat{f_1}+\mathrm{i}\widehat{f_2}.
\end{equation*}

Με αυτόν τον τρόπο η $\widehat{f_1}$ είναι θετική και φραγμένη, ενώ η $\widehat{f_2}$ έχει ρίζα στο 0. Γίνεται κατανοητό ότι ερχόμαστε και πάλι σε αντιστοιχία με την καλή κατάσταση. Προσεγγίζοντας την $\widehat{f}$ με τον ίδιο τρόπο, όπως παραπάνω, κατασκευάζουμε το πολυώνυμο $q=q_1+\mathrm{i}q_2$. Διαιρώντας με αυτό έχουμε: 

\begin{equation*}
\frac{\widehat{f}}{q}=\frac{\frac{f}{g}}{q}=\frac{f}{gq}=\frac{f}{p},
\end{equation*}
όπου:
\begin{equation*}
p=gq=(g_1+\mathrm{i}g_2)(q_1+\mathrm{i}q_2)=g_1q_1-g_2q_2+\mathrm{i}(g_1q_2+g_2q_1)=p_1+\mathrm{i}p_2.
\end{equation*}
`Οπως έχουμε ήδη περιγράψει, η συσσώρευση των ιδιαζουσών τιμών επιβεβαιώνει και την αποτελεσματικότητα της μεθόδου {\en PCGN}. Εύκολα μπορούμε να ελέγξουμε ότι τα Θεωρήματα \ref{thm:Me_cont} και \ref{thm:Me} ισχύουν, κάτω από τις αντίστοιχες υποθέσεις. Το Θεώρημα \ref{thm:eig_clustering} εγγυάται την αποτελεσματικότητα της μεθόδου {\en PGMRES}.

Η υπόθεση ότι η πολλαπλότητα $m_1$ είναι ίση με κάποιον άρτιο ακέραιο και $m_2$ με κάποιον περιττό είναι απαραίτητη, διότι υπάρχουν άρτιες συναρτήσεις με ρίζα περιττής πολλαπλότητας στο 0 (π.χ. $f_1=\vert x\vert$), αλλά οι παράγωγοι αυτών να μην ορίζονται στο 0. Ουσιαστικά, θέλουμε να είναι ομαλή συνάρτηση, στην περιοχή της ρίζας. Το ίδιο ισχύει και για την $f_2$ (π.χ. $f_2=\vert x\vert x$). Αν δεν ισχύει αυτή η υπόθεση, η αποτελεσματικότητα του προρρυθμιστή δεν είναι πάντα σίγουρη.
\subsubsection{Η $f$ έχει ρίζα σε σημείο διαφορετικό του 0}\label{Sss:322}

`Εστω ότι η $f$ έχει μια ρίζα στο σημείο $x_0\in[-\pi,\pi)$, $x_0\neq 0$. `Αρα, $x_0$ είναι μια ρίζα των $f_1$ και $f_2$ με πολλαπλότητες $m_1$ και $m_2$, αντίστοιχα. Σε αυτή την περίπτωση, δεν είναι απαραίτητο ότι η $m_1$ είναι κάποιος άρτιος αριθμός και η $m_2$ κάποιος περιττός. Εφόσον η $f_1$ είναι άρτια συνάρτηση και η $f_2$ περιττή, το $-x_0$ είναι επίσης ένα σημείο ρίζας για τις συναρτήσεις $f_1$ και $f_2$ με τις πολλαπλότητες της ρίζας στο $x_0$. Η συνάρτηση $f_2$ έχει μία επιπλέον ρίζα στο 0, με περιττή πολλαπλότητα. Αφού η $f_1$ έχει ρίζες στο $\pm x_0$, με πολλαπλότητα $m_1$, σε μικρή περιοχή του $\pm x_0$, $I_\epsilon=[-\epsilon-x_0,-x_0+\epsilon]\cup[x_0-\epsilon,x_0+\epsilon]$, η συνάρτηση $f_1$ θα έχει τη μορφή:

\begin{equation}\label{eq:f-c_1}
f_1(x)=c_1(x)(x-x_0)^{m_1}(x+x_0)^{m_1}+\operatorname{o}\left((x-x_0)^{m_1}(x+x_0)^{m_1}\right),
\end{equation}
όπου $c_1(x)$ είναι φραγμένη συνάρτηση μακριά από το 0, η οποία διατηρεί πρόσημο στο $I_\epsilon$. Στο σημείο αυτό, πρέπει να σημειώσουμε ότι η συνάρτηση $f_1$ θα πρέπει να είναι αρκετά ομαλή στα σημεία ριζών. Ειδικότερα, θα πρέπει να είναι ομαλή τάξεως $m_1$. Για παράδειγμα, η άρτια συνάρτηση $\vert x-1\vert\vert x+1\vert$ έχει ρίζες πολλαπλότητας ίσης με 1 στο σημείο $\pm 1$, αλλά δεν είναι παραγωγίσιμη σε αυτό. Αυτή, δε μπορεί να γραφεί στη μορφή (\ref{eq:f-c_1}) και οι ρίζες της δεν αίρονται. Ωστόσο, όταν η $f_1$ είναι ομαλή τάξεως $m_1$, οι ρίζες αυτής αίρονται, διαιρώντας με το τριγωνομετρικό πολυώνυμο:

\begin{equation*}
\begin{split}
&\operatorname{sign}(c_1(x))\left(\sin{\frac{x-x_0}{2}}\right)^{m_1}\left(\sin{\frac{x+x_0}{2}}\right)^{m_1}\\
&=\operatorname{sign}(c_1(x))\frac{1}{2^{\frac{m_1}{2}}}\left(\cos{(x_0)}-\cos{(x)}\right)^{m_1}.
\end{split}
\end{equation*}
Ο συντελεστής $\frac{1}{2^{\frac{m_1}{2}}}$ δεν παίζει κάποιον ρόλο κι επομένως μπορούμε να χρησιμοποιήσουμε την απλούστερη μορφή:

\begin{equation}\label{eq:trig pol}
\operatorname{sign}(c_1(x))\left(\cos{(x_0)}-\cos{(x)}\right)^{m_1}.
\end{equation}

Σχολιάζουμε ότι δε θα μπορούσαμε να χρησιμοποιήσουμε το τριγωνομετρικό πολυώνυμο:

\begin{equation*}
\operatorname{sign}(c_1(x))\left(\sin{(x-x_0)}\right)^{m_1}\left(\sin{(x+x_0)}\right)^{m_1},
\end{equation*}
επειδή αυτό έχει δύο επιπλέον ρίζες $|x_0|-\pi$ και $-|x_0|+\pi$ στο $(-\pi,\pi)$.

Παρόμοια ανάλυση ισχύει και για τη συνάρτηση $f_2$, με τη μικρή διαφορά ότι αυτή έχει και μία επιπλέον ρίζα στο 0, με πολλαπλότητα $m_0$. Συμπεραίνουμε ότι το κατάλληλο τριγωνομετρικό πολυώνυμο για την άρση των ριζών αυτής, είναι το:

\begin{equation}\label{eq:trig pol 2}
q_2(x)=\operatorname{sign}(c_2(x))\left(\cos{(x_0)}-\cos{(x)}\right)^{m_2}\left(\sin{(x)}\right)^{m_0},
\end{equation}
όπου $c_2(x)$ παίζει τον ρόλο της $c_1(x)$ στη ($\ref{eq:f-c_1}$), αλλά αυτή τη φορά στο σύνολο:

\begin{equation*}
I_\epsilon^{'}=[-\epsilon-x_0,-x_0+\epsilon]\cup[-\epsilon,\epsilon]\cup[x_0-\epsilon,x_0+\epsilon].
\end{equation*}
Η ομαλότητα της $f_2$ απαιτείται όπως περιγράφηκε παραπάνω και για την $f_1$.

Στην περίπτωση που $m_1\leq m_2$, χρειάζεται να άρουμε τις ρίζες της $f_1$. `Αρα το τριγωνομετρικό πολυώνυμο $g$, δίνεται από τη σχέση (\ref{eq:trig pol}). Αν $m_1<m_2$, οι ρίζες της $f_2$ στο $\pm x_0$ παραμένουν, αλλά ως ρίζες με μικρότερη πολλαπλότητα. Επομένως, σε αυτή την ειδική περίπτωση, η ομαλότητα της $f_2$ δεν απαιτείται. Για παράδειγμα, έστω ότι $f(x)=(x-1)(x+1)+\mathrm{i}x\vert x-1\vert(x-1)\vert x+1\vert(x+1)$. Παρόλο που η $f_2$ δεν έχει την απαιτούμενη ομαλότητα, μπορούμε να κατασκευάσουμε τον προρρυθμιστή. Το ίδιο ισχύει όταν η πολλαπλότητα $m_2$ δεν είναι ίση με κάποιον ακέραιο αριθμό. Στην περίπτωση όπου $m_2<m_1$, θα πρέπει να χρησιμοποιήσουμε έναν συνδυασμό των (\ref{eq:trig pol}) και (\ref{eq:trig pol 2}). Επομένως, έχουμε:

\begin{equation}\label{eq: trig pol 3}
\begin{split}
g&=g_1+\mathrm{i}g_2=\operatorname{sign}(c_1(x))\left(\cos{(x_0)}-\cos{(x)}\right)^{m_1}\\
&\phantom{=}+\mathrm{i}\operatorname{sign}(c_2(x))\left(\cos{(x_0)}-\cos{(x)}\right)^{m_2}\left(\sin{(x)}\right)^{m_0}.
\end{split}
\end{equation}
Στη συνέχεια ακολουθούμε την ίδια τεχνική για την προσέγγιση της $\widehat{f}=\frac{f}{g}$, με τριγωνομετρικά πολυώνυμα, καταλήγοντας στο ότι ο ταινιωτός {\en Toeplitz} προρρυθμιστής θα έχει ως γεννήτρια συνάρτηση την $p=gq$ και τα θεωρητικά αποτελέσματα τα οποία παρουσιάσαμε παραπάνω ισχύουν για τον προρρυθμισμένο πίνακα $T_n^{-1}(p)T_n(f)$.

\subsubsection{Ρίζες των $f_1$ και $f_2$ σε διαφορετικά σημεία}\label{Sss:323}

Υποθέτουμε ότι η $f_1$ έχει ρίζες στο $\pm x_1\in[-\pi,\pi)$ τάξεως $m_1$ και η $f_2$ έχει ρίζες στο $\pm x_2\in[-\pi,\pi)$ τάξεως $m_2$ και μία επιπλέον ρίζα στο 0, τάξεως $m_0$. Προφανώς, η συνάρτηση $f$ δεν έχει ρίζες, αφού οι $f_1$ και $f_2$ δε μηδενίζονται στα ίδια σημεία, αλλά λαμβάνει τιμές τόσο στο θετικό, όσο και στο αρνητικό ημιεπίπεδο του μιγαδικού επιπέδου. Προκειμένου να άρουμε τις ρίζες της $f_1$ και να κατασκευάσουμε έναν αποτελεσματικό προρρυθμιστή, διαιρούμε με μια πιο συγκεκριμένη μορφή του τριγωνομετρικού πολυωνύμου (\ref{eq: trig pol 3}), η οποία δίνεται παρακάτω:

\begin{equation}\label{eq: trig pol 4}
\begin{split}
g&=g_1+\mathrm{i}g_2=\operatorname{sign}(c_1(x))\left(\cos{(x_1)}-\cos{(x)}\right)^{m_1}\\
&\phantom{=}+\mathrm{i}\operatorname{sign}(c_2(x))\left(\cos{(x_2)}-\cos{(x)}\right)^{m_2}\left(\sin{(x)}\right)^{m_0}.
\end{split}
\end{equation}
Διαιρώντας με το τριγωνομετρικό πολυώνυμο της (\ref{eq: trig pol 4}), αίρουμε τις ρίζες της $f_1$ κι έτσι το εύρος της $\widehat{f_1}$ ανήκει σε ένα φραγμένο σύνολο του θετικού ημιεπιπέδου. Σημειώνουμε ότι αν προσπαθήσουμε να άρουμε τις ρίζες της $f_1$, διαιρώντας με κάποιο τριγωνομετρικό πολυώνυμο, ανάλογο της (\ref{eq:trig pol}), όπως:
\begin{equation*}
\operatorname{sign}(c_1(x))\left(\cos{(x_1)}-\cos{(x)}\right)^{m_1},
\end{equation*}
το φανταστικό μέρος θα τείνει προς το άπειρο στο $\pm x_1$.

Σχολιάζουμε ότι αν η $f_1$ έχει ρίζες στο $\pm x_1$ και η $f_2$ έχει ρίζα μόνο στο σημείο 0 $(x_2=0)$, επιλέγουμε ως $g$ το τριγωνομετρικό πολυώνυμο: $$g=\operatorname{sign}(c_1(x))\left(\cos{(x_1)}-\cos{(x)}\right)^{m_1}+\mathrm{i}\operatorname{sign}(c_2(x))\sin{(x)}^{m_0}.$$

\subsubsection{Ρίζες της $f$ σε πολλά σημεία}

Υποθέτουμε ότι η $f_1$ έχει ρίζες στα μη-μηδενικά σημεία $\pm x_1,\pm x_2, \dots,\pm x_k$, τάξεως $m_1,m_2,\dots,m_k$, αντίστοιχα και η $f_2$ έχει επίσης ρίζες στα ίδια σημεία με πολλαπλότητες $\ell_1,\ell_2,\dots,\ell_k$ και μία επιπλέον ρίζα στο 0, με πολλαπλότητα $\ell_0$.

Αν $m_i\leq\ell_i$, $\forall i=1,2,\dots,k$, μπορούμε να επιλέξουμε το τριγωνομετρικό πολυώνυμο το οποίο αίρει όλες τις ρίζες της $f_1$. Αυτό θα είναι ένα γινόμενο τριγωνομετρικών πολυωνύμων τα οποία δίνονται από τη σχέση (\ref{eq:trig pol}) ως:

\begin{equation}\label{eq:product_real_part}
g=\operatorname{sign}(c_1(x))\prod\limits_{i=1}^{k}\left(\cos{(x_i)}-\cos{(x)}\right)^{m_i}.
\end{equation}
Αν η $f_1$ έχει κι αυτή ρίζα στο 0, με πολλαπλότητα $m_0\leq\ell_0$, το τριγωνομετρικό πολυώνυμο $g$ λαμβάνει τη μορφή:

\begin{equation*}
g=\operatorname{sign}(c_1(x))(2-2\cos{(x)})^{\frac{m_0}{2}}\prod\limits_{i=1}^{k}\left(\cos{(x_i)}-\cos{(x)}\right)^{m_i}.
\end{equation*}
Απαιτούμε την ανισότητα $m_0\leq\ell_0$, διότι διαφορετικά το φανταστικό μέρος της $\widehat{f}$ θα έτεινε προς το άπειρο στο σημείο 0. Η περίπτωση όπου $m_0>\ell_0$ μπορεί να καλυφθεί με έναν διαφορετικό τρόπο, ο οποίος περιγράφεται στη συνέχεια.

Ορίζουμε τα σύνολα δεικτών $Q=\lbrace i: m_i\leq\ell_i\rbrace$ και $R=\lbrace i: m_i>\ell_i\rbrace$. Προφανώς, αν $R=\varnothing$, τότε $\#Q=k$ και αυτή η περίπτωση καλύφθηκε παραπάνω. Θα περιγράψουμε την περίπτωση όπου $R\neq\varnothing$. Σε αυτή δε μπορούμε να χρησιμοποιήσουμε το τριγωνομετρικό πολυώνυμο, το οποίο δίνεται από τη σχέση (\ref{eq:product_real_part}), επειδή στα σημεία των ριζών $\pm x_i:i\in R$, θα προκαλέσουμε το φανταστικό μέρος να τείνει προς το άπειρο. Επομένως, θα πρέπει να χρησιμοποιήσουμε ένα συνδυασμό τριγωνομετρικών πολυωνύμων, έτσι ώστε να άρουμε τις ρίζες των $f_1$ και $f_2$.

Είναι προφανές ότι η $f$ έχει ρίζες στα σημεία $\pm x_i$, με πολλαπλότητα $\min{\lbrace m_i,\ell_i\rbrace}$. Επομένως είναι απαραίτητο, το $g$ να περιέχει τη συνάρτηση: $$\widehat{g}=\prod\limits_{i=1}^k\left(\cos{(x_i)}-\cos{(x)}\right)^{\min{\lbrace m_i,\ell_i\rbrace}},$$ ως όρο του γινομένου. Ο υπόλοιπος όρος θα πρέπει να άρει τις ρίζες που απομένουν. Οι ρίζες της $f_1$ που δεν έχουν αρθεί, είναι σε διαφορετικά σημεία, σε σχέση με αυτές της $f_2$, αφού έχουμε μειώσει την πολλαπλότητα των ριζών (της $f$) κατά $\min{\lbrace m_i,\ell_i\rbrace}$. Οι εναπομείνασες ρίζες της $f_1$ αντιστοιχούν σε δείκτες $i\in R$, ενώ αυτές της $f_2$ αντιστοιχούν σε δείκτες $i\in Q$ κι έτσι οδηγούμαστε σε μια γενίκευση της περίπτωσης ````Ρίζες των $f_1$ και $f_2$ σε διαφορετικά σημεία'''', της υποενότητας \ref{Sss:323}. Επομένως, λαμβάνοντας υπόψιν τη σχέση (\ref{eq: trig pol 4}), αυτός ο όρος θα πρέπει να είναι:
\begin{equation*}
\begin{split}
\widetilde{g}&=\operatorname{sign}(c_1(x))\prod\limits_{i\in R}\left(\cos{(x_i)}-\cos{(x)}\right)^{m_i-\ell_i}\\
&+\mathrm{i}\operatorname{sign}(c_2(x))\left(\sin{(x)}\right)^{\ell_0}\prod\limits_{i\in Q}\left(\cos{(x_i)}-\cos{(x)}\right)^{\ell_i-m_i}.
\end{split}
\end{equation*}

Το γινόμενο των $\widehat{g}$ και $\widetilde{g}$ μας δίνει τη συνάρτηση $g$, η οποία κατόπιν ορισμένων πράξεων γράφεται ως:
\begin{equation}\label{eq:2multiple_roots}
\begin{split}
g&=\operatorname{sign}(c_1(x))\prod\limits_{i=1}^{k}\left(\cos{(x_i)}-\cos{(x)}\right)^{m_i}\\
&+\mathrm{i}\operatorname{sign}(c_2(x))\left(\sin{(x)}\right)^{\ell_0}\prod\limits_{i=1}^{k}\left(\cos{(x_i)}-\cos{(x)}\right)^{\ell_i}.
\end{split}
\end{equation}

Παραπάνω δόθηκε το κατάλληλο τριγωνομετρικό πολυώνυμο για την περίπτωση όπου η συνάρτηση $f$ έχει ρίζες σε πολλά σημεία. Αυτή η επιλογή του $g$ (δηλαδή η $($\ref{eq:2multiple_roots}$)$) καλύπτει επίσης την περίπτωση όπου η $f$ έχει ρίζες σε πολλά σημεία, όπως παραπάνω, ενώ επίσης μπορούν να υπάρχουν σημεία όπου η $f_1$ μηδενίζεται, ενώ η $f_2$ δεν έχει ρίζες και το αντίστροφο. `Εστω $x_{j}\neq 0$ ένα τέτοιο σημείο, με την $f_1$ να έχει πολλαπλότητα ρίζας $m_j$ και $f_2$ να μην έχει ρίζα. Υποθέτουμε ότι $\ell_j=0$. Αναλόγως, αν $x_{j}$ είναι ένα σημείο, όπου η $f_2$ έχει ρίζα πολλαπλότητας $\ell_j$ και η $f_1$ δεν έχει ρίζα, υποθέτουμε ότι $m_j=0$. `Οπως προαναφέρθηκε, αυτή η περίπτωση καλύπτεται από τη σχέση (\ref{eq:2multiple_roots}), όπου κάποιες τιμές των $m_i$ και $\ell_i$ μπορούν να είναι ίσες με 0.

Θα πρέπει να σημειώσουμε ότι αν η $f_1$ έχει ρίζα στο 0, με πολλαπλότητα $m_0$, πολλαπλασιάζουμε τον πρώτο όρο της (\ref{eq:2multiple_roots}) με $(2-2\cos{(x)})^{\frac{m_0}{2}}$, που αντιστοιχεί στην άρση αυτής. Ως εκ τούτου, θα επιτευχθεί η άρση των ριζών και ακολουθούμε την τεχνική προσέγγισης της συνάρτησης $\widehat{f}=\frac{f}{g}$.

\subsection{Δι-διάστατη περίπτωση}
Η προτεινόμενη τεχνική μπορεί να επεκταθεί, κάτω από κατάλληλους μετασχηματισμούς, για την επίλυση μη-συμμετρικών και μη-θετικά ορισμένων δι-διάστατων συστημάτων {\en Toeplitz}. Ωστόσο, για την ανάλυση αυτής της περίπτωσης έχουμε να αντιμετωπίσουμε κάποιες επιπλέον δυσκολίες. Η πρώτη εξ αυτών εντοπίζεται στον τρόπο προσέγγισης της συνάρτησης δύο μεταβλητών, αφού δε μπορούμε να χρησιμοποιήσουμε τη βέλτιστη ομοιόμορφη προσέγγιση. Θα μπορούσαμε να ξεπεράσουμε αυτή τη δυσκολία με χρήση παρεμβολής. Μια ακόμα δυσκολία είναι ότι η $f$ ($f_1$ και/ή $f_2$) μπορεί να έχει καμπύλες ριζών, αντί για απλά σημεία. Αυτή η περίπτωση μπορεί να καλυφθεί με χρήση ανάλογης ανάλυσης, όπως στην εργασία \cite{NSV_2005}. Ωστόσο, στην απλούστερη περίπτωση όπου η $f$ είναι χωριζομένων μεταβλητών, η παραπάνω ανάλυση μπορεί να γενικευτεί, αφού ο {\en BTTB} πίνακας παράγεται ως το άθροισμα τανυστικών γινομένων των μονοδιάστατων πινάκων {\en Toeplitz}. Μετά τον διαχωρισμό των μεταβλητών, μπορούμε να χρησιμοποιήσουμε τον αλγόριθμο {\en Remez} για την κάθε συνάρτηση (μίας μεταβλητής) και να προχωρήσουμε στην κατασκευή του προρρυθμιστή με τα αντίστοιχα τανυστικά γινόμενα. Η ισχύς των παραπάνω ισχυρισμών επιβεβαιώνεται στο Παράδειγμα \ref{exp:6}.

\section{Αριθμητικά αποτελέσματα}
\label{S:4}
Σε αυτή την ενότητα παρουσιάζουμε μια πληθώρα αριθμητικών αποτελεσμάτων, προκειμένου να φανεί η αποτελεσματικότητα της προτεινόμενης τεχνικής προρρύθμισης. Τα αποτελέσματα λήφθηκαν μέσω του {\en\textsc{Matlab}}. Σε όλα τα παραδείγματα το διάνυσμα στήλη του δεύτερου μέλους, $b$, του συστήματος $T_n(f)x=b$, επιλέχθηκε έτσι ώστε η λύση του συστήματος να είναι το διάνυσμα που έχει όλες τις συνιστώσες του ίσες με μονάδα, $(1~1 \cdots 1)^T$. Ως αρχική υπόθεση θεωρήσαμε το μηδενικό διάνυσμα κι επιλέξαμε το κριτήριο τερματισμού: $\frac{\Vert r^{(k)}\Vert_2}{\Vert r^{(0)}\Vert_2}\leq 10^{-6}$, όπου $r^{(k)}=b-Ax^{(k)}$ είναι το διάνυσμα υπόλοιπο της $k$-οστής επανάληψης και $r^{(0)}=b$.

Στους πίνακες επαναλήψεων χρησιμοποιούμε τον ακόλουθο συμβολισμό: Με $I_n$ δηλώνουμε ότι δε χρησιμοποιήθηκε καμία τεχνική προρρύθμισης, το $B$ δηλώνει ότι ως προρρυθμιστής χρησιμοποιήθηκε ο ταινιωτός πίνακας {\en Toeplitz} $T_n(g)$, ενώ το $R_{d_1,d_2}$ δηλώνει τον προρρυθμιστή $T_n(p)$, ο οποίος προέκυψε μετά από βέλτιστη ομοιόμοργη προσέγγιση της $\frac{f}{g}$. Σημειώνουμε ότι $p=gq$, με $q=q_1+\mathrm{i}q_2$ και $d_1$ είναι ο βαθμός του $q_1$, ενώ $d_2$ ο βαθμός του $q_2$. Ακολουθώντας τον ίδιο τρόπο συμβολισμού, με $In_{d_1,d_2}$ δηλώνουμε ότι ο προρρυθμιστής προέκυψε από παρεμβολή στις $\widehat{f}_1$ και $\widehat{f}_2$, με τριγωνομετρικά πολυώνυμα βαθμού $d_1$ και $d_2$, αντίστοιχα. Δίνουμε επίσης και τον αριθμό των ιδιαζουσών τιμών που κυμαίνονται εκτός του διαστήματος $[1-M\epsilon,1+M\epsilon]$ (βλ. Θεώρημα \ref{thm:Me_cont} και \ref{thm:Me}), ως ${\rm SV-out}$.

\begin{exmp}\label{exmp:231}\normalfont
`Εστω $\mathfrak{f}_1(x)=x^2+1+\mathrm{i}\mathfrak{h}_1(x)$, όπου: $$\mathfrak{h}_1(x)=\left\{
     \begin{array}{@{}c@{\thinspace}l}
       -\pi-x &,~-\pi\leq x< -\frac{\pi}{2}\\
       x &,~ -\frac{\pi}{2}\leq x<\frac{\pi}{2} \\
       \pi-x &,~\phantom{-}\frac{\pi}{2}\leq x\leq\pi\\

     \end{array}
   \right..$$
Είναι προφανές ότι η $\mathfrak{h}_1$ είναι μια $2\pi$-περιοδική και συνεχής συνάρτηση. Παρατηρούμε επίσης ότι η $f_1=x^2+1$ είναι μια θετική συνάρτηση. Ο αριθμός των επαναλήψεων που χρειάζονται μέχρι τη σύγκλιση των μεθόδων {\en PGMRES} και {\en PCGN}, δίνεται στον Πίνακα \ref{tab:x^2+1+ig}.

\begin{table}[H]
\centering
\begin{tabular}{ccccc|cccc}
\toprule
\multirow{2}{*}{$n$} & \multicolumn{4}{c|} {\en PGMRES} & \multicolumn{4}{c} {\en PCGN}\\
 & $I_n$ & $R_{4,4}$ & $R_{6,6}$ & $R_{8,6}$ & $I_n$ & $R_{4,4}$ & $R_{6,6}$ & $R_{8,6}$\\\midrule
\phantom{0}256 & 31 & 8 & 7 & 6 & 72 & 37 & 34 & 38\\
\phantom{0}512 & 30 & 8 & 7 & 6 & 74 & 31 & 30 & 31\\
1024 & 29 & 8 & 7 & 6 & 73 & 30 & 30 & 29\\
2048 & 29 & 8 & 6 & 6 & 72 & 30 & 29 & 30\\\bottomrule
\end{tabular}
\caption{\label{tab:x^2+1+ig} Επαναλήψεις ($\mathfrak{f}_1$).}
\end{table}
Παρατηρούμε ότι η {\en PGMRES} συγκλίνει σε πολύ λιγότερες επαναλήψεις σε σύγκριση με την {\en PCGN}. Αυτό εξηγείται από τη διαφορά στον τρόπο συσσώρευσης των ιδιοτιμών και ιδιαζουσών τιμών. `Εχουμε γενική συσσώρευση για τις ιδιάζουσες τιμές του προρρυθμισμένου συστήματος, μερικές εκ των οποίων μπορεί να είναι κοντά στο 0, ενώ έχουμε κύρια συσσώρευση των ιδιοτιμών, εντός ενός ορθογωνίου μακριά από την αρχή των αξόνων.

Το Σχήμα \ref{fig:spectra_x2+1+ig} δείχνει τη συσσώρευση των ιδιοτιμών και ιδιαζουσών τιμών, όταν $n=2048$. Οι μπλε γραμμές στο Σχήμα \ref{fig:spectra_x2+1+iga} συμβολίζουν τις τιμές $1-M\epsilon$ και $1+M\epsilon$. Σημειώνουμε ότι η $f_1$ έχει προσεγγιστεί από το τριγωνομετρικό πολυώνυμο $p_1$, βαθμού $8$ και η $\mathfrak{h}_1$ από το $p_2$, το οποίο έχει βαθμό ίσο με $6$.

`Οπως μπορούμε να δούμε όλες οι ιδιάζουσες τιμές βρίσκονται ανάμεσα στα $1-M\epsilon$ και $1+M\epsilon$ και το Θεώρημα \ref{thm:Me_cont} ισχύει. Στο Σχήμα \ref{fig:spectra_x2+1+igb} μπορούμε να παρατηρήσουμε ότι οι ιδιοτιμές του προρρυθμισμένου συστήματος κυμαίνονται εντός του ορθογωνίου $[0.822,1.178]\times[-0.125,0.125]$, το οποίο έχει πλευρές που αποτελούν φράγματα του πραγματικού και φανταστικού μέρους της $\frac{\mathfrak{f}_1}{p}$, επομένως το Θεώρημα \ref{thm:eig_clustering} ισχύει.

\begin{figure}
    \centering
    \subfloat[Ιδιοτιμές.]{{\label{fig:spectra_x2+1+igb}\includegraphics[width=0.45\linewidth]{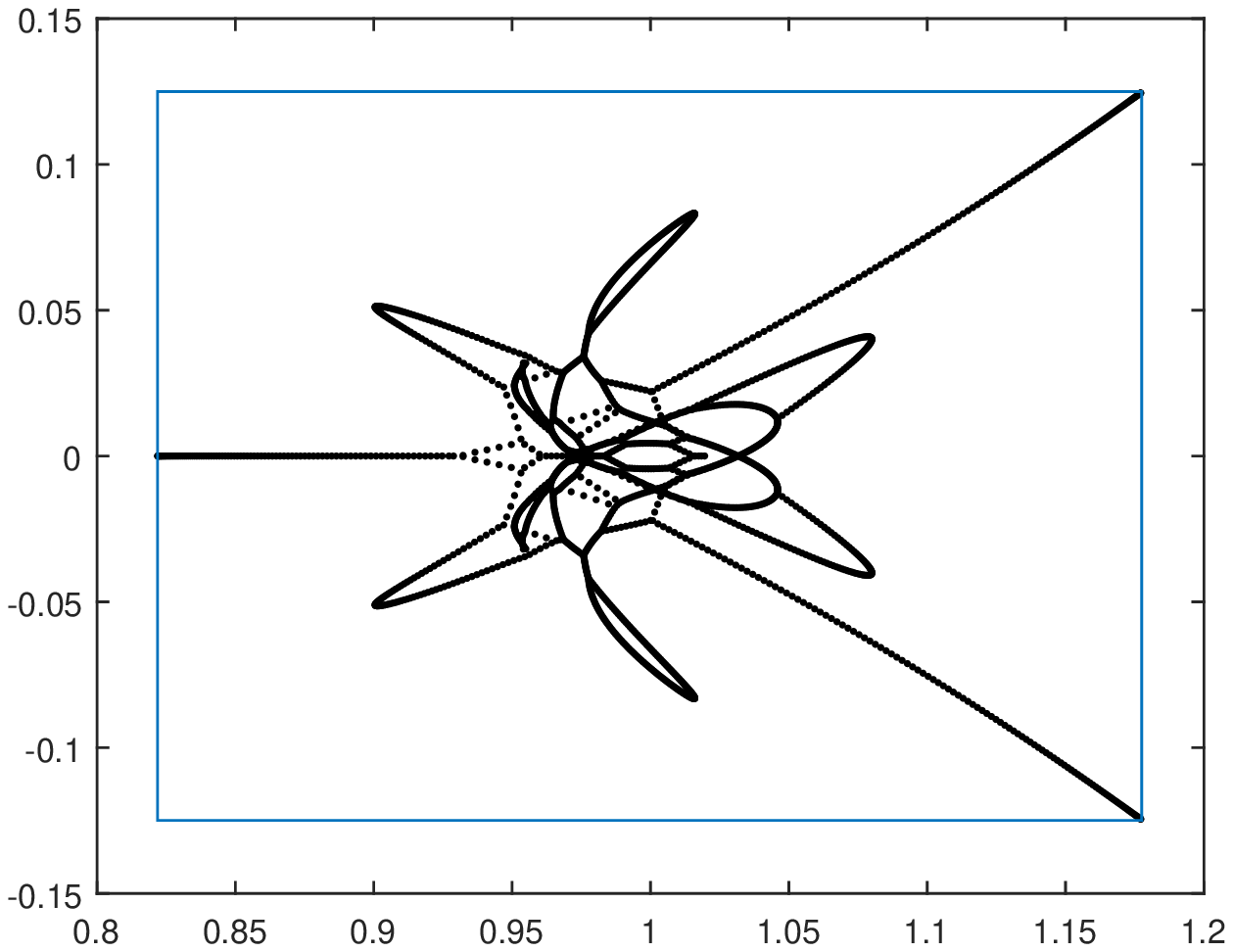}}}%
    \qquad
    \subfloat[Ιδιάζουσες τιμές.]{{\label{fig:spectra_x2+1+iga}\includegraphics[width=0.45\linewidth]{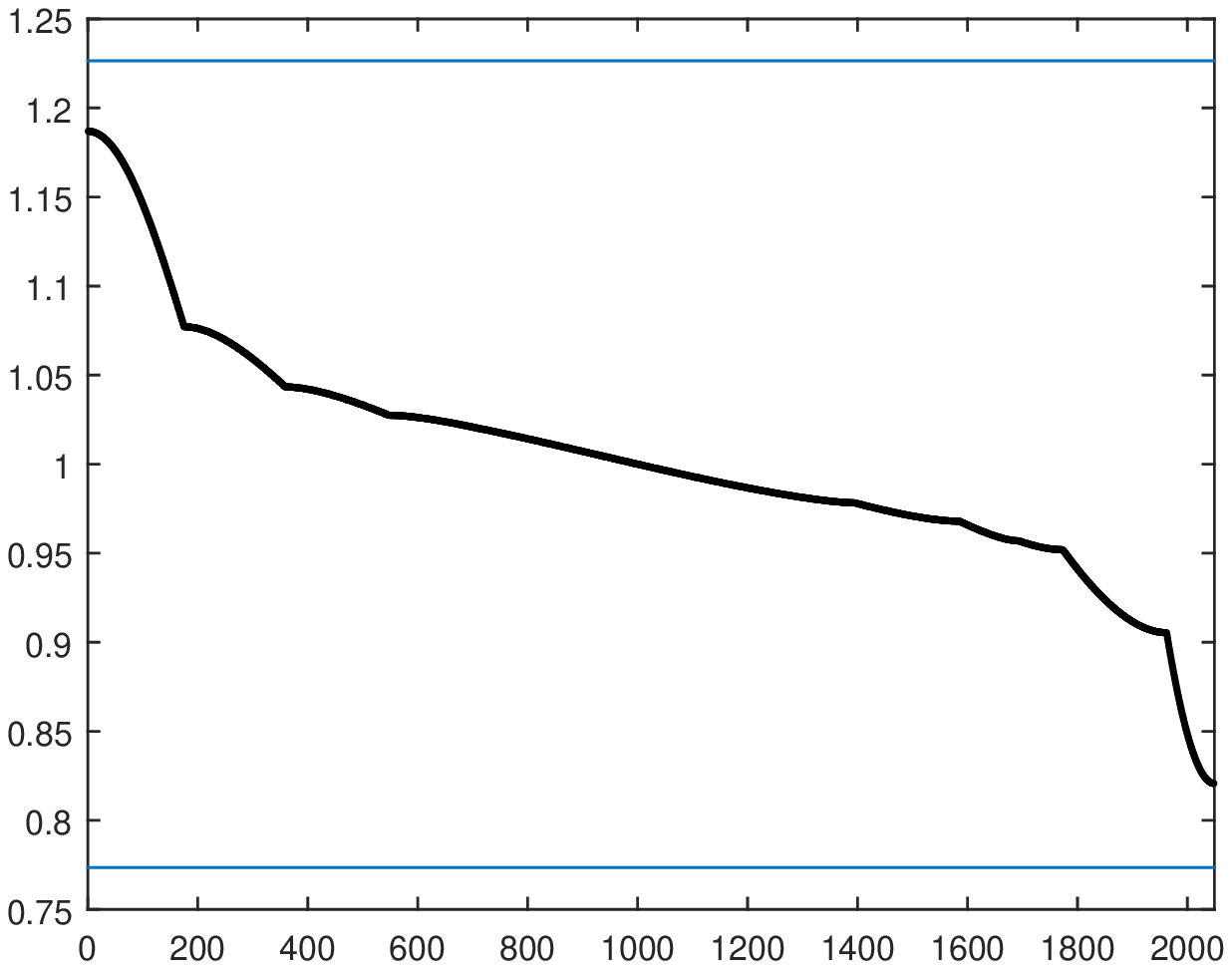}}}%
    \caption{Ιδιοτιμές και ιδιάζουσες τιμές ($\mathfrak{f}_1$).}%
    \label{fig:spectra_x2+1+ig}%
\end{figure}
\end{exmp}

\begin{exmp}\label{exmp:232}\normalfont
`Εστω $\mathfrak{f}_2(x)=x^2+\mathrm{i}x^3$. Προφανώς $f_1=x^2$, $f_2=x^3$ και τόσο η $f_1$, όσο και η $f_2$ έχουν ρίζα στο 0, με πολλαπλότητα $m_1=2$ και $m_2=3$, αντίστοιχα. Επομένως, θα άρουμε τη ρίζα της $\mathfrak{f}_2$, διαιρώντας με το τριγωνομετρικό πολυώνυμο $g(x)=2-2\cos{(x)}$, όπως περιγράφηκε στην υποενότητα \ref{Sss:321}. Τότε, προσεγγίζουμε τη συνάρτηση $\frac{\mathfrak{f}_2}{g}$ με τον τρόπο που περιγράψαμε στην ενότητα \ref{S:3} και κατασκευάζουμε τον προρρυθμιστή.

\begin{figure}
    \centering
    \subfloat[Πραγματικό μέρος.]{{\includegraphics[width=0.45\linewidth]{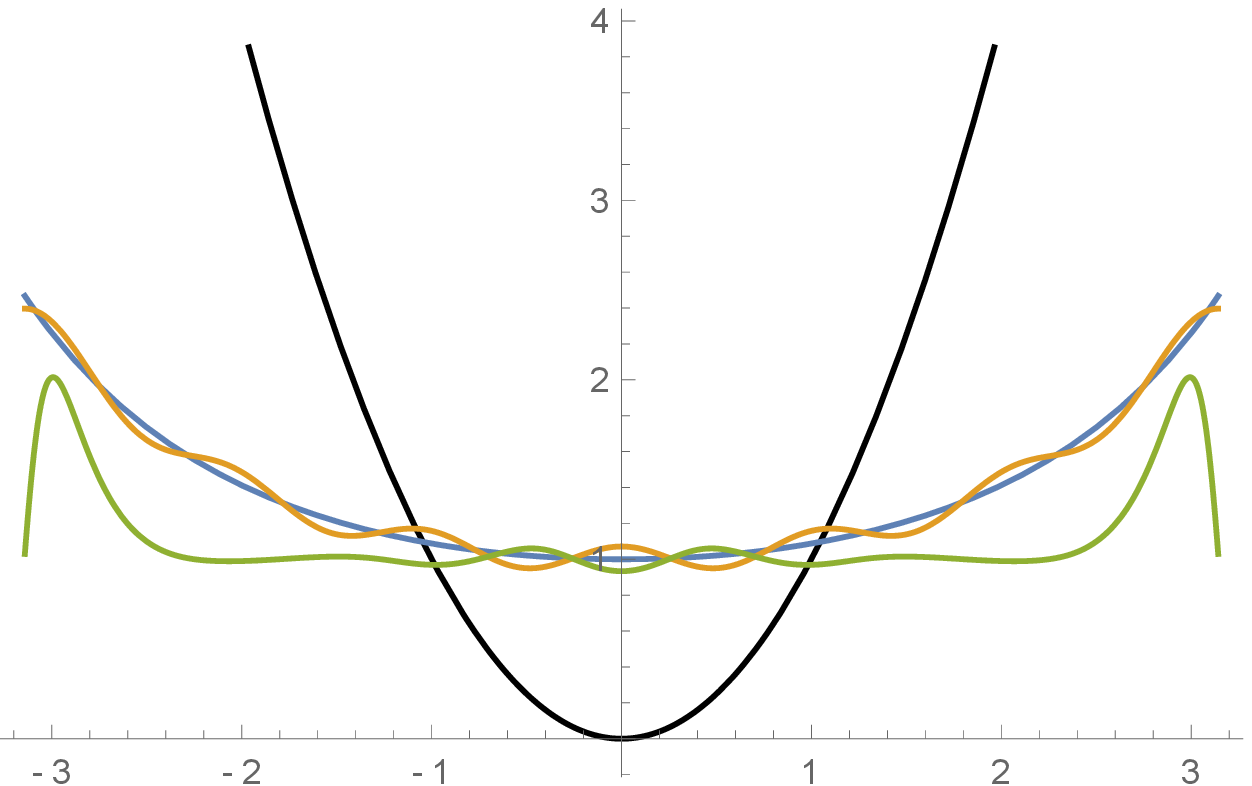}}}%
    \qquad
    \subfloat[Φανταστικό μέρος.]{{\includegraphics[width=0.45\linewidth]{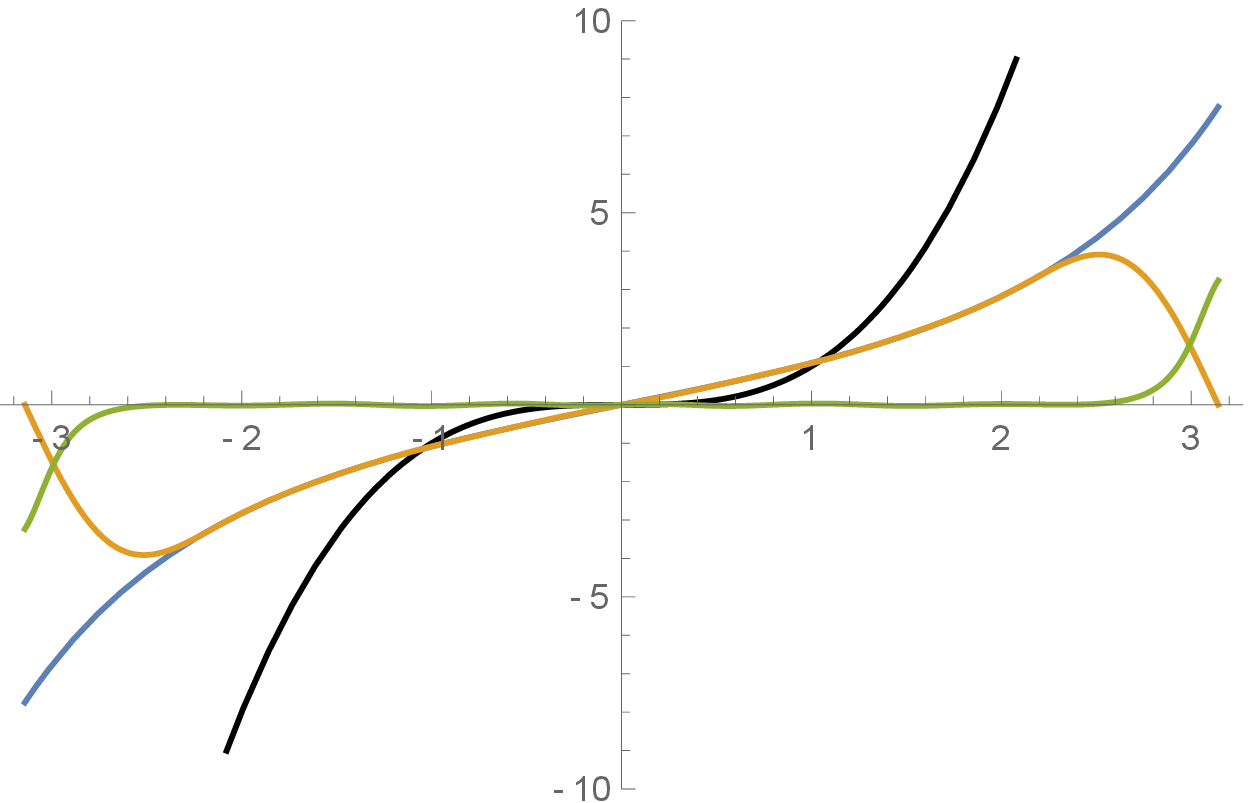}}}%
    \caption{$x^2+\mathrm{i}x^3$.}%
    \label{fig:Re_Im_x2+ix3}%
\end{figure}

Το Σχήμα \ref{fig:Re_Im_x2+ix3} δείχνει το πραγματικό και φανταστικό μέρος των συναρτήσεων $\mathfrak{f}_2(x)=x^2+\mathrm{i}x^3$ (μαύρο), $\widehat{\mathfrak{f}_2}=\frac{\mathfrak{f}_2}{g}$ (μπλε), $q$: βέλτιστη ομοιόμορφη τριγωνομετρική προσέγγιση με πολυώνυμα 6ου βαθμού για τις $\widehat{\mathfrak{f}_2}_1$ και $\widehat{\mathfrak{f}_2}_2$ (πορτοκαλί), και $\frac{\mathfrak{f}_2}{gq}$ (πράσινο).

Ο αριθμός επαναλήψεων, καθώς και το πλήθος ιδιαζουσών τιμών που κυμαίνονται εκτός του $[1-M\epsilon,1+M\epsilon]$, όταν χρησιμοποιούμε τον $R_{6,6}$ ως προρρυθμιστή, δίνονται στον Πίνακα \ref{tab:x^2+ix^3}. Προσεγγίζουμε την $\widehat{\mathfrak{f}_2}_2$ στο διάστημα $[-\frac{5\pi}{7},\frac{5\pi}{7}]$ (που σημαίνει ότι επιλέγουμε $c=\frac{5\pi}{7}$).

\begin{table}[H]
\centering
\begin{tabular}{ccccc|cccc|c}
\toprule
\multirow{2}{*}{$n$} & \multicolumn{4}{c|} {\en PGMRES} & \multicolumn{4}{c|} {\en PCGN} &\\
 & $I_n$ & $B$ & $R_{4,4}$ & $R_{6,6}$ & $I_n$ & $B$ & $R_{4,4}$ & $R_{6,6}$ & ${\rm SV-out}$\\\midrule
\phantom{0}256 & \phantom{$>$}256 & 67 & 24 & 22 & - & 80 & 37 & 35 & 58\\
\phantom{0}512 & $>$500 & 70 & 27 & 26 & - & 93 & 43 & 41 & 114\\
1024 & $>$500 & 69 & 28 & 27 & - & 104 & 47 & 44 & 226\\
2048 & $>$500 & 68 & 28 & 27 & - & 115 & 52 & 48 & 450\\\bottomrule
\end{tabular}
\caption{\label{tab:x^2+ix^3} Επαναλήψεις ($\mathfrak{f}_2$).}
\end{table}
Παρατηρούμε ότι ο αριθμός ιδιαζουσών τιμών που κυμαίνονται εκτός του $I_{\epsilon}=[1-M\epsilon,1+M\epsilon]=[0.931,1.069]$, είναι μικρότερος από $n-\frac{c}{\pi}n$, όπως αναμένονταν από το Θεώρημα \ref{thm:Me}. Για παράδειγμα, όταν $n=2048$, έχουμε 450 ιδιάζουσες τιμές εκτός του $I_{\epsilon}$, ενώ $[n-\frac{c}{\pi}n]=585$.

Η συσσώρευση των ιδιοτιμών και ιδιαζουσών τιμών, του προρρυθμισμένου συστήματος, όταν $n=2048$ δίνεται στο Σχήμα \ref{fig:spectra_x2+ix3}: Το Σχήμα \ref{fig:spectra_x2+ix3a} δείχνει τη συσσώρευση των ιδιαζουσών τιμών, το Σχήμα \ref{fig:spectra_x2+ix3b} δείχνει τις τελευταίες 50 ιδιάζουσες τιμές και το Σχήμα \ref{fig:spectra_x2+ix3c} τη συσσώρευση των ιδιοτιμών, εντός του ορθογωνίου $[0.933,2.015]\times[-3.236,3.236]$. Στα Σχήματα \ref{fig:spectra_x2+ix3a}, \ref{fig:spectra_x2+ix3b} οι πορτοκαλί γραμμές οριοθετούν του διάστημα $I_{\epsilon}$ και η μπλε δείχνει την τιμή $1+M\epsilon^{\prime}$. Τα αστέρια κόκκινου χρώματος συμβολίζουν τις ιδιάζουσες τιμές που κυμαίνονται εκτός του $I_{\epsilon}^{\prime}=[1-M\epsilon,1+M\epsilon^{\prime}]=[0.931,8.535]$ και χαρακτηρίζουν τη γενική συσσώρευση του Θεωρήματος \ref{thm:gen_cluster}. Το Σχήμα \ref{fig:spectra_x2+ix3c} επιβεβαιώνει την ισχύ του Θεωρήματος \ref{thm:eig_clustering}. Αξίζει να σημειωθεί ότι η κύρια μάζα τον ιδιοτιμών συσσωρεύεται πολύ κοντά στο σημείο $(1,0)$.

\begin{figure}[H]%
    \centering
    \subfloat[Ιδιάζουσες τιμές.]{{\label{fig:spectra_x2+ix3a}\includegraphics[width=0.45\linewidth]{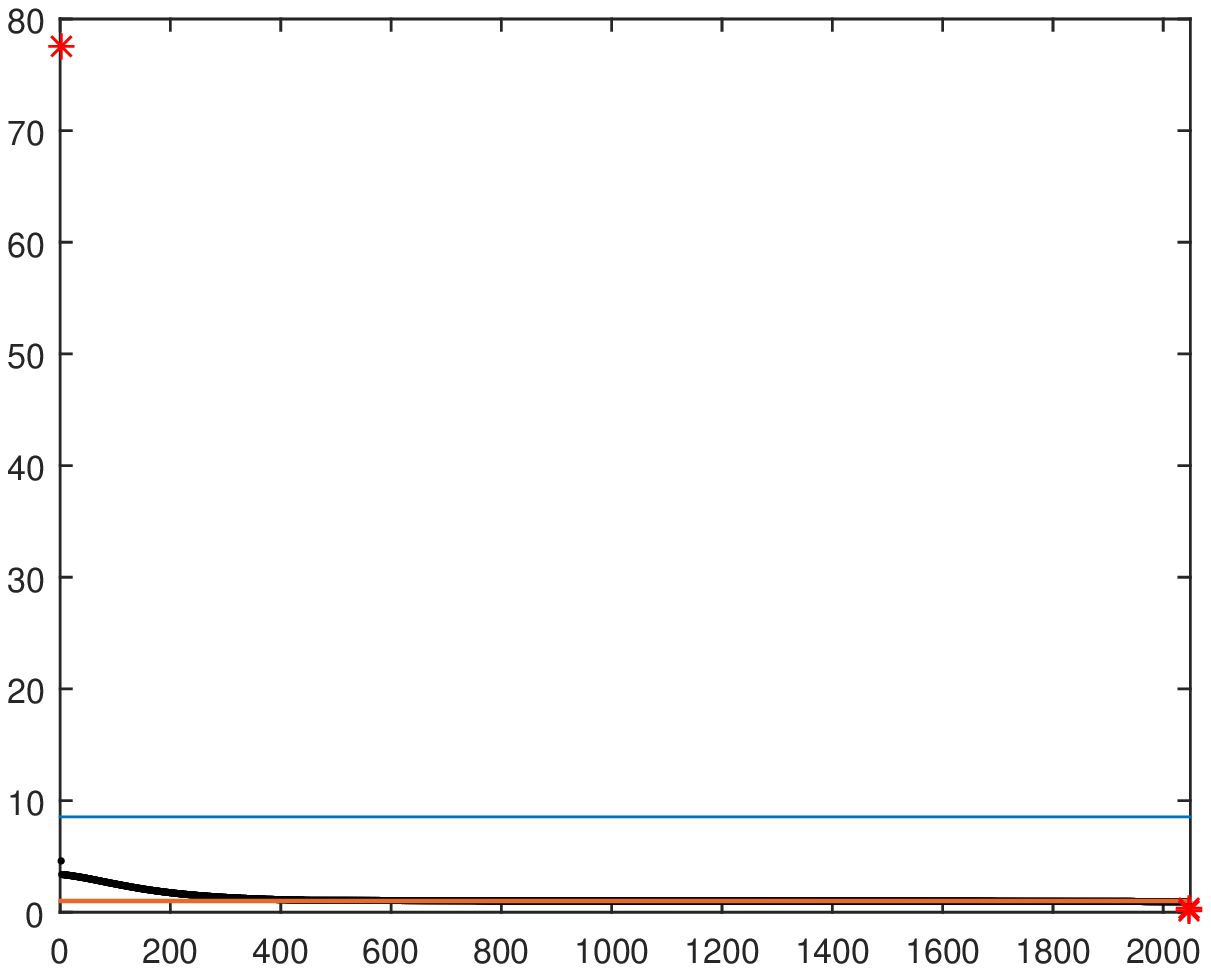}}}%
    \qquad
    \subfloat[Τελευταίες 50 ιδιάζουσες τιμές.]{{\label{fig:spectra_x2+ix3b}\includegraphics[width=0.45\linewidth]{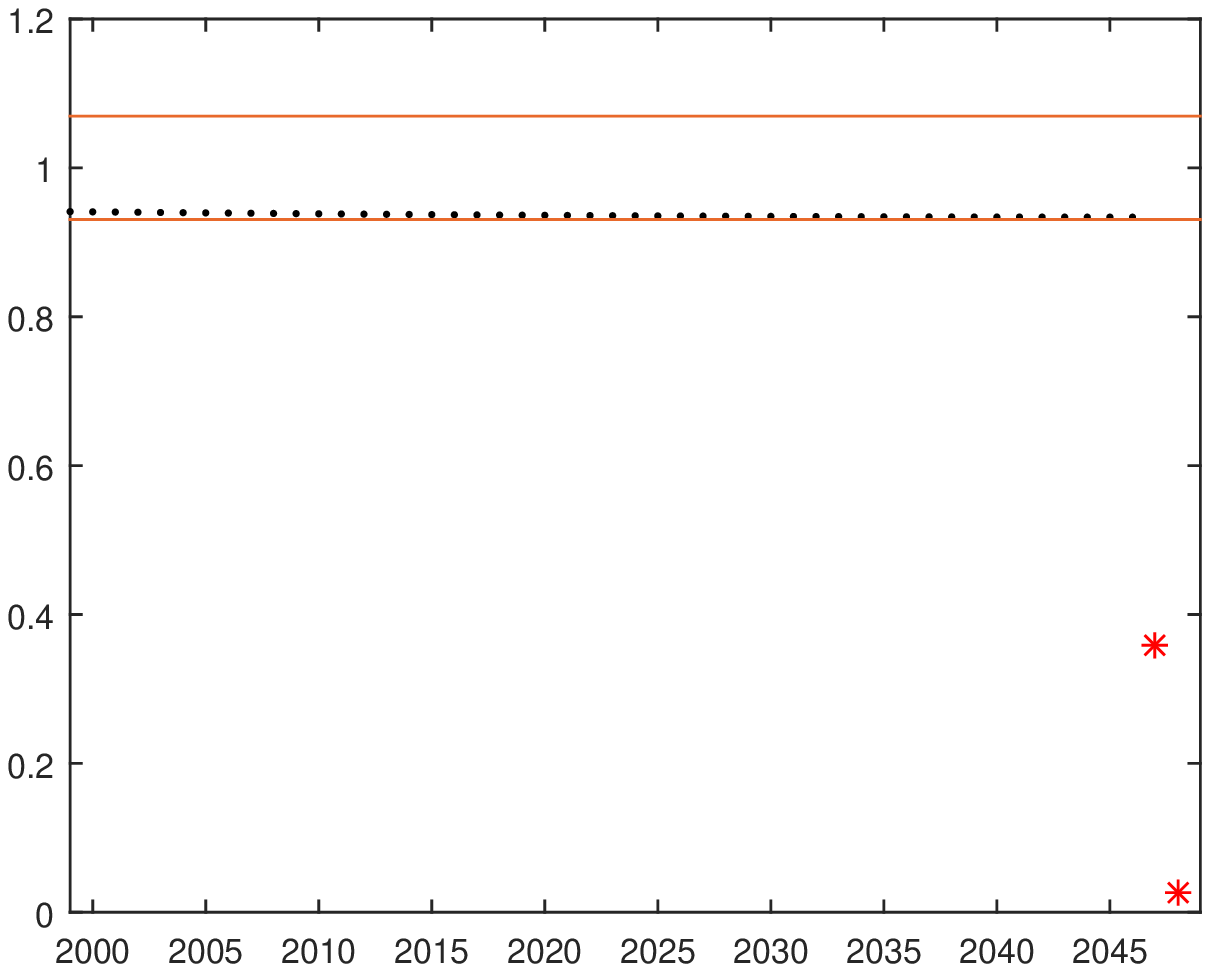}}}%
	\\
	\subfloat[Ιδιοτιμές.]{{\label{fig:spectra_x2+ix3c}\includegraphics[width=0.45\linewidth]{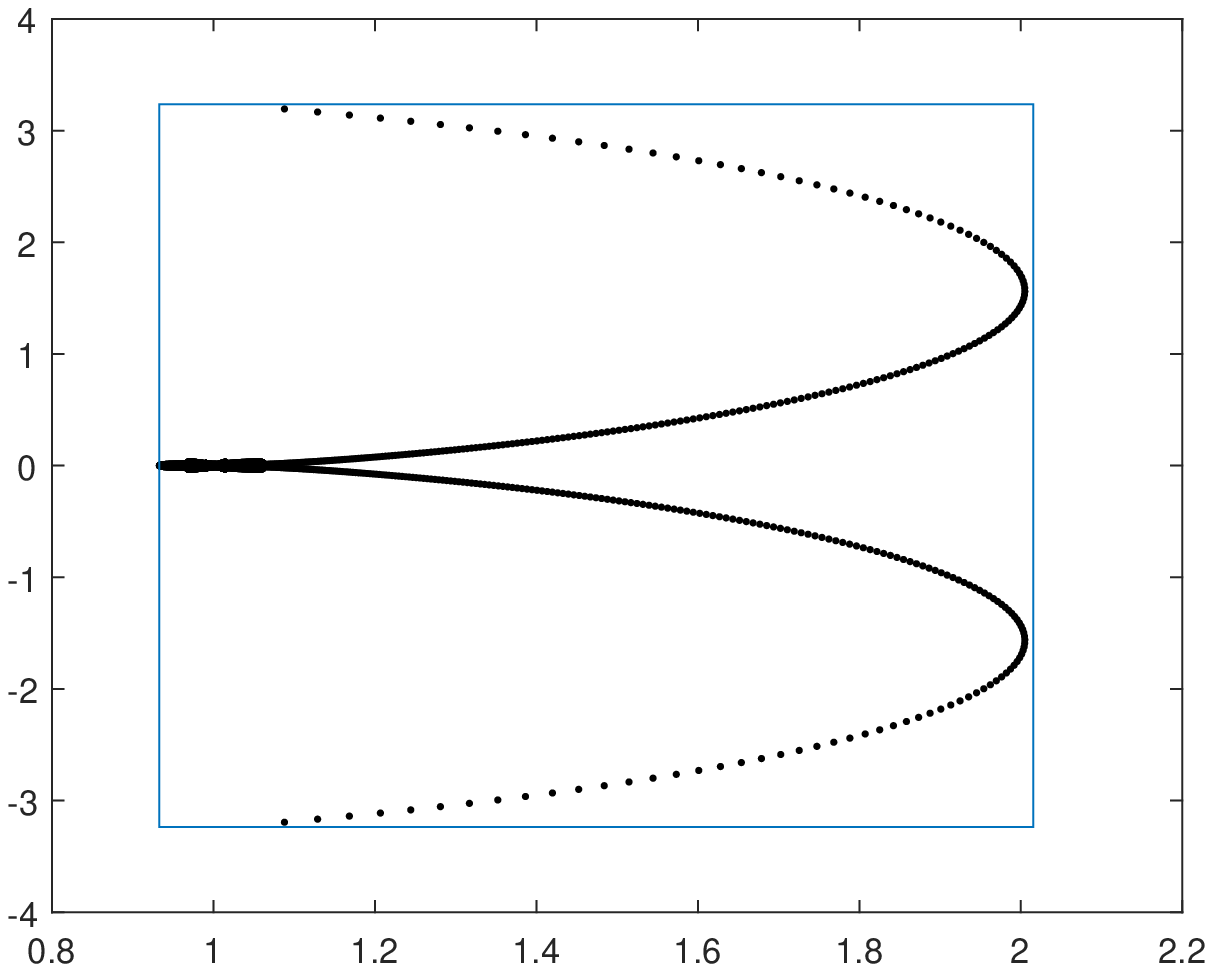}}}%
    \caption{Ιδιοτιμές και ιδιάζουσες τιμές ($\mathfrak{f}_2$).}%
    \label{fig:spectra_x2+ix3}%
\end{figure}
\end{exmp}

Στα παραδείγματα που ακολουθούν επικεντρωνόμαστε μόνο στη συμπεριφορά της μεθόδου {\en PGMRES}, αφού είδαμε ότι αυτή συγκλίνει (στη λύση) σε πολύ λιγότερες επαναλήψεις, σε σχέση με τη μέθοδο {\en PCGN}.

\begin{exmp}\label{exmp:233}\normalfont
`Εστω ότι $\mathfrak{f}_3(x)=x^2+\mathrm{i}x$. Σε αυτό το παράδειγμα $f_1=x^2$, $f_2=x$ και τόσο η $f_1$, όσο και η $f_2$ έχουν ρίζα στο 0, με πολλαπλότητα $m_1=2$ και $m_2=1$, αντίστοιχα. Επομένως, θα άρουμε τις ρίζες της γεννήτριας συνάρτησης χρησιμοποιώντας το τριγωνομετρικό πολυώνυμο $g(x)=2-2\cos{(x)}+\mathrm{i}\sin{(x)}$. Στη συνέχεια, προσεγγίζουμε τη συνάρτηση $\frac{\mathfrak{f}_3}{g}$ και κατασκευάζουμε τον προρρυθμιστή. Σε αυτό το παράδειγμα δίνουμε και τον αριθμό επαναλήψεων όταν χρησιμοποιούμε βέλτιστη ομοιόμορφη προσέγγιση, καθώς επίσης και παρεμβολή με τριγωνομετρικά πολυώνυμα.

Το Σχήμα \ref{fig:Re_Im_x2+ix} δείχνει το πραγματικό και φανταστικό μέρος των $\mathfrak{f}_3$, $\widehat{\mathfrak{f}_3}=\frac{\mathfrak{f}_3}{g}$, $q$ που αφορά στον προρρυθμιστή $R_{4,4}$ και $\frac{\mathfrak{f}_3}{gq}$, όπως έγινε και στο προηγούμενο παράδειγμα.
\begin{figure}
    \centering
    \subfloat[Πραγματικό μέρος.]{{\includegraphics[width=0.45\linewidth]{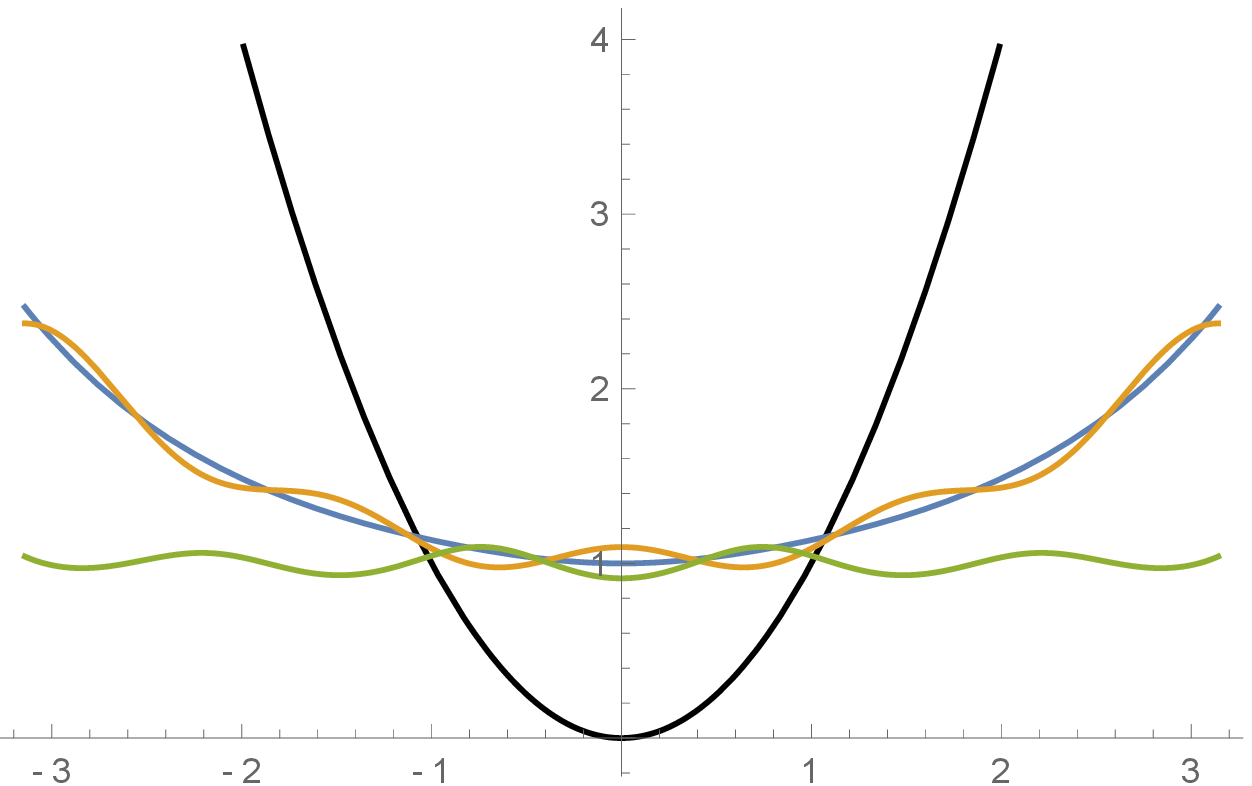}}}%
    \qquad
    \subfloat[Φανταστικό μέρος.]{{\includegraphics[width=0.45\linewidth]{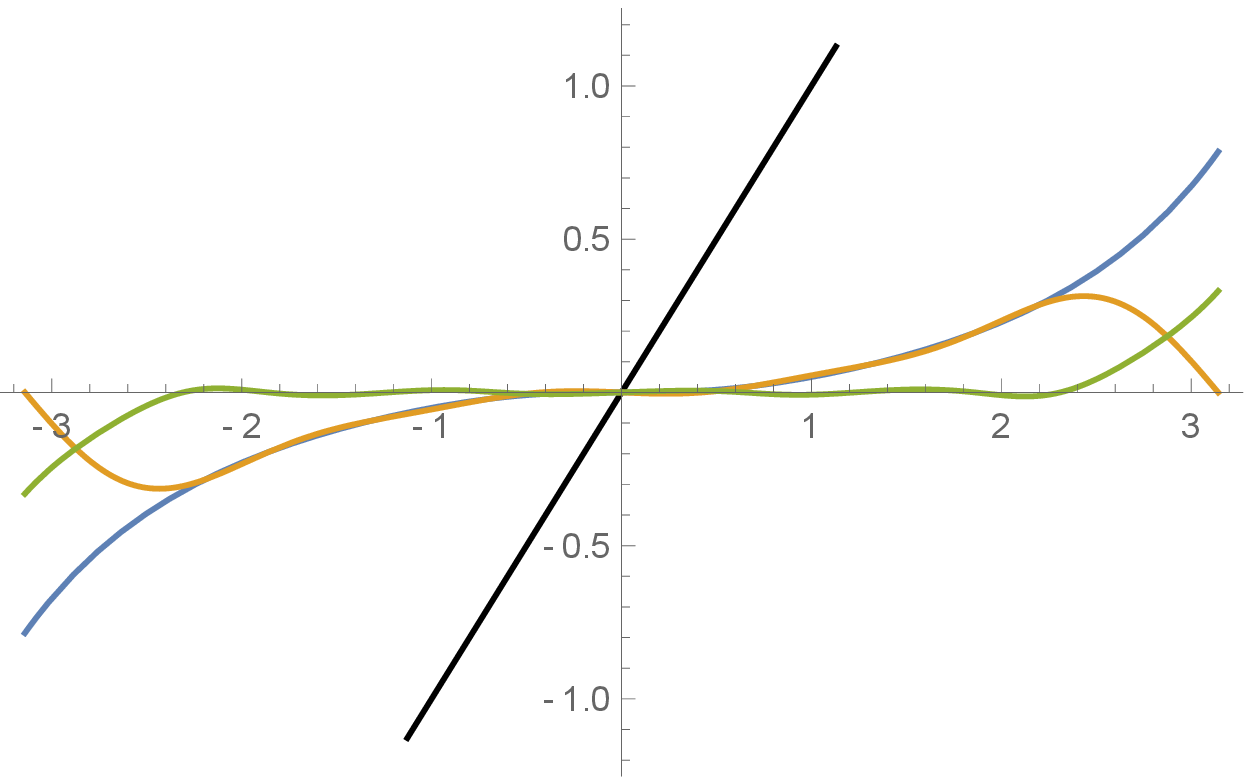}}}%
    \caption{$x^2+\mathrm{i}x$.}%
    \label{fig:Re_Im_x2+ix}%
\end{figure}

Ο αριθμός επαναλήψεων δίνεται στον Πίνακα \ref{tab:x^2+ix}. Σε αυτόν δίνεται επίσης και ο αριθμός των ιδιαζουσών τιμών που κυμαίνονται εκτός του διαστήματος συσσώρευσης και αφορούν στον $R_{4,4}$.

\begin{table}[H]
\centering
\begin{tabular}{cccccccc}
\toprule
 $n$ & $I_n$ & $B$ & $R_{4,4}$ & $In_{4,4}$ & $R_{10,10}$ & $In_{10,10}$ & ${\rm SV-out}$\\\midrule
\phantom{0}256 & \phantom{$>$}256 & 11 & 6 & 6 & 6 & 12 & 2\\
\phantom{0}512 & $>$500 & 11 & 6 & 6 & 6 & 12 & 2\\
1024 & $>$500 & 10 & 6 & 6 & 5 & 12 & 2\\
2048 & $>$500 & 10 & 6 & 5 & 5 & 11 & 2\\\bottomrule
\end{tabular}
\caption{\label{tab:x^2+ix} Επαναλήψεις ($\mathfrak{f}_3$).}
\end{table}

Παρατηρούμε ότι οι προρρυθμιστές που κατασκευάστηκαν από βέλτιστη ομοιόμορφη προσέγγιση και παρεμβολή, έχουν σχεδόν τους ίδιους αριθμούς επαναλήψεων, όταν οι βαθμοί των πολυωνύμων $q_1$ και $q_2$ είναι μικροί. Ωστόσο, όσο οι βαθμοί μεγαλώνουν, ο προρρυθμιστής που προκύπτει μέσω βέλτιστης ομοιόμορφης προσέγγισης έχει καλύτερη συμπεριφορά, δηλαδή είναι πιο αποτελεσματικός από αυτόν που προέκυψε μέσω παρεμβολής. Αυτό οφείλεται στην ταλάντωση του τριγωνομετρικού πολυωνύμου παρεμβολής, μέσω της οποίας δεν εξασφαλίζεται η μείωση του σφάλματος, όσο ο βαθμός του πολυωνύμου αυξάνεται. `Αρα, επιβεβαιώνεται ότι η βέλτιστη ομοιόμορφη προσέγγιση δίνει πιο αποτελεσματικούς προρρυθμιστές.

Η συσσώρευση των ιδιοτιμών και ιδιαζουσών τιμών όταν $n=2048$, χρησιμοποιώντας τον $R_{4,4}$, δίνεται στο Σχήμα \ref{fig:spectra_x2+ix}: το Σχήμα \ref{fig:spectra_x2+ixb} τη συσσώρευση των ιδιοτιμών, ενώ το Σχήμα \ref{fig:spectra_x2+ixa} τη συσσώρευση των ιδιαζουσών τιμών. 

\begin{figure}
    \centering
    \subfloat[Ιδιοτιμές.]{{\label{fig:spectra_x2+ixb}\includegraphics[width=0.45\linewidth]{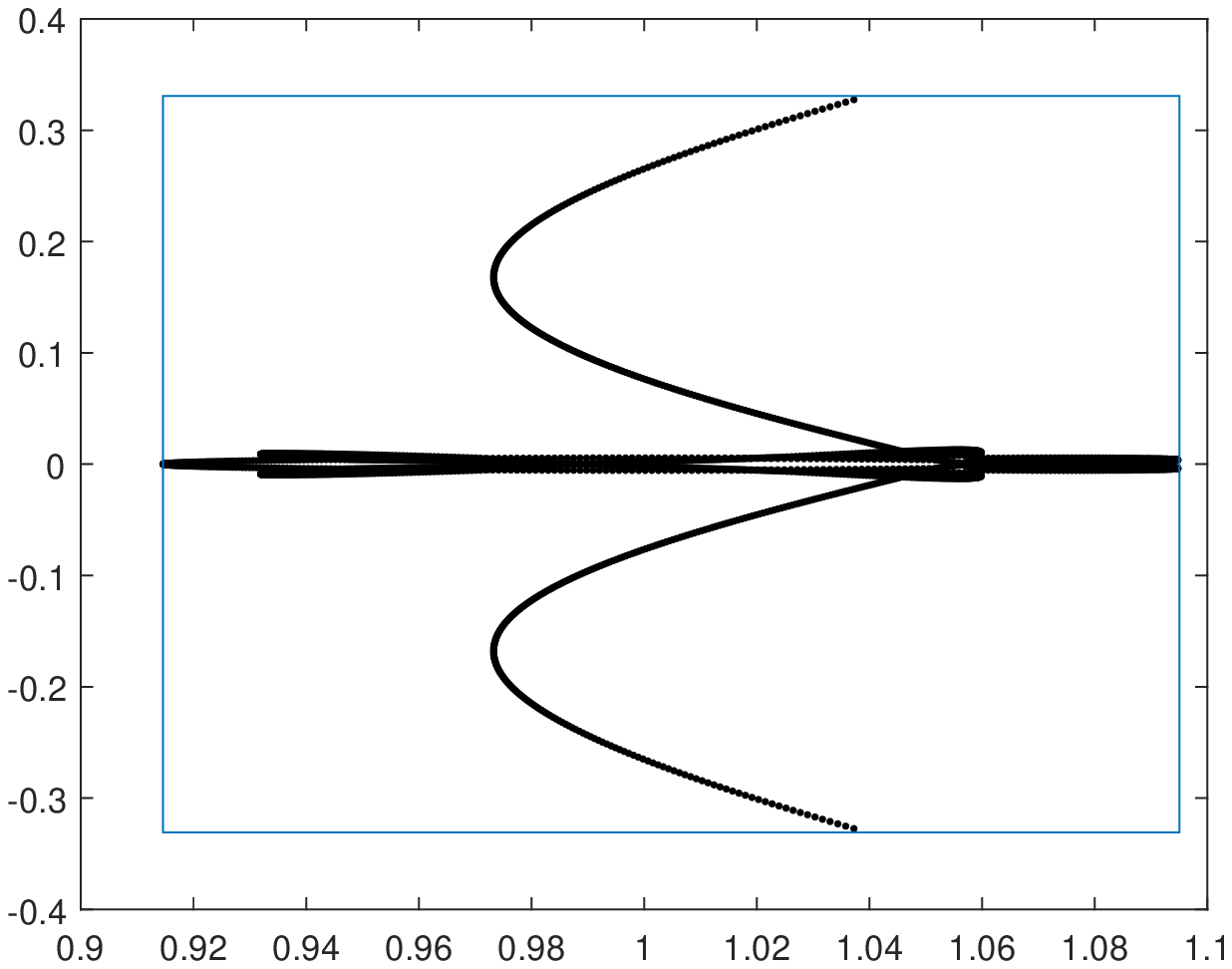}}}%
    \qquad
    \subfloat[Ιδιάζουσες τιμές.]{{\label{fig:spectra_x2+ixa}\includegraphics[width=0.45\linewidth]{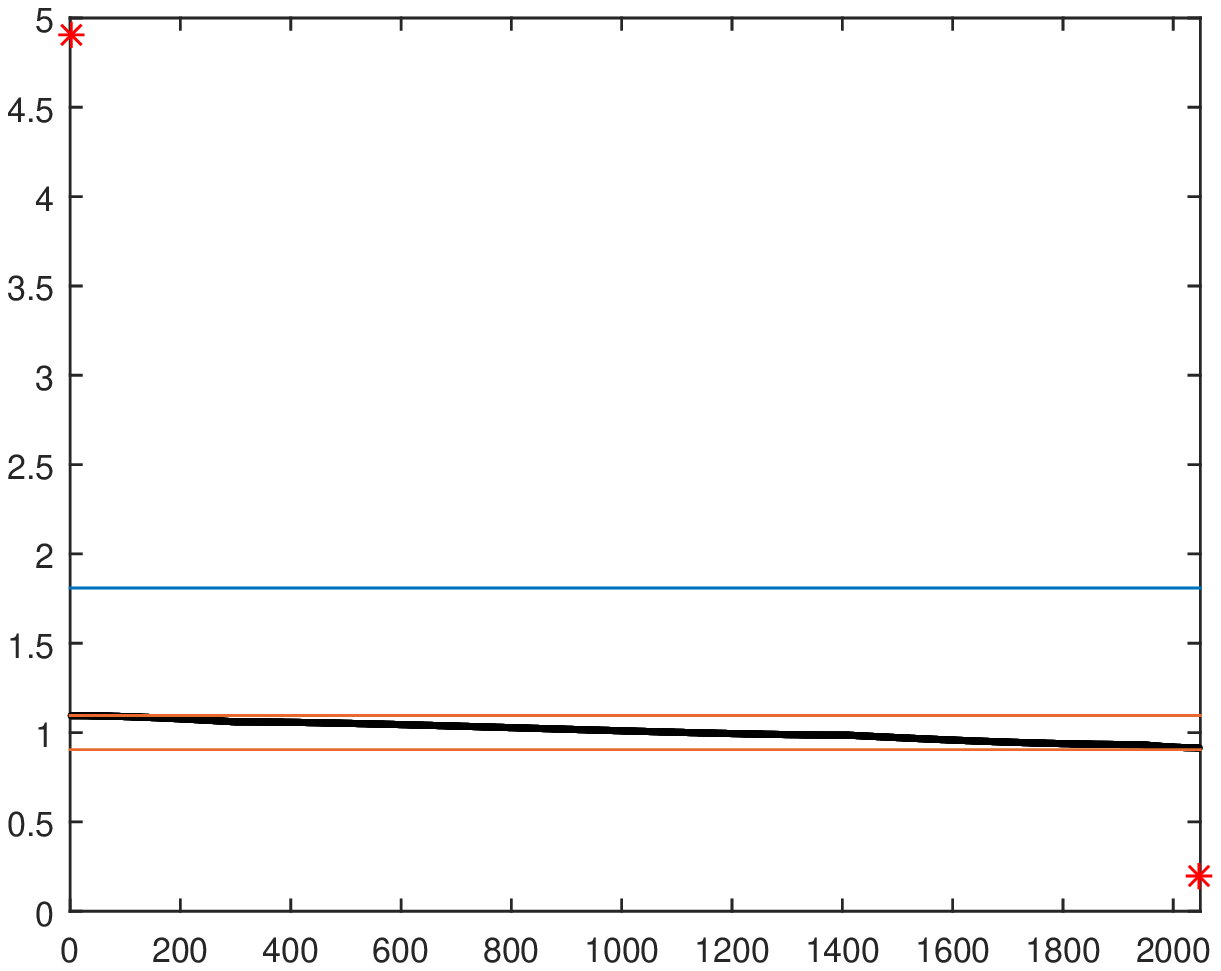}}}%
    \caption{Ιδιοτιμές και ιδιάζουσες τιμές ($\mathfrak{f}_3$).}%
    \label{fig:spectra_x2+ix}%
\end{figure}
Στο Σχήμα \ref{fig:spectra_x2+ixa}, οι πορτοκαλί γραμμές είναι τα άκρα του διαστήματος $I_{\epsilon}=[0.904,1.096]$ και η μπλε γραμμή λαμβάνει την τιμή $1+M\epsilon^{\prime}=1.809$. `Οπως και στο προηγούμενο παράδειγμα, με κόκκινα αστέρια συμβολίζουμε τις ιδιάζουσες τιμές που κυμαίνονται εκτός του διαστήματος γενικής συσσώρευσης. Το σχήμα \ref{fig:spectra_x2+ixb} δείχνει τη συσσώρευση των ιδιοτιμών εντός του ορθογωνίου $[0.915,1.095]\times[-0.331,0.331]$.
\end{exmp}

\begin{exmp}\normalfont
`Εστω $\mathfrak{f}_4(x)=x^2-1+\mathrm{i}\mathfrak{h}_2(x)$, όπου: $$\mathfrak{h}_2(x)=\left\{
     \begin{array}{@{}c@{\thinspace}l}
       -1-x &, -\pi\leq x< -\frac{1}{2}\\
       x &, -\frac{1}{2}\leq x<\frac{1}{2} \\
       1-x &,\phantom{-}\frac{1}{2}\leq x\leq\pi\\

     \end{array}
   \right..$$
Η συνάρτηση $f_1$ έχει μια απλή ρίζα στο $\pm 1$ (με πολλαπλότητα $m_1=1$) και η $f_2=\mathfrak{h}_2$ έχει επίσης ρίζα στο $\pm 1$ με πολλαπλότητα $m_2=1$ και μία ρίζα στο 0, με πολλαπλότητα $m_0=1$. Επομένως, επιλέγουμε ως $g(x)=\cos{(1)}-\cos{(x)}$ και προσεγγίζουμε την $\frac{\mathfrak{f}_4}{g}$ με τριγωνομετρικά πολυώνυμα 4ου βαθμού, έτσι ώστε στη συνέχεια να κατασκευάσουμε τον προρρυθμιστή.

\begin{table}
\centering
\begin{tabular}{ccccc}
\toprule
 $n$ & $I_n$ & $B$ & $R_{4,4}$ & ${\rm SV-out}$\\\midrule
\phantom{0}256 & \phantom{$>$}256 & 15 & 6 & 4\\
\phantom{0}512 & $>$500 & 15 & 6 & 4\\
1024 & $>$500 & 16 & 6 & 4\\
2048 & $>$500 & 15 & 6 & 4\\\bottomrule
\end{tabular}
\caption{\label{tab:x^2-1+ig} Επαναλήψεις ($\mathfrak{f}_4$).}
\end{table}

Ο αριθμός επαναλήψεων και ιδιαζουσών τιμών που κυμαίνονται εκτός του διαστήματος συσσώρευσης δίνονται στον Πίνακα \ref{tab:x^2-1+ig}. Η συσσώρευση των ιδιοτιμών και ιδιαζουσών τιμών, όταν $n=2048$, δίνεται στα Σχήματα \ref{fig:spectra_x2-1+iga} και \ref{fig:spectra_x2-1+igb}, αντίστοιχα. Παρατηρούμε ότι δύο ιδιοτιμές έχουν πραγματικό μέρους μεγαλύτερο από $\operatorname*{ess~sup}\limits_{-\pi\leq x\leq\pi}{\operatorname{Re}\left(\frac{\mathfrak{f}_4(x)}{p(x)}\right)}$, χαρακτηρίζοντας την κύρια συσσώρευση στο ορθογώνιο $[0.888,1.088]\times[-0.253,0.253]$.  
\begin{figure}[H]%
    \centering
    \subfloat[Ιδιοτιμές.]{{\label{fig:spectra_x2-1+iga}\includegraphics[width=0.45\linewidth]{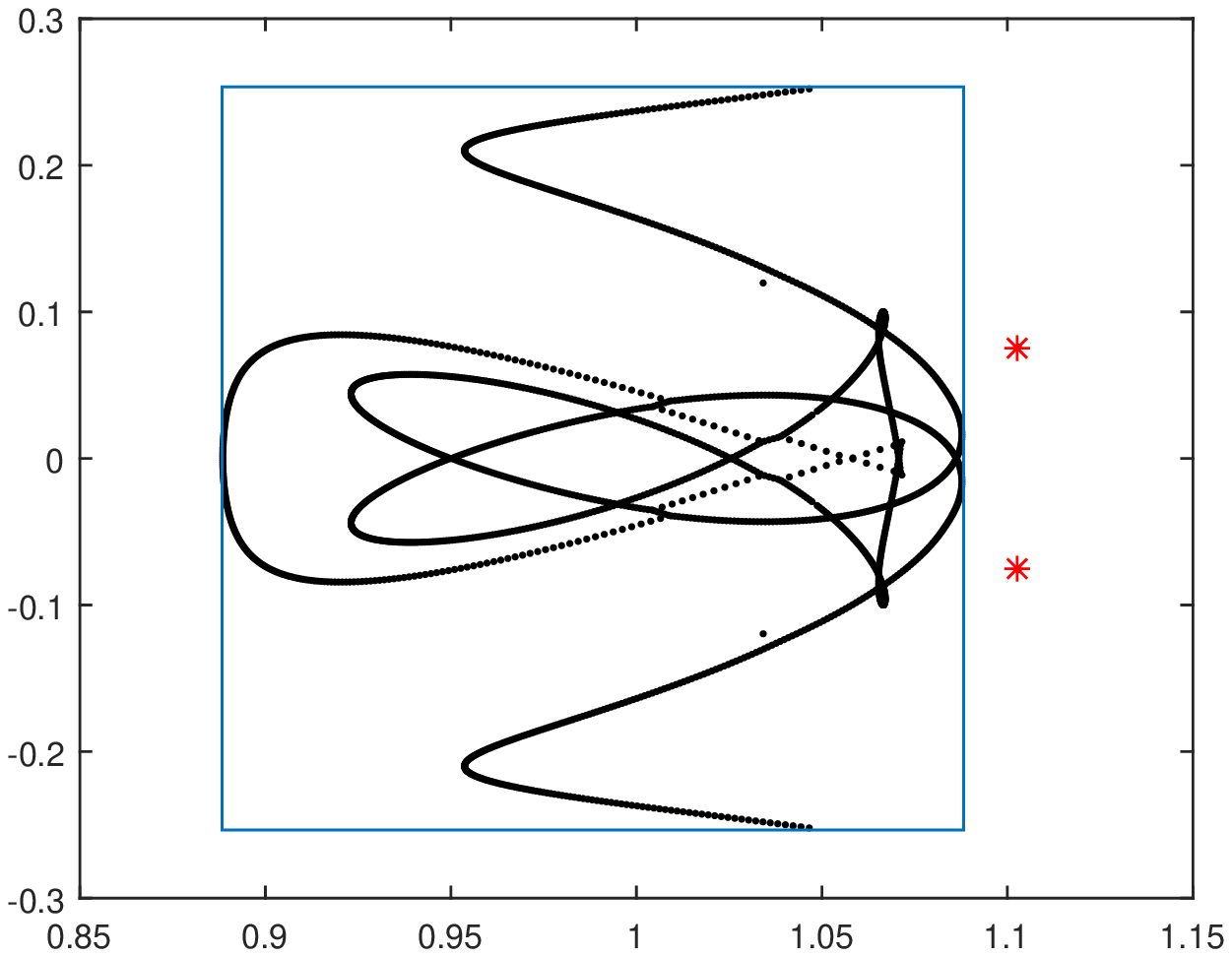}}}   	\qquad
    \subfloat[Ιδιάζουσες τιμές.]{{\label{fig:spectra_x2-1+igb}\includegraphics[width=0.45\linewidth]{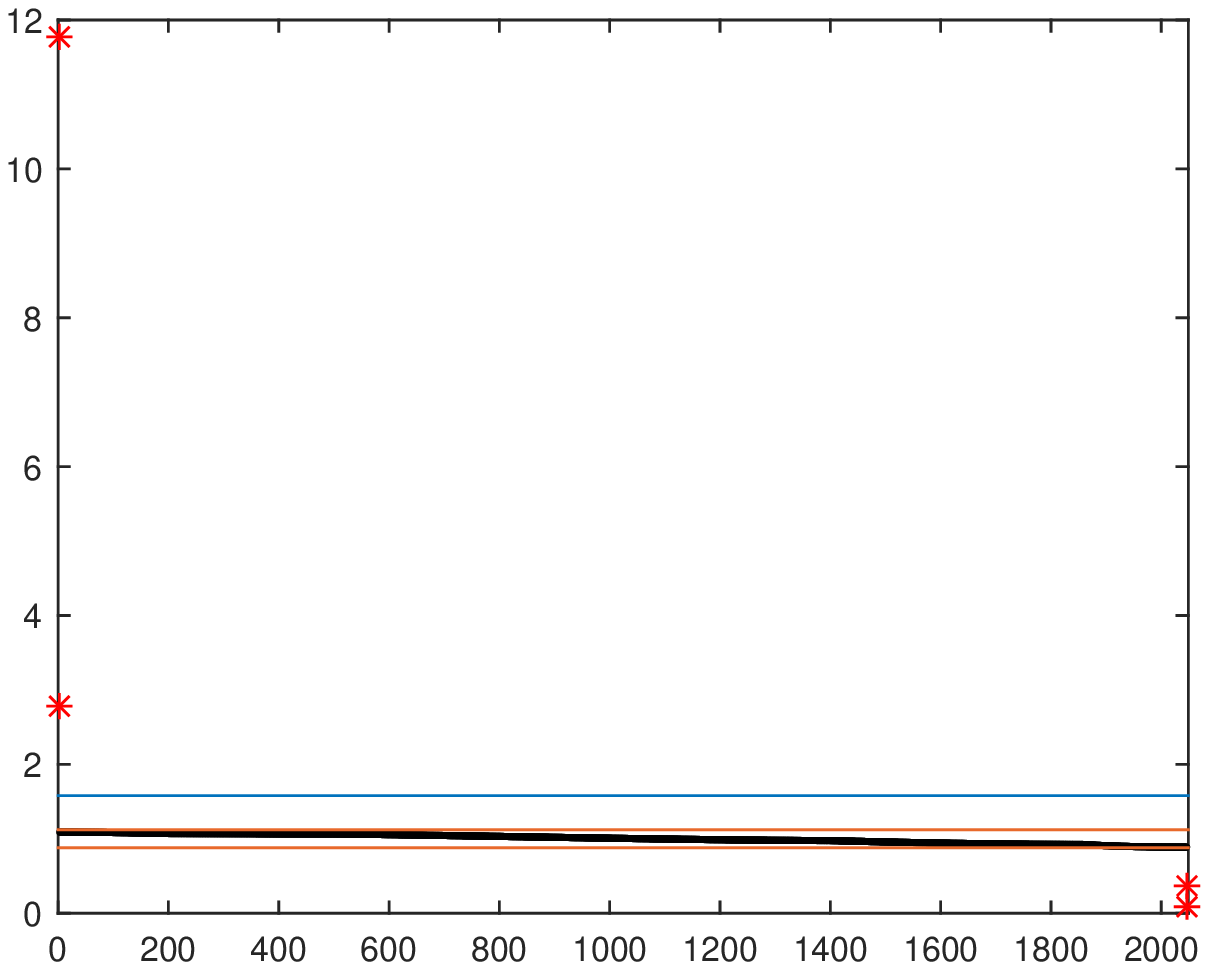}}}%
    \caption{Ιδιοτιμές και ιδιάζουσες τιμές ($\mathfrak{f}_4$).}%
    \label{fig:spectra_x2-1+ig}%
\end{figure}
\end{exmp}

\begin{exmp}\normalfont
`Εστω ότι $\mathfrak{f}_5(x)=(x^2-1)^2+\mathrm{i}x(x^2-4)$. Σε αυτό το παράδειγμα προφανώς το πραγματικό και φανταστικό μέρος της γεννήτριας συνάρτησης έχουν ρίζες σε διαφορετικά σημεία, αφού $f_1=(x^2-1)^2$ και $f_2=x(x^2-4)$. Ειδικότερα, η $f_1$ έχει ρίζα στο $\pm 1$ με πολλαπλότητα $m_1=1$, ενώ η $f_2$ έχει ρίζα στο $0$ και στο $\pm 2$, με πολλαπλότητες $m_0=m_2=1$. Επομένως, επιλέγουμε ως $g(x)=\left(\cos{(1)}-\cos{(x)}\right)^2+\mathrm{i}\sin{(x)}\left(\cos{(2)}-\cos{(x)}\right)$, σύμφωνα με όσα περιγράψαμε στον τρόπο κατασκευής του προρρυθμιστή. Τελικά, κατασκευάζουμε τον $R_{8,6}$.

Ο αριθμός επαναλήψεων με χρήση της μεθόδου {\en PGMRES} δίνεται στον Πίνακα \ref{tab:different_points}. Η κύρια συσσώρευση των ιδιοτιμών εντός του ορθογωνίου $[0.684,1.953]\times[-0.246,0.246]$ φαίνεται στο Σχήμα \ref{fig:spectra_different_points}.

\begin{table}
\centering
\begin{tabular}{cccc}
\toprule
 $n$ & $I_n$ & $B$ & $R_{8,6}$\\\midrule
\phantom{0}256 & 151 & 25 & 12\\
\phantom{0}512 & 197 & 25 & 11\\
1024 & 223 & 25 & 11\\
2048 & 229 & 24 & 11\\\bottomrule
\end{tabular}
\caption{\label{tab:different_points} Επαναλήψεις ($\mathfrak{f}_5$).}
\end{table}

\begin{figure}
    \centering
    \includegraphics[width=0.45\linewidth]{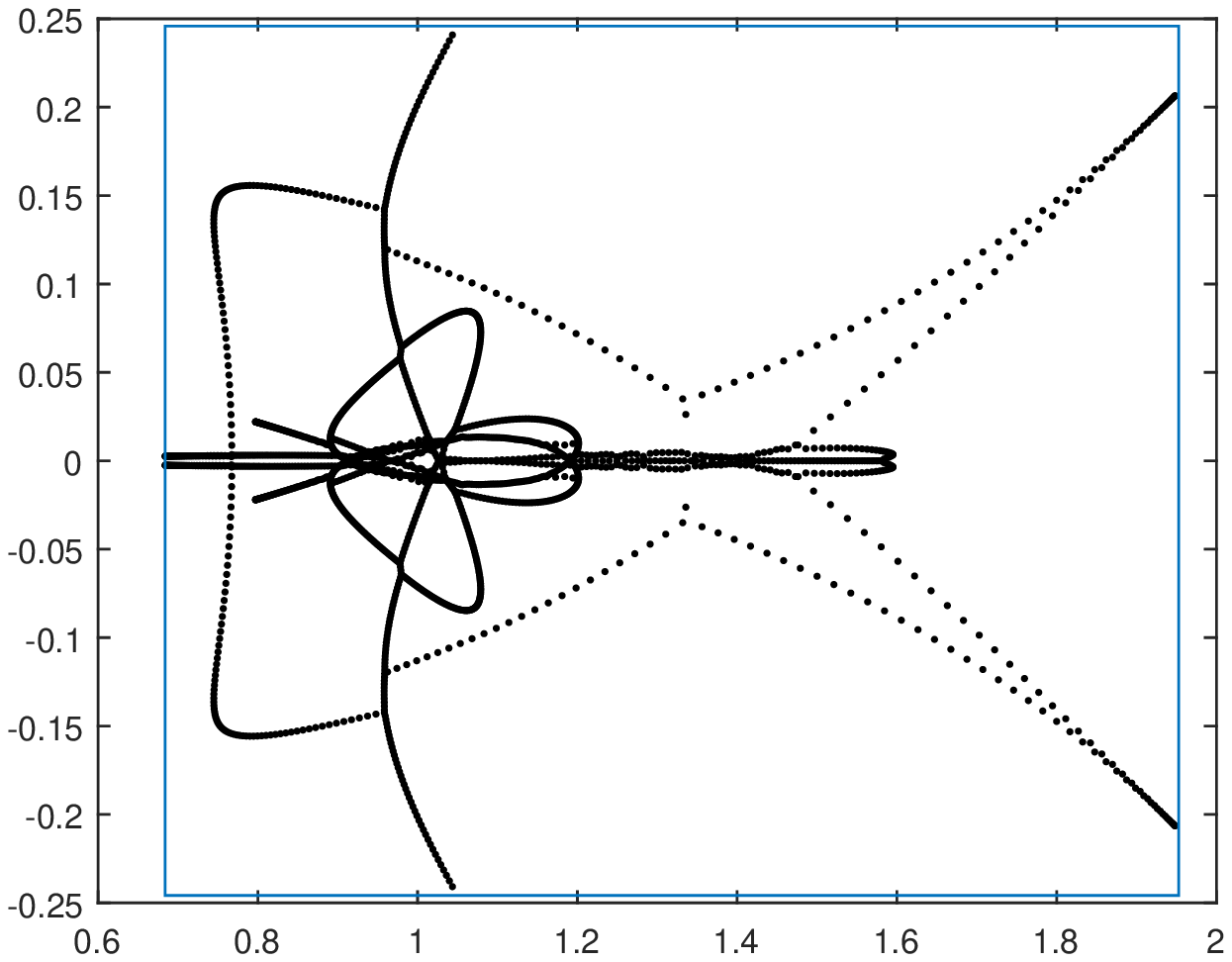}
    \caption{Ιδιοτιμές ($\mathfrak{f}_5$).}%
    \label{fig:spectra_different_points}%
\end{figure}
\end{exmp}

\begin{exmp}\label{exp:6}\normalfont
`Εστω η δι-διάστατη συνάρτηση χωριζομένων μεταβλητών $\mathfrak{f}_6(x,y)=x^2+y^2+\mathrm{i}(x+y)$. Προφανώς, $f_1(x,y)=x^2+y^2$ και $f_2(x,y)=x+y$. Η $f_1$ έχει μια ρίζα στο $(0,0)$ με πολλαπλότητα ίση με 2, ενώ η $f_2$ έχει μια ρίζα στο ίδιο σημείο, αλλά με πολλαπλότητα 1. Ας είναι $\mathfrak{h}(z)=z^2+\mathrm{i}z,~z\in[-\pi,\pi]$, τότε η συνάρτηση $\mathfrak{f}_6$ χωρίζεται ως: $\mathfrak{f}_6(x,y)=\mathfrak{h}(x)+\mathfrak{h}(y)$. Θα άρουμε τη ρίζα της $\mathfrak{h}(z)$, διαιρώντας με $g(z)=2-2\cos{(z)}+\mathrm{i}\sin{(z)}$. Κατόπιν, προσεγγίζουμε την $\frac{h}{g}$ με το τριγωνομετρικό πολυώνυμο $q=q_1+\mathrm{i}q_2$, όπου $q_1$ και $q_2$ είναι 4ου βαθμού. Στη συνέχεια κατασκευάζουμε τον μονοδιάστατο ταινιωτό πίνακα {\en Toeplitz}, $T_n(p)=T_n(gq)$, ο οποίος είναι ο αντίστοιχος προρρυθμιστής για τον $T_n(\mathfrak{h})$. Τελικά, ο δι-διάστατος ταινιωτός {\en Toeplitz} προρρυθμιστής κατασκευάζεται από το τανυστικό γινόμενο $T_{nm}(\widehat{p})=I_n\otimes T_m(p)+T_n(p)\otimes I_m$, όπου $\widehat{p}(x,y)=p(x)+p(y)$.

Ο αριθμός επαναλήψεων για διάφορες διαστάσεις των {\en blocks} δίνεται στον Πίνακα \ref{tab:BTTB}. Η συσσώρευση των ιδιοτιμών και ιδιαζουσών τιμών δίνεται στο Σχήμα \ref{fig:spectra_BTTB}.

\begin{table}
\centering
\begin{tabular}{ccccc}
\hlineB{2}
\diagbox{$n$}{$m$} & 16 & 32 & 64 & 128\\
\hlineB{1.2}
\phantom{0}16 & 6 & 6 & 6 & 6\\
\phantom{0}32 & 6 & 6 & 6 & 6\\
\phantom{0}64 & 6 & 6 & 6 & 6\\
128 & 6 & 6 & 6 & 6\\
\hlineB{2}
\end{tabular}
\caption{\label{tab:BTTB} Επαναλήψεις ($\mathfrak{f}_6$).}
\end{table}

\begin{figure}
    \centering
    \subfloat[Ιδιοτιμές.]{{\label{fig:spectra_BTTBa}\includegraphics[width=0.45\linewidth]{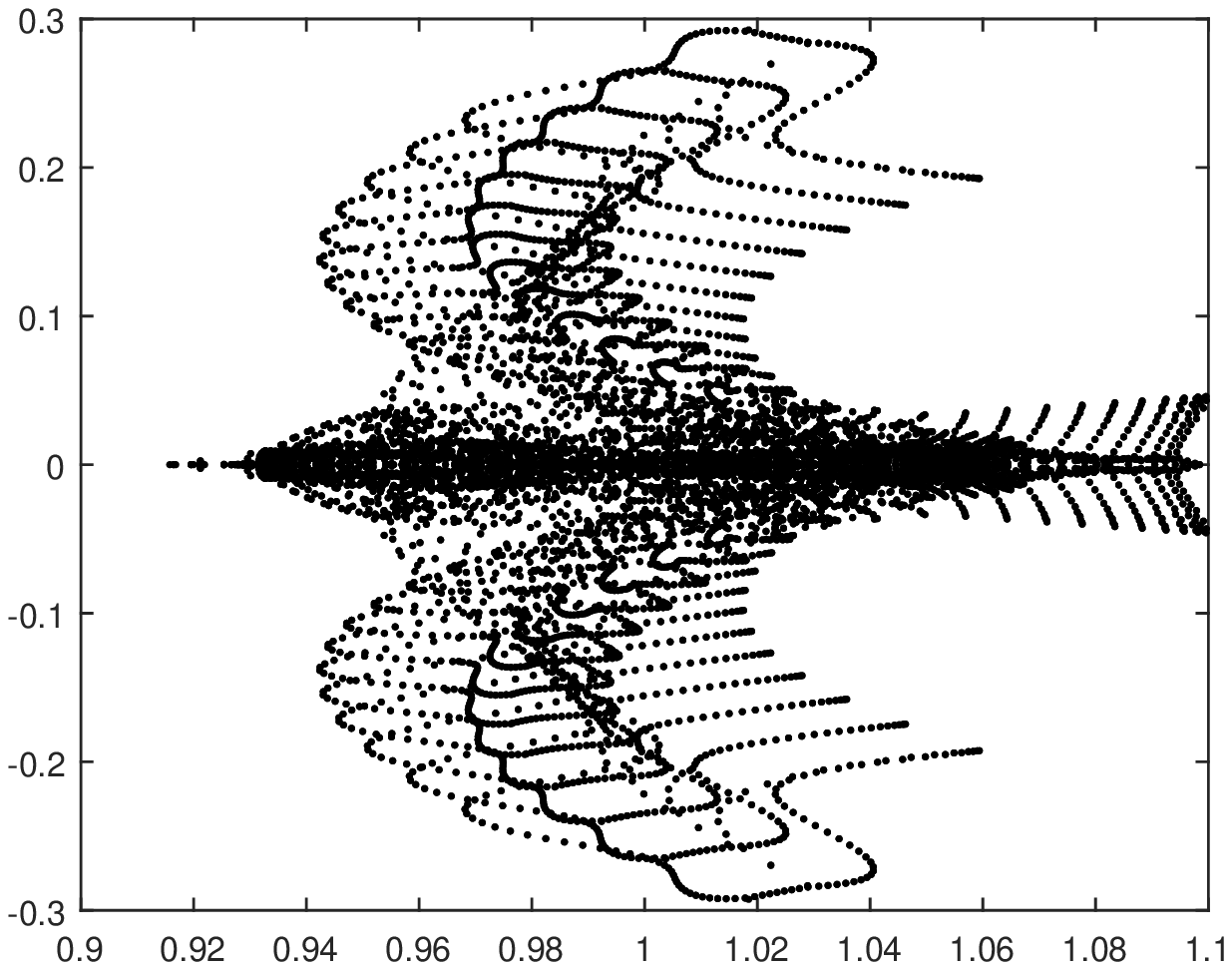}}}%
    \qquad
    \subfloat[Ιδιάζουσες τιμές.]{{\label{fig:spectra_BTTBb}\includegraphics[width=0.45\linewidth]{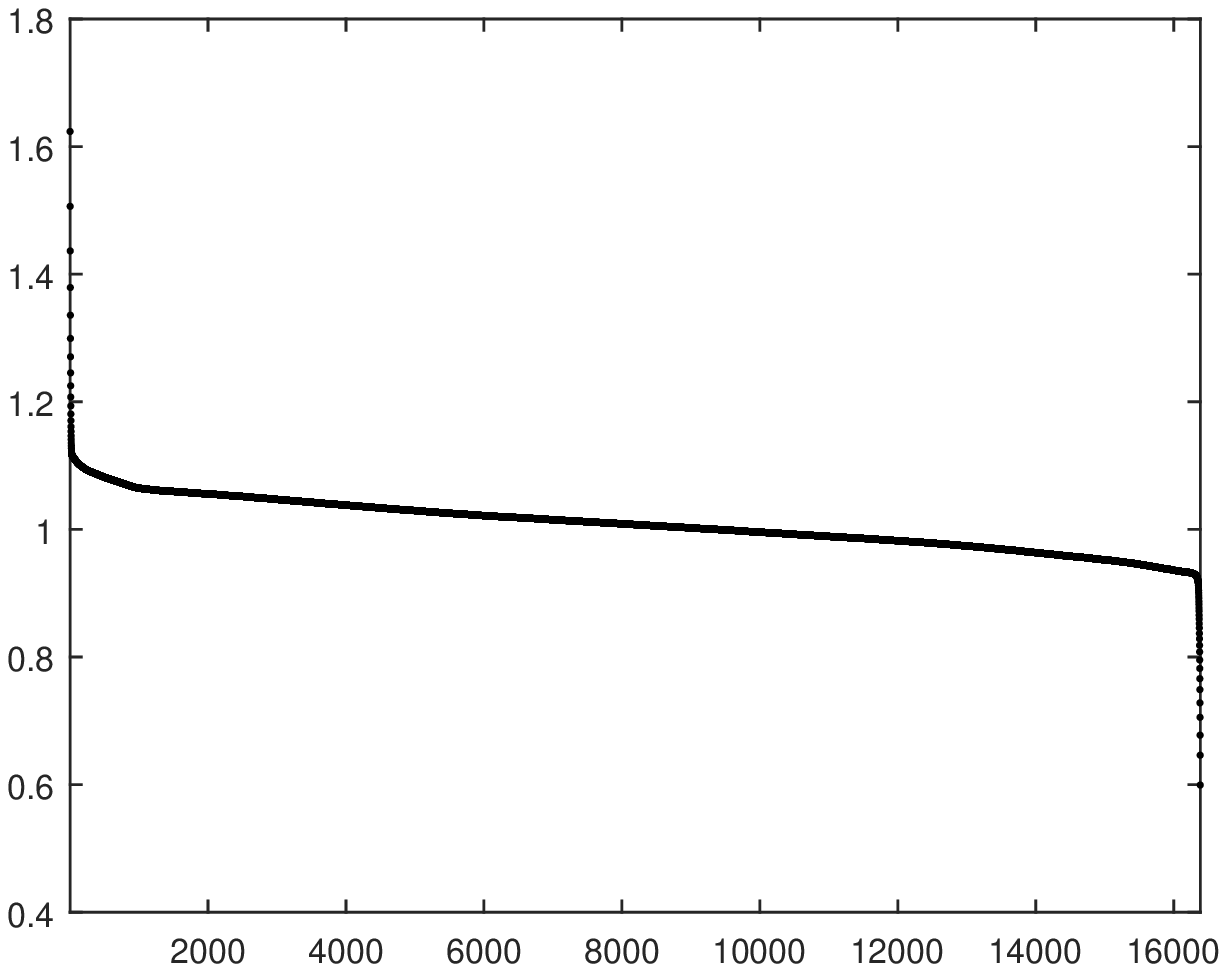}}}%
    \caption{Ιδιοτιμές και ιδιάζουσες τιμές ($\mathfrak{f}_6$).}%
    \label{fig:spectra_BTTB}%
\end{figure}

Το Σχήμα \ref{fig:spectra_BTTBa} δείχνει τη συσσώρευση των ιδιοτιμών του προρρυθμισμένου συστήματος, όταν $n=m=128$ και το Σχήμα \ref{fig:spectra_BTTBb}, την αντίστοιχη συσσώρευση των ιδιαζουσών τιμών.
\end{exmp}

\newpage\thispagestyle{empty}\mbox{}\newpage 

\pagestyle{main}

\chapter{Κυκλοειδείς Προρρυθμιστές}

Μια ευρέως γνωστή κατηγορία προρρυθμιστών αποτελούν οι κυκλοειδείς πίνακες \cite{davis}. `Ενας πίνακας καλείται κυκλοειδής όταν είναι {\en Toeplitz} και κάθε επόμενη γραμμή/στήλη αυτού προκύπτει από μια κυκλική μετατόπιση της προηγούμενης. `Ετσι οι κυκλοειδείς πίνακες έχουν την παρακάτω μορφή:
\begin{equation*}
C_{n}=
\begin{pmatrix}
c_{0} & c_{-1} & c_{-2} & \cdots & c_{2} & c_{1} \\
c_{1} & c_{0} & c_{-1} & \cdots & c_{3} & c_{2} \\
c_{2} & c_{1} & c_{0} & \cdots & c_{4} & c_{3} \\
\vdots & \vdots & \vdots & \ddots & \vdots & \vdots \\
c_{-2} & c_{-3} & c_{-4} & \cdots & c_{0} & c_{-1} \\
c_{-1} & c_{-2} & c_{-3} & \cdots & c_{1} & c_{0}
\end{pmatrix}.
\end{equation*}

Στη βιβλιογραφία κάποιος μπορεί να βρει πολλά είδη κυκλοειδών προρρυθμιστών για συστήματα {\en Toeplitz}, όπως είναι για παράδειγμα αυτός που εισήγαγε ο {\en G.~Strang} το 1986, ο βέλτιστος κυκλοειδής προρρυθμιστής που πρότεινε ο {\en T.~Chan} το 1988 κ.α. Οι προρρυθμιστές του {\en G.~Strang} και του {\en T.~Chan}, κατασκευάζονται από τις τιμές του αρχικού πίνακα {\en Toeplitz} σύμφωνα με τις παρακάτω σχέσεις, αντίστοιχα:
\begin{equation}\label{eq:strang}
s_{k}=
\begin{cases}
t_{k}, &0\leq k\leq \lfloor n/2 \rfloor \\
t_{k-n}, &\lfloor n/2 \rfloor<k\leq n-1 \\
s_{n+k}, &0<-k\leq n-1
\end{cases},
\end{equation}

\begin{equation*}
c_{k}=\begin{cases}
\frac{\left(n-k\right)t_{k}+k t_{k-n}}{n}, &0\leq k\leq n-1\\
c_{n+k}, &0<-k\leq n-1
\end{cases}.
\end{equation*}

Βλέπουμε ότι ο κυκλοειδής προρρυθμιστής του {\en G.~Strang} κατασκευάζεται αναδιπλώνοντας τις κεντρικές διαγωνίους του πίνακα των συντελεστών $T_{n}$, σύμφωνα με τη σχέση (\ref{eq:strang}). Σχολιάζουμε επίσης ότι ο κυκλοειδής προρρυθμιστής του {\en T.~Chan} ονομάζεται βέλτιστος επειδή ελαχιστοποιεί τη νόρμα {\en Frobenius} $\Vert C_{n}-T_{n}\Vert_{F}$, ως προς οποιονδήποτε κυκλοειδή πίνακα $C_{n}$.

`Ολοι οι κυκλοειδείς προρρυθμιστές έχουν μια κοινή χαρακτηριστική ιδιότητα, να διαγωνοποιούνται από τον πίνακα διακριτού μετασχηματισμού {\en Fourier}, ο οποίος δίνεται ως:
\begin{equation*}
\mathcal{F}_n=\frac{1}{\sqrt{n}}\begin{pmatrix}
1 & 1 & 1 & \cdots & 1 \\
1 & \omega & \omega^{2} & \cdots & \omega^{n-1} \\
1 & \omega^{2} & \omega^{4} & \cdots & \omega^{2(n-1)} \\
\vdots & \vdots & \vdots & \ddots & \vdots \\
1 & \omega^{n-1} & \omega^{2(n-1)} & \cdots & \omega^{(n-1)^{2}}
\end{pmatrix},
\end{equation*}
όπου $\omega=\mathrm{e}^{2\pi\mathrm{i}/n}$ και προφανώς $\omega^{j}$, $j=0,\dots,n-1$ είναι οι διακεκριμένες λύσεις της μιγαδικής εξίσωσης $z^{n}-1=0$, δηλαδή οι $n$-οστές ρίζες της μονάδας. Αναφέρουμε ότι ο πίνακας $\mathcal{F}_{n}$ είναι ορθομοναδιαίος (επομένως και αντιστρέψιμος) και συμμετρικός \cite{Chan_Jin,kra,zhan}. Λόγω της προαναφερθείσας ιδιότητας μπορούμε να γράψουμε έναν οποιονδήποτε κυκλοειδή προρρυθμιστή $C_n$, ως:
\begin{equation*}
C_n=\mathcal{F}_n^H\Lambda_n\mathcal{F}_n,
\end{equation*}
όπου $\Lambda_n$ είναι διαγώνιος πίνακας, ο οποίος έχει ως στοιχεία της κυρίας διαγωνίου του τις ιδιοτιμές του $C_n$. Ως εκ τούτου, είναι προφανές ότι μπορούμε να ορίσουμε έναν κυκλοειδή προρρυθμιστή με ιδιοτιμές της αρεσκείας μας. Λόγω του ότι θέλουμε να προρρυθμίσουμε ένα σύστημα {\en Toeplitz}, το οποίο έχει ως γεννήτρια συνάρτηση την $f=f_1+\mathrm{i}f_2$ (με τις ιδιότητες που αναφέραμε στο εισαγωγικό κεφάλαιο), είναι συνετό να κατασκευάσουμε τον κυκλοειδή προρρυθμιστή:
\begin{equation*}
C_n(f)=\mathcal{F}_n^H\Lambda_n(f)\mathcal{F}_n,
\end{equation*}
ο οποίος έχει ως ιδιοτιμές τις τιμές που λαμβάνει η γεννήτρια συνάρτηση $f$ στους κόμβους:
\begin{equation}\label{eq:circulant}
\frac{2(j-1)\pi}{n},~j=1,2,\dots,n.
\end{equation}

Προφανώς, λόγω περιοδικότητας της γεννήτριας συνάρτησης τα διαστήματα $(-\pi,\pi]$ και $[0,2\pi)$ περιέχουν τις ίδιες τιμές, για τους αντίστοιχους κόμβους.

\section{Κατασκευή του προρρυθμιστή}
\label{3S:2}

`Εστω $f=f_1+{\rm i}f_2$, όπου $f_1$ είναι άρτια, $2\pi$-περιοδική συνάρτηση και $f_2$ περιττή κι επίσης $2\pi$-περιοδική, ορισμένες στο $\left(-\pi,\pi\right]$. Είναι προφανές, όπως είδαμε και στο προηγούμενο κεφάλαιο, ότι η $f_2$ έχει ρίζα στο 0, ως περιττή συνάρτηση. Αυτή η ιδιότητα δεν ισχύει πάντα για την $f_1$, η οποία μπορεί να μην έχει και καμία ρίζα στο διάστημα $(-\pi,\pi]$. Επομένως, θα διακρίνουμε και πάλι δύο περιπτώσεις προρρύθμισης του αρχικού συστήματος, βάσει της φύσης της $f$, δηλαδή το αν έχει ρίζες ή όχι.

\subsection{Συστήματα με καλή κατάσταση}\label{3Ss:2.1}

Θα ξεκινήσουμε την περιγραφή κατασκευής του προρρυθμιστή για συστήματα με καλή κατάσταση, δηλαδή συστήματα των οποίων η γεννήτρια συνάρτηση $f$, δεν έχει ρίζες στο πεδίο ορισμού της. Σε αυτή την περίπτωση, ο προτεινόμενος προρρυθμιστής είναι ο κυκλοειδής πίνακας $\mathcal{C}_n(f)$, ο οποίος έχει ως ιδιοτιμές, τις τιμές $f\left(\frac{2\pi(k-1)}{n}\right)$, $k=1,\dots,n$.

Σημειώνουμε ότι η $f$ μπορεί να μην έχει ρίζες, ωστόσο δεν είναι απαραίτητο για την $f_1$ να είναι θετική. Αυτό είναι προφανές αφού οι συναρτήσεις $f_1$ και $f_2$, μπορούν να μηδενίζονται σε διαφορετικά σημεία. Σημειώνουμε επίσης ότι η χρήση του $\mathcal{C}_n(f)$ ως προρρυθμιστή είναι επιτρεπτή, διότι είναι αντιστρέψιμος πίνακας. Η κατασκευή αυτού, γίνεται με $\mathcal{O}(n\log{n})$ πράξεις, μέσω του ταχύ μετασχηματισμού {\en Fourier (FFT)} \cite{Strang}.

\subsection{Συστήματα με κακή κατάσταση}
\label{3Ss:2.2}

Συνεχίζουμε με την περίπτωση όπου η $f$ έχει ρίζες. Σε αυτή, προτείνουμε την εύρεση ενός τριγωνομετρικού πολυωνύμου $g$, έτσι ώστε το πραγματικό μέρος της $\frac{f}{g}$ να είναι θετικό, όπως ακριβώς και στην τεχνική προρρύθμισης του προηγούμενου κεφαλαίου. Η μορφή αυτού δόθηκε στην εργασία \cite{NT_2019} (βλ. ενότητα \ref{S:3}). Ο προτεινόμενος προρρυθμιστής γι αυτή την περίπτωση είναι ο πίνακας $T_n(g)\mathcal{C}_n\left(\frac{f}{g}\right)$, όπου $\mathcal{C}_n\left(\frac{f}{g}\right)$ είναι ο κυκλοειδής πίνακας, που έχει ιδιοτιμές ίσες με $\frac{f}{g}\left(\frac{2\pi(k-1)}{n}\right)$, $k=1,\dots,n$ και $T_n(g)$ είναι ο ταινιωτός πίνακας {\en Toeplitz}, ο οποίος έχει ως γεννήτρια συνάρτηση το τριγωνομετρικό πολυώνυμο $g$. Η κατασκευή του προρρυθμιστή γίνεται επίσης γρήγορα, σε $\mathcal{O}(n\log{n})$ πράξεις, χρησιμοποιώντας και πάλι τον ταχύ μετασχηματισμό {\en Fourier}. Σημειώνουμε ότι ο $T_n(g)\mathcal{C}_n\left(\frac{f}{g}\right)$ μπορεί να χρησιμοποιηθεί ως προρρυθμιστής και στην περίπτωση όπου η $f$ δεν έχει ρίζες, αλλά αντιθέτως η $f_1$ μηδενίζεται στο πεδίο ορισμού της.

\section{Θεωρητικά αποτελέσματα}
\label{3S:3}

Σε αυτή την ενότητα δίνουμε τα θεωρητικά αποτελέσματα, που αφορούν στη συσσώρευση των ιδιοτιμών και ιδιαζουσών τιμών του προρρυθμισμένου συστήματος. Αυτά αποτελούν ικανές συνθήκες για την ταχεία σύγκλιση των μεθόδων {\en PGMRES} και {\en PCGN}, με χρήση του προτεινόμενου προρρυθμιστή. Αρχικά θα μελετήσουμε την περίπτωση όπου η γεννήτρια συνάρτηση του συστήματος είναι συνεχής και στη συνέχεια θα επεκταθούμε και στην περίπτωση όπου αυτή έχει σημεία ασυνέχειας.

Επειδή $T_n(f)=T_n(f_1)+\mathrm{i}T_n(f_2)$, είναι ευρέως γνωστό \cite{chan1993circulant} ότι:
\begin{equation}\label{eq:norm2T}
\Vert T_n(f)\Vert_2\leq 2\Vert f\Vert_\infty.
\end{equation}

Συμβολίζουμε με $\lambda_k(\mathcal{C}_n(f)),~k=1,\dots,n$, τις ιδιοτιμές του κυκλοειδή πίνακα $\mathcal{C}_n(f)$. Από την κατασκευή αυτού, γνωρίζουμε ότι ισχύει:
\begin{equation*}
\lambda_k(\mathcal{C}_n(f))=f\left(\frac{2\pi(k-1)}{n}\right),~k=1,\dots,n.
\end{equation*}

\begin{lem}\label{Lem:1}
`Εστω μια $2\pi$-περιοδική και μιγαδική συνάρτηση $f$. Τότε:
\begin{equation}\label{eq:norm2C}
\Vert \mathcal{C}_n(f)\Vert_2\leq 2\Vert f\Vert_\infty.
\end{equation}
Επιπλέον, αν η $f$ δεν έχει ρίζες στο $(-\pi,\pi]$, ισχύει ότι:
\begin{equation}\label{eq:norm2invC}
\Vert \mathcal{C}^{-1}_n(f)\Vert_2\leq2\left\Vert\frac{1}{f}\right\Vert_\infty.
\end{equation}
\end{lem}

\begin{proof}
Παρατηρούμε ότι $\mathcal{C}_n(f)=\mathcal{C}_n(f_1)+\mathrm{i}\mathcal{C}_n(f_2)$, όπου $\mathcal{C}_n(f_1)$ και $\mathcal{C}_n(f_2)$ είναι Ερμιτιανοί πίνακες. Εύκολα βλέπουμε ότι:
\begin{equation*}
\Vert \mathcal{C}_n(f_1)\Vert_2=\displaystyle\max_k\vert\lambda_k(\mathcal{C}_n(f_1))\vert=\displaystyle\max_k\left\vert f_1\left(\frac{2\pi(k-1)}{n}\right)\right\vert \leq \Vert f_1\Vert_\infty,
\end{equation*}
\begin{equation*}
\Vert \mathcal{C}_n(f_2)\Vert_2=\displaystyle\max_k\vert\lambda_k(\mathcal{C}_n(f_2))\vert=\displaystyle\max_k \left\vert f_2\left(\frac{2\pi(k-1)}{n}\right)\right\vert \leq \Vert f_2\Vert_\infty.
\end{equation*}
`Ετσι, λαμβάνουμε την παρακάτω σχέση:
\begin{equation*}
\Vert \mathcal{C}_n(f)\Vert_2\leq \Vert \mathcal{C}_n(f_1)\Vert_2+\Vert \mathcal{C}_n(f_2)\Vert_2\leq \Vert f_1\Vert_\infty + \Vert f_2\Vert_\infty \leq 2\Vert f\Vert_\infty.
\end{equation*}
Για το άνω φράγμα της $\Vert \mathcal{C}^{-1}_n(f)\Vert_2$, έχουμε:
\begin{equation*}
\Vert \mathcal{C}^{-1}_n(f)\Vert_2=\left\Vert \mathcal{C}_n\left(\frac{1}{f}\right)\right\Vert_2\leq2\left\Vert\frac{1}{f}\right\Vert_\infty.
\end{equation*}
Σημειώνουμε ότι αν η $f$ δεν έχει ρίζες στο $(-\pi,\pi]$: $\vert f(x)\vert\neq 0$, $\forall x\in(-\pi,\pi]$.
\end{proof}

\subsection{Συνεχής περίπτωση}

\begin{thm}\label{Thm:1}
`Εστω $f$ μια $2\pi$-περιοδική και συνεχής, μιγαδική συνάρτηση. Τότε, για κάθε $\varepsilon>0$, υπάρχει σταθερά $M$, τέτοια ώστε για κάθε $n>2M$, $T_n(f)-\mathcal{C}_n(f)=S_n+L_n$, όπου $\Vert S_n\Vert_2\leq\varepsilon$ και ο πίνακας $L_n$ έχει βαθμίδα το πολύ ίση με $2M$.
\end{thm}

\begin{proof}
`Εστω $f$ η συνάρτηση με τις ιδιότητες που περιγράψαμε παραπάνω. Από το θεώρημα προσέγγισης {\en Stone-Weierstrass} \cite{sohrab2003basic}, για κάποιο δεδομένο $\epsilon>0$, υπάρχει ένα τριγωνομετρικό πολυώνυμο
\begin{equation*}
p_M(x)=\displaystyle\sum_{k=-M}^M \rho_k\mathrm{e}^{\mathrm{i}kx},
\end{equation*}
τέτοιο ώστε:
\begin{equation}\label{eq:normfp}
\Vert f-p_M\Vert_\infty\leq\epsilon.
\end{equation}
Για κάθε $n>2M$:
\begin{equation*}
T_n(f)-\mathcal{C}_n(f)=T_n(f-p_M)-\mathcal{C}_n(f-p_M)+T_n(p_M)-\mathcal{C}_n(p_M).
\end{equation*}
Εύκολα παρατηρούμε ότι οι πίνακες $T_n(p_M)$ και $\mathcal{C}_n(p_M)$ διαφέρουν μόνο κατά έναν πίνακα χαμηλής βαθμίδας $L_n$, το πολύ ίσης με $2M$ \cite{potts1999preconditioners}. Χρησιμοποιώντας τις (\ref{eq:norm2T}), (\ref{eq:norm2C}) και (\ref{eq:normfp}), καταλήγουμε στο ότι για τους δύο πρώτους όρους του δεξιού μέλους, της παραπάνω εξίσωσης ισχύει:
\begin{equation*}
\begin{split}
\Vert T_n(f-p_M)-\mathcal{C}_n(f-p_M)\Vert_2&\leq\Vert T_n(f-p_M)\Vert_2+\Vert\mathcal{C}_n(f-p_M)\Vert_2\\
&\leq 2\Vert f-p_M\Vert_\infty+2\Vert f-p_M\Vert_\infty\leq 4\epsilon.
\end{split}
\end{equation*}
Επομένως, $S_n=T_n(f-p_M)-\mathcal{C}_n(f-p_M)$ είναι ένας πίνακας με μικρή νόρμα κι επιλέγοντας $\epsilon=\frac{\varepsilon}{4}$ λαμβάνουμε το ζητούμενο αποτέλεσμα.
\end{proof}

`Οπως αποδείχθηκε στο Λήμμα \ref{Lem:1}, όταν η γεννήτρια συνάρτηση $f$ δεν έχει ρίζες στο $(-\pi,\pi]$, ισχύει η σχέση (\ref{eq:norm2invC}). Συνδυάζοντάς την με το Θεώρημα \ref{Thm:1} και το γεγονός ότι
\begin{equation*}
\mathcal{C}_n^{-1}(f)T_n(f)-I_n=\mathcal{C}_n^{-1}(f)(T_n(f)-\mathcal{C}_n(f))=\mathcal{C}_n^{-1}(f)S_n+\mathcal{C}_n^{-1}(f)L_n,
\end{equation*}
μπορούμε να δώσουμε το παρακάτω πόρισμα.

\begin{cor}\label{cor:S_and_L}
`Εστω $f$ μια $2\pi$-περιοδική και συνεχής, μιγαδική συνάρτηση, η οποία δεν έχει ρίζες στο $(-\pi,\pi]$. Τότε, για κάθε $\varepsilon>0$, υπάρχει σταθερά $M$, τέτοια ώστε για κάθε $n>2M$, $\mathcal{C}_n^{-1}(f)T_n(f)-I_n=\widehat{S}_n+\widehat{L}_n$, όπου $\Vert\widehat{S}_n\Vert_2\leq\varepsilon$ και ο πίνακας $\widehat{L}_n$ έχει βαθμίδα το πολύ ίση με $2M$.
\end{cor}

\begin{thm}\label{Thm:2}
`Εστω $f$ μια $2\pi$-περιοδική, συνεχής και μιγαδική συνάρτηση, η οποία δεν έχει ρίζες στο $(-\pi,\pi]$. Τότε, για κάθε $\varepsilon>0$, το διάστημα $[1-\varepsilon,1+\varepsilon]$ αποτελεί ένα σύνολο κύριας συσσώρευσης των ιδιαζουσών τιμών του $\mathcal{C}_n^{-1}(f)T_n(f)$. 
\end{thm}

\begin{proof}
Είναι ευρέως γνωστό ότι οι ιδιάζουσες τιμές του $\mathcal{C}_n^{-1}(f)T_n(f)$ είναι οι τετραγωνικές ρίζες των ιδιοτιμών του πίνακα που σχετίζεται με τις κανονικές εξισώσεις, δηλαδή του $(\mathcal{C}_n^{-1}(f)T_n(f))^H\mathcal{C}_n^{-1}(f)T_n(f)$.

Για να μελετήσουμε τη συσσώρευση των ιδιοτιμών του προαναφερθέντος πίνακα, ακολουθούμε την απόδειξη του Θεωρήματος 2, της \cite{chan1993circulant} και καταλήγουμε στο ότι το πολύ $4M$ ιδιοτιμές του $(\mathcal{C}_n^{-1}(f)T_n(f))^H\mathcal{C}_n^{-1}(f)T_n(f)-I_n$, έχουν απόλυτη τιμή μεγαλύτερη του $\varepsilon$. Το τελευταίο αποτέλεσμα είναι ισοδύναμο με την κύρια συσσώρευση των ιδιαζουσών τιμών του προρρυθμισμένου συστήματος $\mathcal{C}_n^{-1}(f)T_n(f)$ στο $[1-\varepsilon,1+\varepsilon]$.
\end{proof}

Παρακάτω, με $\mathcal{T}_n$ συμβολίζουμε τον βέλτιστο κυκλοειδή προρρυθμιστή που προτάθηκε από τον {\en T.~Chan} στην \cite{chan1988optimal}.

\begin{thm}\label{sing val_cont}
`Εστω $f$ μια $2\pi$-περιοδική και συνεχής, μιγαδική συνάρτηση, η οποία έχει ρίζες στο $(-\pi,\pi]$. `Εστω επίσης $g$ ένα τριγωνομετρικό πολυώνυμο τέτοιο ώστε η $\frac{f}{g}$ να μην έχει ρίζες στο $(-\pi,\pi]$. Τότε, για κάθε $\varepsilon>0$, το διάστημα $[1-\varepsilon,1+\varepsilon]$ αποτελεί σύνολο κύριας συσσώρευσης των ιδιαζουσών τιμών του $\mathcal{P}_n^{-1}\left(\frac{f}{g}\right)T_n^{-1}(g)T_n(f)$, όπου $\mathcal{P}_n\left(\frac{f}{g}\right)$ είναι είτε ο κυκλοειδής πίνακας $\mathcal{C}_n\left(\frac{f}{g}\right)$, είτε ο βέλτιστος κυκλοειδής προρρυθμιστής $\mathcal{T}_n$, που προέκυψε από τον $T_n\left(\frac{f}{g}\right)$.
\end{thm}

\begin{proof}
Παρακάτω δίνουμε την απόδειξη για την περίπτωση όπου $\mathcal{P}_n\left(\frac{f}{g}\right)=\mathcal{C}_n\left(\frac{f}{g}\right)$. Η απόδειξη για τη χρήση του $\mathcal{T}_n$ ως προρρυθμιστή θα είναι ακριβώς ίδια, λαμβάνοντας υπόψιν το Πόρισμα 1 της \cite{chan1993circulant}, το οποίο είναι το αντίστοιχο του Πορίσματος \ref{cor:S_and_L}.

`Εστω $g$ το τριγωνομετρικό πολυώνυμο που έχει βαθμό $\operatorname{deg}(g)=d$. Τότε, ο $T_n(g)$ θα είναι ένας ταινιωτός πίνακας {\en Toeplitz} με πλάτος ταινίας $2d+1$. Ακολουθώντας την ίδια τεχνική με το Θεώρημα \ref{Thm:2} και το Θεώρημα \ref{thm:gen_cluster}, θα αναλύσουμε το αντίστοιχο προρρυθμισμένο σύστημα των κανονικών εξισώσεων. `Εχουμε:
\begin{equation*}
\begin{split}
&\mathcal{C}_n^{-1}\left(\frac{f}{g}\right)T_n^{-1}(g)T_n(f)T_n(\bar{f})T_n^{-1}(\bar{g})\mathcal{C}_n^{-1}\left(\frac{\bar{f}}{\bar{g}}\right)=\\
&\mathcal{C}_n^{-1}\left(\frac{f}{g}\right)T_n^{-1}(g)\left[T_n(g)T_n\left(\frac{f}{g}\right)+L_{1}\right]\\
&\phantom{\mathcal{C}_n^{-1}\left(\frac{f}{g}\right)T_n^{-1}(g)}
\left[T_n\left(\frac{\bar{f}}{\bar{g}}\right)T_n(\bar{g})+L_{1}^H\right]T_n^{-1}(\bar{g})\mathcal{C}_n^{-1}\left(\frac{\bar{f}}{\bar{g}}\right)=\\
&\left[\mathcal{C}_n^{-1}\left(\frac{f}{g}\right)T_n\left(\frac{f}{g}\right)+L_2\right]\left[T_n\left(\frac{\bar{f}}{\bar{g}}\right)\mathcal{C}_n^{-1}\left(\frac{\bar{f}}{\bar{g}}\right)+L_2^H\right]=\\
&\mathcal{C}_n^{-1}\left(\frac{f}{g}\right)T_n\left(\frac{f}{g}\right)T_n\left(\frac{\bar{f}}{\bar{g}}\right)\mathcal{C}_n^{-1}\left(\frac{\bar{f}}{\bar{g}}\right)+L_3.
\end{split}
\end{equation*}
Οι πίνακες $L_1$ και $L_1^H$ παραπάνω είναι χαμηλής βαθμίδας, το πολύ ίσης με $2d$ ($d$ είναι ο βαθμός του $g$). Το ίδιο ισχύει και για τους $L_2$ και $L_2^H$. `Ετσι, καταλήγουμε στο ότι ο $L_3$ είναι Ερμιτιανός πίνακας με βαθμίδα το πολύ ίση με $4d$.

`Εστω ότι $p_M$ είναι ένα τριγωνομετρικό πολυώνυμο προσέγγισης της $\frac{f}{g}$ (όπως στο Θεώρημα \ref{Thm:1}). Επειδή η $\frac{f}{g}$ δεν έχει ρίζες, από το Θεώρημα \ref{Thm:2} έχουμε ότι $\forall\varepsilon>0$, το πολύ $4M$ ιδιάζουσες τιμές του $\mathcal{C}_n^{-1}(\frac{f}{g})T_n(\frac{f}{g})$ θα κυμαίνονται εκτός του $[1-\varepsilon,1+\varepsilon]$. Λαμβάνοντας υπόψιν και αυτές που κυμαίνονται εκτός του διαστήματος, λόγω του πίνακα $L_3$, έχουμε ότι $\forall\varepsilon>0$, το πολύ $4M+4d$ ιδιάζουσες τιμές του προρρυθμισμένου συστήματος $\mathcal{C}_n^{-1}(\frac{f}{g})T_n^{-1}(g)T_n(f)$ θα κυμαίνονται εκτός του διαστήματος $[1-\varepsilon,1+\varepsilon]$ κι έτσι η απόδειξη ολοκληρώνεται.
\end{proof}

\begin{thm}\label{Thm:5}
`Εστω $f$ μια $2\pi$-περιοδική, συνεχής και μιγαδική συνάρτηση, η οποία δεν έχει ρίζες στο $(-\pi,\pi]$. Τότε, οι ιδιοτιμές του $\mathcal{C}_n^{-1}(f)T_n(f)$ συσσωρεύονται, με την έννοια της κύριας συσσώρευσης, γύρω από το σημείο $(1,0)$. 
\end{thm}

\begin{proof}
Γνωρίζουμε ότι ένας πίνακας $A$ μπορεί να γραφεί ως το άθροισμα του Ερμιτιανού του και αντι-Ερμιτιανού του μέρους, $A=H(A)+SH(A)=\frac{A+A^H}{2}+\frac{A-A^H}{2}$. Για να μελετήσουμε τη συσσώρευση των ιδιοτιμών του $A$, μπορούμε να μελετήσουμε τη συσσώρευση των ιδιοτιμών του $H(A)$, καθώς και του $SH(A)$, λαμβάνοντας υπόψιν ότι $\operatorname{Re}(\lambda(A))\in\operatorname{range}(H(A))$ και $\operatorname{Im}(\lambda(A))\in\operatorname{range}(SH(A))$ \cite{bendixson,garren1968bounds,hirsch}.

Από το Θεώρημα \ref{Thm:1}, μπορούμε να αντικαταστήσουμε το $T_n(f)$ με $\mathcal{C}_n(f)+S_n+L_n$, όπου $\Vert S_n\Vert_2<\epsilon$ και $\operatorname{rank}(L_n)\leq2M$. Επομένως, 
\begin{equation*}
\mathcal{C}_n^{-1}(f)T_n(f)=\mathcal{C}_n^{-1}(f)\left[\mathcal{C}_n(f)+S_n+L_n\right]=I_n+\mathcal{C}_n^{-1}(f)S_n+\mathcal{C}_n^{-1}(f)L_n.
\end{equation*}
Το Ερμιτιανό του μέρος είναι:
\begin{equation*}
\begin{split}
H(\mathcal{C}_n^{-1}(f)T_n(f))&=H(I_n+\mathcal{C}_n^{-1}(f)S_n+\mathcal{C}_n^{-1}(f)L_n)\\
&=I_n+H(\mathcal{C}_n^{-1}(f)S_n)+H(\mathcal{C}_n^{-1}(f)L_n).
\end{split}
\end{equation*}
Σύμφωνα με την ανάλυση του Θεωρήματος \ref{Thm:1}, ο τελευταίος όρος είναι ένας πίνακας χαμηλής βαθμίδας, το πολύ ίσης με $4M$ και σχετίζεται με το ότι το πολύ $4M$ ιδιοτιμές του προρρυθμισμένου πίνακα έχουν πραγματικό μέρος εκτός του διαστήματος συσσώρευσης. Μένει να μελετήσουμε το φάσμα του $H(\mathcal{C}_n^{-1}(f)S_n)$.
\begin{equation*}
\begin{split}
\vert\lambda(H(\mathcal{C}_n^{-1}(f)S_n))\vert&=\left\vert\lambda\left(\frac{\mathcal{C}_n^{-1}(f)S_n+S_n^H\mathcal{C}^{-1}(\bar{f})}{2}\right)\right\vert\\
&\leq\left\Vert\frac{\mathcal{C}_n^{-1}(f)S_n+S_n^H\mathcal{C}^{-1}(\bar{f})}{2}\right\Vert_2\\
&\leq\left\Vert\mathcal{C}_n^{-1}(f)S_n\right\Vert_2\leq\left\Vert\mathcal{C}_n^{-1}(f)\right\Vert_2\epsilon\leq2\left\Vert\frac{1}{f}\right\Vert_\infty\epsilon.
\end{split}
\end{equation*}
Επιλέγοντας ως $\varepsilon=2\left\Vert\frac{1}{f}\right\Vert_\infty\epsilon$, έχουμε ότι τα πραγματικά μέρη των ιδιοτιμών, του προρρυθμισμένου πίνακα, συσσωρεύονται στο $[1-\varepsilon,1+\varepsilon]$.

Το αντι-Ερμιτιανό μέρος του προρρυθμισμένου πίνακα είναι:
\begin{equation*}
SH(\mathcal{C}_n^{-1}(f)T_n(f))=SH(\mathcal{C}_n^{-1}(f)S_n)+SH(\mathcal{C}_n^{-1}(f)L_n).
\end{equation*}
Μέσω της ίδιας ανάλυσης λαμβάνουμε ότι ο τελευταίος όρος σχετίζεται με το ότι το πολύ $4M$ ιδιοτιμές του προρρυθμισμένου πίνακα έχουν φανταστικό μέρος εκτός του διαστήματος συσσώρευσης. Μελετώντας ανάλογα το $SH(\mathcal{C}_n^{-1}(f)S_n)$ έχουμε:
\begin{equation*}
\vert\lambda(SH(\mathcal{C}_n^{-1}(f)S_n))\vert\leq\left\Vert\mathcal{C}_n^{-1}(f)S_n\right\Vert_2\leq2\left\Vert\frac{1}{f}\right\Vert_\infty\epsilon.
\end{equation*}
Επομένως, τα φανταστικά μέρη των ιδιοτιμών του προρρυθμισμένου πίνακα συσσωρεύονται στο $[-\varepsilon,\varepsilon]$, και το θεώρημα αποδείχθηκε.
\end{proof}

\begin{thm}\label{thm:Theorem 6}
`Εστω $f$ μια $2\pi$-περιοδική, συνεχής και μιγαδική συνάρτηση, η οποία έχει ρίζες στο $(-\pi,\pi]$. `Εστω επίσης $g$ ένα τριγωνομετρικό πολυώνυμο τέτοιο ώστε η $\frac{f}{g}$ να μην έχει ρίζες στο $(-\pi,\pi]$. Τότε, οι ιδιοτιμές του $\mathcal{C}_n^{-1}\left(\frac{f}{g}\right)T_n^{-1}(g)T_n(f)$ συσσωρεύονται, με την έννοια της κύριας συσσώρευσης, γύρω από το σημείο $(1,0)$.
\end{thm}

\begin{proof}
`Εστω ότι $h$ συμβολίζει το πηλίκο {\en Rayleigh} του Ερμιτιανού μέρους του προρρυθμισμένου συστήματος $\mathcal{C}_n^{-1}\left(\frac{f}{g}\right)T_n^{-1}(g)T_n(f)$. `Εχουμε:
\begin{equation*}
\begin{split}
h&=\frac{1}{2}\frac{x^H\left[\mathcal{C}_n^{-1}\left(\frac{f}{g}\right)T_n^{-1}(g)T_n(f)+T_n(\bar{f})T_n^{-1}(\bar{g})\mathcal{C}_n^{-1}\left(\frac{\bar{f}}{\bar{g}}\right)\right]x}{x^Hx}\\
&=\frac{1}{2}\frac{x^H \mathcal{C}_n^{-1}\left(\frac{f}{g}\right)T_n^{-1}(g)\left[T_n(g)T_n\left(\frac{f}{g}\right)+L_{1}\right]x}{x^Hx}\\
&\phantom{hspace{2pc}}+\frac{1}{2}\frac{x^H\left[T_n\left(\frac{\bar{f}}{\bar{g}}\right)T_n(\bar{g})+L_{1}^H\right]T_n^{-1}(\bar{g})\mathcal{C}_n^{-1}\left(\frac{\bar{f}}{\bar{g}}\right)x}{x^Hx}\\
&=\frac{1}{2}\frac{x^H\left[\mathcal{C}_n^{-1}\left(\frac{f}{g}\right)T_n\left(\frac{f}{g}\right)+T_n\left(\frac{\bar{f}}{\bar{g}}\right)\mathcal{C}_n^{-1}\left(\frac{\bar{f}}{\bar{g}}\right)\right]x}{x^Hx}+\frac{1}{2}\frac{x^H L_2 x}{x^Hx},
\end{split}
\end{equation*}
όπου $L_2=\mathcal{C}_n^{-1}\left(\frac{f}{g}\right)T_n^{-1}(g)L_1+L_1^HT_n^{-1}(\bar{g})\mathcal{C}_n^{-1}\left(\frac{\bar{f}}{\bar{g}}\right)$ είναι ένας Ερμιτιανός πίνακας χαμηλής βαθμίδας, το πολύ ίσης με $4d$ ($d$ είναι ο βαθμός του $g$).

Το πρώτο πηλίκο {\en Rayleigh} μας δίνει το εύρος του $H\left(\mathcal{C}_n^{-1}\left(\frac{f}{g}\right)T_n\left(\frac{f}{g}\right)\right)$, για όλα τα $x\in\mathbb{R}^n$. Επομένως, τα πραγματικά μέρη των ιδιοτιμών του $\mathcal{C}_n^{-1}\left(\frac{f}{g}\right)T_n\left(\frac{f}{g}\right)$ συσσωρεύονται στο $\operatorname{range}\left(H\left(\mathcal{C}_n^{-1}\left(\frac{f}{g}\right)T_n\left(\frac{f}{g}\right)\right)\right)$. Λόγω του ότι η συνάρτηση $\frac{f}{g}$ δεν έχει ρίζες, το Θεώρημα \ref{Thm:5} ισχύει για την $\frac{f}{g}$, αντί της $f$. Αυτό σημαίνει ότι τα πραγματικά μέρη των ιδιοτιμών του $\mathcal{C}_n^{-1}\left(\frac{f}{g}\right)T_n\left(\frac{f}{g}\right)$ συσσωρεύονται στο $[1-\varepsilon,1+\varepsilon]$, με το πολύ $4M$ ιδιοτιμές εκτός του διαστήματος (συσσώρευσης). Λαμβάνοντας υπόψιν και τον πίνακα χαμηλής βαθμίδας $L_2$, συμπεραίνουμε ότι τα πραγματικά μέρη των ιδιοτιμών του προρρυθμισμένου συστήματος συσσωρεύονται στο $[1-\varepsilon,1+\varepsilon]$ κι έχουμε το πολύ $4M+4d$ ιδιοτιμές εκτός του διαστήματος.

Μελετώντας με τον ίδιο τρόπο το αντι-Ερμιτιανό μέρος του προρρυθμισμένου πίνακα $\mathcal{C}_n^{-1}\left(\frac{f}{g}\right)T_n^{-1}(g)T_n(f)$, έχουμε ότι τα φανταστικά μέρη των ιδιοτιμών του προρρυθμισμένου συστήματος συσσωρεύονται στο $[-\varepsilon,\varepsilon]$ με το πολύ $4M+4d$ ιδιοτιμές εκτός του διαστήματος.
\end{proof}

\begin{rem}
Είναι προφανές ότι τα Θεωρήματα \ref{Thm:5} και \ref{thm:Theorem 6}, ισχύουν και για τον βέλτιστο κυκλοειδή προρρυθμιστή. Οι αποδείξεις είναι ακριβώς οι ίδιες.
\end{rem}

\subsection{Κατά τμήματα συνεχής περίπτωση}

Στην προηγούμενη υποενότητα μελετήσαμε την περίπτωση όπου η $f=f_1+\mathrm{i}f_2$ είναι συνεχής συνάρτηση. `Ενα λογικό ερώτημα το οποίο γεννιέται είναι: ````Τι ισχύει όταν η $f_2$ παρουσιάζει ασυνέχεια στο $-\pi$ και $\pi$?''''. Γενικότερα, ````Τι ισχύει αν η $f$ είναι κατά τμήματα συνεχής συνάρτηση στο $(-\pi,\pi]$?''''. Θα αποδείξουμε ανάλογα θεωρήματα, προκειμένου να απαντήσουμε σε αυτά τα ερωτήματα. Από εδώ και στο εξής, υποθέτουμε ότι η $f_1$ είναι πραγματική, άρτια και $2\pi$-περιοδική συνάρτηση, η $f_2$ πραγματική, περιττή κι επίσης $2\pi$-περιοδική, και υποθέτουμε επίσης ότι η $f$ είναι κατά τμήματα συνεχής στο $(-\pi,\pi]$.

Στην υποενότητα αυτή, χρησιμοποιούμε τον $\mathcal{C}_n(f)$ ή τον $\mathcal{C}_n\left(\frac{f}{g}\right)$ (όταν χρειάζεται η άρση των ριζών) για να κατασκευάσουμε τον προρρυθμιστή. Τα θεωρητικά αποτελέσματα, τα οποία θα αποδείξουμε παρακάτω, ισχύουν και για τον βέλτιστο κυκλοειδή προρρυθμιστή $\mathcal{T}_n(f)$ ή $\mathcal{T}_n\left(\frac{f}{g}\right)$. Οι αποδείξεις μπορούν να γίνουν με τον ίδιο τρόπο, αν παρακάτω χρησιμοποιήσουμε το Λήμμα 8 της \cite{yeung1993circulant}, αντί του Λήμματος \ref{lem:g function}.

Σημειώνουμε ότι για κατά τμήματα συνεχείς, καθώς επίσης και για συνεχείς συναρτήσεις, ο βέλτιστος κυκλοειδής προρρυθμιστής $\mathcal{T}_n\left(\frac{f}{g}\right)$, σε συνδυασμό με τον ταινιωτό πίνακα {\en Toeplitz} $T_n(g)$, χρησιμοποιείται για πρώτη φορά. Όσον αφορά στην κατασκευή του προρρυθμιστή $\mathcal{T}_n\left(\frac{f}{g}\right)$, ο υπολογισμός του $T_n\left(\frac{f}{g}\right)$ απαιτείται, ενώ ο $\mathcal{C}_n\left(\frac{f}{g}\right)$ κατασκευάζεται εύκολα από τη συνάρτηση $\frac{f}{g}$.

\begin{thm}\label{thm:T-C_piecewise}
`Εστω $f$ μια συνάρτηση χωρίς ρίζες, η οποία γράφεται ως $f=f_1+\mathrm{i}f_2$, $f_1$ $2\pi$-περιοδική και άρτια, ενώ $f_2$ $2\pi$-περιοδική και περιττή συνάρτηση, έχοντας σημεία ασυνέχειας $\xi_1,\xi_2,\dots,$ $\xi_{\nu}\in(0,2\pi]$ με εύρος ασυνέχειας {\en (jump of discontinuity)}, τη φραγμένη ποσότητα:
\begin{equation*}
\alpha_k=\lim\limits_{x\rightarrow\xi_{k}^+}f(x)-\lim\limits_{x\rightarrow\xi_{k}^-}f(x).
\end{equation*}
Τότε, $\mathcal{O}\left(\log{n}\right)$ ιδιοτιμές του $\Delta_n=T_n(f)-\mathcal{C}_n(f)$ κυμαίνονται εκτός του ορθογωνίου $[-\varepsilon,\varepsilon]^2$ του μιγαδικού επιπέδου.
\end{thm}

Θα δώσουμε την απόδειξη του Θεωρήματος \ref{thm:T-C_piecewise}, αφού πρώτα δώσουμε το παρακάτω λήμμα, με την απόδειξή του.

\begin{lem}\label{lem:g function}
`Εστω $\xi$ ένα τυχαίο σημείο στο διάστημα $(0,2\pi]$ και $g$ ορισμένη ως: 
\begin{equation*}
\mathfrak{g}(x)=
\begin{cases}
x+\pi-\xi,~&0< x\leq\xi,\\
x-\pi-\xi,~&\xi<x\leq2\pi.
\end{cases}
\end{equation*}
Τότε, $T_n(\mathfrak{g})-\mathcal{C}_n(\mathfrak{g})=A_n+B_n$, όπου $\Vert A_n\Vert_F$ είναι φραγμένη από μια σταθερά ανεξάρτητη της διάστασης $n$ και $\Vert B_n\Vert_F$ τείνει στο άπειρο όπως ο $\log{n}$ $\left(\Vert B_n\Vert_F=\mathcal{O}(\log{n})\right)$.
\end{lem}

\begin{proof}
Αρχικά, είναι εύκολο να βρει κανείς \cite{yeung1993circulant} ότι οι διαγώνιοι του πίνακα $T_n(\mathfrak{g})$ δίνονται ως:
\begin{equation}\label{eq:Toeplitz_g}
t_k=
\begin{cases}
0,~&k=0,\\
\frac{\mathrm{i}}{k}\mathrm{e}^{-\mathrm{i}k\xi},~&k\neq0.
\end{cases}
\end{equation}

\begin{figure}[h]
\centering
\begin{tikzpicture}[axis/.style={very thick, ->, >=stealth'}]
\draw[axis] (0,0) --(10,0);
\draw[axis] (3,-4) --(3,4);
\draw (4,-3.14) --(3+6.28,3.14-1);
\draw[dotted] (0,-0.86) --(3,3.14-1);
\draw (3,3.14-1) --(4,3.14);
\node[below] at (-3.14+4,-0.15) {$-\pi+\xi$};
\node[below] at (2.75,-0.15) {0};
\node[below] at (4.2,-0.15) {$\xi$};
\node[below] at (3.25+4,-0.15) {$\pi+\xi$};
\node[below] at (3+6.28,-0.15) {$2\pi$};
\node[below] at (2.75,3.29) {$\pi$};
\node[below] at (2.45,3.42-1) {$\pi-\xi$};
\node[below] at (2.55,-2.9) {$-\pi$};
\draw[dashed] (3,-3.14) --(4,-3.14);
\draw[dashed] (3,3.14-1) --(3+6.28,3.14-1);
\draw[dashed] (3,3.14) --(4,3.14);
\draw[dashed] (3+6.28,3.14-1) --(3+6.28,0);
\draw[dashed] (4,3.14) --(4,-3.14);
\end{tikzpicture}
\caption{Γραφική παράσταση της $\mathfrak{g}(x)$.} \label{fig:g(x)}
\end{figure}

Χωρίς βλάβη της γενικότητας, επιτρέψτε μας να υποθέσουμε ότι $\frac{2\pi m}{n}\leq\xi<\frac{2\pi(m+1)}{n}<\pi$, όπου $m\in\mathbb{N}$. Θα υπολογίσουμε τις τιμές $c_k$ του κυκλοειδούς πίνακα $\mathcal{C}_n(\mathfrak{g})=\mathcal{F}_n^H\Delta(\mathfrak{g})\mathcal{F}_n$. `Εχουμε:
\begin{equation}\label{eq:Circulant entries}
c_k=\frac{1}{n}\displaystyle\sum\limits_{j=0}^{n-1}\mathrm{e}^{-\mathrm{i}kj2\pi/n}\mathfrak{g}\left(\frac{2\pi j}{n}\right)
\end{equation}

Για $k=0$ είναι προφανές, από τη σχέση (\ref{eq:Circulant entries}), ότι η κύρια διαγώνιος του $\mathcal{C}_n(\mathfrak{g})$ είναι το άθροισμα των ιδιοτιμών του, διαιρεμένο με $n$:
\begin{equation*}
c_0=\frac{1}{n}\displaystyle\sum\limits_{j=0}^{n-1}\mathfrak{g}\left(\frac{2\pi j}{n}\right).
\end{equation*}

Αναλυτικότερα, για τις τιμές του $\mathcal{C}_n(\mathfrak{g})$ έχουμε:
\begin{equation*}
\begin{split}
c_k&=\frac{1}{n}\displaystyle\sum\limits_{j=0}^{m}\left[\frac{2\pi\left(j+\frac{n}{2}\right)}{n}-\xi\right]\mathrm{e}^{-\mathrm{i}kj2\pi/n}+\frac{1}{n}\displaystyle\sum\limits_{j=m+1}^{n-1}\left[\frac{2\pi\left(j-\frac{n}{2}\right)}{n}-\xi\right]\mathrm{e}^{-\mathrm{i}kj2\pi/n}\\
&=\frac{1}{n}\displaystyle\sum\limits_{j=0}^{m}\left[\frac{2\pi\left(j+\frac{n}{2}\right)}{n}-\xi\right]\mathrm{e}^{-\mathrm{i}kj2\pi/n}\\
&\phantom{2pc}+\frac{1}{n}\displaystyle\sum\limits_{j=m+1-n}^{-1}\left[\frac{2\pi\left(j+\frac{n}{2}\right)}{n}-\xi\right]\mathrm{e}^{-\mathrm{i}k(j+n)2\pi/n}\\
&=\frac{1}{n}\displaystyle\sum\limits_{j=m+1-n}^{m}\left[\frac{2\pi\left(j+\frac{n}{2}\right)}{n}-\xi\right]\mathrm{e}^{-\mathrm{i}kj2\pi/n}\\
&=\frac{1}{n}\displaystyle\sum\limits_{j=0}^{n-1}\left[\frac{2\pi\left(j+m+1-\frac{n}{2}\right)}{n}-\xi\right]\mathrm{e}^{-\mathrm{i}k(j+m+1-n)2\pi/n}\\
&=\frac{1}{n}\left(\frac{2\pi\left(m+\frac{1}{2}\right)}{n}-\xi\right)\displaystyle\sum\limits_{j=0}^{n-1}\mathrm{e}^{-\mathrm{i}k(j+m+1)2\pi/n}\\
&\phantom{2pc}+\frac{1}{n}\displaystyle\sum\limits_{j=0}^{n-1}\frac{2\pi\left(j+\frac{1}{2}-\frac{n}{2}\right)}{n}\mathrm{e}^{-\mathrm{i}k(j+m+1)2\pi/n}\\
&=S_k^{(1)}+S_k^{(2)}.
\end{split}
\end{equation*}
\begin{itemize}
\item Για $k=0$, $S_0^{(1)}=\frac{1}{n}\left(\frac{2\pi\left(m+\frac{1}{2}\right)}{n}-\xi\right)\displaystyle\sum\limits_{j=0}^{n-1}1=\theta$, όπου: $$\theta=\frac{2\pi\left(m+\frac{1}{2}\right)}{n}-\xi$$ και από τον τρόπο που ορίσαμε το $m$, ισχύει $\vert\theta\vert\leq\frac{\pi}{n}$.
\item Για $k\neq 0$, $S_k^{(1)}=0$, επειδή οι όροι του αθροίσματος είναι ισαπέχοντα σημεία στον μοναδιαίο κύκλο (οι $n$-οστές ρίζες της μονάδας) και το άθροισμα αυτών είναι ίσο με μηδέν.
\end{itemize}
Τώρα θα υπολογίσουμε το άθροισμα $S_k^{(2)}$:

\begin{equation*}
\begin{split}
S_k^{(2)}&=\frac{1}{n}\displaystyle\sum\limits_{j=0}^{n-1}\frac{2\pi\left(j+\frac{1}{2}-\frac{n}{2}\right)}{n}\mathrm{e}^{-\mathrm{i}k(j+m+1)2\pi/n}\\
&=\frac{1}{n}\displaystyle\sum\limits_{j=0}^{\frac{n}{2}-1}\frac{2\pi\left(j+\frac{1}{2}-\frac{n}{2}\right)}{n}\mathrm{e}^{-\mathrm{i}k(j+m+1)2\pi/n}\\
&\phantom{2pc}+\frac{1}{n}\displaystyle\sum\limits_{j=\frac{n}{2}}^{n-1}\frac{2\pi\left(j+\frac{1}{2}-\frac{n}{2}\right)}{n}\mathrm{e}^{-\mathrm{i}k(j+m+1)2\pi/n}\\
&=\frac{1}{n}\displaystyle\sum\limits_{j=1}^{\frac{n}{2}}\frac{2\pi\left(j-\frac{1}{2}-\frac{n}{2}\right)}{n}\mathrm{e}^{-\mathrm{i}k(j+m)2\pi/n}\\
&\phantom{2pc}+\frac{1}{n}\displaystyle\sum\limits_{j=1}^{\frac{n}{2}}\frac{2\pi\left(\frac{n}{2}-j+\frac{1}{2}\right)}{n}\mathrm{e}^{-\mathrm{i}k(n-j+m+1)2\pi/n}\\
&=\frac{1}{n}\displaystyle\sum\limits_{j=1}^{n/2}\frac{2\pi\left(\frac{n}{2}-j+\frac{1}{2}\right)}{n}\left(\mathrm{e}^{-\mathrm{i}k(m+1-j)2\pi/n}-\mathrm{e}^{-\mathrm{i}k(m+j)2\pi/n}\right).
\end{split}
\end{equation*}
Από την παραπάνω ανάλυση, είναι ξεκάθαρο ότι $c_0=S_0^{(1)}+S_0^{(2)}=\theta$. Επιπλέον, συμπεραίνουμε ότι οι τιμές $c_k$, για $k\neq 0$ δίνονται ως $c_k=S_k^{(2)}$. Οπότε, έχουμε:

\begin{equation*}
\begin{split}
c_k&=\frac{1}{n}\mathrm{e}^{-\mathrm{i}k\left(m+\frac{1}{2}\right)2\pi/n}\displaystyle\sum\limits_{j=1}^{n/2}\frac{2\pi}{n}\left(\frac{n}{2}-j+\frac{1}{2}\right)\left(\mathrm{e}^{\mathrm{i}k\left(j-\frac{1}{2}\right)2\pi/n}-\mathrm{e}^{-\mathrm{i}k\left(j-\frac{1}{2}\right)2\pi/n}\right)\\
&=\frac{2\mathrm{i}}{n}\mathrm{e}^{-\mathrm{i}k\left(m+\frac{1}{2}\right)2\pi/n}\displaystyle\sum\limits_{j=1}^{n/2}\frac{2\pi}{n}\left(\frac{n}{2}-j+\frac{1}{2}\right)\sin{\left[k\left(j-\frac{1}{2}\right)\frac{2\pi}{n}\right]}.
\end{split}
\end{equation*}
Θα υπολογίσουμε την ποσότητα:
\begin{equation}\label{eq:circ_entries}
\frac{2}{n}\displaystyle\sum\limits_{j=1}^{n/2}\frac{2\pi}{n}\left(\frac{n}{2}-j+\frac{1}{2}\right)\sin{\left[k\left(j-\frac{1}{2}\right)\frac{2\pi}{n}\right]}.
\end{equation}
Πολλαπλασιάζοντας με $\sin{\left(\frac{k\pi}{n}\right)}$, έχουμε:
\begin{equation*}
\begin{split}
&\frac{2}{n}\displaystyle\sum\limits_{j=1}^{n/2}\frac{2\pi}{n}\left(\frac{n}{2}-j+\frac{1}{2}\right)\sin{\left[k\left(j-\frac{1}{2}\right)\frac{2\pi}{n}\right]}\sin{\left(\frac{k\pi}{n}\right)}\\
&=\frac{1}{n}\displaystyle\sum\limits_{j=1}^{n/2}\frac{2\pi}{n}\left(\frac{n}{2}-j+\frac{1}{2}\right)\left\lbrace\cos{\left[k\left(j-1\right)\frac{2\pi}{n}\right]}-\cos{\left(kj\frac{2\pi}{n}\right)}\right\rbrace\\
&=\frac{1}{n}\displaystyle\sum\limits_{j=0}^{n/2-1}\frac{2\pi}{n}\left(\frac{n}{2}-j-\frac{1}{2}\right)\cos{\left(kj\frac{2\pi}{n}\right)}\\
&\phantom{2pc}-\frac{1}{n}\displaystyle\sum\limits_{j=1}^{n/2}\frac{2\pi}{n}\left(\frac{n}{2}-j+\frac{1}{2}\right)\cos{\left(kj\frac{2\pi}{n}\right)}\\
&=\frac{1}{n}\frac{2\pi}{n}\frac{n-1}{2}-\frac{1}{n}\frac{2\pi}{n}\displaystyle\sum\limits_{j=1}^{n/2-1}\cos{\left(kj\frac{2\pi}{n}\right)}-\frac{1}{n}\frac{\pi}{n}\cos{\left(k\pi\right)}.
\end{split}
\end{equation*}
`Οταν $k$ είναι περιττός αριθμός, η παραπάνω ποσότητα είναι ίση με $\frac{1}{n}\frac{2\pi}{n}\frac{n-1}{2}+\frac{1}{n}\frac{\pi}{n}=\frac{\pi}{n}$. Αντιστοίχως, όταν $k$ είναι άρτιος, $\frac{1}{n}\frac{2\pi}{n}\frac{n-1}{2}-\frac{1}{n}\frac{2\pi}{n}\displaystyle\sum\limits_{j=0}^{n/2-1}\cos{\left(kj\frac{2\pi}{n}\right)}+\frac{1}{n}\frac{2\pi}{n}-\frac{1}{n}\frac{\pi}{n}=\frac{\pi}{n}$.

Επομένως, η σχέση (\ref{eq:circ_entries}) μπορεί να γραφεί ως $\frac{\pi}{n\sin{\left(\frac{k\pi}{n}\right)}}$ κι έτσι οι τιμές $c_k$, όταν $k\neq0$ δίνονται ως:
\begin{equation}\label{eq:circ_entries2}
c_k=\frac{\mathrm{i}\pi}{n\sin{\left(\frac{k\pi}{n}\right)}}\mathrm{e}^{-\mathrm{i}k\left(m+\frac{1}{2}\right)2\pi/n}.
\end{equation}

Μένει να υπολογίσουμε τον πίνακα (διαφορά πινάκων) $T_n(\mathfrak{g})-\mathcal{C}_n(\mathfrak{g})$. Η κύρια διαγώνιός του δίνεται ως $d_0=-c_0=-\theta$ και λαμβάνοντας υπόψη τις (\ref{eq:Toeplitz_g}) και (\ref{eq:circ_entries2}) παρατηρούμε ότι οι $k$-οστές διαγώνιοί του $d_k$, όταν $k\neq 0$ δίνονται ως:

\begin{equation}\label{eq:Tng-Cng}
\begin{split}
d_k&=t_k-c_k=\frac{\mathrm{i}}{k}\mathrm{e}^{-\mathrm{i}k\xi}-\frac{\mathrm{i}\pi}{n\sin{\left(\frac{k\pi}{n}\right)}}\mathrm{e}^{-\mathrm{i}k\left(m+\frac{1}{2}\right)2\pi/n}\\
&=\frac{\mathrm{i}}{k}\left(\mathrm{e}^{-\mathrm{i}k\xi}-\mathrm{e}^{-\mathrm{i}k\left(m+\frac{1}{2}\right)2\pi/n}\right)+\mathrm{i}\left(\frac{1}{k}-\frac{\pi}{n\sin{\left(\frac{k\pi}{n}\right)}}\right)\mathrm{e}^{-\mathrm{i}k\left(m+\frac{1}{2}\right)2\pi/n}.
\end{split}
\end{equation}

Θα αποδείξουμε ότι ο πρώτος όρος της παραπάνω σχέσης είναι της τάξεως $\frac{1}{n}$. `Εχουμε:
\begin{equation*}
\begin{split}
\frac{\mathrm{i}}{k}\left(\mathrm{e}^{-\mathrm{i}k\xi}-\mathrm{e}^{-\mathrm{i}k\left(m+\frac{1}{2}\right)2\pi/n}\right)&=\frac{\mathrm{i}}{k}\mathrm{e}^{-\mathrm{i}k\xi}\left(1-\mathrm{e}^{-\mathrm{i}k\left(\left(m+\frac{1}{2}\right)2\pi/n-\xi\right)}\right)\\
&=\frac{\mathrm{i}}{k}\mathrm{e}^{-\mathrm{i}k\xi}\left(1-\mathrm{e}^{-\mathrm{i}k\theta}\right).
\end{split}
\end{equation*}
Αναπτύσοντας κατά {\en Taylor} την ποσότητα $\mathrm{e}^{-\mathrm{i}k\theta}$ γύρω από το 0, λαμβάνουμε:
\begin{equation*}
\frac{\mathrm{i}}{k}\mathrm{e}^{-\mathrm{i}k\xi}\left(1-\mathrm{e}^{-\mathrm{i}k\theta}\right)=\frac{\mathrm{i}}{k}\mathrm{e}^{-\mathrm{i}k\xi}\left(-\mathrm{i}k\theta\mathrm{e}^{-\mathrm{i}k\widehat{\theta}}\right)=\theta\mathrm{e}^{-\mathrm{i}k(\xi+\widehat{\theta})},~\vert\widehat{\theta}\vert<\vert\theta\vert.
\end{equation*}
Προφανώς, $\left\vert\theta\mathrm{e}^{-\mathrm{i}k(\xi+\widehat{\theta})}\right\vert=\vert\theta\vert=\mathcal{O}\left(\frac{1}{n}\right)$.

Γράφουμε τη διαφορά πινάκων $T_n(\mathfrak{g})-\mathcal{C}_n(\mathfrak{g})$ ως το άθροισμα δύο πινάκων $T_n(\mathfrak{g})-\mathcal{C}_n(\mathfrak{g})=\widetilde{A}_n+\widetilde{B}_n$, όπου οι τιμές του $\widetilde{A}_n$ είναι της τάξεως $\mathcal{O}\left(\frac{1}{n}\right)$ και αυτές του $\widetilde{B}_n$ είναι $\Omega\left(\frac{1}{n}\right)$. Επομένως, η $k$-οστή διαγώνιος του $\widetilde{A}_n$ αποτελείται από τις τιμές $\frac{\mathrm{i}}{k}\left(\mathrm{e}^{-\mathrm{i}k\xi}-\mathrm{e}^{-\mathrm{i}k\left(m+\frac{1}{2}\right)2\pi/n}\right)$.

Στη συνέχεια, θα εκτιμήσουμε την τάξη του δεύτερου όρου της (\ref{eq:Tng-Cng}). Αφού $\mathrm{e}^{-\mathrm{i}k\left(m+\frac{1}{2}\right)2\pi/n}$ είναι μέτρου 1, προσπαθούμε να εκτιμήσουμε την τάξη των όρων $\frac{1}{k}-\frac{\pi}{n\sin{\left(\frac{k\pi}{n}\right)}}$ για τις διάφορες τιμές του $k=1,2,\dots,n-1$.

Διακρίνουμε τρεις περιπτώσεις, λαμβάνοντας υπόψη το μέγεθος του $k$ σε σχέση με το $n$.\\
1) $k\sim n$ και $n-k\sim n$: Αυτό σημαίνει ότι $k=\alpha n$, όπου $0<\alpha<1$ είναι μια σταθερά ανεξάρτητη του $n$. Τότε:
\begin{equation*}
\begin{split}
&\frac{1}{k}-\frac{\pi}{n\sin{\left(\frac{k\pi}{n}\right)}}=\frac{1}{k}-\frac{1}{\sin{\left(\alpha\pi\right)}}\frac{\pi}{n}=\frac{1}{k}-\beta\frac{\pi}{n}=-\frac{1}{n-k}+\left(\frac{1}{k}-\beta\frac{\pi}{n}+\frac{1}{n-k}\right)\\
&\phantom{\frac{1}{k}-\frac{\pi}{n\sin{\left(\frac{k\pi}{n}\right)}}}=-\frac{1}{n-k}+\left(\frac{1}{\alpha n}-\beta\frac{\pi}{n}+\frac{1}{(1-\alpha)n}\right)=-\frac{1}{n-k}+\mathcal{O}\left(\frac{1}{n}\right).
\end{split}
\end{equation*}
2) $n-k=o(n)\Leftrightarrow k=n-o(n)$: Παρατηρούμε ότι, $\sin{\left(\frac{k\pi}{n}\right)}=\sin{\left((n-k)\frac{\pi}{n}\right)}$. Εφαρμόζοντας ανάπτυγμα {\en Taylor} έχουμε:
\begin{equation*}
\sin{\left((n-k)\frac{\pi}{n}\right)}=(n-k)\frac{\pi}{n}-\frac{1}{6}(n-k)^3\frac{\pi^3}{n^3}\cos{\widetilde{\theta}},~0<\widetilde{\theta}<(n-k)\frac{\pi}{n}.
\end{equation*}
Επομένως:
\begin{equation*}
\begin{split}
&\frac{1}{k}-\frac{\pi}{n\sin{\left(\frac{k\pi}{n}\right)}}=\frac{1}{k}-\frac{1}{(n-k)-\frac{1}{6}(n-k)^3\frac{\pi^2}{n^2}\cos{\widetilde{\theta}}}\\
&\phantom{\frac{1}{k}-\frac{\pi}{n\sin{\left(\frac{k\pi}{n}\right)}}}\simeq\frac{1}{k}-\frac{1}{(n-k)}+\frac{1}{6}(n-k)\frac{\pi^2}{n^2}\cos{\widetilde{\theta}}=-\frac{1}{n-k}+\mathcal{O}\left(\frac{1}{n}\right).
\end{split}
\end{equation*}
3) $k=o(n)$: Με παρόμοιο τρόπο
\begin{equation*}
\begin{split}
&\frac{1}{k}-\frac{\pi}{n\sin{\left(\frac{k\pi}{n}\right)}}=\frac{1}{k}-\frac{1}{k-\frac{1}{6}k^3\frac{\pi^2}{n^2}\cos{\widehat{\theta}}}\simeq\frac{1}{k}-\frac{1}{k}+\frac{1}{6}k\frac{\pi^2}{n^2}\cos{\widehat{\theta}}\\
&\phantom{\frac{1}{k}-\frac{\pi}{n\sin{\left(\frac{k\pi}{n}\right)}}}=-\frac{1}{n-k}+\frac{1}{n-k}+\frac{1}{6}k\frac{\pi^2}{n^2}\cos{\widehat{\theta}}\\
&\phantom{\frac{1}{k}-\frac{\pi}{n\sin{\left(\frac{k\pi}{n}\right)}}}=-\frac{1}{n-k}+\mathcal{O}\left(\frac{1}{n}\right),~0<\widehat{\theta}<k\frac{\pi}{n}.
\end{split}
\end{equation*}

Χωρίζουμε τον πίνακα $\widetilde{B}_n$ ως $\widetilde{B}_n=\widehat{A}_n+\widehat{B}_n$, όπου ο $\widehat{A}_n$ έχει τιμές που χαρακτηρίζονται ως $\mathcal{O}\left(\frac{1}{n}\right)$ (στις τρεις περιπτώσεις που προαναφέραμε) και $\left(\widehat{B}_n\right)_k=-\frac{\mathrm{i}}{n-k}\mathrm{e}^{-\mathrm{i}k\left(m+\frac{1}{2}\right)2\pi/n}$, $k>0$. Επειδή ο πίνακας $\widehat{B}_n$ είναι Ερμιτιανός, οι τιμές των διαγωνίων με αρνητικό δείκτη $k$, θα είναι οι συζυγείς των αντίστοιχων τιμών για $k>0$. `Αρα, $\left(\widehat{B}_n\right)_{-k}=\overline{\left(\widehat{B}_n\right)}_{k}=\frac{\mathrm{i}}{n-k}\mathrm{e}^{\mathrm{i}k\left(m+\frac{1}{2}\right)2\pi/n}$, $k>0$.

Τότε, χωρίζουμε επίσης τον $\widehat{B}_n$ ως:
\begin{equation*}
\widehat{B}_n=
\begin{pmatrix}
V_{n/2} & U_{n/2}\\
U_{n/2}^H & V_{n/2}
\end{pmatrix}=
\begin{pmatrix}
V_{n/2} & 0\\
0 & V_{n/2}
\end{pmatrix}+
\begin{pmatrix}
0 & U_{n/2}\\
U_{n/2}^H & 0
\end{pmatrix}=A_n^{\prime}+B_n.
\end{equation*}
Από την παραπάνω ανάλυση, συμπεραίνουμε ότι $A_n^{\prime}=\mathcal{O}\left(\frac{1}{n}\right)$. Συνοψίζοντας, ο πίνακας $T_n(\mathfrak{g})-\mathcal{C}_n(\mathfrak{g})$ γράφεται ως $A_n+B_n$, όπου $A_n=\widetilde{A}_n+\widehat{A}_n+A_n^{\prime}$ είναι ένας Ερμιτιανός πίνακας {\en Toeplitz}, με τιμές τάξεως $\mathcal{O}\left(\frac{1}{n}\right)$. Επομένως, η νόρμα {\en Frobenius} αυτού φράσσεται από μια σταθερά ανεξάρτητη του $n$ κι έτσι τόσο οι ιδιοτιμές του, όσο και οι ιδιάζουσες τιμές αυτού, συσσωρεύονται κατά κύριο τρόπο γύρω από το μηδέν \cite{Benedetto_Serra,tyrtyshnikov1995circulant}.

Μένει να μελετήσουμε τον πίνακα $B_n$. Για να εκπληρώσουμε αυτόν τον σκοπό, ακολουθούμε την ίδια τεχνική που παρουσιάζεται στο Λήμμα 8 της \cite{yeung1993circulant}. Θεωρούμε τον πίνακα:
\begin{equation*}
J_n=\begin{pmatrix}
0 &\cdots & 0 & 1\\
\vdots &  & 1 &0\\
0 &\reflectbox{$\ddots$} &  &\vdots\\
1 & 0 &\cdots & 0
\end{pmatrix}
\end{equation*}
και τους ορθομοναδιαίους πίνακες $P_n$ και $Q_n$, όπως φαίνεται παρακάτω:
\begin{equation*}
P_n=\operatorname{diag}\left(-1,-\mathrm{e}^{\mathrm{i}\zeta},\dots,-\mathrm{e}^{\mathrm{i}(\frac{n}{2}-1)\zeta}\right),
\end{equation*}
\begin{equation*}
Q_n=\operatorname{diag}\left(-\mathrm{i}\mathrm{e}^{\mathrm{i}\frac{n}{2}\zeta},-\mathrm{i}\mathrm{e}^{\mathrm{i}(\frac{n}{2}+1)\zeta},\dots,-\mathrm{i}\mathrm{e}^{\mathrm{i}(n-1)\zeta}\right),
\end{equation*}
όπου $\zeta=\left(m+\frac{1}{2}\right)\frac{2\pi}{n}$. Τότε,
$U_{n/2}=P_{n/2}^H\mathcal{H}_{n/2}J_{n/2}Q_{n/2}$ όπου $\mathcal{H}_{n/2}$ είναι ο $\frac{n}{2}\times\frac{n}{2}$ πίνακας {\en Hilbert}. `Ετσι, ο $B_n$ μπορεί να γραφεί ως:
\begin{equation*}
\begin{split}
&\begin{pmatrix}
0 & U_{n/2}\\
U_{n/2}^H & 0
\end{pmatrix}=
\begin{pmatrix}
P_{n/2}^H & 0\\
0 & Q_{n/2}^H
\end{pmatrix}
\begin{pmatrix}
0 & \mathcal{H}_{n/2}J_{n/2}\\
J_{n/2}\mathcal{H}_{n/2} & 0
\end{pmatrix}
\begin{pmatrix}
P_{n/2} & 0\\
0 & Q_{n/2}
\end{pmatrix}\\
&=\frac{1}{2}
\begin{pmatrix}
P_{n/2}^H & 0\\
0 & Q_{n/2}^H
\end{pmatrix}
\begin{pmatrix}
I_{n/2} & I_{n/2}\\
J_{n/2} & -J_{n/2}
\end{pmatrix}
\begin{pmatrix}
\mathcal{H}_{n/2} & 0\\
0 & -\mathcal{H}_{n/2}
\end{pmatrix}\cdot\\
&\hspace{2pc}\cdot
\begin{pmatrix}
I_{n/2} & J_{n/2}\\
I_{n/2} & -J_{n/2}
\end{pmatrix}
\begin{pmatrix}
P_{n/2} & 0\\
0 & Q_{n/2}
\end{pmatrix}\\
&=\frac{1}{2}
\begin{pmatrix}
P_{n/2}^H & P_{n/2}^H\\
Q_{n/2}^HJ_{n/2} & -Q_{n/2}^HJ_{n/2}
\end{pmatrix}
\begin{pmatrix}
\mathcal{H}_{n/2} & 0\\
0 & -\mathcal{H}_{n/2}
\end{pmatrix}
\begin{pmatrix}
P_{n/2} & J_{n/2}Q_{n/2}\\
P_{n/2} & -J_{n/2}Q_{n/2}
\end{pmatrix}.
\end{split}
\end{equation*}
Παρατηρούμε ότι ο $B_n$ είναι όμοιος με τον μεσαίο πίνακα του παραπάνω γινομένου, δηλαδή του
\begin{equation*}
\begin{pmatrix}
\mathcal{H}_{n/2} & 0\\
0 & -\mathcal{H}_{n/2}
\end{pmatrix},
\end{equation*}
αφού ο πίνακας
\begin{equation*}
\frac{1}{\sqrt{2}}
\begin{pmatrix}
P_{n/2}^H & P_{n/2}^H\\
Q_{n/2}^HJ_{n/2} & -Q_{n/2}^HJ_{n/2}
\end{pmatrix}
\end{equation*}
είναι ορθογώνιος. Επειδή ο $\mathcal{H}_{n/2}$ είναι ο $\frac{n}{2}\times\frac{n}{2}$ πίνακας {\en Hilbert}, συμπεραίνουμε ότι η νόρμα $\Vert B_n\Vert_F$ τείνει στο άπειρο όπως ο $\log{\left(\frac{n}{2}\right)}$, το οποίο ισοδύναμα δηλώνει ότι το πλήθος των ιδιοτιμών του $B_n$, που έχουν απόλυτη τιμή μεγαλύτερη του $\varepsilon>0$, είναι της τάξεως $\mathcal{O}\left(\log{\left(\frac{n}{2}\right)}\right)$ \cite{widom1966hankel,yeung1993circulant}.
\end{proof}

{\it Απόδειξη }του Θεωρήματος \ref{thm:T-C_piecewise}. Θεωρούμε ότι οι συναρτήσεις $f_1$ και $f_2$ έχουν σημεία ασυνέχειας $\xi_1,\xi_2,\dots,\xi_{\nu}\in(0,2\pi]$ με εύρη ασυνέχειας, τη φραγμένη ποσότητα $\alpha_k$ για την $f_1$ και $\beta_k$ για την $f_2$ στο $\xi_k,~k=1,2,\dots,\nu$. Σε περίπτωση που η $f_1$ έχει ασυνέχεια στο σημείο $\xi_k$, όπου η $f_2$ είναι συνεχής, θέτουμε $\beta_k=0$. Φυσικά, ακολουθούμε την ίδια λογική για την αντίστροφη περίπτωση.

Στη συνέχεια ακολουθούμε την τεχνική που εφάρμοσαν οι {\en R.~Chan} και {\en M-C.~Yeung} στην \cite{yeung1993circulant}, αλλά στην περίπτωσή μας, έχουμε να μελετήσουμε μιγαδικές συναρτήσεις, αντί για πραγματικές. Θεωρούμε τις συναρτήσεις $\mathfrak{g}_k$ για κάθε αντίστοιχο σημείο $\xi_k$, όπως κάναμε στο Λήμμα \ref{lem:g function}:
\begin{equation*}
\mathfrak{g}_k(x)=
\begin{cases}
x+\pi-\xi_k,~&0< x\leq\xi_k,\\
x-\pi-\xi_k,~&\xi_k<x\leq2\pi.
\end{cases}~k=1,2,\dots,\nu,
\end{equation*}
η οποία είναι ασυνεχής στα σημεία $\xi_k$, με αντίστοιχο εύρος ασυνέχειας ίσο με $-2\pi$.

Ακολούθως, προσθαφαιρούμε στην $f$ τη συνάρτηση,
\begin{equation*}
\widehat{\mathfrak{g}}(x)=\displaystyle\sum\limits_{k=1}^{\nu}\left(\frac{\alpha_k}{2\pi}+\mathrm{i}\frac{\beta_k}{2\pi}\right)\mathfrak{g}_k(x),
\end{equation*}
\begin{equation*}
\text{οπότε }f(x)=f_1(x)+\mathrm{i}f_2(x)=f_1(x)+\displaystyle\sum\limits_{k=1}^{\nu}\frac{\alpha_k}{2\pi}\mathfrak{g}_k(x)+\mathrm{i}f_2(x)+\mathrm{i}\sum\limits_{k=1}^{\nu}\frac{\beta_k}{2\pi}\mathfrak{g}_k(x)-\widehat{\mathfrak{g}}(x).
\end{equation*}
`Οπως στην \cite{yeung1993circulant}, είναι προφανές ότι η $h_1(x)=f_1(x)+\displaystyle\sum\limits_{k=1}^{\nu}\frac{\alpha_k}{2\pi}\mathfrak{g}_k(x)$ και η $h_2(x)=f_2(x)+\displaystyle\sum\limits_{k=1}^{\nu}\frac{\beta_k}{2\pi}\mathfrak{g}_k(x)$ είναι και οι δύο συνεχείς στο $(0,2\pi]$. Σχηματίζοντας τη διαφορά $\Delta_n=T_n(f)-\mathcal{C}_n(f)$ λαμβάνουμε ότι:
\begin{equation}\label{Delta_n}
\Delta_n=T_n(f)-\mathcal{C}_n(f)=T_n(h)-\mathcal{C}_n(h)-\left(T_n(\widehat{\mathfrak{g}})-\mathcal{C}_n(\widehat{\mathfrak{g}})\right),
\end{equation}
όπου $h(x)=h_1(x)+\mathrm{i}h_2(x)$. Επειδή η $h$ είναι μια συνεχής συνάρτηση, από το Θεώρημα \ref{Thm:1}, ο πίνακας $T_n(h)-\mathcal{C}_n(h)$ έχει κύρια συσσώρευση των ιδιοτιμών στο 0. Πρέπει να εκτιμήσουμε τη συμπεριφορά του $T_n(\widehat{\mathfrak{g}})-\mathcal{C}_n(\widehat{\mathfrak{g}})$. `Εχουμε:
\begin{equation*}
T_n(\widehat{\mathfrak{g}})-\mathcal{C}_n(\widehat{\mathfrak{g}})=\displaystyle\sum\limits_{k=1}^{\nu}\frac{\alpha_k}{2\pi}\left(T_n(\mathfrak{g}_k)-\mathcal{C}_n(\mathfrak{g}_k)\right)+\mathrm{i}\displaystyle\sum\limits_{k=1}^{\nu}\frac{\beta_k}{2\pi}\left(T_n(\mathfrak{g}_k)-\mathcal{C}_n(\mathfrak{g}_k)\right).
\end{equation*}
Χρησιμοποιώντας το Λήμμα \ref{lem:g function}, λαμβάνουμε ότι κάθε πίνακας $T_n(\mathfrak{g}_k)-\mathcal{C}_n(\mathfrak{g}_k)=A_{n,k}+B_{n,k}$, όπου $\Vert A_{n,k}\Vert_F\leq c_k<\infty$ ($c_k:$ σταθερά ανεξάρτητη της $n$) και $\Vert B_{n,k}\Vert_F=\mathcal{O}\left(\log{n}\right)$. Επειδή $T_n(\widehat{\mathfrak{g}})-\mathcal{C}_n(\widehat{\mathfrak{g}})$ είναι ένας γραμμικός συνδυασμός των πινάκων $T_n(\mathfrak{g}_k)-\mathcal{C}_n(\mathfrak{g}_k)$, $k=1,2,\dots,\nu$ ισχύει ότι $\Vert T_n(\widehat{\mathfrak{g}})-\mathcal{C}_n(\widehat{\mathfrak{g}})\Vert_F=\mathcal{O}\left(\log{n}\right)$. Ως εκ τούτου, $\mathcal{O}\left(\log{n}\right)$ ιδιοτιμές του $\Delta_n=T_n(f)-\mathcal{C}_n(f)$ κυμαίνονται εκτός του ορθογωνίου $[-\varepsilon,\varepsilon]^2$, του μιγαδικού επιπέδου.\qed

Στη συνέχεια θα μελετήσουμε τη συσσώρευση των ιδιοτιμών και ιδιαζουσών τιμών του προρρυθμισμένου πίνακα $\mathcal{C}_n^{-1}(f)T_n(f)$, όταν η $f$ είναι κατά τμήματα συνεχής, με πεπερασμένο πλήθος σημείων ασυνέχειας.

\begin{thm}\label{sing_val_discont}
`Εστω $f$ μια μιγαδική συνάρτηση, όπως στο Θεώρημα \ref{thm:T-C_piecewise}. Τότε, οι ιδιάζουσες τιμές του προρρυθμισμένου πίνακα $\mathcal{C}_n^{-1}(f)T_n(f)$ συσσωρεύονται γύρω από το 1, με την έννοια της γενικής συσσώρευσης, δηλαδή για κάθε $\varepsilon>0$, $\mathcal{O}(\log{n})$ ιδιάζουσες τιμές κυμαίνονται εκτός του $[1-\varepsilon,1+\varepsilon]$.
\end{thm}

\begin{proof}
Είναι προφανές ότι για τον προρρυθμισμένο πίνακα ισχύει:
\begin{equation*}
\begin{split}
\mathcal{C}_n^{-1}(f)T_n(f)&=\mathcal{C}_n^{-1}(f)(T_n(f)-\mathcal{C}_n(f))+I_n=\mathcal{C}_n^{-1}(f)\Delta_n+I_n\\
&\hspace{-6pt}\overset{(\ref{Delta_n})}{=}\mathcal{C}_n^{-1}(f)\Delta_n(h)-\mathcal{C}_n^{-1}(f)\Delta_n(\widehat{\mathfrak{g}})+I_n,
\end{split}
\end{equation*}
όπου $\Delta_n(h)=T_n(h)-\mathcal{C}_n(h)$ και $\Delta_n(\widehat{\mathfrak{g}})=T_n(\widehat{\mathfrak{g}})-\mathcal{C}_n(\widehat{\mathfrak{g}})$. Από την απόδειξη του Θεωρήματος \ref{thm:T-C_piecewise} έχουμε ότι $\Delta_n(\widehat{\mathfrak{g}})=A_n+B_n$, όπου $\Vert A_n\Vert_F<\infty$, ανεξάρτητη της διάστασης $n$ και $\Vert B_n\Vert_F=\mathcal{O}\left(\log{n}\right)$. Επομένως, η παραπάνω σχέση γράφεται ως:
\begin{equation*}
\mathcal{C}_n^{-1}(f)T_n(f)-I_n=\mathcal{C}_n^{-1}(f)\left(\Delta_n(h)-A_n\right)-\mathcal{C}_n^{-1}(f)B_n.
\end{equation*}

Υποθέτουμε ότι $A=\mathcal{C}_n^{-1}(f)\left(\Delta_n(h)-A_n\right)$ και $B=-\mathcal{C}_n^{-1}(f)B_n$. Γνωρίζουμε ότι οι ιδιάζουσες τιμές του $A$ συσσωρεύονται κατά κύριο τρόπο γύρω από το 0 και ο $B_n$, ο οποίος είναι όμοιος με τον $\mathcal{H}_n$, έχει γενική συσσώρευση των ιδιαζουσών τιμών γύρω από το 0. `Εστω ότι για κάποιο συγκεκριμένο $\varepsilon>0$, $k$ ιδιάζουσες τιμές είναι μεγαλύτερες από $\varepsilon$. Από την ανισότητα του {\en Weyl} \cite{horn1994topics} έχουμε ότι:
\begin{equation*}
\sigma_{j+k}(A+B)\leq\sigma_{k+1}(A)+\sigma_j(B)\leq\varepsilon+\sigma_j(B),~1\leq j+k\leq n.
\end{equation*}

Συνεπώς, $\mathcal{O}(\log{n})$ ιδιάζουσες τιμές του προρρυθμισμένου πίνακα κυμαίνονται εκτός του $[1-\varepsilon,1+\varepsilon]$ και η απόδειξη ολοκληρώθηκε.
\end{proof}

\begin{lem}\label{lem: Frobenius of product}
`Εστω $\lbrace\mathcal{A}_n\rbrace$ και $\lbrace\mathcal{B}_n\rbrace$ δύο ακολουθίες πινάκων, με $\mathcal{A}_n$ να είναι αντιστρέψιμος για κάθε $n\in\mathbb{N}$. `Εστω επίσης ότι οι $\lbrace\mathcal{A}_n\rbrace$ και $\lbrace\mathcal{A}_n^{-1}\rbrace$ έχουν φραγμένη νόρμα $\Vert\cdot\Vert_2$, το οποίο σημαίνει ότι υπάρχουν θετικές σταθερές $c$ και $d$ τέτοιες ώστε $\Vert \mathcal{A}_n\Vert_2\leq c$ και $\Vert \mathcal{A}_n^{-1}\Vert_2\leq d$, για οποιονδήποτε $n\in\mathbb{N}$. Τότε, για τη νόρμα {\en Frobenius} ισχύει ότι, $\Vert \mathcal{A}_n\mathcal{B}_n\Vert_F\sim\Vert \mathcal{B}_n\Vert_F$, που σημαίνει ότι υπάρχουν $\alpha,\beta>0$ έτσι ώστε, $\alpha\Vert \mathcal{B}_n\Vert_F\leq\Vert \mathcal{A}_n\mathcal{B}_n\Vert_F\leq\beta\Vert \mathcal{B}_n\Vert_F$.
\end{lem}

\begin{proof}
Γνωρίζουμε ότι $\Vert \mathcal{A}_n\mathcal{B}_n\Vert_F=\Vert \mathcal{B}_n\mathcal{A}_n\Vert_F=\left(\operatorname{tr}\left(\mathcal{A}_n^H\mathcal{B}_n^H\mathcal{B}_n\mathcal{A}_n\right)\right)^{\frac{1}{2}}$.

`Εστω $\lambda_j$, $j=1,2,\dots,n$ οι ιδιοτιμές του $\mathcal{A}_n^H\mathcal{B}_n^H\mathcal{B}_n\mathcal{A}_n$ για κάποιο σταθερό $n$ και $\mu_j$ οι ιδιοτιμές του $\mathcal{B}_n^H\mathcal{B}_n$, ταξινομημένες σε μη-αύξουσα σειρά. `Εστω επίσης ότι με $w_j$, $j=1,2,\dots,n$ συμβολίζουμε τα ιδιοδιανύσματα του $\mathcal{B}_n^H\mathcal{B}_n$ και $W_j=\operatorname{span}\lbrace w_j,w_{j+1},\dots,w_n\rbrace$, $\widetilde{W}_j=\operatorname{span}\lbrace w_1,w_2,\dots,w_j\rbrace$. Χρησιμοποιούμε το {\en min-max} θεώρημα των {\en Courant-Fisher}, για να συσχετίσουμε τις ιδιοτιμές $\lambda_j$ με τις $\mu_j$.

\begin{equation*}
\begin{split}
\lambda_j&=\min_{\mathcal{V}:\operatorname{dim}(\mathcal{V})=n-j+1}{\max_{x\in\mathcal{V}}{\frac{x^H\mathcal{A}_n^H\mathcal{B}_n^H\mathcal{B}_n\mathcal{A}_nx}{x^Hx}}}\leq\max_{x\in \mathcal{A}_n^{-1}W_j}{\frac{x^H\mathcal{A}_n^H\mathcal{B}_n^H\mathcal{B}_n\mathcal{A}_nx}{x^Hx}}\\
&=\max_{y\in W_j}{\frac{y^H\mathcal{B}_n^H\mathcal{B}_ny}{y^H\mathcal{A}_n^{-H}\mathcal{A}_n^{-1}y}}=\max_{y\in W_j}{\frac{y^H\mathcal{B}_n^H\mathcal{B}_ny}{y^Hy}\cdot\frac{y^Hy}{y^H\mathcal{A}_n^{-H}\mathcal{A}_n^{-1}y}}\\
&\leq\max_{y\in W_j}{\frac{y^H\mathcal{B}_n^H\mathcal{B}_ny}{y^Hy}}\cdot\max_{y\in W_j}{\frac{y^Hy}{y^H\mathcal{A}_n^{-H}\mathcal{A}_n^{-1}y}}\\
&=\min_{\mathcal{V}:\operatorname{dim}(\mathcal{V})=n-j+1}{\max_{y\in\mathcal{V}}{\frac{y^H\mathcal{B}_n^H\mathcal{B}_ny}{y^Hy}}}\cdot\bar{c}_j=\bar{c}_j\mu_j,\text{ όπου }\frac{1}{d^2}\leq\bar{c}_j\leq c^2.
\end{split}
\end{equation*}

Από την άλλη,
\begin{equation*}
\begin{split}
\lambda_j&=\max_{\mathcal{V}:\operatorname{dim}(\mathcal{V})=j}{\min_{x\in\mathcal{V}}{\frac{x^H\mathcal{A}_n^H\mathcal{B}_n^H\mathcal{B}_n\mathcal{A}_nx}{x^Hx}}}\geq\min_{x\in \mathcal{A}_n^{-1}\widetilde{W}_j}{\frac{x^H\mathcal{A}_n^H\mathcal{B}_n^H\mathcal{B}_n\mathcal{A}_nx}{x^Hx}}\\
&=\min_{y\in\widetilde{W}_j}{\frac{y^H\mathcal{B}_n^H\mathcal{B}_ny}{y^H\mathcal{A}_n^{-H}\mathcal{A}_n^{-1}y}}\geq\min_{y\in\widetilde{W}_j}{\frac{y^H\mathcal{B}_n^H\mathcal{B}_ny}{y^Hy}}\cdot\min_{y\in\widetilde{W}_j}{\frac{y^Hy}{y^H\mathcal{A}_n^{-H}\mathcal{A}_n^{-1}y}}\\
&=\max_{\mathcal{V}:\operatorname{dim}(\mathcal{V})=j}{\min_{y\in\mathcal{V}}{\frac{y^H\mathcal{B}_n^H\mathcal{B}_ny}{y^Hy}}}\cdot\underline{c}_j=\underline{c}_j\mu_j,\text{ όπου }\frac{1}{d^2}\leq\underline{c}_j\leq c^2.
\end{split}
\end{equation*}

Επομένως, από το θεώρημα ενδιαμέσων τιμών υπάρχουν $\widetilde{c}_j$ ($\underline{c}_j\leq\widetilde{c}_j\leq\bar{c}_j$), έτσι ώστε $\lambda_j=\widetilde{c}_j\mu_j$, $j=1,2,\dots,n$.

Παίρνοντας τη νόρμα {\en Frobenius}, λαμβάνουμε:
\begin{equation*}
\begin{split}
\Vert \mathcal{A}_n\mathcal{B}_n\Vert_F&=\left(\operatorname{tr}\left(\mathcal{A}_n^H\mathcal{B}_n^H\mathcal{B}_n\mathcal{A}_n\right)\right)^{\frac{1}{2}}=\bigg(\sum_{j=1}^n\lambda_j\bigg)^{\frac{1}{2}}=\bigg(\sum_{j=1}^n\widetilde{c}_j\mu_j\bigg)^{\frac{1}{2}}\\
&=\bigg(\widetilde{c}\sum_{j=1}^n\mu_j\small\bigg)^{\frac{1}{2}}=\sqrt{\widetilde{c}}\bigg(\sum_{j=1}^n\mu_j\bigg)^{\frac{1}{2}}=c^\prime\left(\operatorname{tr}\left(\mathcal{B}_n^H\mathcal{B}_n\right)\right)^{\frac{1}{2}}=c^\prime\Vert \mathcal{B}_n\Vert_F,
\end{split}
\end{equation*}
όπου, από το θεώρημα ενδιαμέσων τιμών, $\frac{1}{d^2}\leq\widetilde{c}\leq c^2$ και $\frac{1}{d}\leq c^{\prime}\leq c$. `Ετσι, η απόδειξη ολοκληρώθηκε.
\end{proof}

\begin{thm}\label{eig_discont}
`Εστω $f$ μια μιγαδική συνάρτηση, όπως περιγράφηκε στο Θεώρημα \ref{thm:T-C_piecewise}. Τότε, οι ιδιοτιμές του $\mathcal{C}_n^{-1}(f)T_n(f)$ συσσωρεύονται γύρω από το $(1,0)$, με την έννοια της γενικής συσσώρευσης, που σημαίνει ότι για κάθε $\varepsilon>0$, $\mathcal{O}(\log{n})$ ιδιοτιμές κυμαίνονται εκτός του ορθογωνίου $[1-\varepsilon,1+\varepsilon]\times[-\varepsilon,\varepsilon]$ του μιγαδικού επιπέδου.
\end{thm}

\begin{proof}
Για να μελετήσουμε το φάσμα των ιδιοτιμών διαχωρίζουμε τον προρρυθμισμένο πίνακα στο Ερμιτιανό και αντι-Ερμιτιανό του μέρος. Το Ερμιτιανό μέρος γράφεται ως:
\begin{equation*}
\frac{1}{2}\left[\mathcal{C}_n^{-1}(f)T_n(f)+T_n(\bar{f})\mathcal{C}_n^{-1}(\bar{f})\right]=\frac{1}{2}\left[\mathcal{C}_n^{-1}(f)\Delta_n+\Delta_n^H\mathcal{C}_n^{-1}(\bar{f})\right]+I_n,
\end{equation*}
κι έτσι για να αποδείξουμε ότι έχει γενική συσσώρευση γύρω από το 1 με $\mathcal{O}(\log{n})$ ιδιοτιμές εκτός του διαστήματος συσσώρευσης, πρέπει να αποδείξουμε ότι ο Ερμιτιανός πίνακας $\mathcal{C}_n^{-1}(f)\Delta_n+\Delta_n^H\mathcal{C}_n^{-1}(\bar{f})$ έχει γενική συσσώρευση γύρω από το 0 με $\mathcal{O}(\log{n})$ ιδιοτιμές εκτός του διαστήματος συσσώρευσης. Επειδή η $f$ δεν έχει ρίζες, έχουμε ότι $c\leq\Vert\mathcal{C}_n^{-1}(f)\Vert_2\leq C$, όπου $c, C$ είναι θετικές σταθερές. Στο Θεώρημα \ref{thm:T-C_piecewise} αποδείξαμε ότι $\Vert\Delta_n\Vert_F=\mathcal{O}(\log{n})$. Χρησιμοποιώντας το Λήμμα \ref{lem: Frobenius of product} με $\mathcal{A}_n=\mathcal{C}_n^{-1}(f)$ και $\mathcal{B}_n=\Delta_n$, λαμβάνουμε ότι $\Vert\mathcal{C}_n^{-1}(f)\Delta_n\Vert_F=\mathcal{O}(\log{n})$. Παίρνοντας τη νόρμα {\en Frobenius} έχουμε
\begin{equation*}
\Vert\mathcal{C}_n^{-1}(f)\Delta_n+\Delta_n^H\mathcal{C}_n^{-1}(\bar{f})\Vert_F\leq\Vert\mathcal{C}_n^{-1}(f)\Delta_n\Vert_F+\Vert\Delta_n^H\mathcal{C}_n^{-1}(\bar{f})\Vert_F=\mathcal{O}(\log{n}).
\end{equation*}

Επομένως, για κάποιο $\varepsilon>0$, το πολύ $\mathcal{O}(\log{n})$ ιδιοτιμές του $\mathcal{C}_n^{-1}(f)\Delta_n+\Delta_n^H\mathcal{C}_n^{-1}(\bar{f})$ κυμαίνονται εκτός του διαστήματος $[-\varepsilon,\varepsilon]$, το οποίο ισοδύναμα μας δίνει τη γενική συσσώρευση του πραγματικού μέρους των ιδιοτιμών (του προρρυθμισμένου συστήματος) στο $[1-\varepsilon,1+\varepsilon]$.

Θεωρούμε το αντι-Ερμιτιανό μέρος του προρρυθμισμένου πίνακα:
\begin{equation*}
\frac{1}{2}\left[\mathcal{C}_n^{-1}(f)T_n(f)-T_n(\bar{f})\mathcal{C}_n^{-1}(\bar{f})\right]=\frac{1}{2}\left[\mathcal{C}_n^{-1}(f)\Delta_n-\Delta_n^H\mathcal{C}_n^{-1}(\bar{f})\right].
\end{equation*}

Ομοίως έχουμε ότι
\begin{equation*}
\Vert\mathcal{C}_n^{-1}(f)\Delta_n-\Delta_n^H\mathcal{C}_n^{-1}(\bar{f})\Vert_F\leq\Vert\mathcal{C}_n^{-1}(f)\Delta_n\Vert_F+\Vert\Delta_n^H\mathcal{C}_n^{-1}(\bar{f})\Vert_F=\mathcal{O}(\log{n}).
\end{equation*}
Η παραπάνω σχέση μας δίνει τη γενική συσσώρευση του φανταστικού μέρους των ιδιοτιμών στο $[-\varepsilon,\varepsilon]$.

Από την άλλη:
\begin{equation*}
\begin{split}
\mathcal{O}(\log{n})=\Vert\mathcal{C}_n^{-1}(f)\Delta_n\Vert_F&\leq\frac{1}{2}\Vert\mathcal{C}_n^{-1}(f)\Delta_n+\Delta_n^H\mathcal{C}_n^{-1}(\bar{f})\Vert_F\\
&\phantom{hspace{2pc}}+\frac{1}{2}\Vert\mathcal{C}_n^{-1}(f)\Delta_n-\Delta_n^H\mathcal{C}_n^{-1}(\bar{f})\Vert_F.
\end{split}
\end{equation*}

Αυτό σημαίνει ότι είτε το Ερμιτιανό, είτε το αντι-Ερμιτιανό μέρος (είτε και τα δύο) έχουν νόρμα {\en Frobenius} της τάξης $\mathcal{O}(\log{n})$. Λαμβάνοντας υπόψη τις ιδιότητες ισοδυναμίας των όρων $\Vert\mathcal{C}_n^{-1}(f)\Delta_n\Vert_F$, $\Vert\Delta_n\Vert_F$ και $\Vert \mathcal{H}_{n/2}\Vert_F$, λαμβάνουμε ότι $\mathcal{O}(\log{n})$ ιδιοτιμές κυμαίνονται εκτός του ορθογωνίου $[1-\varepsilon,1+\varepsilon]\times[-\varepsilon,\varepsilon]$ του μιγαδικού επιπέδου κι έτσι η απόδειξη ολοκληρώθηκε.
\end{proof}

\begin{rem}
Λόγω της περιοδικότητας της συνάρτησης $f$, όλα τα παραπάνω θεωρήματα ισχύουν και στο $(-\pi,\pi]$.
\end{rem}

Συνεχίζουμε με τη μελέτη της περίπτωσης όπου η γεννήτρια συνάρτηση $f$ έχει πεπερασμένα σημεία ριζών, καθώς επίσης και πεπερασμένα σημεία ασυνέχειας. Αρχικά, μελετάμε τη συμπεριφορά των ιδιαζουσών τιμών.

\begin{thm}\label{sing_val_discont_jumps}
`Εστω $f=f_1+\mathrm{i}f_2$, όπου $f_1$ είναι άρτια και $f_2$ περιττή, $2\pi$-περιοδική, έχοντας $k$ ρίζες $x_1,x_2,\dots,x_k$ στο $(-\pi,\pi]$ και $\nu$ σημεία ασυνέχειας $\xi_1,\xi_2,\dots,\xi_\nu$, επίσης στο $(-\pi,\pi]$, με αντίστοιχα εύρη ασυνέχειας
\begin{equation*}
\alpha_j=\lim\limits_{x\rightarrow\xi_{j}^+}f(x)-\lim\limits_{x\rightarrow\xi_{j}^-}f(x),
\end{equation*}
και υποθέτουμε ότι τα σημεία $x_j$ είναι διαφορετικά από τα $\xi_j$. `Εστω επίσης $g$, το τριγωνομετρικό πολυώνυμο τέτοιο ώστε η $\frac{f}{g}$ να μην έχει ρίζες στο $(-\pi,\pi]$. Τότε, για κάθε $\varepsilon>0$, το διάστημα $[1-\varepsilon,1+\varepsilon]$ αποτελεί σύνολο γενικής συσσώρευσης των ιδιαζουσών τιμών του προρρυθμισμένου πίνακα $\mathcal{C}_n^{-1}\left(\frac{f}{g}\right)T_n^{-1}(g)T_n(f)$, με $\mathcal{O}(\log{n})$ ιδιάζουσες τιμές εκτός του διαστήματος.
\end{thm}

\begin{proof}
Ακολουθούμε την απόδειξη του Θεωρήματος \ref{sing val_cont} προκειμένου να λάβουμε το αντίστοιχο αποτέλεσμα για τον πίνακα των κανονικών εξισώσεων,
\begin{equation*}
\begin{split}
&\mathcal{C}_n^{-1}\left(\frac{f}{g}\right)T_n^{-1}(g)T_n(f)T_n(\bar{f})T_n^{-1}(\bar{g})\mathcal{C}_n^{-1}\left(\frac{\bar{f}}{\bar{g}}\right)=\\
&\mathcal{C}_n^{-1}\left(\frac{f}{g}\right)T_n\left(\frac{f}{g}\right)T_n\left(\frac{\bar{f}}{\bar{g}}\right)\mathcal{C}_n^{-1}\left(\frac{\bar{f}}{\bar{g}}\right)+L,
\end{split}
\end{equation*}
όπου $L$ είναι πίνακας χαμηλής βαθμίδας, το πολύ ίσης με $4d$ ($d$ είναι ο βαθμός του τριγωνομετρικού πολυωνύμου $g$). Επομένως, οι ιδιάζουσες τιμές του προρρυθμισμένου πίνακα $\mathcal{C}_n^{-1}\left(\frac{f}{g}\right)T_n^{-1}(g)T_n(f)$ συμπεριφέρονται όπως αυτές του $\mathcal{C}_n^{-1}\left(\frac{f}{g}\right)T_n^{-1}\left(\frac{f}{g}\right)$ με τη διαφορά $4d$ επιπλέον ιδιαζουσών τιμών εκτός του διαστήματος συσσώρευσης, που προκύπτουν από τον $L$.

Η μόνη διαφορά από το Θεώρημα \ref{sing val_cont} εντοπίζεται στο ότι η συνάρτηση $\frac{f}{g}$ παρουσιάζει ασυνέχεια στα σημεία $\xi_j$, $j=1,2,\dots,\nu$ με πεπερασμένο εύρος ασυνέχειας $\beta_j=\frac{\alpha_j}{g(\xi_j)}$, $j=1,2,\dots,\nu$. Στη συνέχεια, εφαρμόζοντας το Θεώρημα \ref{sing_val_discont} για την $\frac{f}{g}$ καταλήγουμε στο επιθυμητό αποτέλεσμα.
\end{proof}

Η συμπεριφορά των ιδιοτιμών δίνεται στο επόμενο θεώρημα.

\begin{thm}
`Εστω $f$ μια μιγαδική συνάρτηση και $g$ τριγωνομετρικό πολυώνυμο, όπως περιγράφηκε στο Θεώρημα \ref{sing_val_discont_jumps}. Τότε, οι ιδιοτιμές του πίνακα $\mathcal{C}_n^{-1}(\frac{f}{g})T_n^{-1}(g)T_n(f)$ συσσωρεύονται γύρω από το $(1,0)$, με την έννοια της γενικής συσσώρευσης. Ισοδύναμα, για κάθε $\varepsilon>0$, $\mathcal{O}(\log{n})$ ιδιοτιμές κυμαίνονται εκτός του ορθογωνίου $[1-\varepsilon,1+\varepsilon]\times[-\varepsilon,\varepsilon]$ του μιγαδικού επιπέδου.
\end{thm}

\begin{proof}
Για το Ερμιτιανό μέρος του $\mathcal{C}_n^{-1}(\frac{f}{g})T_n^{-1}(g)T_n(f)$ ισχύει:
\begin{equation*}
\begin{split}
H&=\frac{1}{2}\left[\mathcal{C}_n^{-1}\left(\frac{f}{g}\right)T_n^{-1}(g)T_n(f)+T_n(\bar{f})T_n^{-1}(\bar{g})\mathcal{C}_n^{-1}\left(\frac{\bar{f}}{\bar{g}}\right)\right]\\
&=\frac{1}{2}\left[\mathcal{C}_n^{-1}\left(\frac{f}{g}\right)T_n^{-1}(g)\left(T_n(g)T_n\left(\frac{f}{g}\right)+L_1\right)\right]\\
&\phantom{2pc}+\frac{1}{2}\left[\left(T_n\left(\frac{\bar{f}}{\bar{g}}\right)T_n(\bar{g})+L_1^H\right)T_n^{-1}(\bar{g})\mathcal{C}_n^{-1}\left(\frac{\bar{f}}{\bar{g}}\right)\right]\\
&=\frac{1}{2}\left[\mathcal{C}_n^{-1}\left(\frac{f}{g}\right)T_n\left(\frac{f}{g}\right)+L_2+T_n\left(\frac{\bar{f}}{\bar{g}}\right)\mathcal{C}_n^{-1}\left(\frac{\bar{f}}{\bar{g}}\right)+L_2^H\right]\\
&=\frac{1}{2}\left[\mathcal{C}_n^{-1}\left(\frac{f}{g}\right)T_n\left(\frac{f}{g}\right)+T_n\left(\frac{\bar{f}}{\bar{g}}\right)\mathcal{C}_n^{-1}\left(\frac{\bar{f}}{\bar{g}}\right)\right]+L_3,
\end{split}
\end{equation*}
όπου $L_3$ είναι πίνακας χαμηλής βαθμίδας, το πολύ ίσης με $4d$ ($d$ είναι ο βαθμός του τριγωνομετρικού πολυωνύμου $g$). Παρατηρούμε ότι ο $H$, διαφέρει από το Ερμιτιανό μέρος του πίνακα $\mathcal{C}_n^{-1}\left(\frac{f}{g}\right)T_n\left(\frac{f}{g}\right)$, μόνο κατά τον $L_3$.

Επειδή ο $\mathcal{C}_n^{-1}\left(\frac{f}{g}\right)$ έχει φραγμένη νόρμα $\Vert\cdot\Vert_2$ και η $\frac{f}{g}$ έχει σημεία ασυνέχειας, μπορούμε να χρησιμοποιήσουμε το Λήμμα \ref{lem: Frobenius of product} και το Θεώρημα \ref{eig_discont} για να λάβουμε ότι το Ερμιτιανό μέρος του $\mathcal{C}_n^{-1}\left(\frac{f}{g}\right)T_n\left(\frac{f}{g}\right)$ έχει γενική συσσώρευση των ιδιοτιμών, γύρω από το 1 με $\mathcal{O}(\log{n})$ ιδιοτιμές εκτός του διαστήματος συσσώρευσης.

Εργαζόμενοι ανάλογα για το αντι-Ερμιτιανό μέρος λαμβάνουμε τη γενική συσσώρευση, γύρω από το 0 με $\mathcal{O}(\log{n})$ ιδιοτιμές εκτός του διαστήματος συσσώρευσης και η απόδειξη ολοκληρώθηκε.
\end{proof}


\begin{rem}
Αποδείξαμε τη γενική συσσώρευση των ιδιοτιμών και ιδιαζουσών τιμών, του προρρυθμισμένου πίνακα, όταν η γεννήτρια συνάρτηση $f$ έχει σημεία ασυνέχειας. Αυτό, εκ πρώτης όψεως είναι ένα αρνητικό αποτέλεσμα σε σχέση με τη συνεχή περίπτωση, όπου αποδείχθηκε κύρια συσσώρευση. Ωστόσο, η συνάρτηση του λογαρίθμου, η οποία χαρακτηρίζει τη γενική συσσώρευση, τείνει πολύ αργά προς το άπειρο και στα αριθμητικά αποτελέσματα η γενική συσσώρευση δε γίνεται αισθητή. Αντιθέτως, η συσσώρευση των ιδιοτιμών και ιδιαζουσών τιμών μοιάζει να είναι κύρια.
\end{rem}

\section{Αριθμητικά αποτελέσματα}
\label{sec:experiments}

Σε αυτή την ενότητα δίνουμε διάφορα αριθμητικά παραδείγματα, τα οποία έρχονται σε συμφωνία με τα θεωρητικά αποτελέσματα τα οποία αποδείξαμε, σχετικά με την προτεινόμενη τεχνική προρρύθμισης. Στα αριθμητικά πειράματα το διάνυσμα $b$, του δεξιού μέλους του αρχικού συστήματος, επιλέχθηκε και πάλι έτσι ώστε η λύση του συστήματος να είναι το διάνυσμα, του οποίου όλες οι συνιστώσες είναι ίσες με μονάδα, δηλαδή το $(1~1~\cdots~1)^T$. Ως αρχική προσέγγιση επιλέξαμε το μηδενικό διάνυσμα και ως κριτήριο τερματισμού: $\frac{\Vert r(k)\Vert_2}{\Vert r(0)\Vert_2}\leq 10^{-6}$, όπου $r(k)$ συμβολίζει, όπως και στο προηγούμενο κεφάλαιο, το διάνυσμα υπόλοιπο της $k$-οστής επανάληψης και $r(0)=b$.

Στους πίνακες δίνουμε τον αριθμό επαναλήψεων των μεθόδων {\en PGMRES} και {\en PCGN}, έως ότου να έχουμε την επιθυμητή σύγκλιση στη λύση του συστήματος. Χρησιμοποιούμε τον εξής συμβολισμό: $n$ είναι η διάσταση του συστήματος, με $I_n$ δηλώνουμε ότι δε χρησιμοποιήθηκε κανένας προρρυθμιστής, το $\mathcal{C}_n$ συμβολίζει τον προτεινόμενο προρρυθμιστή, ενώ το $\mathcal{T}_n$ συμβολίζει τον βέλτιστο κυκλοειδή προρρυθμιστή \cite{chan1988optimal,chan1993circulant,yeung1993circulant}. Με χρήση παρόμοιου συμβολισμού, όταν γίνεται άρση των ριζών της γεννήτριας συνάρτησης, με $B\mathcal{C}_n$ και $B\mathcal{T}_n$ θα δηλώνουμε τον αντίστοιχο ````ταινιωτο-επί-κυκλοειδή'''' {\en (Band-times-Circulant)} προρρυθμιστή.

\begin{exmp}\label{exp: ex1}\normalfont
Ως πρώτο παράδειγμα αυτού του κεφαλαίου επιλέγουμε τη $2\pi$-περιοδική και συνεχή συνάρτηση που είδαμε και στο Παράδειγμα \ref{exmp:231}, $\mathfrak{f}_1(x)=x^2+1+\mathrm{i}\mathfrak{h}_1(x)$, όπου: $$\mathfrak{h}_1(x)=\left\{
     \begin{array}{@{}c@{\thinspace}l}
       -\pi-x &,~-\pi\leq x< -\frac{\pi}{2}\\
       x &,~ -\frac{\pi}{2}\leq x<\frac{\pi}{2} \\
       \pi-x &,~\phantom{-}\frac{\pi}{2}\leq x\leq\pi\\
     \end{array}
   \right..$$

Προφανώς, το πραγματικό μέρος της $\mathfrak{f}_1$ είναι μια θετική συνάρτηση στο $(-\pi,\pi]$. Συνεπώς, για την επίλυση του συστήματος χρησιμοποιούμε τους κυκλοειδείς προρρυθμιστές $\mathcal{C}_n$ και $\mathcal{T}_n$. Ο Πίνακας \ref{tab:g1} δείχνει τον αριθμό επαναλήψεων, μέχρι τη σύγκλιση των μεθόδων {\en PGMRES} και {\en PCGN}. Παρατηρούμε ότι οι προρρυθμιστές συγκλινουν, στη λύση του συστήματος, με τον ίδιο αριθμό επαναλήψεων. `Οπως θα δούμε στα παραδείγματα που ακολουθούν, τόσο ο $\mathcal{C}_n$, όσο και ο $\mathcal{T}_n$ είναι αποτελεσματικοί, όμως σημειώνουμε ότι στα περισσότερα εξ αυτών, ο $\mathcal{C}_n$ συγκλίνει στη λύση με λιγότερες επαναλήψεις σε σύγκριση με τον $\mathcal{T}_n$. Επιπλέον, στο Σχήμα \ref{fig:x2+ix_eig} παρατηρούμε ότι η συσσώρευση των ιδιοτιμών είναι πολύ πιο πυκνή όταν χρησιμοποιούμε τον $B\mathcal{C}_n$, αντί του $B\mathcal{T}_n$.

\begin{table}
\centering
\begin{tabular}{cccc|ccc}
\toprule
\multirow{2}{*}{$n$} & \multicolumn{3}{c|} {\en PGMRES} & \multicolumn{3}{c} {\en PCGN} \\
 & $I_n$ & $\mathcal{C}_n$ & $\mathcal{T}_n$ & $I_n$ & $\mathcal{C}_n$ & $\mathcal{T}_n$ \\\midrule
\phantom{0}256 & 31 & 5 & 5 & 72 & 6 & 6 \\
\phantom{0}512 & 30 & 5 & 5 & 74 & 6 & 6 \\
1024 & 29 & 5 & 5 & 73 & 6 & 6 \\
2048 & 29 & 4 & 4 & 72 & 6 & 6 \\\bottomrule
\end{tabular}
\caption{Επαναλήψεις ($\mathfrak{f}_1$).}\label{tab:g1}
\end{table}
\end{exmp}

\begin{exmp}\label{exp: ex4}\normalfont
Σε αυτό το παράδειγμα η γεννήτρια συνάρτηση του πίνακα {\en Toeplitz} είναι η $\mathfrak{f}_2(x)=x^2+\mathrm{i}x^3$, βλ. επίσης Παράδειγμα \ref{exmp:232}. Αυτή έχει μια ρίζα στο 0 και το φανταστικό της μέρος έχει σημείο ασυνέχειας στο $\pi$. Η πολλαπλότητα της ρίζας εξαρτάται από το πραγματικό μέρος της $\mathfrak{f}_2$, γεγονός το οποίο μας οδηγεί στην επιλογή $g(x)=2-2\cos(x)$, για την άρση των ριζών. Η αναγκαιότητα, καθώς και τα πλεονεκτήματα της προρρύθμισης είναι εξώφθαλμα, όπως φαίνεται στον Πίνακα \ref{tab:x2+ix3}.

Η αποτελεσματικότητα της προτεινόμενης τεχνικής προρρύθμισης φαίνεται στον Πίνακα \ref{tab:x2+ix3}. Εκεί παρατηρούμε μια ελαφρώς καλύτερη συμπεριφορά του $B\mathcal{C}_n$ σε σχέση με τον ````ταινιωτό-επί-βέλτιστο κυκλοειδή'''' {\en (Band-times-optimal Circulant)}, όταν παίρνουμε τη λύση μέσω της μεθόδου {\en PCGN} (βλ. επίσης \cite{tyrtyshnikov1995circulant}).

\begin{table}[H]
\centering
\begin{tabular}{ccccc|cccc}
\toprule
\multirow{2}{*}{$n$} & \multicolumn{4}{c|} {\en PGMRES} & \multicolumn{4}{c} {\en PCGN} \\
 & $I_n$ & $\mathcal{T}_n$ & $B\mathcal{C}_n$ & $B\mathcal{T}_n$ & $I_n$ & $\mathcal{T}_n$ & $B\mathcal{C}_n$ & $B\mathcal{T}_n$ \\\midrule
\phantom{0}256 & \phantom{$>$}256 & 22 & 7 & 7 & - & \phantom{0}61 & 13 & 14 \\
\phantom{0}512 & $>$500 & 28 & 7 & 7 & - & \phantom{0}90 & 14 & 17 \\
1024 & $>$500 & 36 & 7 & 7 & - & 140 & 15 & 17 \\
2048 & $>$500 & 39 & 7 & 8 & - & 273 & 16 & 19 \\\bottomrule
\end{tabular}
\caption{Επαναλήψεις ($\mathfrak{f}_2$).}\label{tab:x2+ix3}
\end{table}

\begin{figure}
    \centering
    \subfloat[Ιδιοτιμές όταν $n=256$.]{{\label{fig:4a}\includegraphics[width=0.45\linewidth]{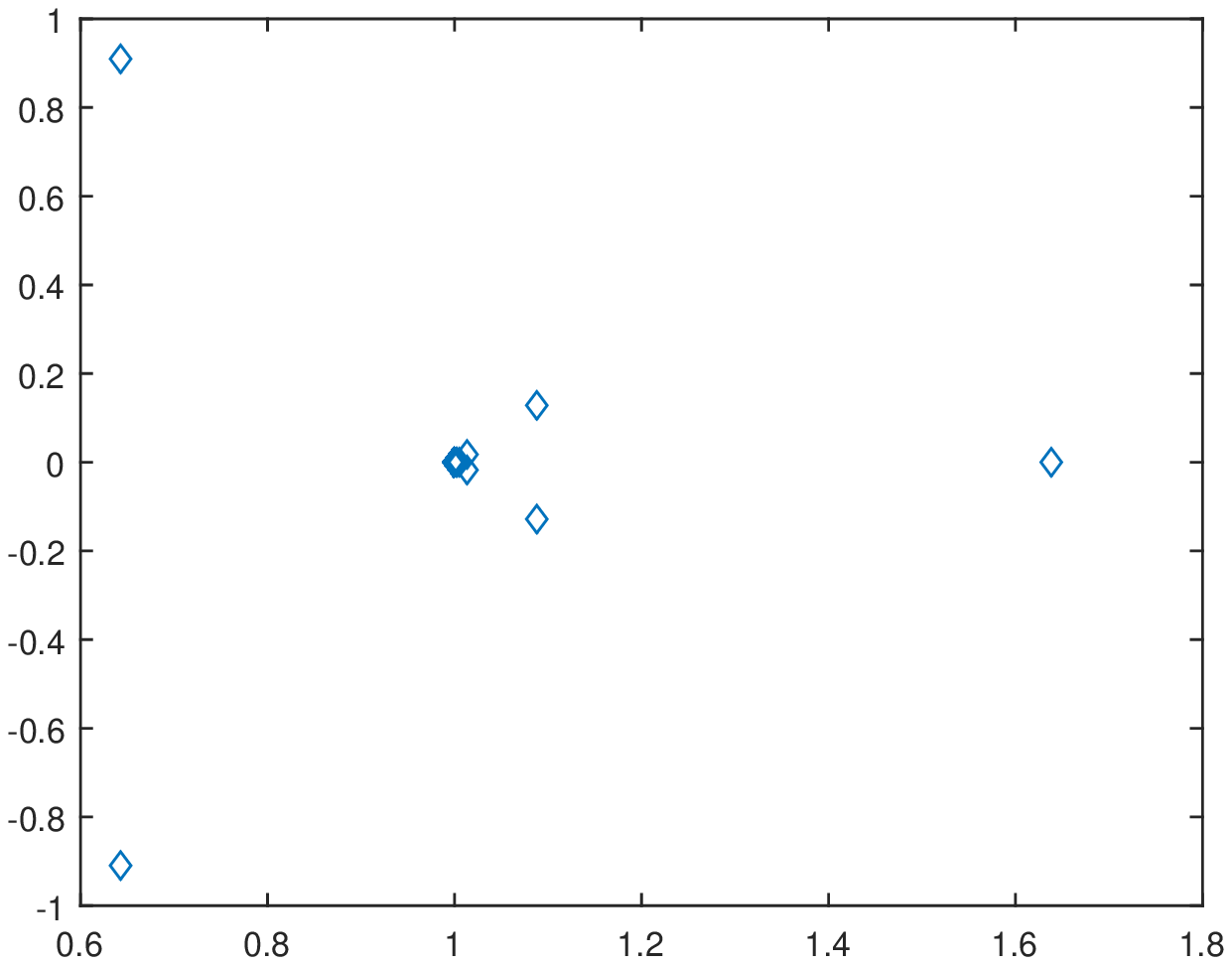}}}%
    \qquad
    \subfloat[Ιδιοτιμές για διάφορες διαστάσεις.]{{\label{fig:4b}\includegraphics[width=0.45\linewidth]{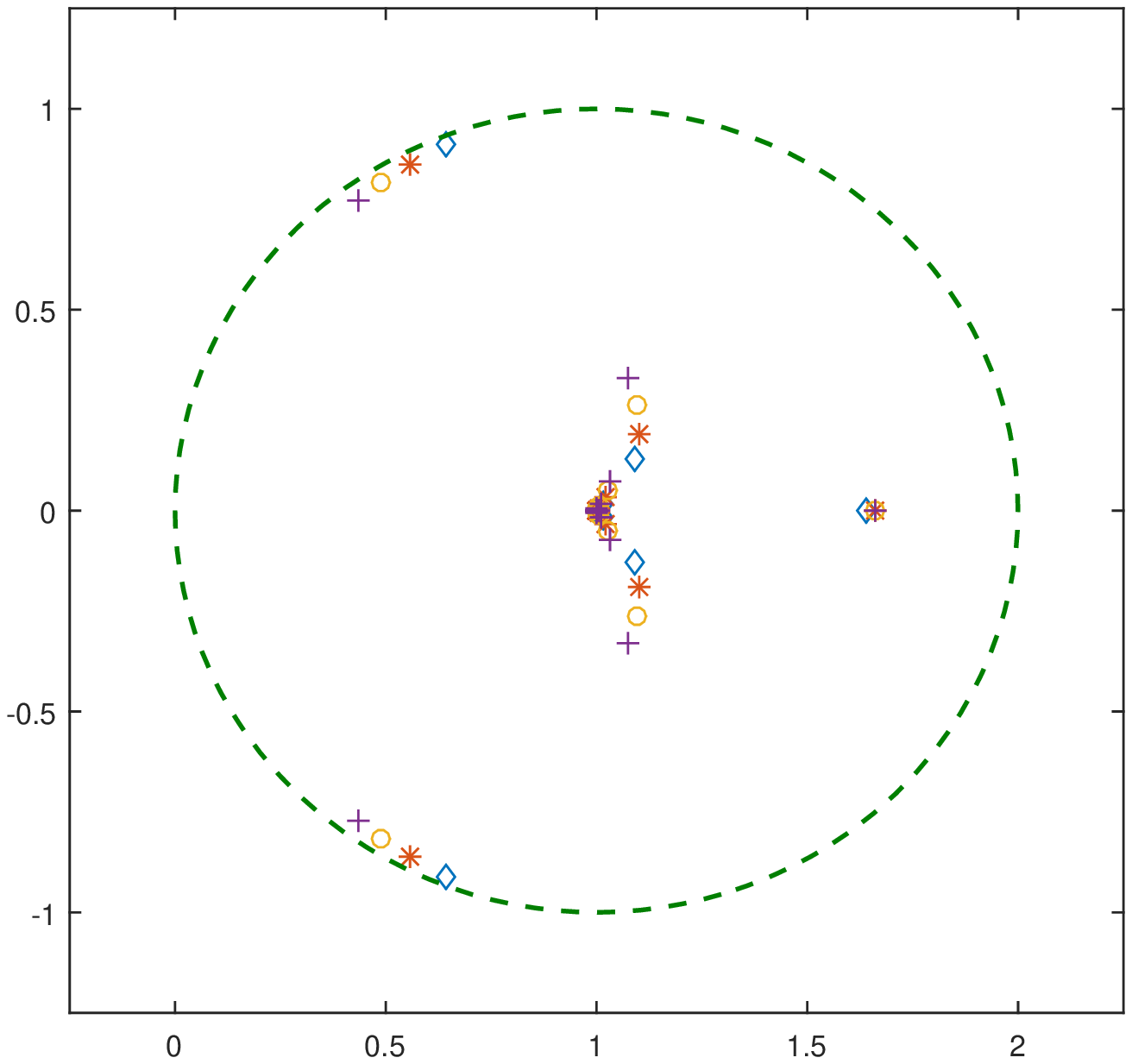}}}%
    \caption{Ιδιοτιμές ($\mathfrak{f}_2$).}
    \label{fig:x2+ix3}
\end{figure}

Στο Σχήμα \ref{fig:4b} δίνουμε τη συσσώρευση των ιδιοτιμών του προρρυθμισμένου πίνακα $\mathcal{C}_n^{-1}\left(\frac{\mathfrak{f}_2}{g}\right)T_n^{-1}(g)T_n(\mathfrak{f}_2)$ από τη διάσταση $n=256$ (μπλε διαμάντια), έως $n=2048$ (μωβ σταυροί). Με πορτοκαλί αστέρια και κίτρινους κύκλους συμβολίζουμε τις αντίστοιχες ιδιοτιμές για $n=512$ και $n=1024$, αντίστοιχα. Ο κύκλος με την πράσινη διακεκομμένη γραμμή, έχει κέντρο το $(1,0)$ και ακτίνα ίση με 1. Παρατηρούμε ότι όλες οι ιδιοτιμές είναι εντός του κύκλου.
\end{exmp}

\begin{exmp}\label{exp: ex3}\normalfont
Σε αυτό το παράδειγμα η γεννήτρια συνάρτηση του πίνακα {\en Toeplitz} είναι η $\mathfrak{f}_3(x)=x^2+\mathrm{i}x$, που μελετήθηκε με διαφορετική τεχνική προρρύθμισης στο Παράδειγμα \ref{exmp:233}. Αυτή, όπως και στο προηγούμενο παράδειγμα, έχει ρίζα στο 0 και το φανταστικό της μέρος παρουσιάζει ασυνέχεια στο $\pi$. Η διαφορά εντόπίζεται στο ότι η πολλαπλότητα της ρίζας εξαρτάται από το φανταστικό μέρος και όχι από το πραγματικό. Η άρση των ριζών είναι απαραίτητη για να επιτύχουμε ταχύτερη σύγκλιση στη λύση του συστήματος. Ο ταινιωτός πίνακας {\en Toeplitz} θα έχει ως γεννήτρια συνάρτηση την $g(x)=2-2\cos(x)+\mathrm{i}\sin(x)$ (βλ. υποενότητα \ref{Ss:32}).

Ο Πίνακας \ref{tab:x2+ix} δείχνει τον αριθμό επαναλήψεων για τις μεθόδους {\en PGMRES} και {\en PCGN}. Στα Σχήματα \ref{fig:x2+ix_sv} και \ref{fig:x2+ix_eig} παρουσιάζουμε τη συσσώρευση ιδιαζουσών τιμών και των ιδιοτιμών, αντίστοιχα, όταν $n=256$. Πιο συγκεκριμένα, αυτές που αφορούν στον προρρυθμιστή $B\mathcal{C}_n$ σημειώνονται με μπλε διαμάντια, ενώ αυτές που αφορούν στον $B\mathcal{T}_n$, με πορτοκαλί αστέρια. Σε αυτό το παράδειγμα είναι ολοφάνερη η αναγκαιότητα προρρύθμισης, διότι χωρίς αυτή, η λύση του συστήματος δίνεται σε επαναλήψεις ίσες με τη διάσταση $n$.

\begin{table}
\centering
\begin{tabular}{ccccc|cccc}
\toprule
\multirow{2}{*}{$n$} & \multicolumn{4}{c|} {\en PGMRES} & \multicolumn{4}{c} {\en PCGN} \\
 & $I_n$ & $\mathcal{T}_n$ & $B\mathcal{C}_n$ & $B\mathcal{T}_n$ & $I_n$ & $\mathcal{T}_n$ & $B\mathcal{C}_n$ & $B\mathcal{T}_n$ \\\midrule
\phantom{0}256 & \phantom{$>$}256 & 9 & 5 & 5 & - & 10 & 7 & 7 \\
\phantom{0}512 & $>$500 & 9 & 5 & 5 & - & 11 & 7 & 7 \\
1024 & $>$500 & 9 & 5 & 5 & - & 11 & 8 & 7 \\
2048 & $>$500 & 9 & 5 & 5 & - & 11 & 8 & 8 \\\bottomrule
\end{tabular}
\caption{Επαναλήψεις ($\mathfrak{f}_3$).}\label{tab:x2+ix}
\end{table}

\begin{figure}[htbp]
    \centering
    \includegraphics[width=0.75\linewidth]{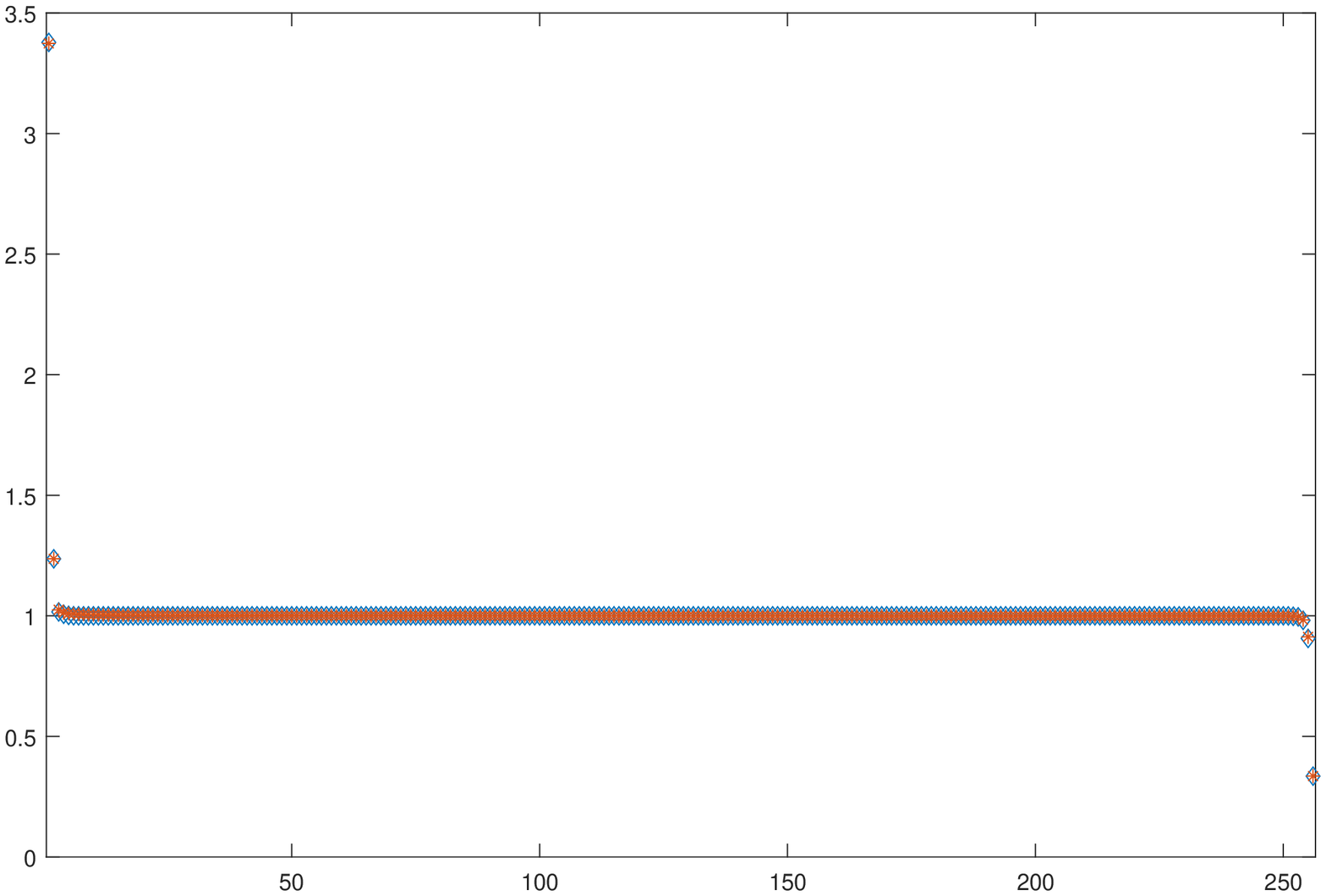}%
    \caption{Ιδιάζουσες τιμές ($\mathfrak{f}_3$).}
    \label{fig:x2+ix_sv}
\end{figure}

\begin{figure}[htbp]%
    \centering
	\subfloat[Ιδιοτιμές.]{{\label{fig:3c}\includegraphics[width=0.45\linewidth]{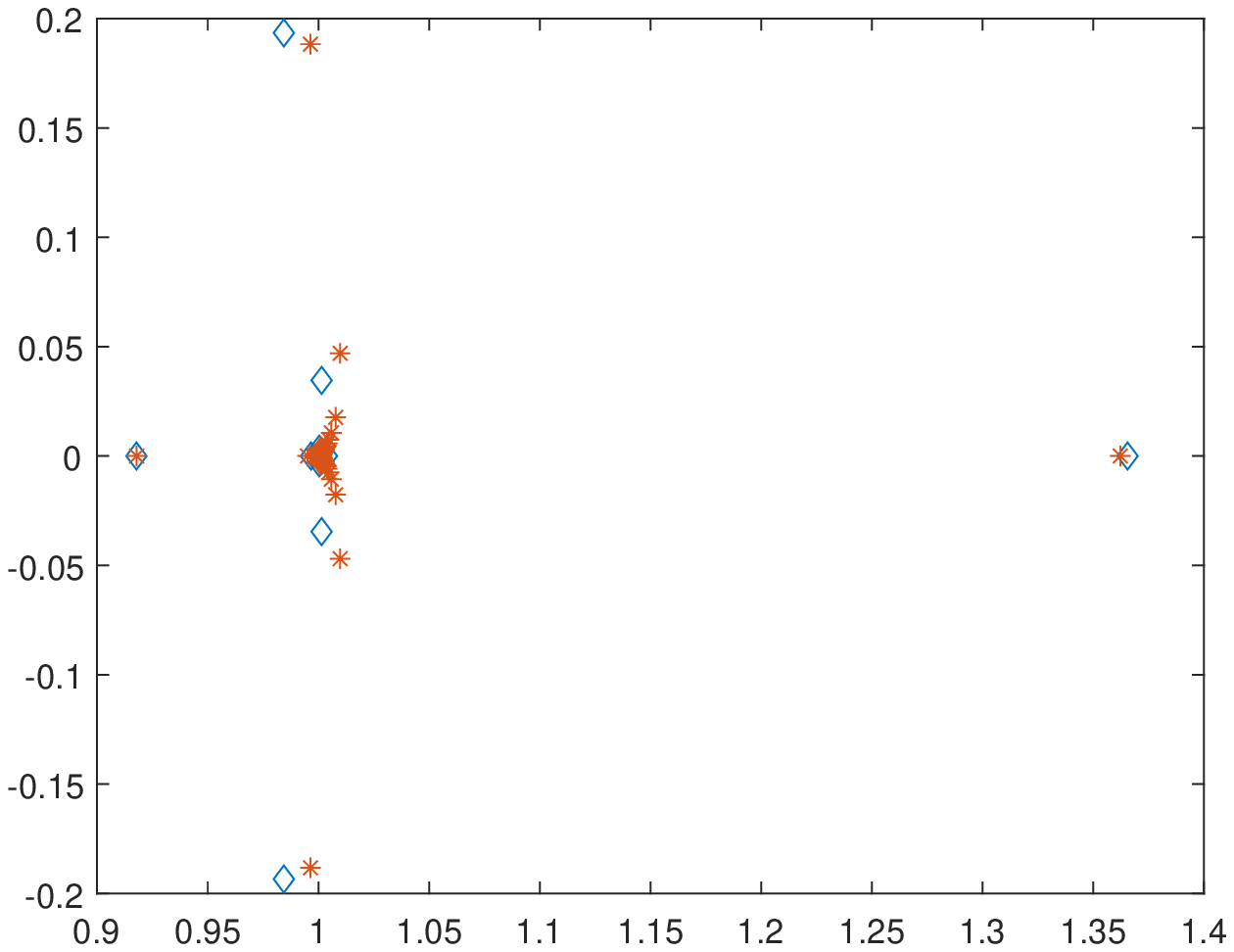}}}%
    \qquad
    \subfloat[Ιδιοτιμές κοντά στο $(1,0)$.]{{\label{fig:3d}\includegraphics[width=0.45\linewidth]{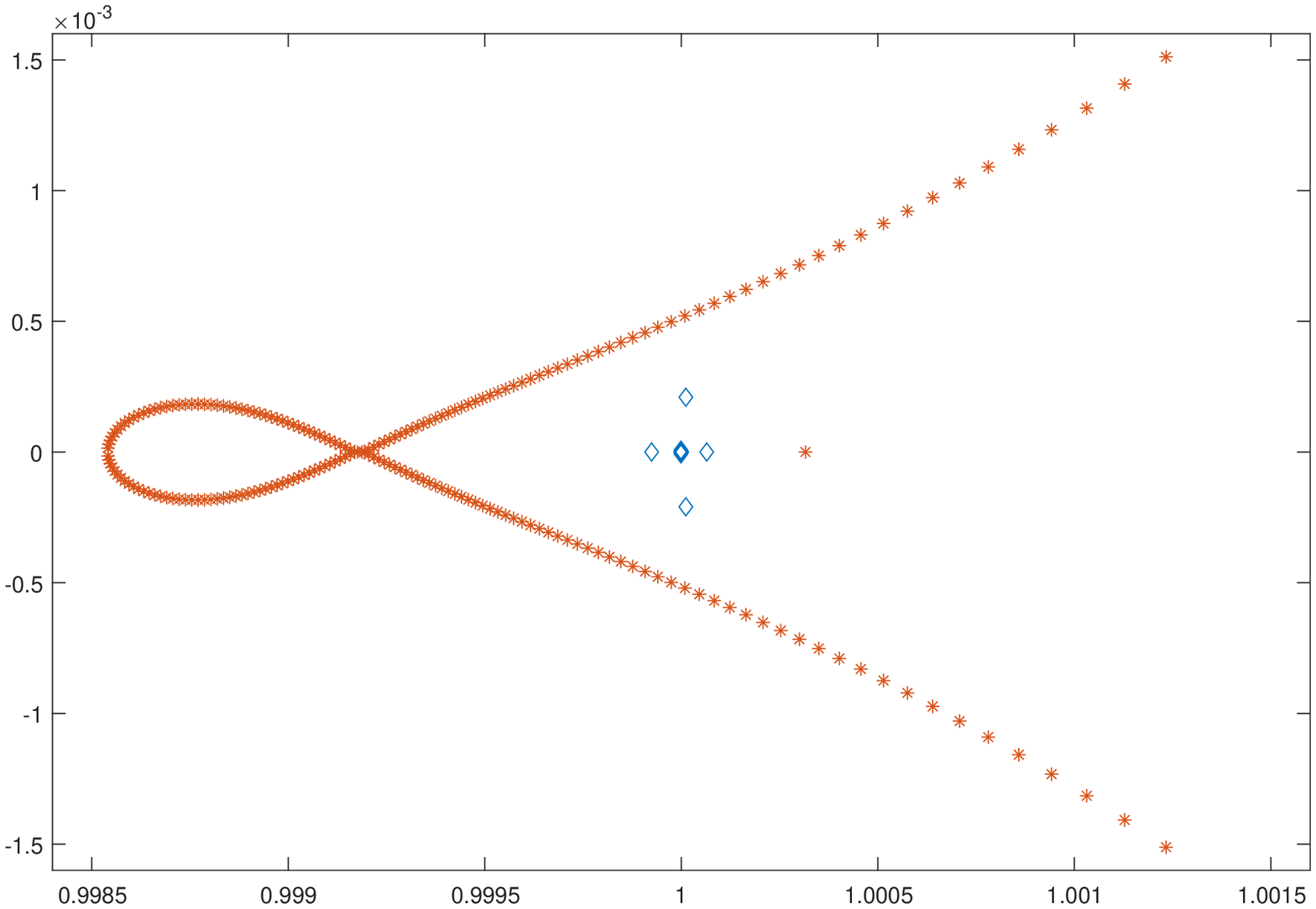}}}    
    \caption{Ιδιοτιμές ($\mathfrak{f}_3$).}
    \label{fig:x2+ix_eig}
\end{figure}

Θα θέλαμε να σημειώσουμε ότι οι αριθμοί επαναλήψεων μεταξύ του $\mathcal{T}_n$ και $B\mathcal{T}_n$ δε διαφέρουν σε μεγάλο βαθμό, γεγονός το οποίο δεν ισχύει στο προηγούμενο παράδειγμα. Εκεί, η διαφορά είναι αισθητή και αυτό οφείλεται στην τάξη της ρίζας, η οποία ήταν ίση με 2, ενώ στο παρόν παράδειγμα ίση με 1. `Οσο μεγαλύτερη δηλαδή είναι η τάξη της ρίζας, τόσο μεγαλύτερη είναι και η διαφορά στην αποτελεσματικότητα των προρρυθμιστών $\mathcal{T}_n$ και $B\mathcal{T}_n$.

Αν και φαίνεται ότι οι ιδιάζουσες τιμές των προρρυθμισμένων συστημάτων ταυτίζονται, στο Σχήμα \ref{fig:x2+ix_eig} παρατηρούμε ότι η χρήση του $B\mathcal{C}_n$, αντί του $B\mathcal{T}_n$ οδηγεί σε ένα πιο πυκνό σύνολο συσσώρευσης των ιδιοτιμών, σε μια μικρή περιοχή κοντά στο $(1,0)$.
\end{exmp}

\begin{exmp}\label{exp: ex5}\normalfont
Δίνουμε ένα παράδειγμα, στο οποίο η γεννήτρια συνάρτηση του πίνακα {\en Toeplitz} είναι η $\mathfrak{f}_7(x)=x^2-1+\mathrm{i}x^3$. Αυτή δεν έχει ρίζες στο $(-\pi,\pi]$, αλλά το πραγματικό της μέρος παίρνει τόσο θετικές, όσο και αρνητικές τιμές. Προφανώς, το φανταστικό της μέρος έχει ασυνέχεια στο $\pi$. Χρησιμοποιούμε τους κυκλοειδείς προρρυθμιστές $\mathcal{C}_n$ και $\mathcal{T}_n$.

\begin{table}[htbp]
\centering
\begin{tabular}{cccc|ccc}
\toprule
\multirow{2}{*}{$n$} & \multicolumn{3}{c|} {\en PGMRES} & \multicolumn{3}{c} {\en PCGN} \\
 & $I_n$ & $\mathcal{C}_n$ & $\mathcal{T}_n$ & $I_n$ & $\mathcal{C}_n$ & $\mathcal{T}_n$ \\\midrule
\phantom{0}256 & 160 & 9 & \phantom{0}9 & - & 11 & 13 \\
\phantom{0}512 & 248 & 9 & \phantom{0}9 & 362 & 12 & 13 \\
1024 & 326 & 9 & \phantom{0}9 & 397 & 12 & 13 \\
2048 & 370 & 9 & 10 & 415 & 12 & 13 \\\bottomrule
\end{tabular}
\caption{Επαναλήψεις ($\mathfrak{f}_7$).}\label{tab:x2-1+ix3}
\end{table}

\begin{figure}[htbp]%
    \centering
    \includegraphics[width=0.75\linewidth]{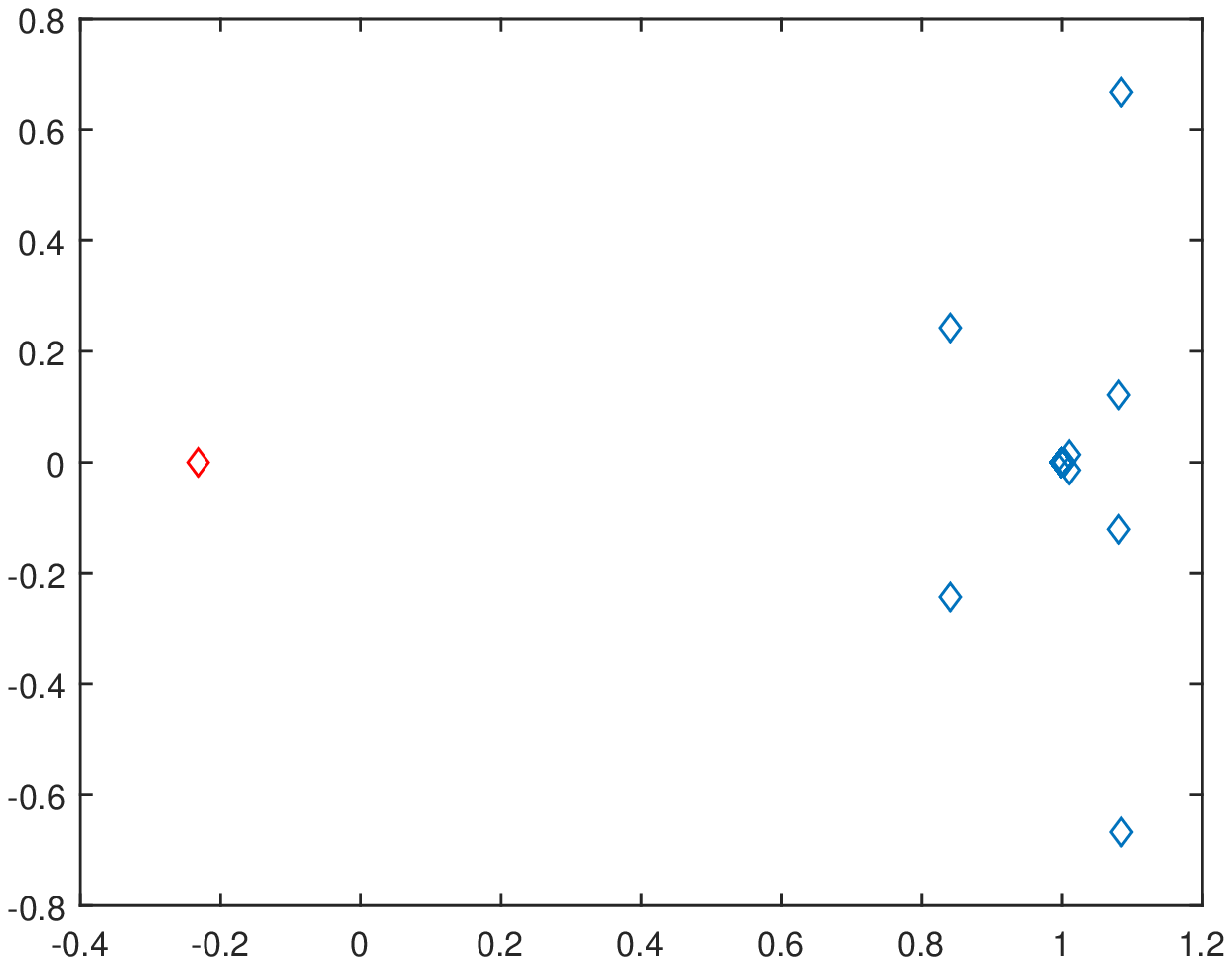}%
    \caption{Ιδιοτιμές ($\mathfrak{f}_7$).}
    \label{fig:x2-1+ix3_eig}
\end{figure}

Οι αριθμοί επαναλήψεων φαίνονται στον Πίνακα \ref{tab:x2-1+ix3}. Στο Σχήμα \ref{fig:x2-1+ix3_eig} δίνουμε τη συσσώρευση των ιδιοτιμών του $\mathcal{C}_n^{-1}(f)T_n(\mathfrak{f}_7)$, όταν $n=256$. Βλέπουμε ότι και πάλι αυτές συσσωρεύονται γύρω από το $(1,0)$. Σχολιάζουμε ότι αν και υπάρχει μια αρνητική ιδιοτιμή, η οποία συμβολίζεται ως κόκκινο διαμάντι, η μέθοδος {\en PGMRES} είναι αποτελεσματική. Γνωρίζουμε ότι αν μια ιδιοτιμή του προρρυθμισμένου πίνακα δεν ανήκει στον δίσκο με κέντρο το $(1,0)$ και ακτίνα ίση με 1, ο αριθμός επαναλήψεων αυξάνεται κατά 1 \cite{gmati2007comment}.
\end{exmp}

\begin{rem}
Διάφορες τεχνικές προρρύθμισης, με χρήση κυκλοειδών πινάκων, μπορούν να βρεθούν στις \cite{honwathen,pestanawathen} και \cite{potts}. Στις πρώτες δύο οι συγγραφείς συμμετρικοποιούν τον αρχικό πίνακα συντελεστών και λύνουν το σύστημα που προκύπτει με την Προρρυθμισμένη μέθοδο Ελαχίστων Υπολοίπων {\en (PMINRES)} \cite{minres}. Στην τελευταία, οι συγγραφείς αναλύουν προρρυθμιστές οι οποίοι ανήκουν σε τριγωνομετρική άλγεβρα, για μη-συμμετρικά συστήματα {\en Toeplitz}. Επικεντρώνονται στην επίλυση του συστήματος με τη μέθοδο {\en PCGN}, αλλά χρησιμοποιούν και {\en PGMRES}. Σημειώνουμε ότι στις παραπάνω εργασίες δεν έγινε χρήση ταινιωτών-επί-άλγεβρα {\en (Band-times-Algebra)} προρρυθμιστών. Θεωρητικά αποτελέσματα για τη συσσώρευση των ιδιαζουσών τιμών μπορούν επίσης να βρεθούν στην \cite{barakitis}.
\end{rem}

Στη συνέχεια, θα δώσουμε τον αριθμό επαναλήψεων χρησιμοποιώντας τη μέθοδο {\en PGMRES}, για τα Παραδείγματα 1 και 3 της \cite{honwathen}. Προσαρμόσαμε τις επιλογές στα αριθμητικά πειράματα, έτσι ώστε να είναι ίδιες με αυτές της \cite{honwathen}, για να έχουμε αμεσότερη σύγκριση. Πιο συγκεκριμένα, αλλάξαμε το διάνυσμα $b$, έτσι ώστε αυτό να έχει όλες τις συνιστώσες του ίσες με 1. Επιπλέον αλλάξαμε το κριτήριο τερματισμού σε $\frac{\Vert r(k)\Vert_2}{\Vert r(0)\Vert_2}< 10^{-7}$ (ακριβώς όπως στην \cite{honwathen}).

\begin{exmp}\label{exp: ex6}\normalfont
Σε αυτό το παράδειγμα, πίνακας συντελεστών είναι ο πίνακας {\en Grcar}:
\begin{equation*}
G_n=\begin{bmatrix}
1 & 1 & 1 & 1 & 0 &\cdots & 0\\
-1 &\ddots &\ddots  &\ddots &\ddots &\ddots &\vdots\\
0 &\ddots &\ddots  &\ddots &\ddots &\ddots & 0\\
\vdots &\ddots &\ddots  &\ddots &\ddots &\ddots &1\\
\vdots &\ddots &\ddots  &\ddots &\ddots &\ddots &1\\
\vdots &\ddots &\ddots  &\ddots &\ddots &\ddots &1\\
0 &\cdots &\cdots  &\cdots  & 0 & -1 & 1\\
\end{bmatrix}.
\end{equation*}

Ο $G_n$ έχει ως γεννήτρια συνάρτηση το τριγωνομετρικό πολυώνυμο $\mathfrak{f}_8(x)=1+\cos{(2x)}+\cos{(3x)}+\mathrm{i}\left[-2\sin{(x)}-\sin{(2x)}-\sin{(3x)}\right]$. Οι αριθμοί επαναλήψεων εφαρμόζοντας {\en PGMRES} και τους προρρυθμιστές $\mathcal{C}_n$ και $\mathcal{T}_n$, καθώς επίσης και {\en PMINRES} με τον προρρυθμιστή που προτάθηκε στην \cite{honwathen} (αυτός είναι ο $\vert\mathcal{T}_n\vert$) δίνεται στον Πίνακα \ref{tab:Grcar}. Δίνουμε επίσης και τους αριθμούς επαναλήψεων, με χρήση του $\vert\mathcal{C}_n\vert$ ως προρρυθμιστή, του οποίου η αποτελεσματικότητα αποδείχθηκε στην \cite{hon_simax}.

\begin{table}[H]
\centering
\begin{tabular}{cccc|cc}
\toprule
\multirow{2}{*}{$n$} & \multicolumn{3}{c|} {\en PGMRES} & \multicolumn{2}{c} {\en PMINRES} \\
 & $I_n$ & $\mathcal{C}_n$ & $\mathcal{T}_n$ & $\vert\mathcal{C}_n\vert$ & $\vert\mathcal{T}_n\vert$\\\midrule
\phantom{0}128 & \phantom{0}94 & 4 & 6 & 9 & 13\\
\phantom{0}256 & 158 & 4 & 6 & 9 & 12\\
\phantom{0}512 & 218 & 4 & 6 & 9 & 11\\
1024 & 213 & 4 & 5 & 9 & 11\\\bottomrule
\end{tabular}
\caption{Επαναλήψεις ($\mathfrak{f}_8$).}
\label{tab:Grcar}
\end{table}

Αν και το κόστος ανά επανάληψη της μεθόδου {\en PMINRES}, είναι λιγότερο από αυτό της {\en PGMRES}, διότι το κόστος της τελευταίας αυξάνεται από επανάληψη σε επανάληψη, παρατηρούμε ότι οι αριθμοί επαναλήψεων που δίνονται από τον $\mathcal{C}_n$ είναι επαρκώς μικρότεροι από αυτούς της {\en PMINRES}. Οι 4 επαναλήψεις είναι πολύ λίγες κι έτσι το κόστος ανά επανάληψη είναι περίπου το ίδιο με αυτό της {\en PMINRES}.

Παρατηρούμε ότι ο $\mathcal{C}_n$ είναι πιο αποτελεσματικός από τον βέλτιστο κυκλοειδή προρρυθμιστή. Αυτό ισχύει κι όταν ο πίνακας συντελεστών του συστήματος δίνεται από τη συνάρτηση υπερβολικού ημιτόνου του $G_n$. Χρησιμοποιούμε τους προρρυθμιστές $\sinh{\mathcal{C}_n}$ και $\sinh{\mathcal{T}_n}$ και δίνουμε τα αντίστοιχα αποτελέσματα στον Πίνακα \ref{tab:sinhGrcar}.

\begin{table}[H]
\centering
\begin{tabular}{cccc}
\toprule
$n$ & $I_n$ & $\sinh{\mathcal{C}_n}$ & $\sinh{\mathcal{T}_n}$\\\midrule
\phantom{0}64 & \phantom{0}64 & 8 & 14 \\
128 & 124 & 9 & 13\\
256 & 240 & 9 & 11\\
512 & 486 & 9 & 10\\\bottomrule
\end{tabular}
\caption{Επαναλήψεις για τον $\sinh{G_n}$.}
\label{tab:sinhGrcar}
\end{table}
\end{exmp}

\begin{exmp}\label{exp: ex8}\normalfont
Κατόπιν διακριτοποίησης ολοκληρω-διαφορικών {\en (integro-differential)} εξισώσεων, υπάρχει περίπτωση ο πίνακας συντελεστών να σχετίζεται με την εκθετική συνάρτηση ενός πίνακα {\en Toeplitz} \cite{kressner}. `Εστω $\mathfrak{f}_2(x)=x^2+\mathrm{i}x^3$ η γεννήτρια συνάρτηση του $T_n(\mathfrak{f}_2)$. Μπορούμε να χρησιμοποιήσουμε τον $\mathrm{e}^{B\mathcal{C}_n}$ ως προρρυθμιστή για τον $\mathrm{e}^{T_n(\mathfrak{f}_2)}$. Στον Πίνακα \ref{tab:exp} δίνουμε τους αριθμούς επαναλήψεων, εφαρμόζοντας {\en PGMRES}. Σημειώνεται ότι αν και ο πίνακας $\mathrm{e}^{T_n(\mathfrak{f}_2)}$ δεν είναι {\en Toeplitz}, ο προτεινόμενος προρρυθμιστής επιτυγχάνει την ταχεία σύγκλιση στη λύση του συστήματος.

\begin{table}[H]
\centering
\begin{tabular}{ccc}
\toprule
$n$ & $I_n$ & $\mathrm{e}^{B\mathcal{C}_n}$\\\midrule
\phantom{0}256 & \phantom{$>$}255 & 11 \\
\phantom{0}512 & \phantom{$>$}497 & 12 \\
1024 & $>$500 & 12 \\
2048 & $>$500 & 13 \\\bottomrule
\end{tabular}
\caption{Επαναλήψεις για τον $\mathrm{e}^{T_n(\mathfrak{f}_2)}$.}
\label{tab:exp}
\end{table}
\end{exmp}

\newpage\thispagestyle{empty}\mbox{}\newpage 

\pagestyle{main}

\chapter{Συστήματα {\en Toeplitz} με `Αγνωστη Γεννήτρια Συνάρτηση}

Σε αυτό το κεφάλαιο μελετάμε την προρρύθμιση $n\times n$ μη συμμετρικών, πραγματικών συστημάτων {\en Toeplitz}, όταν η γεννήτρια συνάρτηση του πίνακα συντελεστών $T_n$ δεν είναι γνωστή εκ των προτέρων, όμως γνωρίζουμε ότι μια γεννήτρια συνάρτηση $f$, η οποία σχετίζεται με την ακολουθία πινάκων $\{T_n\}$, $T_n=T_n(f)$, όντως υπάρχει. Γίνεται κατάλληλη προσαρμογή, τόσο των ταινιωτών προρρυθμιστών του δευτέρου κεφαλαίου, όσο και των κυκλοειδών/ταινιωτών-επί-κυκλοειδών προρρυθμιστών του προηγούμενου κεφαλαίου. Αναλύεται ο τρόπος κατασκευής των προρρυθμιστών, από τις τιμές του πίνακα συντελεστών και μελετάται η συσσώρευση των ιδιοτιμών και ιδιαζουσών τιμών του προρρυθμισμένου συστήματος.

\section{Ταινιωτοί προρρυθμιστές}

Θα ξεκινήσουμε από την παρουσίαση των ταινιωτών {\en Toeplitz} προρρυθμιστών, δίνοντας τον τρόπο κατασκευής αυτών και μελετώντας τόσο τη συνεχή, όσο και την ασυνεχή περίπτωση. Λόγω του ότι δεν είναι γνωστό σε ποια από τις δύο περιπτώσεις βρισκόμαστε, θα δώσουμε έναν τρόπο εύρεσης πιθανών σημείων ασυνέχειας. Στο τέλος της ενότητας θα δώσουμε ορισμένα αριθμητικά παραδείγματα, τα οποία υποδεικνύουν την αποτελεσματικότητα του προτεινόμενου προρρυθμιστή.

\subsection{Κατασκευή του προρρυθμιστή}\label{Sss:411}
Αρχικά, θα θέλαμε να σημειώσουμε ότι αφού ο πίνακας $T_n$ είναι πραγματικός και μη-συμμετρικός, προκύπτει από μια συνάρτηση $f=f_1+\mathrm{i}f_2$, όπου $f_1$ είναι άρτια, $f_2$ περιττή και $\mathrm{i}$ είναι η φανταστική μονάδα. Για να εκτιμήσουμε τις ρίζες της $f$, θα πρέπει να προσεγγίσουμε τις συναρτήσεις $f_1$ και $f_2$, που την απαρτίζουν, χρησιμοποιώντας τις τιμές του αρχικού πίνακα συντελεστών. Αφού επιλέξουμε ένα ισοκατανεμημένο πλέγμα $G_n=\lbrace \theta_j\rbrace$, όπου
\begin{equation*}
\theta_j=-\pi+\frac{2j\pi}{n+1},~j=1,\dots,n,
\end{equation*}
θα προσεγγίσουμε τις συναρτήσεις $f_1$ και $f_2$ (στο $G_n$). Μια προφανής προσέγγιση αποτελεί το ανάπτυγμα {\en Fourier}, διότι η γεννήτρια συνάρτηση του πίνακα $T_n$ είναι εξ ορισμού το ίδιο το ανάπτυγμα {\en Fourier}, αναλόγως με τη διάσταση $n$. Το πηλίκο {\en Rayleigh} με χρήση κατάλληλων διανυσμάτων αποτελεί ακόμη ένα μαθηματικό εργαλείο για την προσέγγιση της $f$ \cite{Serra_1999}. Ωστόσο, αν αυτό εφαρμοστεί σε όλο το πλέγμα $G_n$, το υπολογιστικό κόστος υπερβαίνει το $\mathcal{O}(n\log{n})$ κι έτσι η χρήση του γίνεται πρακτικά απαγορευτική. Αυτός είναι ο λόγος που προτιμάμε τον υπολογισμό του αναπτύγματος {\en Fourier} στα σημεία του $G_n$, ο οποίος μπορεί να γίνει σε $\mathcal{O}(n\log{n})$, χρησιμοποιώντας τον ταχύ μετασχηματισμό {\en Fourier}. $\forall j=1,\dots,n$ έχουμε:
\begin{equation}\label{Fourier exp}
f(\theta_j)\simeq F_{n-1}(\theta_j)=\sum\limits_{k=-n+1}^{n-1}t_k\mathrm{e}^{\mathrm{i}k\theta_j}.
\end{equation}
Χωρίζοντας το πραγματικό και φανταστικό μέρος, των λαμβανόμενων τιμών, προσεγγίζουμε τις συναρτήσεις $f_1$ και $f_2$, αντίστοιχα. Παρακάτω θα αναλύσουμε τη διαδικασία επιλογής μιας πιθανής ρίζας για τη συνάρτηση $f_1$. Παρόμοια ανάλυση ισχύει και για τη συνάρτηση $f_2$.

Γνωρίζοντας ότι η $f_1$ είναι άρτια συνάρτηση, συμπεραίνουμε ότι αυτή θα μπορούσε να έχει ρίζες είτε ανάμεσα σε δύο διαδοχικά σημεία $\theta_j$ και $\theta_{j+1}$, όπου λαμβάνει διαφορετικό πρόσημο (σε αυτή την περίπτωση η συνάρτηση $f_1$ τέμνει τον άξονα), είτε ανάμεσα σε δύο σημεία $\theta_{j-1}$ και $\theta_{j+1}$, με $f_1(\theta_j)$ να λαμβάνει μια πολύ μικρή τιμή, σχεδόν ίση με μηδέν κι επιπλέον η ακολουθία $\lbrace f_1(\theta_i),~i=1,\dots,n\rbrace$ να αλλάζει τοπικά μονοτονία στο $\theta_j$ (σε αυτή την περίπτωση η $f_1$ εφάπτεται στον άξονα). Λαμβάνοντας υπόψη τη λεπτότητα του πλέγματος  $G_n$, οι παραπάνω περιπτώσεις μπορούν να ενοποιηθούν ως εξής: Επιλέγουμε το $\theta_i$, $1\leq i\leq n$ ως σημείο πιθανής ρίζας αν $\vert\operatorname{Re}\left(F_{n-1}(\theta_i)\right)\vert$ λαμβάνει πολύ μικρές τιμές κοντά στο μηδέν, π.χ. $\vert\operatorname{Re}\left(F_{n-1}(\theta_i)\right)\vert<10^{-6}$. Επισημαίνουμε ότι σε περίπτωση που $\vert\operatorname{Re}\left(F_{n-1}(\theta_i)\right)\vert=\vert \operatorname{Re}\left(F_{n-1}(\theta_{i+1})\right)\vert\simeq 0$, θέτουμε ως σημείο πιθανής ρίζας τον μέσο όρο των τιμών $\theta_i$ και $\theta_{i+1}$.

Σχολιάζουμε ότι η παραπάνω διαδικασία εφαρμόζεται σε συστήματα {\en Toeplitz} μεγάλης διάστασης. Για συστήματα μικρής διάστασης θα ήταν καλό να προχωρήσουμε και σε μια τεχνική εκλέπτυνσης κοντά στα σημεία πιθανών ριζών, με χρήση του πηλίκου {\en Rayleigh}, όπως έγινε στις \cite{korovkin, NSV_2005, Serra_1999, capizzano2000some}).

Για την εύρεση του κατάλληλου τριγωνομετρικού πολυωνύμου, το οποίο αίρει την κακή κατάσταση του αρχικού συστήματος, είναι αναγκαία και η εύρεση της πολλαπλότητας κάθε ρίζας. `Εστω $m_i^1$ και $m_i^2$ οι πολλαπλότητες των ριζών της $f_1$ και $f_2$, αντίστοιχα, οι οποίες αφορούν στο σημείο πιθανής ρίζας $x_i$, $i=1,2,\dots,\rho$. Από εδώ και στο εξής με $m_0^1$ και $m_0^2$ θα συμβολίζουμε τις πολλαπλότητες της ρίζας στο $x_0=0$, για την $f_1$ και $f_2$, αντίστοιχα.

Θα περιγράψουμε την εκτίμηση της πολλαπλότητας των ριζών για τη συνάρτηση $f_1$. Αρχικά, υποθέτουμε για απλούστευση ότι $\operatorname{Re}(F_{n-1}(\theta_j))\geq0$, $\forall j=1,2,\dots,n$. Επιλέγουμε τον άνω αριστερά κύριο τετραγωνικό υποπίνακα του $T_n$, μιας συγκεκριμένης μικρής διάστασης, π.χ. $64\times 64$. Σκοπός μας είναι να μελετήσουμε τη συμπεριφορά της ιδιοτιμής που αντιστοιχεί στη ρίζα που εξετάζουμε πηγαίνοντας από κάποια μικρή διάσταση στη διπλάσιά της. Λαμβάνοντας τον $k\times k$ κύριο υποπίνακα του $T_n$, π.χ. $k=16$, υπολογίζουμε το συμμετρικό μέρος του πίνακα, $S_k^1$, το οποίο αντιστοιχεί στην $f_1$. Σημειώνουμε ότι για κάθε ρίζα υπάρχει μια αντίστοιχη ιδιοτιμή, η οποία τείνει στο 0, όσο η μεταβλητή $k$ παίρνει μεγαλύτερες τιμές. `Εστω $x_i$ το σημείο όπου το $\operatorname{Re}\left(F_{n-1}(\theta_j)\right)$, $j=1,2,\dots,n$ λαμβάνει την ελάχιστη τιμή. Εφαρμόζουμε τη μέθοδο των Αντιστρόφων Δυνάμεων {\en (Inverse Power)} \cite{demmel,trefethen} στον $S_k^1$, με αρχικό διάνυσμα $\Theta_{i,k}$, ορισμένο ως:
\begin{equation*}
\Theta_{i,k}=\frac{1}{\sqrt{k}}\left(1,\mathrm{e}^{\mathrm{i}x_i},\dots,\mathrm{e}^{\mathrm{i}(k-1)x_i}\right)^T.
\end{equation*}
Εκτιμούμε την πολλαπλότητα $m_i^1$ εφαρμόζοντας την τεχνική που προτάθηκε στην \cite{NSV_2005}. Με πιο απλά λόγια, εξετάζουμε κατά πόσο η απόσταση μεταξύ δύο διαδοχικών ιδιοτιμών (κατ" αναλογία με το $k$) γίνεται όλο και πιο μικρή. Διπλασιάζοντας τη μεταβλητή $k$ ξανά και ξανά (ας πούμε από $16$ σε $32$ και τέλος σε $64$), εκτιμούμε το λόγο:
\begin{equation*}
s_i^1=\frac{\widetilde{\lambda^1}_{i,k}-\widetilde{\lambda^1}_{i,2k}}{\widetilde{\lambda^1}_{i,2k}-\widetilde{\lambda^1}_{i,4k}}.
\end{equation*}
Με $\widetilde{\lambda^1}_{i,k}$ συμβολίζουμε την πλησιέστερη στο 0 ιδιοτιμή (που αντιστοιχεί στη ρίζα $x_i$) του συμμετρικού μέρους του $k\times k$ κύριου υποπίνακα του $T_n$, υπολογισμένη με λίγες επαναλήψεις της μεθόδου Αντιστρόφων Δυνάμεων.

Εκτιμούμε την πολλαπλότητα $m_i^1$, ως τον κοντινότερο ακέραιο στο $\log_2{(s_i^1)}$. Αυτό αποδεικνύεται ωε εξής:
Είναι γνωστό ότι η ιδιοτιμή $\lambda^1_{i,k}$ που αντιστοιχεί στη ρίζα $x_i$ με πολλαπλότητα $m_i^1$, τείνει στο 0 με ταχύτητα $\mathcal{O}\left(\frac{1}{k^{m_i^1}}\right)$. Επομένως, αυτή γράφεται ως:
\begin{equation*}
\lambda^1_{i,k}=c\frac{1}{k^{m_i^1}}+o\left(\frac{1}{k^{m_i^1}}\right).
\end{equation*}
Τότε, η προσέγγιση $\widetilde{\lambda^1}_{i,k}$ της $\lambda^1_{i,k}$ είναι η $\widetilde{\lambda^1}_{i,k}\simeq c\frac{1}{k^{m_i^1}}$ και ο λόγος $s_i^1$ προσεγγίζεται ως:
\begin{equation*}
s_i^1\simeq\frac{c\frac{1}{k^{m_i^1}}-c\frac{1}{\left(2k\right)^{m_i^1}}}{c\frac{1}{\left(2k\right)^{m_i^1}}-c\frac{1}{\left(4k\right)^{m_i^1}}}=\frac{c\frac{1}{k^{m_i^1}}\left(1-\frac{1}{2^{m_i^1}}\right)}{c\frac{1}{k^{m_i^1}}\frac{1}{2^{m_i^1}}\left(1-\frac{1}{2^{m_i^1}}\right)}=2^{m_i^1}.
\end{equation*}
Ως επακόλουθο, $\log\limits_2{s_i^1}\simeq m_i^1$.

Αν η γραφική παράσταση της $f_1$, λαμβάνει θετικές και αρνητικές τιμές, προφανώς υπάρχουν σημεία ριζών όπου υπάρχει τομή με τον άξονα. Για την εκτίμηση της πολλαπλότητας αυτών των ριζών, δεν μπορούμε να χρησιμοποιήσουμε τη διαδικασία που περιγράψαμε παραπάνω. Αυτή μπορεί να εφαρμοστεί μόνο για ρίζες που αντιστοιχούν σε τοπικά ελάχιστα. `Ετσι, θα προσπαθήσουμε να εκτιμήσουμε την πολλαπλότητα αυτών των ριζών για τη συνάρτηση $\vert f_1\vert$, όπου τα σημεία τομής μετατρέπονται σε τοπικά ελάχιστα. Φυσικά, οι πολλαπλότητες των ριζών της $\vert f_1\vert$ παραμένουν ίδιες με αυτές της $f_1$.

Επομένως, πρέπει να υπολογίσουμε τον πίνακα $\widehat{S}_k^1 = T_k(\vert f_1\vert)$, για μικρές τιμές της μεταβλητής $k$ (π.χ., 16, 32 και 64), επειδή δε θέλουμε να αυξήσουμε την πολυπλοκότητα του αλγορίθμου. Θα χρησιμοποιήσουμε τις τιμές $\vert\operatorname{Re}(F_{n -1})\vert$ στο $G_n\cup\lbrace -\pi,\pi\rbrace$. Τότε,

\begin{equation}\label{eq:Simpson_integral}
\begin{split}
\left(\widehat{S}_k^1\right)_{rq}&=\frac{1}{2\pi}\int\limits_{-\pi}^{\pi}\vert f_1(x)\vert\mathrm{e}^{-\mathrm{i}(r-q)x}\mathrm{d}x\\
&\simeq\frac{1}{2\pi}\int\limits_{-\pi}^{\pi}\left\vert\operatorname{Re}\left(F_{n-1}(x)\right)\right\vert\mathrm{e}^{-\mathrm{i}(r-q)x}\mathrm{d}x.
\end{split}
\end{equation}
Μπορούμε να προσεγγίσουμε το τελευταίο ολοκλήρωμα με χρήση του σύνθετου κανόνα του {\en Simpson}, σε σημεία του $G_n\cup\lbrace -\pi,\pi\rbrace$, ή οποιαδήποτε άλλη μέθοδο αριθμητικής ολοκλήρωσης. Ακολουθώντας την τεχνική που προαναφέραμε για τον $\widehat{S}_k^1$, υπολογίζουμε τις πολλαπλότητες των ριζών της $f_1$.

Είναι προφανές ότι για την εκτίμηση των πολλαπλοτήτων που έχουν οι ρίζες της συνάρτησης $f_2$, δε μπορούμε να αποφύγουμε τον υπολογισμό του πίνακα $\widehat{S}_k^2=T_k(\vert f_2\vert)$, επειδή η $f_2$ λαμβάνει πάντα θετικές και αρνητικές τιμές ως περιττή συνάρτηση. `Ολα τα υπόλοιπα παραμένουν ίδια, με τη διαφορά ότι το φανταστικό μέρος $\operatorname{Im}(F_{n-1})$ παίρνει τη θέση του $\operatorname{Re}(F_{n-1})$. Θα θέλαμε να σχολιάσουμε ότι σε περίπτωση που υπάρχει μια ρίζα $x_i$ όπου $f_1(x_i)=0$, ενώ $f_2(x_i)\neq0$ (ή $f_2(x_i)=0$ και $f_1(x_i)\neq0$), τότε θέτουμε $m_i^2=0$ (ή $m_i^1=0$).

Αναφέρουμε ότι η τεχνική για την οποία γίνεται λόγος κι έχει να κάνει με τον υπολογισμό των πολλαπλοτήτων όταν η $f_\ell$, $\ell=1,2$ λαμβάνει θετικές και αρνητικές τιμές, αποτελεί μια βελτίωση της αντίστοιχης τεχνικής των \cite{NSV_2005, Serra_1999}, όπου οι συγγραφείς μελέτησαν μη-αρνητικές συναρτήσεις.

Λόγω της λεπτότητας του πλέγματος $G_n$, είναι εύκολο να εκτιμήσουμε τα σημεία πιθανής ασυνέχειας των συναρτήσεων $f_1$ και $f_2$. Για απλούστευση, θα περιγράψουμε αυτή την εκτίμηση για την $f_1$. Το μόνο που έχουμε να κάνουμε είναι να ελέγξουμε το μέγεθος του λόγου $\operatorname{Re}\left(F_{n-1}(\theta_j)\right)-\operatorname{Re}\left(F_{n-1}(\theta_{j+1})\right)$, $j=1,2,\dots,n$, προς $h=\frac{2\pi}{n+1}$, όπου $\theta_{n+1}=\theta_1$. Αυτό σημαίνει ότι ελέγχουμε το πόσο διαφέρουν τα πραγματικά μέρη του αναπτύγματος $F_{n-1}$, για δύο διαδοχικά σημεία, σε σχέση με το βήμα $h$. Αν
\begin{equation*}
\frac{\vert\operatorname{Re}\left(F_{n-1}(\theta_j)\right)-\operatorname{Re}\left(F_{n-1}(\theta_{j+1})\right)\vert}{h}=\Omega(n),
\end{equation*}
(αρκούντως μεγάλο), τότε υποθέτουμε ότι υπάρχει ένα σημείο ασυνέχειας ή ταχεία μεταβολή της συνάρτησης {\en (unbounded variation)}, στο διάστημα $[\theta_j,\theta_{j+1}]$. Προφανώς, ακριβώς η ίδια ανάλυση μπορεί να γίνει και για την $f_2$. Αυτό σημαίνει ότι αν αντικαταστήσουμε το $\operatorname{Re}$ με $\operatorname{Im}$, μπορούμε να εκτιμήσουμε τα σημεία ασυνέχειας της $f_2$.

Η εκτίμηση των σημείων ασυνέχειας απαιτείται για την πολυωνυμική προσέγγιση της $\frac{f}{g_n}$, όπου $g_n$ είναι το κατάλληλο τριγωνομετρικό πολυώνυμο το οποίο αίρει τις εκτιμώμενες ρίζες της συνάρτησης $f$. Όπως περιγράψαμε στο δεύτερο κεφάλαιο, ο προρρυθμιστής γίνεται ιδιαίτερα πιο αποτελεσματικός αν αποκλείσουμε μια περιοχή κοντά στο σημείο ασυνέχειας, πριν εφαρμόσουμε τον αλγόριθμο βέλτιστης ομοιόμορφης προσέγγισης του {\en Remez}. Εκεί, η απόσταση από το σημείο ασυνέχειας, επιλέχθηκε εμπειρικά να είναι ίση με $\frac{2\pi}{7}$. Ως παράδειγμα, στο Σχήμα \ref{fig:discont_appr} παρουσιάζουμε την προσέγγιση της συνάρτησης $x^3$ στο $[-\pi,\pi]$, όταν αποκλείουμε κάποια σημεία κοντά στο $\pm\pi$ (όπου εντοπίζεται η ασυνέχεια), καθώς κι όταν δεν προχωρούμε σε κάποια εξαίρεση σημείων. Η προσέγγιση έγινε με τριγωνομετρικά πολυώνυμα τετάρτου βαθμού και όπως παρατηρείται, το σφάλμα της προσέγγισης είναι κατά πολύ μεγαλύτερο στη δεύτερη περίπτωση.

\begin{center}
\begin{figure}
    \includegraphics[width=1\linewidth]{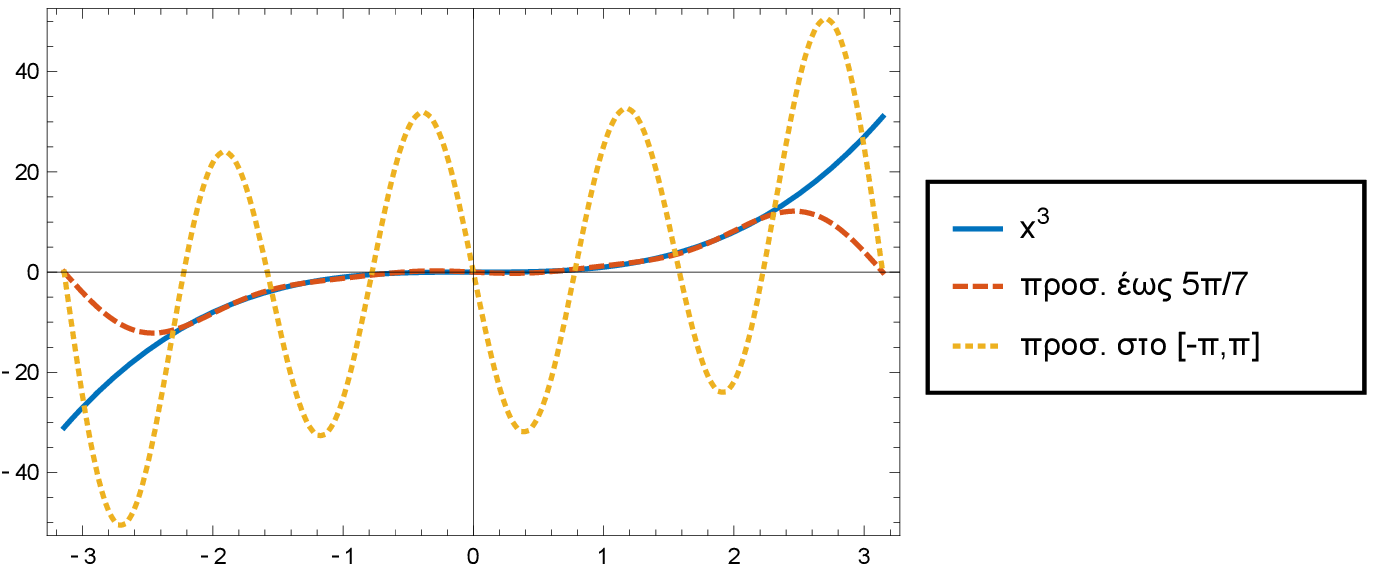}  
    \caption{Προσέγγιση της $x^3$.}
    \label{fig:discont_appr}
\end{figure}
\end{center}

Προκειμένου να αποφύγουμε την κακή κατάσταση θα πρέπει να βρούμε ένα τριγωνομετρικό πολυώνυμο, το οποίο έχει τις ίδιες ρίζες με τη γεννήτρια συνάρτηση $f$ ώστε να άρουμε τις ρίζες αυτής. `Εστω $x_i$, $i=1,2,\dots,\rho$ οι ακριβείς μη-μηδενικές ρίζες της συνάρτησης $f$ στο διάστημα $(0,\pi]$, με πολλαπλότητες $m_i^1$ για την $f_1$ και $m_i^2$ για την $f_2$, αντίστοιχα, $i=1,2,\dots,\rho$. Λόγω του ότι η $f_1$ είναι άρτια και η $f_2$ περιττή, αν $x_i$ είναι μια ρίζα στο $(0,\pi]$, το σημείο $-x_i$ είναι επίσης μια ρίζα στο $[-\pi,0)$, με την ίδια πολλαπλότητα. Τότε, όπως περιγράψαμε εκτενώς στο δεύτερο κεφάλαιο, η μορφή του τριγωνομετρικού πολυωνύμου $g$ δίνεται ως:

Αν $m_i^1\leq m_i^2$, $\forall i=1,2,\dots,\rho:$

\begin{equation*}
g=\operatorname{sign}(c_1(x))\prod\limits_{i=1}^{\rho}\left(\cos{(x_i)}-\cos{(x)}\right)^{m_i^1}.
\end{equation*}

Σε περίπτωση που η $f_1$ έχει επίσης ρίζα στο 0, με πολλαπλότητα $m_0^1\leq m_0^2$:

\begin{equation*}
g=\operatorname{sign}(c_1(x))(2-2\cos{(x)})^{\frac{m_0^1}{2}}\prod\limits_{i=1}^{\rho}\left(\cos{(x_i)}-\cos{(x)}\right)^{m_i^1}.
\end{equation*}

Ωστόσο, αν υπάρχει τουλάχιστον μια ρίζα $x_j:$ $m_j^1>m_j^2$:

\begin{equation}\label{eq:1multiple_roots}
\begin{split}
g&=\operatorname{sign}(c_1(x))\prod\limits_{i=1}^{\rho}\left(\cos{(x_i)}-\cos{(x)}\right)^{m_i^1}\\
&+\mathrm{i}\operatorname{sign}(c_2(x))\left(\sin{(x)}\right)^{m_0^2}\prod\limits_{i=1}^{\rho}\left(\cos{(x_i)}-\cos{(x)}\right)^{m_i^2}.
\end{split}
\end{equation}

Σημειώνουμε ότι αν η $f_1$ έχει μια ρίζα στο 0, με πολλαπλότητα $m_0^1$, πολλαπλασιάζουμε τον πρώτο όρο της (\ref{eq:1multiple_roots}) με $(2-2\cos{(x)})^{\frac{m_0^1}{2}}$. Οι συναρτήσεις $c_1$ και $c_2$ ορίζονται με τον τρόπο που περιγράψαμε στο δεύτερο κεφάλαιο. Πρακτικά, τα πρόσημα αυτών $\operatorname{sign}(c_1)$ και $\operatorname{sign}(c_2)$ επιλέγονται έτσι ώστε $\operatorname{Re}\left(\frac{f}{g}\right)>0$.

Θα δώσουμε ένα απλό παράδειγμα, ώστε να γίνει περισσότερο κατανοητός στον αναγνώστη ο τρόπος με τον οποίο επιλέγουμε το τριγωνομετρικό πολυώνυμο $g$. `Εστω $f(x)=(x-x_1)(x+x_1)+\mathrm{i}x(x-x_1)^2(x+x_1)^2$, $x_1\in(0,\pi]$, $x\in[-\pi,\pi]$. `Οπως φαίνεται η $f$ έχει ρίζες στο $\pm x_1$ με πολλαπλότητα ίση με 1. Παρατηρούμε ότι η πολλαπλότητα της ρίζας για την $f_1$ είναι $m_1^1=1$, ενώ για την $f_2$ είναι $m_1^2=2$ και $m_0^2=1$. Επειδή $m_1^1<m_1^2$, επιλέγουμε το τριγωνομετρικό πολυώνυμο $g(x)=\cos{(x_1)}-\cos{(x)}$, για την άρση της κακής κατάστασης. Τότε, η $\frac{f}{g}$ γράφεται:
\begin{equation*}
\frac{f(x)}{g(x)}=\frac{(x-x_1)(x+x_1)}{\cos{(x_1)}-\cos{(x)}}+\mathrm{i}\frac{x(x-x_1)^2(x+x_1)^2}{\cos{(x_1)}-\cos{(x)}}.
\end{equation*}
Είναι εύκολο να ελέγξει κανείς ότι το πραγματικό μέρος της παραπάνω συνάρτησης δεν έχει ρίζες και είναι θετικό μακριά από το 0. Από την άλλη, το φανταστικό μέρος έχει ρίζες στο 0 και $\pm x_1$, με πολλαπλότητα ίση με 1. Ωστόσο, το γεγονός αυτό δεν αποτελεί πρόβλημα, διότι προσπαθούμε να πετύχουμε συσσώρευση του φανταστικού μέρους των ιδιοτιμών κοντά στο 0. Τονίζουμε ότι η συνάρτηση $\frac{f}{g}$ δεν έχει ρίζες, διότι $\operatorname{Re}\left(\frac{f}{g}\right)>0$.

Σκοπός μας είναι να μελετήσουμε κατά ποιον τρόπο το σφάλμα κατά την εκτίμηση μιας ρίζας της συνάρτησης $f$ επηρεάζει τη σύγκλιση της μεθόδου {\en PGMRES}. Σχολιάζουμε ότι στην εργασία \cite{NSV_2006i} οι συγγραφείς έδωσαν αποτελέσματα για τον δείκτη κατάστασης του προρρυθμισμένου πίνακα, για δι-διάστατους {\en (two-level)} θετικά ορισμένους πίνακες {\en Toeplitz}, όπου η συνάρτηση $f$ είναι δύο μεταβλητών, μη-αρνητική και άρτια, έχοντας ρίζες άρτιας τάξης. Σημειώνουμε ότι τα ίδια αποτελέσματα, λαμβάνοντας ανάλογες υποθέσεις/θεωρήσεις, ισχύουν και στην περίπτωσή μας όταν οι ρίζες έχουν άρτιες πολλαπλότητες. Επειδή αυτό δε συμβαίνει πάντα, θα πρέπει να βρούμε κάποιον εναλλακτικό τρόπο σύγκλισης, μέσω της μελέτης του σφάλματος. Μελετούμε κι εδώ ξεχωριστά το συμμετρικό και αντισυμμετρικό μέρος του προρρυθμισμένου συστήματος.

Πρακτικά εκτιμούμε τις ρίζες στα σημεία $\widetilde{x}_i\simeq x_i$, $i=1,2,\dots,\rho$. Υποθέτουμε ότι οι πολλαπλότητες των ριζών υπολογίστηκαν με ακρίβεια, όπως επίσης και η ρίζα στο 0. Σε περίπτωση που εκτιμούμε δύο ρίζες κοντά στο 0 με απόσταση $o(1)$, θεωρούμε ότι έχουμε μία ρίζα στο 0. Επομένως, δημιουργείται ένα σφάλμα κατά την εκτίμηση των ριζών, που βρίσκονται μακριά από το 0. Το τριγωνομετρικό πολυώνυμο $g_n$, το οποίο επιτυγχάνει την άρση των εκτιμώμενων ριζών, δίνεται από τη διαδικασία που περιγράψαμε παραπάνω, με τη διαφορά ότι στις αντίστοιχες σχέσεις το $\widetilde{x}_i$ παίρνει τη θέση του $x_i$.

Για την απλούστευση της ανάλυσης, υποθέτουμε ότι η $f$ έχει δύο ρίζες στα σημεία $\pm x_1$, $x_1\neq 0$ με πολλαπλότητα κάποιον ακέραιο αριθμό $\alpha>0$. Σε περίπτωση που έχουμε περισσότερα από ένα ζεύγη ριζών, η ανάλυση γενικεύεται άμεσα. `Εστω $f=f_1+\mathrm{i}f_2$ και υποθέτουμε ότι η $f_1$ και $f_2$ έχουν ρίζες στο $\pm x_1$ τάξεως $\alpha$ και $\beta$, αντίστοιχα, με $\beta\geq\alpha$, ενώ η $f_2$ έχει και μια επιπλέον ρίζα στο 0. Σε αυτή την περίπτωση, σύμφωνα με την προαναφερθείσα ανάλυση, το τριγωνομετρικό πολυώνυμο που αίρει την κακή κατάσταση θα έπρεπε να είναι το $g(x)=c\left(\cos(x_1)-\cos(x)\right)^\alpha$. Ωστόσο, όπως τονίσαμε, πρακτικά έχουμε ένα σφάλμα στην εκτίμηση της ρίζας, $\widetilde{x}_1-x_1=\varepsilon$. Επομένως, αν χρησιμοποιήσουμε ως προρρυθμιστή τον ταινιωτό πίνακα $T_n(g_n)$, όπου $g_n(x)=c\left(\cos(\widetilde{x}_1)-\cos(x)\right)^\alpha$ \cite{serra2003practical}, θα πρέπει να μελετήσουμε τη συμπεριφορά του φάσματος του προρρυθμισμένου πίνακα $T_n^{-1}(g_n)T_n(f)=T_n^{-1}(g_n)T_n(f_1)+T_n^{-1}(g_n)T_n(\mathrm{i}f_2)$. Ο πρώτος όρος του αθροίσματος γράφεται ως:
\begin{equation*}
\begin{split}
T_n^{-1}(g_n)T_n(f_1)&=T_n^{-1}(g_n)\left[T_n(g_n)T_n\left(\frac{f_1}{g_n}\right)+L_1\right]\\
&=T_n\left(\frac{f_1}{g_n}\right)+L_2=T_n\left(h_1\frac{g}{g_n}\right)+L_2,
\end{split}
\end{equation*}
όπου $L_1$ και $L_2$ είναι πίνακες χαμηλής βαθμίδας, ίσης με $2\alpha$ και $h_1$ είναι μια θετική και φραγμένη συνάρτηση. Το μη-φραγμένο μέρος της γεννήτριας συνάρτησης του τελευταίου πίνακα είναι ο λόγος $\frac{g}{g_n}$, που έχει ρίζα στο $x_1$ και πόλο στο $\widetilde{x}_1$. Είναι εύκολο να δει κανείς ότι αυτός ο λόγος είναι φραγμένος στο σύνολο $K=[-\pi,\pi]\backslash\left[(-\widetilde{x}_1-\varepsilon,-\widetilde{x}_1+\varepsilon)\cup(\widetilde{x}_1-\varepsilon,\widetilde{x}_1+\varepsilon)\right]$ και μη-φραγμένος στα διαστήματα $(-\widetilde{x}_1-\varepsilon,-\widetilde{x}_1+\varepsilon)$ και $(\widetilde{x}_1-\varepsilon,\widetilde{x}_1+\varepsilon)$. Υποθέσαμε, χωρίς βλάβη της γενικότητας ότι $\varepsilon=\widetilde{x}_1-x_1>0$.

Λόγω της ισοκατανομής των ιδιοτιμών των πινάκων {\en Toeplitz} \cite{Grenander} το πολύ $4\varepsilon n$ ιδιοτιμές μπορούν να κυμαίνονται εκτός του $[a,b]$, όπου $a=\min\limits_{x\in K}\frac{f_1(x)}{g_n(x)}$ και $b=\max\limits_{x\in K}\frac{f_1(x)}{g_n(x)}$. Το μέγεθος του $\varepsilon$ εξαρτάται από το πόσο ομαλή είναι η συνάρτηση $f_1$. Σύμφωνα με την \cite{Serra_1999}, αν η $f_1$ είναι συνεχής το σφάλμα του αναπτύγματος {\en Fourier}, για την προσέγγιση της $f_1$, είναι $\varepsilon=\mathcal{O}\left(\frac{\log{n}}{n}\right)$, επομένως $\mathcal{O}(\log{n})$ ιδιοτιμές μπορούν να κυμαίνονται εκτός του διαστήματος $[a,b]$.

Το αντισυμμετρικό μέρος του προρρυθμισμένου πίνακα μας δίνει:
\begin{equation*}
\begin{split}
T_n^{-1}(g_n)T_n(\mathrm{i}f_2)&=T_n^{-1}(g_n)\left[T_n(g_n)T_n\left(\frac{\mathrm{i}f_2}{g_n}\right)+L_1^\prime\right]\\
&=T_n\left(\mathrm{i}h_2\frac{g}{g_n}\right)+L_2^\prime,
\end{split}
\end{equation*}
όπου $h_2$ είναι μια φραγμένη και περιττή συνάρτηση, η οποία έχει μια ρίζα στο 0 με την πολλαπλότητα που έχει και η $f_2$, καθώς επίσης και ρίζες στο $\pm x_1$ με πολλαπλότητα $\beta-\alpha$. Οι πίνακες $L_{1}^\prime$, $L_{2}^\prime$ έχουν χαμηλή βαθμίδα ίση με $2\alpha$. Ακολουθώντας την ανάλυση που περιγράψαμε για το συμμετρικό μέρος, έχουμε το πολύ $\mathcal{O}(\log{n})$ ιδιοτιμές εκτός του $[-c,c]$, όπου $c=\max\limits_{x\in K}\frac{f(x)}{g_n(x)}$.

Αν $\beta<\alpha$, τότε $g(x)=c\left(\cos(x_1)-\cos(x)\right)^\alpha+\mathrm{i}c^\prime\sin(x)^\gamma\left(\cos(x_1)-\cos(x)\right)^\beta$ κι έχουμε ότι:
\begin{equation}\label{eq:zw}
\begin{split}\frac{f}{g_n}&=\frac{z^\beta}{z_n^\beta}\frac{h_2\sin(x)^\gamma-\mathrm{i}h_1 z^{\alpha-\beta}}{c^\prime\sin(x)^\gamma-\mathrm{i}c z_n^{\alpha-\beta}}=\frac{z^\beta}{z_n^\beta}w,
\end{split}
\end{equation}
όπου $z=\cos(x_1)-\cos(x)$, $z_n=\cos(\widetilde{x}_1)-\cos(x)$, $h_1, h_2$ είναι φραγμένες, θετικές και άρτιες συναρτήσεις. Ο πρώτος λόγος, $\frac{z^\beta}{z_n^\beta}$, αποτελεί τον μη-φραγμένο παράγοντα του $\frac{f}{g_n}$, ενώ ο δεύτερος, 
$w$, είναι φραγμένος. Είναι εύκολο να δούμε ότι το πραγματικό μέρος του $w$ είναι φραγμένή, θετική και άρτια συνάρτηση, ενώ το φανταστικό φραγμένη και περιττή. Επομένως,
\begin{equation*}
T_n^{-1}(g_n)T_n(f)=T_n\left(\frac{f}{g_n}\right)+\widehat{L}_1=T_n\left(w\frac{z^\beta}{z_n^\beta}\right)+\widehat{L}_2.
\end{equation*}
Ακολουθώντας την ίδια ανάλυση, με χρήση των ιδιοτήτων ισοκατανομής των ιδιοτιμών του συμμετρικού και αντισυμμετρικού μέρους του $T_n\left(w\frac{z^\beta}{z_n^\beta}\right)$, όπως και στην πρώτη περίπτωση, καταλήγουμε σε ανάλογα αποτελέσματα. `Ετσι, το πολύ $\mathcal{O}(\log{n})$ ιδιοτιμές κυμαίνονται εκτός του ορθογωνίου $[a,b]\times[-c,c]$. Περισσότερες λεπτομέρειες δίνονται στην απόδειξη του Θεωρήματος \ref{thm:2}.

`Οπως προαναφέρθηκε, στο πρόβλημα που μελετάμε, η γεννήτρια συνάρτηση $f$ δεν είναι γνωστή εκ των προτέρων. Ωστόσο, έχουμε ήδη υπολογίσει το ανάπτυγμα {\en Fourier} στο $G_n$. Οπότε, υπολογίζουμε το λόγο $\frac{f}{g}$ προσεγγίζοντας τον από το λόγο $\widehat{f}=\frac{F_{n-1}}{g_n}$ σε κάποιο υποσύνολο του $G_n$. Μένει να προσαρμόσουμε κατάλληλα την τεχνική προσέγγισης που προτάθηκε στην \cite{NT_2019}, για τη βέλτιστη ομοιόμορφη προσέγγιση, με χρήση του αλγορίθμου {\en Remez}, της $\widehat{f_1}=\operatorname{Re}(\widehat{f})$ και $\widehat{f_2}=\operatorname{Im}(\widehat{f})$ με άρτια και περιττά τριγωνομετρικά πολυώνυμα κατάλληλων βαθμών, αντίστοιχα. Αρχικά, επιλέγουμε ένα ισοκατανεμημένο σύνολο σημείων από το πλέγμα $G_n$, έστω $G_k$, $k<\!\!<n$, στο $(0,\pi)$. Για να βελτιώσουμε την απόδοση του προρρυθμιστή, καλό θα ήταν να αποκλείσουμε κάποια σημεία του $G_k$, τα οποία ανήκουν σε περιοχές ασυνέχειας της $\widehat{f_1}$ ή της $\widehat{f_2}$ (μειώνοντας το σφάλμα προσέγγισης από τον αλγόριθμο {\en Remez}). Εξαιτίας σφαλμάτων κατά την εκτίμηση των πιθανών ριζών και λαμβάνοντας υπόψη τη σχέση (\ref{eq:zw}), η συνάρτηση $\widehat{f_1}$, ή η $\widehat{f_2}$ μπορεί να μην είναι φραγμένες σε μικρές περιοχές που περιέχουν τις ρίζες. Για να μειώσουμε το σφάλμα προσέγγισης, καλό είναι να αποκλείσουμε επίσης κάποια σημεία τα οποία ανήκουν σε τέτοιες περιοχές. Τέλος, σχηματίζουμε τα σύνολα $G_k^1$ και $G_k^2$ για την $\widehat{f_1}$ και $\widehat{f_2}$, αντίστοιχα, κι εφαρμόζουμε τον αλγόριθμο προσέγγισης {\en Remez}.

Προσεγγίζουμε τις $\widehat{f_1}$ και $\widehat{f_2}$ με τις $q_1$ και $q_2$, αντίστοιχα. Τότε, ορίζουμε την $q=q_1+\mathrm{i}q_2$ και σχηματίζουμε τον ταινιωτό πίνακα {\en Toeplitz} $T_n(p_n)$, όπου $p_n=g_nq$. Συμβολίζουμε με $d_1$ και $d_2$ τους βαθμούς των τριγωνομετρικών πολυωνύμων $q_1$ και $q_2$, αντίστοιχα. Από εδώ και στο εξής θα συμβολίζουμε τον προρρυθμιστή $T_n(p_n)$, ως $R_{d_1,d_2}$, για να φαίνονται οι βαθμοί των πολυωνύμων προσέγγισης. Με χρήση του $R_{d_1,d_2}$ ως προρρυθμιστή, μπορούμε να λύσουμε μη-συμμετρικά και πραγματικά συστήματα {\en Toeplitz}, με αποτελεσματικό τρόπο. Θεωρητικά αποτελέσματα για την περίπτωση που η $f$ είναι γνωστή εκ των προτέρων, δόθηκαν στο δεύτερο κεφάλαιο. Παρακάτω παρουσιάζουμε το αντίστοιχο του Θεωρήματος \ref{thm:eig_clustering}, το οποίο αφορά στη συσσώρευση των ιδιοτιμών, για το πρόβλημα που μας απασχολεί, δηλαδή όταν η $f$ δεν είναι γνωστή εκ των προτέρων.

\begin{thm}\label{thm:2}
Ας είναι $T_n$ ο $n\times n$, πραγματικός πίνακας {\en Toeplitz}, με άγνωστη γεννήτρια συνάρτηση. `Εστω $p_n$ το τριγωνομετρικό πολυώνυμο το οποίο προέκυψε από την προτεινόμενη διαδικασία, με σφάλματα στις εκτιμήσεις των ριζών τάξεως το πολύ $\mathcal{O}\left(\frac{\log{n}}{n}\right)$. Τότε, οι ιδιοτιμές του προρρυθμισμένου συστήματος $T_n(p_n)^{-1}T_n$ συσσωρεύονται στο ορθογώνιο $[a,b]\times[-c,c]$, όπου $a=\min\limits_{x\in K}\operatorname{Re}\left(\frac{F_{n-1}(x)}{p_n(x)}\right)$, $b=\max\limits_{x\in K}\operatorname{Re}\left(\frac{F_{n-1}(x)}{p_n(x)}\right)$, $c=\max\limits_{x\in K}\operatorname{Im}\left(\frac{F_{n-1}(x)}{p_n(x)}\right)$ και $K=[-\pi,\pi]\backslash\bigcup\limits_i\left((-\beta_i,-\alpha_i)\cup(\alpha_i,\beta_i)\right)$, με $(\alpha_i,\beta_i)$ να είναι διαστήματα που περιέχουν τις ρίζες $x_i$ και τους πόλους $\widetilde{x}_i$, $i=1,2,\dots,\rho$, έχοντας μήκος $\beta_i-\alpha_i=\mathcal{O}\left(\frac{\log{n}}{n}\right)$, καθώς επίσης και $x_i-\alpha_i$, $\widetilde{x}_i-\alpha_i$, $\beta_i-x_i$, $\beta_i-\widetilde{x}_i$ είναι της τάξεως $\mathcal{O}\left(\frac{\log{n}}{n}\right)$. Τότε, το πολύ $\mathcal{O}(\log{n})$ ιδιοτιμές κυμαίνονται εκτός του ορθογωνίου.
\end{thm}

\begin{proof}
Αν μπορούσαμε να προσδιορίσουμε τις ρίζες με ακρίβεια, θα χρησιμοποιούσαμε ως προρρυθμιστή τον πίνακα $T_n(p)$, όπου $p=gq$. Τότε, από το Θεώρημα \ref{thm:eig_clustering} θα προέκυπτε μια κύρια συσσώρευση στο ορθογώνιο $[a,b]\times[-c,c]$. Ωστόσο, χρησιμοποιούμε ως προρρυθμιστή τον πίνακα $T_n(p_n)$. `Εχουμε επιλέξει το μήκος του $[\alpha_i,\beta_i]$ να είναι $\mathcal{O}\left(\frac{\log{n}}{n}\right)$, επειδή η απόσταση μεταξύ της ρίζας $x_i$ και του πόλου $\widetilde{x}_i$ είναι τάξεως το πολύ $\mathcal{O}\left(\frac{\log{n}}{n}\right)$. `Εχουμε επιλέξει επίσης τις αποστάσεις $x_i-\alpha_i$, $\widetilde{x}_i-\alpha_i$, $\beta_i-x_i$, $\beta_i-\widetilde{x}_i$ να είναι της τάξεως $\mathcal{O}\left(\frac{\log{n}}{n}\right)$, έτσι ώστε η συνάρτηση $\bar{f}=\frac{F_{n-1}}{p_n}$ να είναι φραγμένη (άνω και κάτω), ανεξαρτήτως της διάστασης $n$, στο σύνολο $K$. Εύκολα βλέπουμε ότι ο λόγος $\frac{\cos{(x_i)}-\cos{(x)}}{\cos{(\widetilde{x}_i)}-\cos{(x)}}$, ο οποίος χαρακτηρίζει τη συνάρτηση $\bar{f}$ ως μη φραγμένη στο $\widetilde{x}_i$ και δηλώνει ότι μηδενίζεται στο $x_i$, είναι φραγμένος ανεξαρτήτως της διάστασης $n$, έξω από το διάστημα $(\alpha_i,\beta_i)$. Αυτό εξασφαλίζει ότι η $\bar{f}$ είναι φραγμένη σε ένα ορθογώνιο $[a,b]\times[-c,c]$, όταν ορίζεται στο σύνολο $K$. Οι ιδιοτιμές που κυμαίνονται εκτός του $[a,b]\times[-c,c]$ εξαρτώνται από αυτά τα διαστήματα. Για να εκτιμήσουμε πόσες είναι αυτές, χρησιμοποιούμε το θεώρημα ισοκατανομής των ιδιοτιμών του {\en Szeg{\"o}}. Αναλυτικότερα, θα πρέπει να ελέγξουμε ξεχωριστά το πραγματικό και φανταστικό μέρος της $\bar{f}$. Αρχικά, σταθεροποιούμε έναν ακέραιο $n$, ο οποίος είναι αρκετά μεγάλος και ουσιαστικά είναι η διάσταση του αρχικού συστήματος προς λύση $T_nx=b$. Στη συνέχεια ορίζουμε ως $N$, την ακέραια μεταβλητή που χρησιμοποιούμε στο θεώρημα του {\en Szeg{\"o}}. Θα δώσουμε την απόδειξη για το πραγματικό μέρος της $\bar{f}$, καθώς αυτή που αφορά στο φανταστικό μέρος είναι ανάλογη. 

Θέτουμε $\bar{f}_1=\operatorname{Re}(\bar{f})$. Προφανώς η $\bar{f}_1$ είναι μη-φραγμένη στο σύνολο 
\begin{equation*}
\bigcup\limits_{i}\left((-\beta_i,-\alpha_i)\cup(\alpha_i,\beta_i)\right)
\end{equation*}
και δε μπορούμε να χρησιμοποιήσουμε το θεώρημα του {\en Szeg{\"o}}. Αφού ο $n$ είναι σταθερός και η $\bar{f}_1$ προέρχεται από το πλέγμα $G_n$, μπορούμε να προσεγγίσουμε την $\bar{f}_1$ με την $\widetilde{f}_1$, η οποία είναι φραγμένη και παράγει τον ίδιο πίνακα $T_n(\bar{f}_1)$. Αυτό γίνεται ως εξής: Αν δεν υπάρχει κανένας πόλος ανάμεσα σε δύο διαδοχικά σημεία του $G_n$, η $\widetilde{f}_1$ λαμβάνει την ίδια τιμή με την $\bar{f}_1$. Αν υπάρχει κάποιος πόλος ανάμεσα σε δύο διαδοχικά σημεία $w_j$ και $w_{j+1}$, η $\widetilde{f}_1$ λαμβάνει την τιμή του ευθυγράμμου τμήματος
\begin{equation*}
\frac{\bar{f}_1(w_{j+1})-\bar{f}_1(w_{j})}{w_{j+1}-w_j}(x-w_j)+\bar{f}_1(w_{j}),~x\in[w_j,w_{j+1}].
\end{equation*}
Αυτό σημαίνει ότι η $\widetilde{f}_1(w_j)$ είναι τοπικό ελάχιστο/μέγιστο, ενώ η $\widetilde{f}_1(w_{j+1})$ είναι τοπικό μέγιστο/ελάχιστο, αντίστοιχα, αλλά φράσονται καθώς η διάσταση είναι σταθερή. Επομένως, η $\widetilde{f}_1$ είναι μια φραγμένη και συνεχής συνάρτηση. Για να εφαρμόσουμε το θεώρημα του {\en Szeg{\"o}}, θεωρούμε τη συνεχή και φραγμένη συνάρτηση $F_\eta$:
\begin{equation*}
F_\eta(z)=
\begin{cases}
1,~z\leq a-\eta, z\geq b+\eta\\
0,~z\in[a,b]\\
\end{cases},
\end{equation*}
$\eta$ αρκούντως μικρό. `Εχουμε:
\begin{equation*}
\begin{split}
&\limsup\limits_{N\rightarrow\infty}\frac{1}{N}\#\lbrace\lambda_j(T_n(\widetilde{f}_1))< a~\vee~\lambda_j(T_n(\widetilde{f}_1))> b\rbrace\\
&=\frac{1}{2\pi}\int\limits_{-\pi}^{\pi}F_\eta(\widetilde{f}_1(x))\mathrm{d}x\leq\frac{1}{2\pi}\int\limits_{\bigcup\limits_{i=1}^\rho\left((-\beta_i,-\alpha_i)\cup(\alpha_i,\beta_i)\right)}1\mathrm{d}x\\
&=\sum\limits_{i=1}^\rho 2(\beta_i-\alpha_i)=2\sum\limits_{i=1}^\rho c_i\frac{\log{n}}{n}=c\frac{\log{n}}{n},
\end{split}
\end{equation*}
όπου $\#$ δηλώνει τον πληθικό αριθμό του συνόλου και $\vee$ τη λογική διάζευξη {\en (OR)}.
`Αρα, $\limsup\limits_{N\rightarrow\infty}\#\lbrace\lambda_j(T_n(\widetilde{f}_1))< a~\vee~\lambda_j(T_n(\widetilde{f}_1))> b\rbrace\leq c\frac{\log{n}}{n}N$.

Πηγαίνοντας πίσω στο σταθερό $n$ και δεδομένου ότι επιλέγουμε κάποιο (αρκούντως) μικρό $\eta$, καταλήγουμε στο ότι ο αριθμός των ιδιοτιμών που κυμαίνονται εκτός του διαστήματος $[a,b]$ (κατεύθυνση του πραγματικού άξονα) είναι το πολύ $c\frac{\log{n}}{n}n=c\log{n}$, δηλαδή της τάξεως $\mathcal{O}(\log{n})$. Το ίδιο αποτέλεσμα λαμβάνεται για την κατεύθυνση του φανταστικού άξονα. Επομένως, $\mathcal{O}(\log{n})$ ιδιοτιμές του προρρυθμισμένου πίνακα κυμαίνονται εκτός του ορθογωνίου $[a,b]\times[-c,c]$.
\end{proof}

\begin{rem}
Αν δεν έχουν εκτιμηθεί ρίζες σε σημεία διαφορετικά του 0 ή αν όλες οι ρίζες έχουν εκτιμηθεί με ακρίβεια, τότε από το Θεώρημα \ref{thm:eig_clustering}, η συσσώρευση των ιδιοτιμών στο ορθογώνιο είναι κύρια. Επιπλέον, αν η γεννήτρια συνάρτηση του $T_n$ είναι επαρκώς ομαλή ή κάποιο τριγωνομετρικό πολυώνυμο, περιπτώσεις όπου το σφάλμα προσέγγισης του αναπτύγματος {\en Fourier} είναι της τάξεως $\mathcal{O}\left(\frac{1}{n}\right)$ \cite{Serra_1999}, ακολουθώντας ακριβώς την απόδειξη του Θεωρήματος \ref{thm:2}, καταλήγουμε σε κύρια συσσώρευση των ιδιοτιμών του προρρυθμισμένου πίνακα. 
\end{rem}

\begin{rem}
Όσον αφορά στη συσσώρευση των ιδιαζουσών τιμών, από το Θεώρημα \ref{thm:gen_cluster}, αυτή επιτυγχάνεται για τον προρρυθμισμένο πίνακα αν η $f$ είναι γνωστή εκ των προτέρων και ανήκει στην κλάση $\mathrm{L}^2([-\pi,\pi])$. Στην περίπτωσή μας, όπου η συνάρτηση είναι άγνωστη η συσσώρευση των ιδιαζουσών τιμών εξακολουθεί να έχει την ίδια φύση (γενική συσσώρευση). Ωστόσο, αν υπάρχουν σφάλματα κατά την εκτίμηση των ριζών, αυτή γίνεται ακόμα χειρότερη. Επομένως, η σύγκλιση της μεθόδου {\en PCGN}, είναι πιο αργή από αυτή της μεθόδου {\en PGMRES}. Αυτό φαίνεται και στον Πίνακα \ref{tab:it_1} του Παραδείγματος \ref{ex:1}.
\end{rem}

Ο Αλγόριθμος \ref{1algo} περιγράφει την κατασκευή του προτεινόμενου προρρυθμιστή, σε μορφή ψευδοκώδικα.

\setcounter{algorithm}{1}
\begin{breakablealgorithm}
\caption{Κατασκευή του Προρρυθμιστή.}
\label{1algo}
\textbf{Είσοδος:} $n\in\mathbb{N}$, $T_n$: $n\times n$ μη-συμμετρικός, πραγματικός πίνακας {\en Toeplitz}.
\begin{algorithmic}[1]
\STATE Κατασκευάστε το ισοκατανεμημένο πλέγμα $G_n$, με σημεία:\\ $\theta_j=-\pi+\frac{2\pi j}{n+1}$, $j=1,2,\dots,n$.
\STATE\textbf{για }{$j=1,2,\dots,n$}
\STATE\quad Υπολογίστε το ανάπτυγμα {\en Fourier}: $F_{n-1}(\theta_j)=\sum\limits_{k=-n+1}^{n-1}t_k\mathrm{e}^{\mathrm{i}k\theta_j}$, $\theta_j\in G_n$.
\STATE\quad Εκτιμήστε τις $f_1$ και $f_2$ ως τις τιμές των $F_{n-1}^1(\theta_j)=\operatorname{Re}(F_{n-1}(\theta_j))$ και\\\hspace{.75pc} $F_{n-1}^2(\theta_j)=\operatorname{Im}(F_{n-1}(\theta_j))$, αντίστοιχα.
\STATE\textbf{τέλος για}
\STATE Επιλέξτε σημεία $\theta_i\in G_n$, κοντά στα τοπικά ελάχιστα της $\vert F_{n-1}^\ell\vert$, $\ell=1,2$, τέτοια ώστε $\vert F_{n-1}^\ell(\theta_i)\vert\simeq0$ και θεωρήστε τα ως πιθανές ρίζες $x_i$, $i=1,\dots,\rho$.
\STATE Επιλέξτε διαστήματα $[\theta_j,\theta_{j+1}]$, όπου είναι πιθανόν να υπάρχουν σημεία ασυνέχειας του $F_{n-1}$.
\STATE Εκτιμήστε τις πολλαπλότητες των ριζών της $f_\ell$, $\ell=1,2$:
\STATE\quad\textbf{αν} $F_{n-1}^\ell$ λαμβάνει θετικές και αρνητικές τιμές
	\STATE\qquad Υπολογίστε το $\vert F_{n-1}^\ell\vert$ στο $G_n$.
	\STATE\qquad Υπολογίστε το $\widehat{S}_{4k}^\ell\simeq T_{4k}(\vert F_{n-1}^\ell\vert)$ με χρήση του σύνθετου κανόνα του \\\hspace{1.85pc}{\en Simpson}, $k<\!\!<n$.
	\STATE\qquad Για κάθε ρίζα $x_i\in G_n$, εκτιμήστε την αντίστοιχη ιδιοτιμή $\lambda_{i,k}^\ell$ του $\widehat{S}_k^\ell$ \\\hspace{1.8pc}χρησιμοποιώντας λίγες επαναλήψεις της μεθόδου Αντιστρόφων \\\hspace{1.8pc}Δυνάμεων με αρχικό διάνυσμα $\Theta_{i,k}=\frac{1}{\sqrt{k}}\left(1,\mathrm{e}^{\mathrm{i}x_i},\mathrm{e}^{2\mathrm{i}x_i},\dots,\mathrm{e}^{(k-1)\mathrm{i}x_i}\right)^T$.
\STATE\qquad Επαναλάβετε το 12 για τους $\widehat{S}_{2k}^\ell$ και $\widehat{S}_{4k}^\ell$ για να λάβετε τις $\lambda_{i,2k}^\ell$ και $\lambda_{i,4k}^\ell$,\\\hspace{1.8pc} αντίστοιχα.
\STATE\qquad Υπολογίστε το λόγο $s_i^\ell=\frac{\lambda_{i,k}^\ell-\lambda_{i,2k}^\ell}{\lambda_{i,2k}^\ell-\lambda_{i,4k}^\ell}$ και τις πολλαπλότητες $m_i^\ell$ ως τους\\\hspace{1.85pc} πλησιέστερους ακέραιους στον $\log_2{s_i^\ell}$.
\STATE\quad\textbf{αλλιώς}
	\STATE\qquad Υπολογίστε το συμμετρικό και αντισυμμετρικό μέρος (του $T_n$),\\\hspace{1.8pc} $S_n^1=\frac{T_n+T_n^T}{2}$ και $S_n^2=\frac{T_n-T_n^T}{2}$, αντίστοιχα.
	\STATE\qquad Επαναλάβετε τα 12 - 14 για τον $S_k^\ell$ αντί του $\widehat{S}_k^\ell$.
\STATE\quad\textbf{τέλος αν}
\STATE Επιλέξτε το τριγωνομετρικό πολυώνυμο $g_n$, ώστε $\operatorname{Re}\left(\frac{F_{n-1}(\theta_j)}{g_n(\theta_j)}\right)>0$.
\STATE Ορίστε το $G_k$, υποσύνολο του $G_n$ με $k$ ισοκατανεμημένα σημεία στο $(0,\pi)$.
\STATE Εξαιρέστε τα σημεία του $G_k$ τα οποία είναι κοντά σε ρίζες ή σημεία ασυνέχειας της $f_\ell$, και ορίστε το νέο σύνολο ως $G_k^\ell$, $\ell=1,2$.
\STATE Εκτιμήστε τις $\operatorname{Re}\left(\frac{f}{g}\right)$ και $\operatorname{Im}\left(\frac{f}{g}\right)$, ως $\widehat{f}_1$ και $\widehat{f}_2$, αντίστοιχα, υπολογισμένες στο $G_k^\ell$, $\ell=1,2$.
\STATE Προσεγγίστε την $\widehat{f}_\ell$, με ένα κατάλληλο τριγωνομετρικό πολυώνυμο $q_\ell$, με βέλτιστη ομοιόμορφη προσέγγιση, χρησιμοποιώντας ως κόμβους τα σημεία του $G_k^\ell$, $\ell=1,2$.
\STATE Κατασκευάστε τον ταινιωτό προρρυθμιστή {\en Toeplitz} $T_n(p_n)$, όπου $p_n=g_n q$ και $q=q_1+\mathrm{i}q_2$.
\end{algorithmic}
\end{breakablealgorithm}

\subsection{Αριθμητικά αποτελέσματα}

Σε αυτή την υποενότητα παρουσιάζουμε μια πληθώρα αριθμητικών παραδειγμάτων, τα οποία δείχνουν την αποτελεσματικότητα της προτεινόμενης τεχνικής προρρύθμισης. Δίνουμε τις επαναλήψεις που χρειάζονται για την επιθυμητή σύγκλιση στη λύση του συστήματος, με σφάλμα το πολύ ίσο με $10^{-6}$. Στους Πίνακες \ref{tab:it_1}, \ref{tab:it_2} και \ref{tab:it_3} με $I_n$ δηλώνουμε ότι δε χρησιμοποιήθηκε κανένας προρρυθμιστής.

\setcounter{thm}{2}
\begin{exmp}\label{ex:1}\normalfont
Παίρνοντας τον πίνακα {\en Toeplitz} ο οποίος προκύπτει από τη συνάρτηση $\mathfrak{f}_2(x)=x^2+\mathrm{i}x^3$, θα εκτιμήσουμε τις ρίζες αυτής, καθώς και την πολλαπλότητα αυτών, από τα στοιχεία του πίνακα. 

\begin{figure}[H]
\centering
    \subfloat[Εκτίμηση της $x^2$.]{{\label{fig:1a}\includegraphics[width=0.45\linewidth]{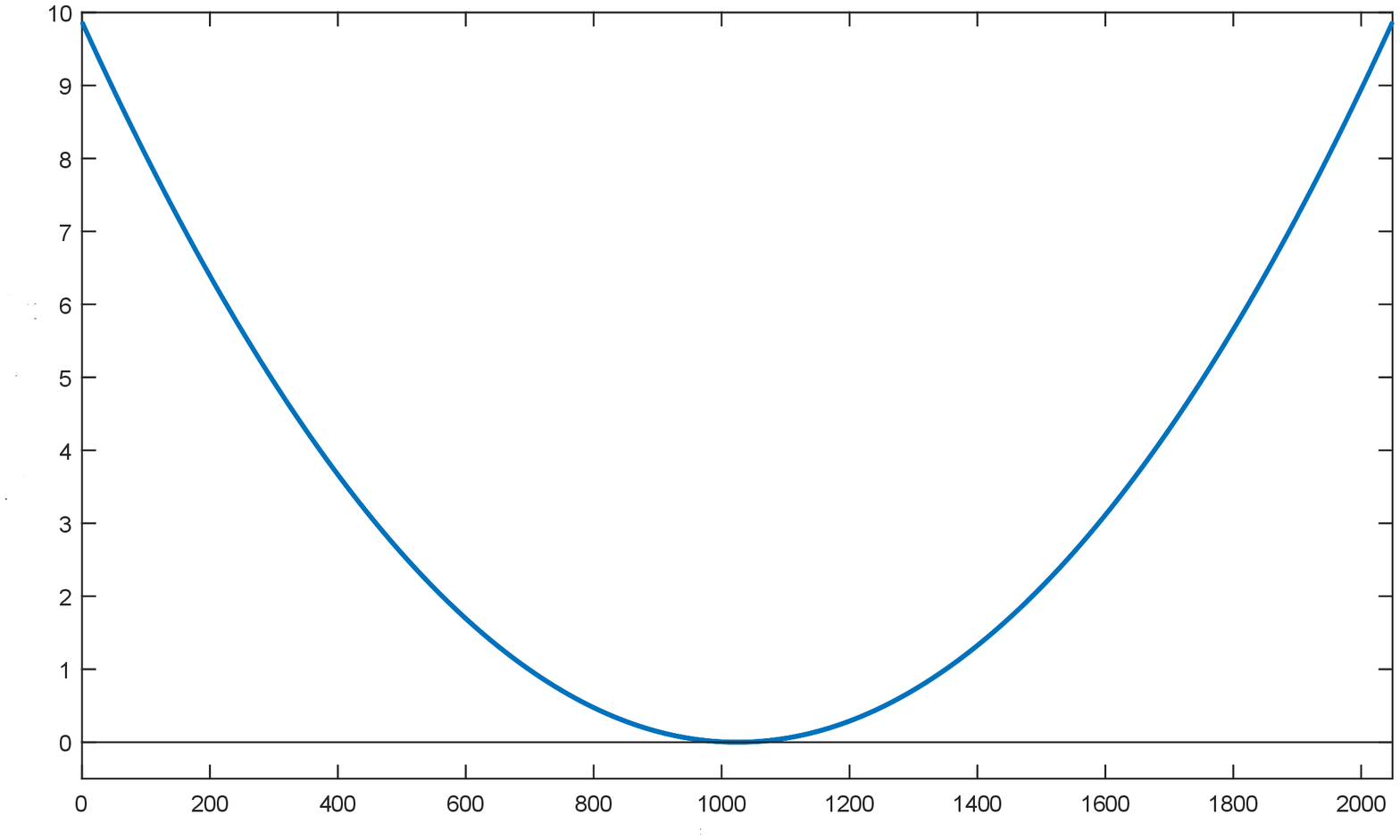}}}%
    \qquad
    \subfloat[Εκτίμηση της $x^3$.]{{\label{fig:1b}\includegraphics[width=0.45\linewidth]{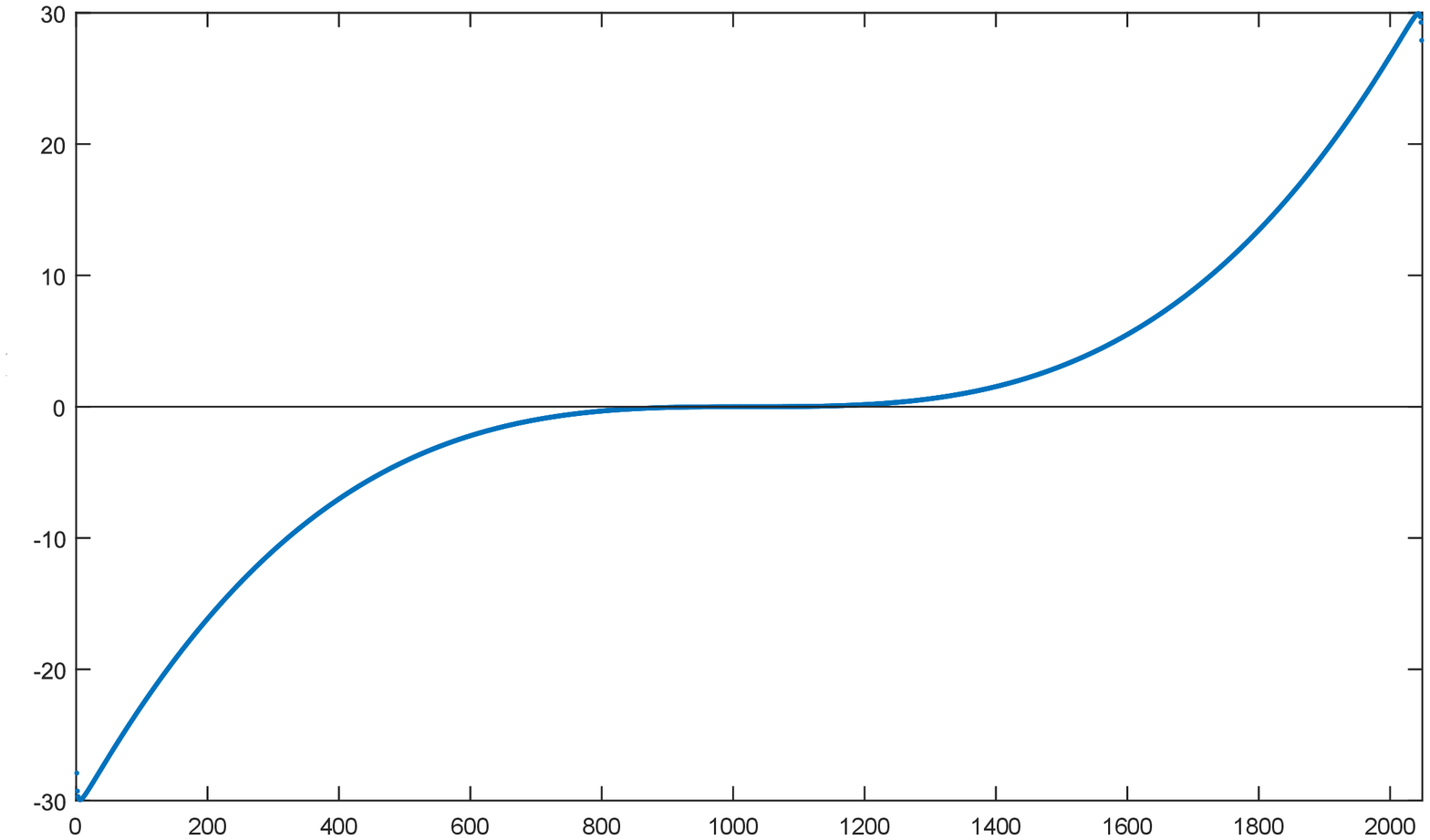}}}%
    \caption{Ανάπτυγμα {\en Fourier} στο $G_{2048}$.}
    \label{fig:example11}
\end{figure}

Τα Σχήματα \ref{fig:example11} και \ref{fig:example12} δείχνουν τις τιμές που λαμβάνει το ανάπτυγμα {\en Fourier} για το πραγματικό και φανταστικό μέρος της (άγνωστης) γεννήτριας συνάρτησης $\mathfrak{f}_2$.

\begin{figure}[H]
\centering
    \subfloat[Εκτίμηση της $x^2$ (μεγ.).]{{\includegraphics[width=0.45\linewidth]{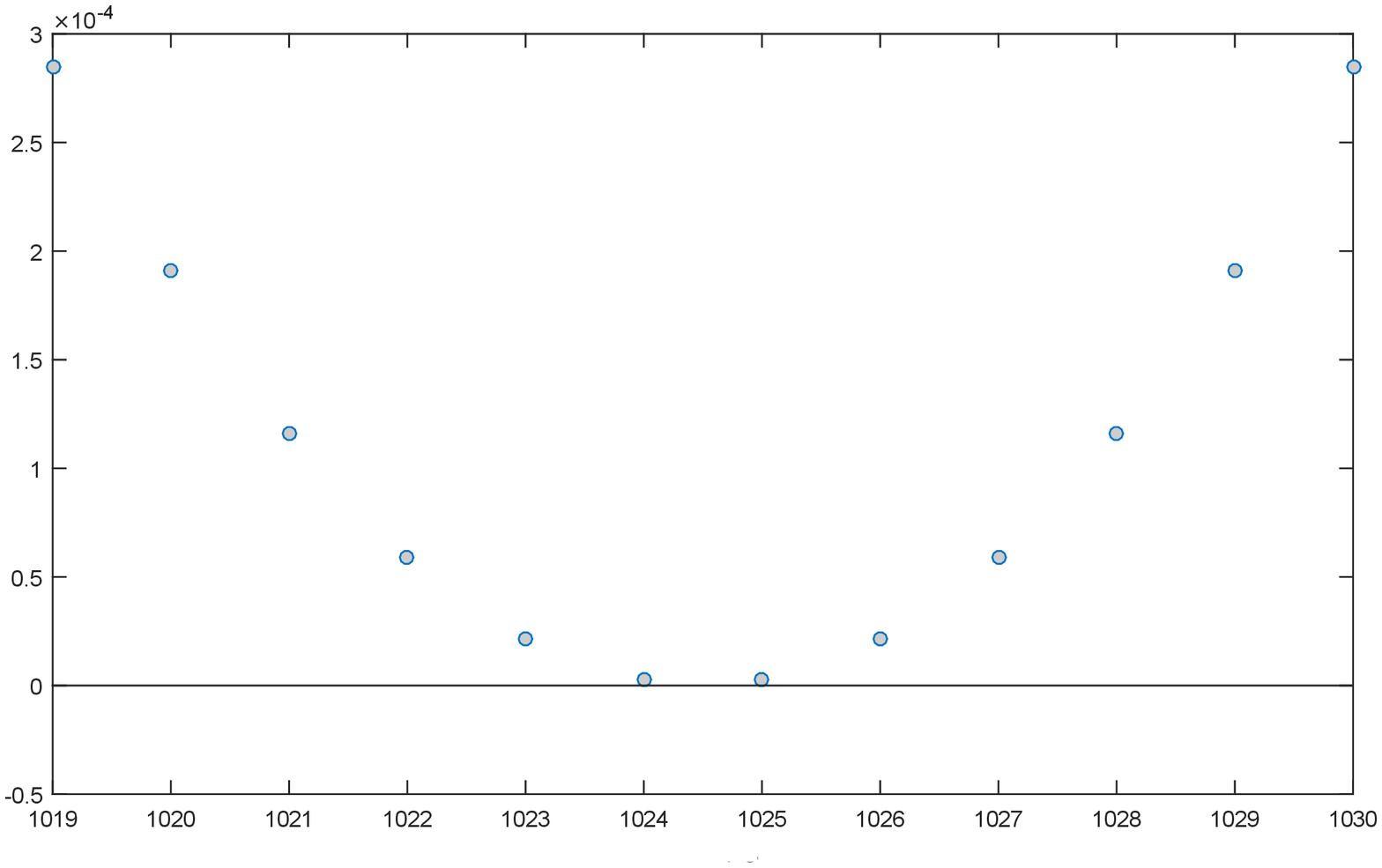}}}%
    \qquad
    \subfloat[Εκτίμηση της $x^3$ (μεγ.).]{{\includegraphics[width=0.45\linewidth]{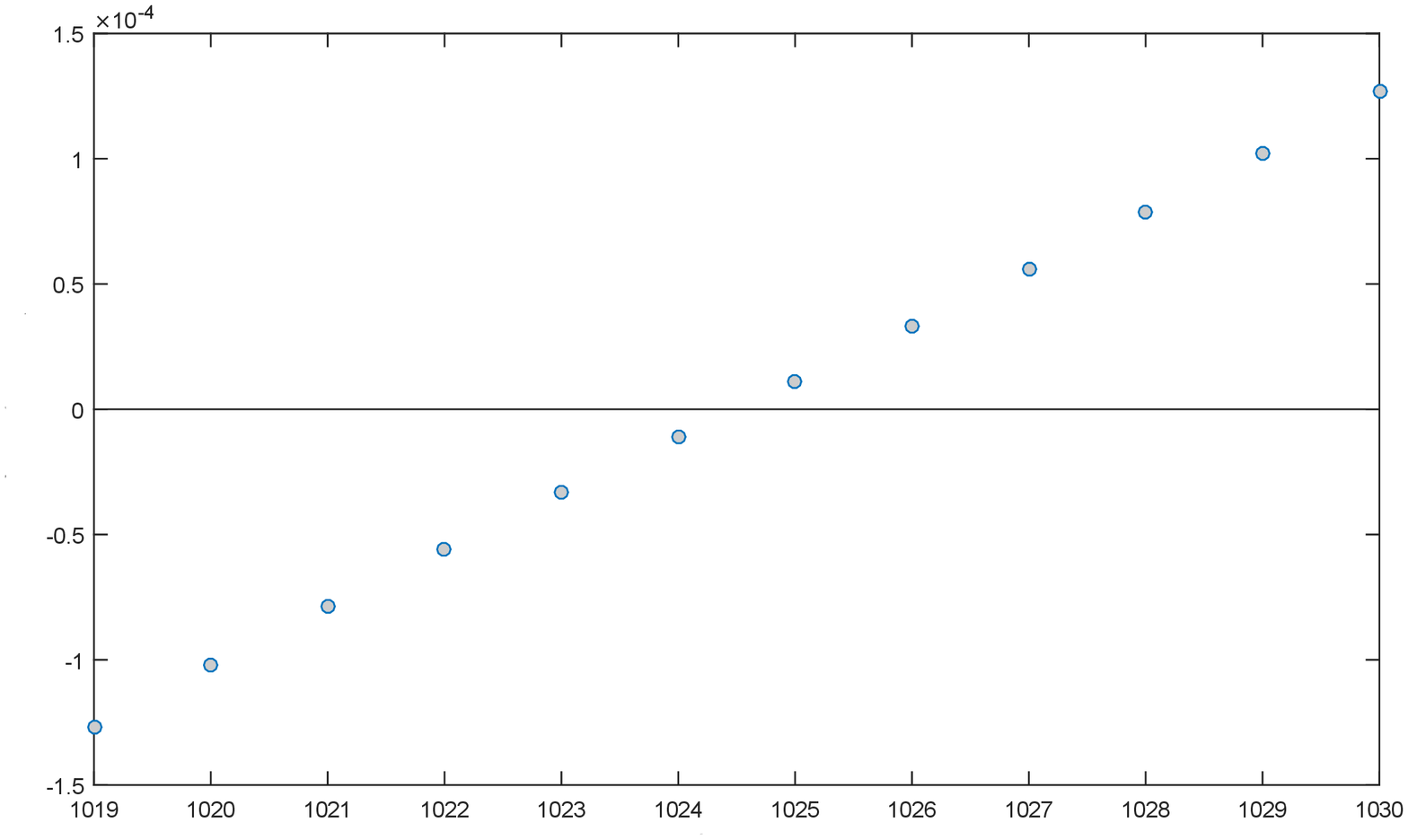}}}%
    \caption{Ανάπτυγμα {\en Fourier} κοντά στην αρχή των αξόνων.}
    \label{fig:example12}
\end{figure}

Εκτελώντας 4 επαναλήψεις της μεθόδου Αντιστρόφων Δυνάμεων καταλήγουμε στο ότι το πραγματικό μέρος της $\mathfrak{f}_2$ έχει μια ρίζα στο 0 με πολλαπλότητα $m_0^1=2$, ενώ το φανταστικό μέρος αυτής έχει ρίζα στο ίδιο σημείο με πολλαπλότητα $m_0^2=3$ (για περισσότερες λεπτομέρειες βλ. Πίνακα \ref{tab:1}). Επομένως, εκτιμήσαμε ότι η $\mathfrak{f}_2$ έχει μια ρίζα στο 0 και $m_0^1<m_0^2$. Στον Πίνακα \ref{tab:1}, με $\vert\lambda_0^2\vert$ δηλώνουμε την ελάχιστη ιδιοτιμή του $\widehat{S}_k^2$ (υπολογισμένη με τη μέθοδο των Αντιστρόφων Δυνάμεων). Προκειμένου να άρουμε την κακή κατάσταση επιλέγουμε το τριγωνομετρικό πολυώνυμο $g_n=g=2-2\cos(x)$, και προσεγγίζουμε την $\frac{F_{n-1}}{g}$ με τριγωνομετρικά πολυώνυμα τετάρτου βαθμού. Στη συνέχεια, κατασκευάζουμε τον προρρυθμιστή $R_{4,4}$ κι επιλύουμε το σύστημα με χρήση της μεθόδου {\en PGMRES} και {\en PCGN}. Οι επαναλήψεις των μεθόδων για διάφορες διαστάσεις δίνονται στον Πίνακα \ref{tab:it_1}. 

\begin{table}[H]
\centering
\begin{tabular}{ccccc}
\toprule
$k$ & $\lambda_0^1$ & $\log_2{(s_0^1)}$ & $\vert\lambda_0^2\vert$ & $\log_2{(s_0^2)}$\\\midrule
16 & 0.0351 & 1.9224 & 0.0133 & 2.6634\\
32 & 0.0092 & & 0.0023 &\\
64 & 0.0024 & & 0.0005 &\\\bottomrule
\end{tabular}
\caption{Πολλαπλότητα των ριζών ($\mathfrak{f}_2$).}
\label{tab:1}
\end{table}

Σημειώνουμε ότι η λύση του συστήματος χωρίς προρρύθμιση λαμβάνεται σχεδόν μετά από $n$ επαναλήψεις ($n$ είναι η διάσταση του συστήματος). Στο Σχήμα \ref{fig:example13} φαίνεται η συσσώρευση των ιδιοτιμών και ιδιαζουσών τιμών του προρρυθμισμένου συστήματος.

\begin{table}[H]
\centering
\begin{tabular}{ccc|cc}
\toprule
\multirow{2}{*}{$n$} & \multicolumn{2}{c|} {\en PGMRES} & \multicolumn{2}{c} {\en PCGN} \\
  & $I_n$ & $R_{4,4}$ & $I_n$ & $R_{4,4}$\\\midrule
1024 & $>$500 & 28 & - & 45\\
2048 & $>$500 & 28 & - & 49\\
4096 & $>$500 & 28 & - & 54\\
8192 & $>$500 & 27 & - & 57\\\bottomrule
\end{tabular}
\caption{Επαναλήψεις ($\mathfrak{f}_2$).}
\label{tab:it_1}
\end{table}

\begin{figure}[H]
\centering
    \subfloat[Ιδιοτιμές.]{{\label{fig:3a}\includegraphics[width=0.45\linewidth]{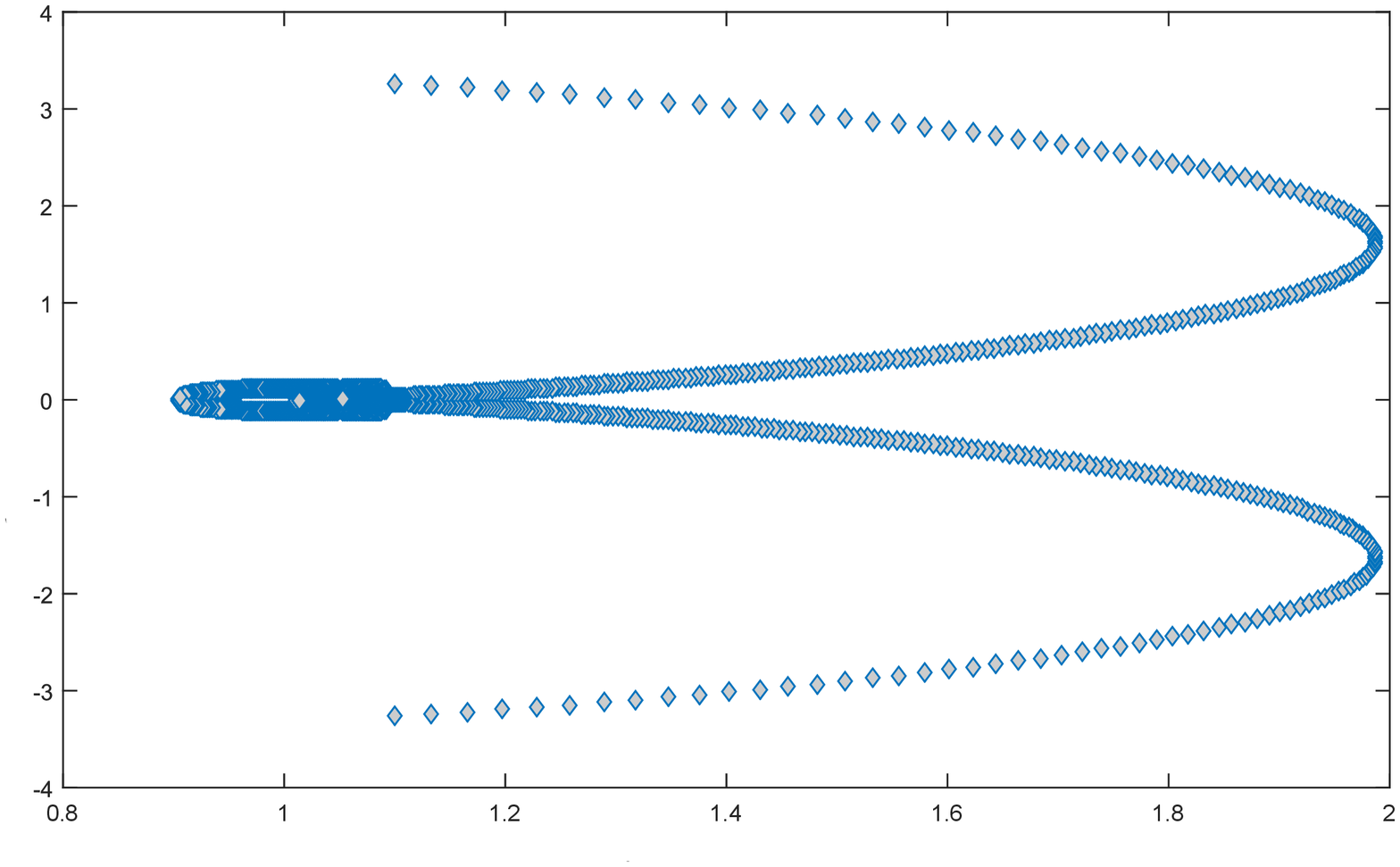}}}%
    \qquad
    \subfloat[Ιδιάζουσες τιμές.]{{\label{fig:3b}\includegraphics[width=0.45\linewidth]{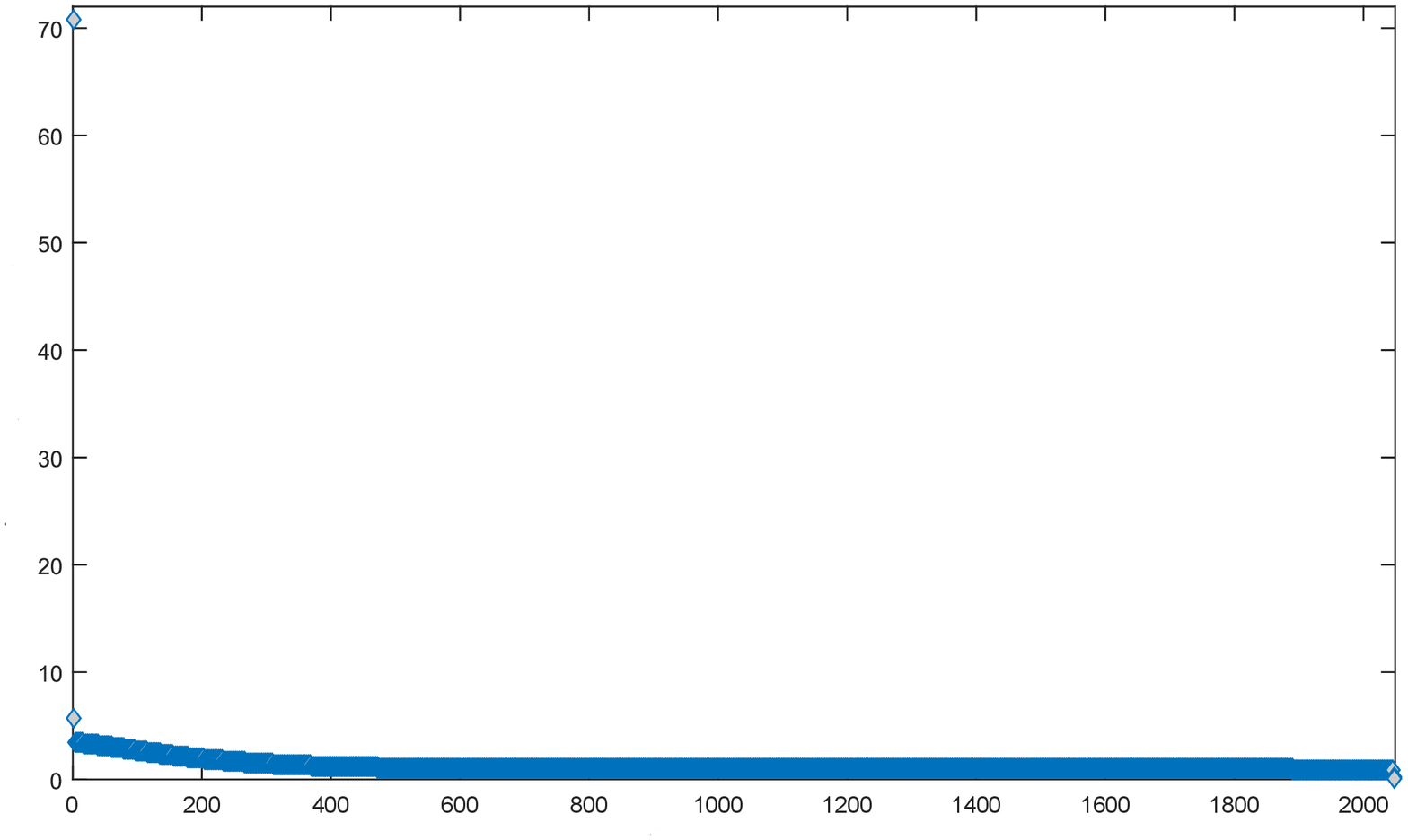}}}%
    \caption{Ιδιοτιμές και ιδιάζουσες τιμές ($\mathfrak{f}_2$).}
    \label{fig:example13}
\end{figure}
\end{exmp}

\begin{exmp}\normalfont
Θα δώσουμε εν συντομία, ένα ανάλογο παράδειγμα για τον πίνακα {\en Toeplitz}, ο οποίος έχει ως γεννήτρια συνάρτηση την $\mathfrak{f}_3(x)=x^2+\mathrm{i}x$. Εδώ, η ρίζα της $\mathfrak{f}_3$ είναι απλή. Το γεγονός αυτό επιβεβαιώθηκε δουλεύοντας ακριβώς με τον ίδιο τρόπο, όπως στο Παράδειγμα \ref{ex:1}. Πιο συγκεκριμένα, καταλήξαμε στο ότι $m_0^1=2$ και $m_0^2=1$, αφού $\log_2{(s_0^1)}=1.9224$ και $\log_2{(s_0^2)}=0.9718$. Επομένως, η $\mathfrak{f}_3$ έχει μια ρίζα στο 0 και $m_0^1>m_0^2$, που σημαίνει ότι επιλέγουμε ως $g$ το τριγωνομετρικό πολυώνυμο $2-2\cos(x)+\mathrm{i}\sin(x)$. Η λύση του συστήματος λαμβάνεται μετά από μόλις 6 επαναλήψεις, όταν  $n=2048$, και 5 επαναλήψεις όταν $n=4096$ και $8192$, με χρήση της μεθόδου {\en PGMRES} και τον $R_{4,4}$ ως προρρυθμιστή. Αξίζει να αναφερθεί ότι η λύση του συστήματος λαμβάνεται στον ίδιο αριθμό επαναλήψεων με εκείνον της τεχνικής προρρύθμισης του δευτέρου κεφαλαίου, δηλαδή στην περίπτωση όπου η $\mathfrak{f}_3$ είναι γνωστή εκ των προτέρων. Αυτό πιθανότατα συμβαίνει διότι οι ρίζες έχουν εκτιμηθεί με ακρίβεια.
\end{exmp}

\begin{exmp}\label{exmp:414}\normalfont
Σε αυτό το παράδειγμα ασχολούμαστε με την επίλυση του συστήματος {\en Toeplitz} που προκύπτει από τη γεννήτρια συνάρτηση $\mathfrak{f}_9(x)=(x^2-1)^2+\mathrm{i}x(x^2-1)$. Αυτή έχει ρίζες στο $\pm1$, σημεία τα οποία δεν ανήκουν στο πλέγμα $G_n$. Αν μπορούσαμε να εκτιμήσουμε τις τιμές της $\mathfrak{f}_9$ με ακρίβεια στο $G_n$, τότε θα αναμέναμε ένα σφάλμα, στην εκτίμηση της ρίζας, της τάξεως $\mathcal{O}\left(\frac{1}{n}\right)$. Ωστόσο, προσεγγίζουμε την $\mathfrak{f}_9$ μέσω του αναπτύγματος {\en Fourier} του $T_n$ κι έτσι το σφάλμα εξαρτάται από τη φύση της (άγνωστης) $\mathfrak{f}_9$, δηλαδή το πόσο ομαλή είναι. Στο παράδειγμα μας αναμένουμε σφάλμα της τάξεως $\mathcal{O}\left(\frac{\log{n}}{n}\right)$, καθώς εντοπίζεται ασυνέχεια για το φανταστικό μέρος στο $\pm\pi$ \cite{Serra_1999}. Φυσικά, όσο μεγαλώνει η διάσταση του αρχικού συστήματος, τόσο ακριβέστερη γίνεται η εκτίμηση της ρίζας (βλ. Πίνακα \ref{tab:est_1}).

\begin{table}[H]
\centering
\begin{tabular}{ccc}
\toprule
$n$ & $x_1^1$ & $x_1^2$\\\midrule
2048 & 1.0012 & 0.9981\\
4096 & 1.0007 & 0.9991\\
8192 & 0.9996 & 0.9996\\\bottomrule
\end{tabular}
\caption{Εκτίμηση της ρίζας ($\mathfrak{f}_9$).}
\label{tab:est_1}
\end{table}


Στον Πίνακα \ref{tab:err_1} παρουσιάζουμε τον αριθμό επαναλήψεων, χρησιμοποιώντας τη μέθοδο {\en PGMRES} και τον $R_{8,4}$, ως προρρυθμιστή, όταν $n=2048$, για διάφορα υποθετικά σφάλματα στην εκτίμηση των ριζών. `Οπως είναι φυσικό, δίνουμε επίσης τον αριθμό επαναλήψεων για τις ρίζες που εκτιμήθηκαν μέσω της διαδικασίας που περιγράψαμε, όπου τα σφάλματα ήταν $\varepsilon_1=0.0012$ και $\varepsilon_2=0.0019$. Εκτελώντας 4 επαναλήψεις της μεθόδου Αντιστρόφων Δυνάμεων υπολογίζουμε ότι $\log_2{(s_1^1)}=1.6780$, $\log_2{(s_0^2)}=0.8882$ και $\log_2{(s_1^2)}=1.0005$. Αυτό σημαίνει ότι $m_0^1=0$, $m_1^1=2$, $m_0^2=1$ και $m_1^2=1$. Επομένως, επιλέγουμε ως τριγωνομετρικό πολυώνυμο το $g_n(x)=(\cos(x_1^1)-\cos(x))^2+\mathrm{i}\sin(x)(\cos(x_1^2)-\cos(x)).$

\begin{table}[H]
\centering
\begin{tabular}{ccccccc}
\hlineB{2}
\diagbox{$\varepsilon_1$}{$\varepsilon_2$} & 0 & 0.0001 & 0.0005 & 0.0010 & 0.0019 & 0.0020\\
\hlineB{1.2}
0 & 7 & 7 & 8 & 9 & 12 & 13\\
0.0001 & 7 & 7 & 8 & 9 & 12 & 13\\
0.0005 & 7 & 7 & 8 & 9 & 12 & 13\\
0.0010 & 7 & 7 & 8 & 9 & 13 & 13\\
0.0012 & 7 & 7 & 8 & 9 & 13 & 13\\
0.0020 & 7 & 7 & 8 & 9 & 13 & 13\\
\hlineB{2}
\end{tabular}
\caption{Επαναλήψεις με υποθετικά σφάλματα ($\mathfrak{f}_9$).}
\label{tab:err_1}
\end{table}

Στον Πίνακα \ref{tab:err_1} παρατηρούμε ότι ο αριθμός επαναλήψεων είναι σχεδόν ο ίδιος κατά μήκος των στηλών. Αυτό σημαίνει ότι το σφάλμα $\varepsilon_1$ δεν παίζει τοσό καθοριστικό ρόλο για τη σύγκλιση του προρρυθμισμένου συστήματος, όσο το $\varepsilon_2$. Αυτό ισχύει διότι η πολλαπλότητα της ρίζας του φανταστικού μέρους της $\mathfrak{f}_9$ είναι μικρότερη από την πολλαπλότητα της αντίστοιχης ρίζας, του πραγματικού μέρους της $\mathfrak{f}_9$. Μπορούμε να εξηγήσουμε αυτό το φαινόμενο αναλύοντας την $\frac{\mathfrak{f}_9}{g_n}$.
\begin{equation*}
\begin{split}
\frac{\mathfrak{f}_9}{g_n}&=\frac{(x^2-1)^2+\mathrm{i}x(x^2-1)}{(\cos(x_1^1)-\cos(x))^2+\mathrm{i}\sin(x)(\cos(x_1^2)-\cos(x))}\\
&=\frac{h_1(\cos(1)-\cos(x))^2+\mathrm{i}h_2\sin(x)(\cos(1)-\cos(x))}{(\cos(x_1^1)-\cos(x))^2+\mathrm{i}\sin(x)(\cos(x_1^2)-\cos(x))}.
\end{split}
\end{equation*}
Λαμβάνοντας υπόψη ότι η μη-μηδενική ρίζα έχει εκτιμηθεί σε διαφορετικά σημεία $x_1^1$ και $x_1^2$ για το πραγματικό και φανταστικό μέρος, αντίστοιχα, η συνάρτηση $\frac{\mathfrak{f}_9}{g_n}$ δεν έχει πόλο, αλλά λαμβάνει πολύ μεγάλη τιμή για $x=x_1^1$ και $x=x_1^2$. Μελετώντας τον παρονομαστή της συνάρτησης σε μια περιοχή του 1, που περιέχει τα σημεία $x_1^1$ και $x_1^2$, παρατηρούμε ότι ο δεύτερος όρος $\sin(x)(\cos(x_1^2)-\cos(x))$ κυριαρχεί επί του πρώτου $(\cos(x_1^1)-\cos(x))^2$, στην προαναφερθείσα περιοχή, εκτός από μια μικρότερη σε μέγεθος περιοχή του $x_1^2$, με μήκος της τάξεως $\mathcal{O}\left(\frac{1}{n^2}\right)$, όπου ο πρώτος όρος υπερταιρεί επί του δεύτερου. Επομένως, το σφάλμα $\varepsilon_1$ της εκτίμησης του $x_1^1$ έχει επιρροή σε μια μικρή περιοχή της τάξεως $\mathcal{O}\left(\frac{1}{n^2}\right)$, αλλά η τεχνική μας είναι κατασκευασμένη στο πλέγμα $G_n$ με βήμα $\frac{2\pi}{n+1}$ κι αυτό σημαίνει ότι μια τόσο μικρή περιοχή δε μπορεί να ανιχνευθεί. Με άλλα λόγια το $\varepsilon_1$ δεν παίζει και τόσο καθοριστικό ρόλο στην αύξηση των επαναλήψεων.

\begin{figure}[H]
    \centering
    \includegraphics[width=0.9\linewidth]{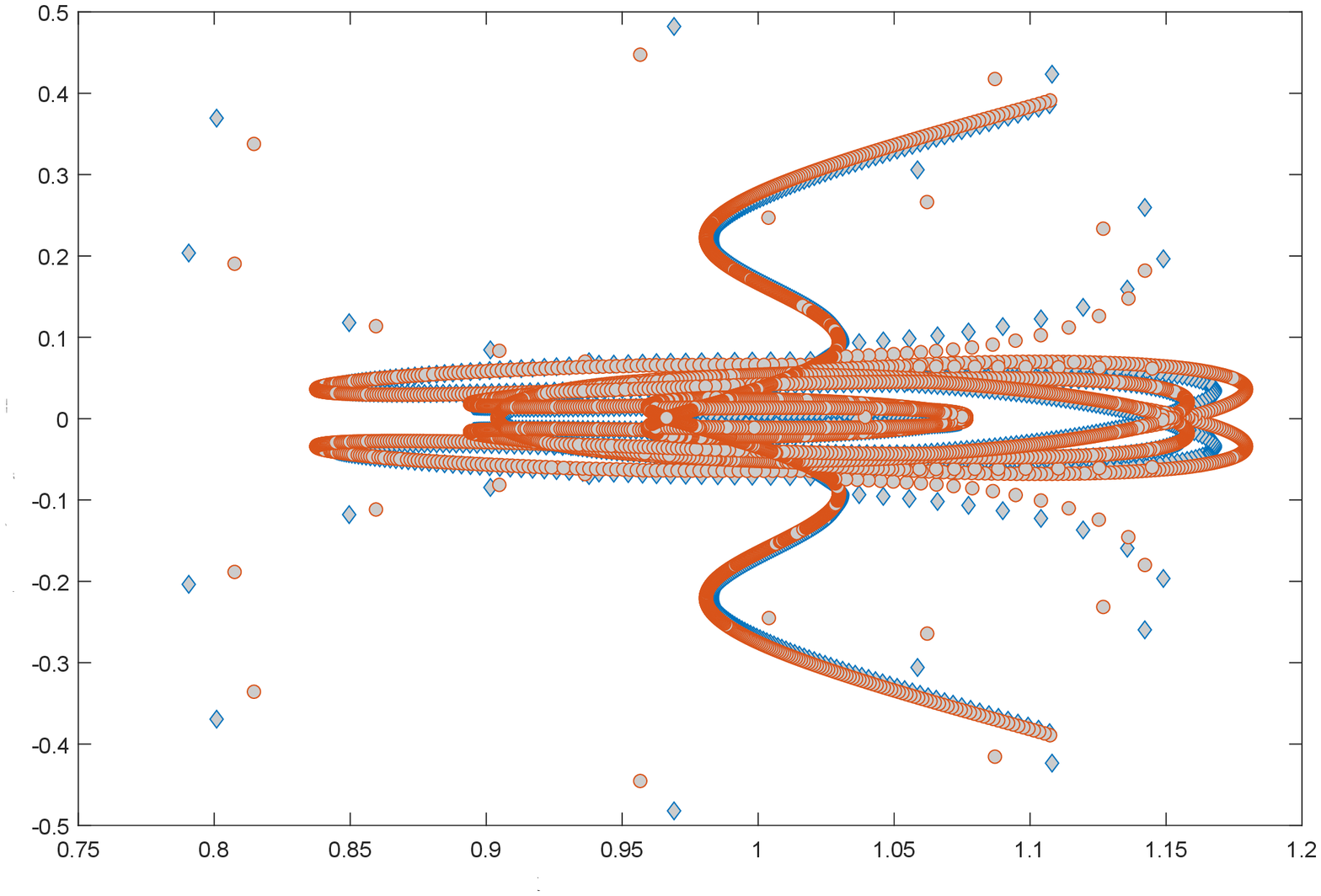}%
    \caption{Ιδιοτιμές ($\mathfrak{f}_9$).}
    \label{fig:example31}
\end{figure}

Στο σχήμα \ref{fig:example31} παρουσιάζουμε τη συσσώρευση των ιδιοτιμών, για το προρρυθμισμένο σύστημα, όταν $n=2048$ (μπλε διαμάντια) και $n=4096$ (πορτοκαλί κύκλοι). Παρατηρούμε ότι οι ιδιοτιμές μακριά από το $(1,0)$, οι οποίες εμφανίζονται ως διακεκριμένα/απομονωμένα σημεία, σχηματίζουν ζεύγη (μπλε διαμαντιών-πορτοκαλί κύκλων), τα οποία κυμαίνονται εκτός της συσσώρευσης. Το γεγονός αυτό εξηγεί τη συσσώρευση με το πολύ $\mathcal{O}(\log{n})$ ιδιοτιμές εκτός του φάσματος. Σημειώστε ότι ο ρυθμός αύξησης της τάξεως $\mathcal{O}(\log{n})$, είναι ιδιαιτέρως αργός, όταν διπλασιάζεται η διάσταση $n$. Αυτό φαίνεται επίσης στον Πίνακα \ref{tab:it_2}, όπου οι επαναλήψεις που δίνονται μέσω της μεθόδου {\en PGMRES} δεν αυξάνονται όταν διπλασιάζουμε το $n$.

\begin{table}[H]
\centering
\begin{tabular}{ccc}
\toprule
$n$ & $I_n$ & $R_{8,4}$ \\\midrule
2048 & $>$500 & 13 \\
4096 & $>$500 & 12 \\
8192 & $>$500 & 12 \\\bottomrule
\end{tabular}
\caption{Επαναλήψεις ($\mathfrak{f}_9$).}
\label{tab:it_2}
\end{table}
\end{exmp}

\begin{exmp}\normalfont
Ως το τελευταίο παράδειγμα αυτής της υποενότητας, παρουσιάζουμε την επίλυση ενός συστήματος {\en Toeplitz}, το οποίο γεννάται από μια συνεχή συνάρτηση που έχει ρίζες στο $\pm 1$ με πολλαπλότητες $m_1^1=m_1^2=1$. Πιο συγκεκριμένα, ο πίνακας {\en Toeplitz} γεννάται από την $\mathfrak{f}_{10}(x)=(x^2-1)+\mathrm{i}\mathfrak{h}_3(x)$, όπου $\mathfrak{h}_3(x)$ είναι η τεθλασμένη γραμμή, ορισμένη ως:
\begin{equation*}
\mathfrak{h}_3(x)=
\begin{cases}
x+\pi,~&x\in[-\pi,-\pi+\frac{1}{2})\\
\frac{x}{-2\pi+3}+\frac{1}{-2\pi+3},~&x\in[-\pi+\frac{1}{2},-\frac{1}{2})\\
\frac{x}{2\pi-3},~&x\in[-\frac{1}{2},\frac{1}{2})\\
\frac{x}{-2\pi+3}-\frac{1}{-2\pi+3},~&x\in[\frac{1}{2},\pi-\frac{1}{2})\\
x-\pi,~&x\in[\pi-\frac{1}{2},\pi]
\end{cases}
\end{equation*}

Το φανταστικό μέρος του αναπτύγματος {\en Fourier} της $\mathfrak{h}_3$, υπολογισμένο στο πλέγμα $G_n$ όταν $n=2048$, δίνεται στο Σχήμα \ref{fig:example41}.

\begin{figure}
    \centering
    \includegraphics[width=0.9\linewidth]{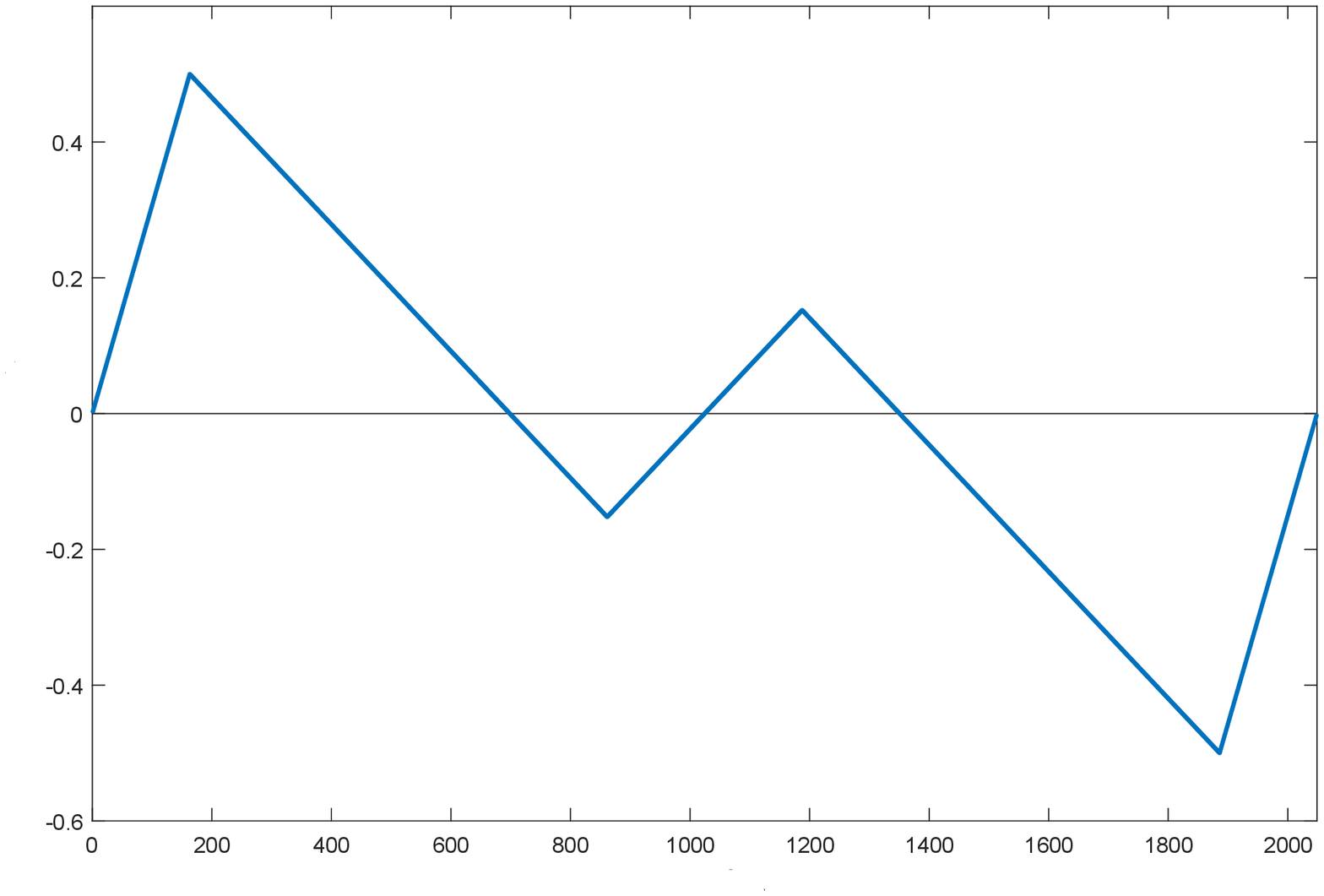}%
    \caption{Ανάπτυγμα {\en Fourier} της $\mathfrak{h}_3$.}
    \label{fig:example41}
\end{figure}

Θεωρώντας τη γεννήτρια συνάρτηση ως άγνωστη και χρησιμοποιώντας τον προτεινόμενο αλγόριθμο, εκτιμήσαμε τις μη-μηδενικές ρίζες στα (ίδια) σημεία $x_1^1=x_1^2$ και μετά από 3 επαναλήψεις της μεθόδου Αντίστροφων Δυνάμεων, συμπεράναμε ότι $m_1^1=m_1^2=1$ και $m_0^2=1$. Επομένως, το τριγωνομετρικό πολυώνυμο το οποίο αίρει τις εκτιμώμενες ρίζες της $\mathfrak{f}_{10}$, δίνεται ως $g_n(x)=\cos{(x_1^1)-\cos{(x)}}$. Στον Πίνακα \ref{tab:it_3} παρουσιάζουμε τον αριθμό επαναλήψεων, με χρήση της {\en PGMRES} όταν δε χρησιμοποιούμε κάποιον προρρυθμιστή, καθώς κι όταν γίνεται προρρύθμιση με τον $R_{4,4}$.

\begin{table}
\centering
\begin{tabular}{ccc}
\toprule
$n$ & $I_n$ & $R_{4,4}$ \\\midrule
2048 & $>$500 & 11 \\
4096 & $>$500 & 12 \\
8192 & $>$500 & 12 \\\bottomrule
\end{tabular}
\caption{Επαναλήψεις ($\mathfrak{f}_{10}$).}
\label{tab:it_3}
\end{table}

Στο Σχήμα \ref{fig:example43}, παρουσιάζουμε τη συσσώρευση των ιδιοτιμών του προρρυθμισμένου συστήματος.

\begin{figure}[htbp]%
\centering
    \subfloat[Ιδιοτιμές.]{{\includegraphics[width=0.45\linewidth]{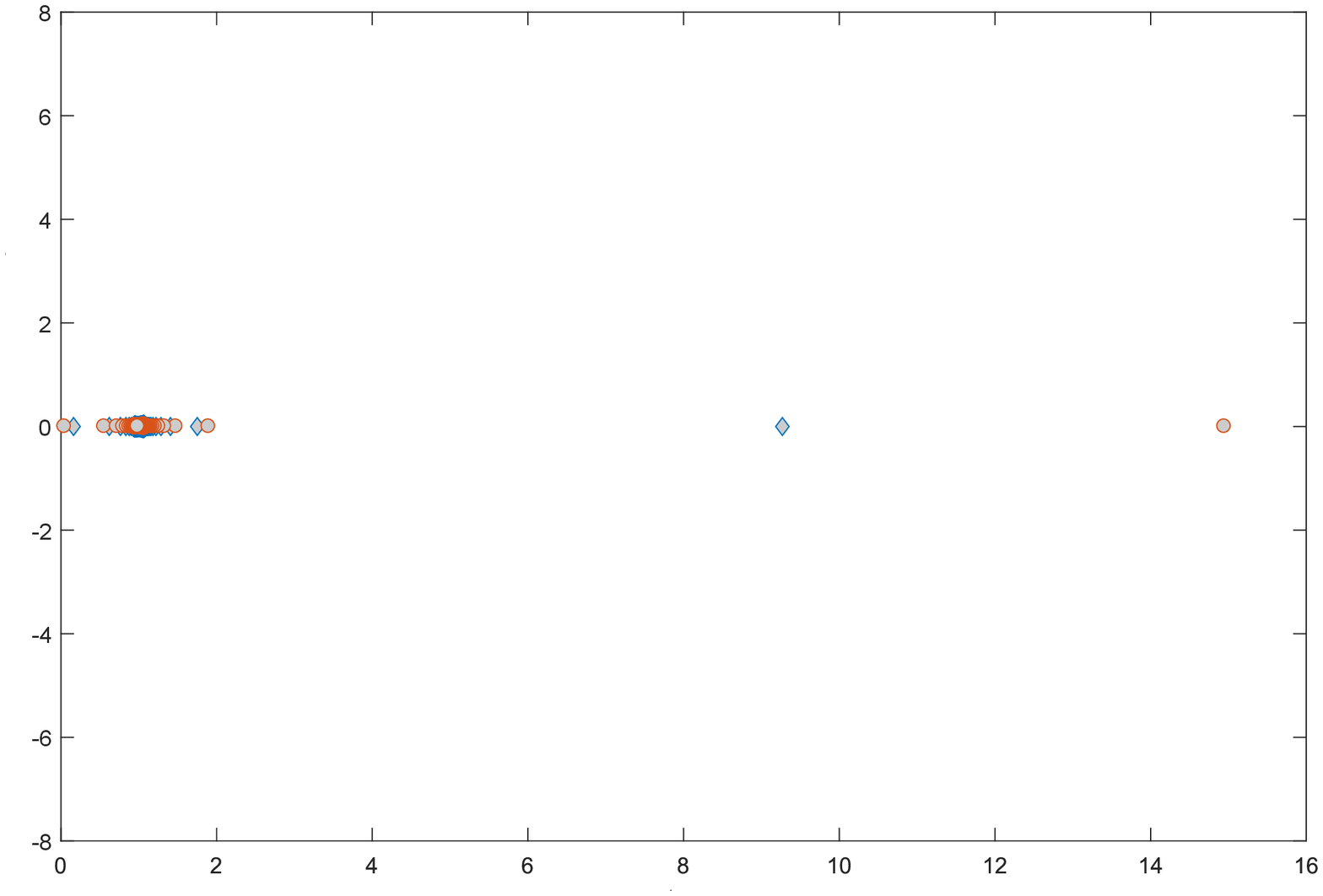}}}%
    \qquad
    \subfloat[Ιδιοτιμές κοντά στο $(1,0)$.]{{\includegraphics[width=0.45\linewidth]{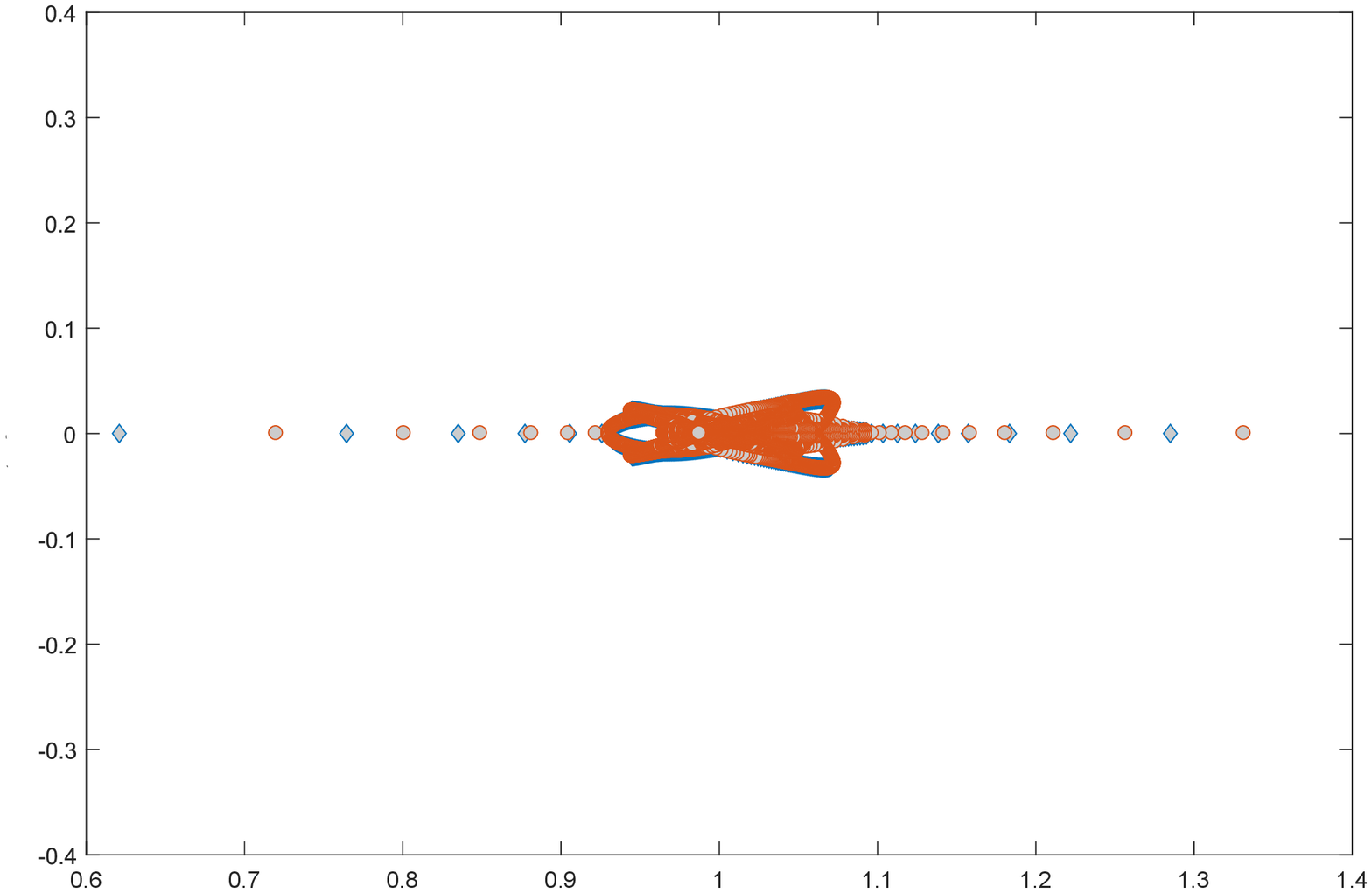}}}%
    \caption{Ιδιοτιμές ($\mathfrak{f}_{10}$).}
    \label{fig:example43}
\end{figure}
\end{exmp}

\section{Κυκλοειδείς προρρυθμιστές}

Σε αυτή την υποενότητα αρχικά θα περιγράψουμε τον τρόπο με τον οποίο κατασκευάζουμε τον προτεινόμενο προρρυθμιστή και στη συνέχεια θα δώσουμε έναν αλγόριθμο για την κατασκευή αυτού, σε μορφή ψευδοκώδικα. Ο προτεινόμενος προρρυθμιστής λαμβάνει δύο μορφές. Αυτή του κυκλοειδούς πίνακα, για συστήματα με καλή κατάσταση (η γεννήτρια συνάρτηση δεν έχει ρίζες), καθώς και αυτή του ταινιωτού-επί-κυκλοειδή πίνακα, για συστήματα με κακή κατάσταση (η γεννήτρια συνάρτηση έχει ρίζες). Αυτομάτως καταλαβαίνει κανείς ότι το πρώτο βήμα είναι να μελετήσουμε αν η γεννήτρια συνάρτηση έχει ρίζες ή όχι. Εφόσον βρεθούν ρίζες, προφανώς θα πρέπει να προσεγγιστούν. Το ίδιο θα πρέπει να γίνει και για τις πολλαπλότητες των ριζών, έτσι ώστε να βρεθεί το κατάλληλο τριγωνομετρικό πολυώνυμο που οδηγεί στην άρση της κακής κατάστασης.

\subsection{Κατασκευή του προρρυθμιστή}
`Οπως και στην προηγούμενη ενότητα, όπου προτάθηκε η χρήση ταινιωτών προρρυθμιστών, έτσι κι εδώ αρχικά θα προσεγγίσουμε τη γεννήτρια συνάρτηση του πίνακα {\en Toeplitz}, με τη βοήθεια του αναπτύγματος {\en Fourier}, $F_{n-1}$, στο ισοκατανεμημένο πλέγμα $G_n$, το οποίο έχει ως σημεία τα $\theta_j=\frac{2(j-1)\pi}{n},~j=1,\dots,n$. Επιλέξαμε αυτό το πλέγμα, το οποίο διαφέρει από το πλέγμα της προηγούμενης ενότητας (βλ. επίσης \cite{chay}), διότι τα σημεία αυτού είναι και σημεία όπου λαμβάνονται οι ιδιοτιμές του κυκλοειδή προρρυθμιστή που προσπαθούμε να κατασκευάσουμε, όπως φαίνεται στη σχέση (\ref{eq:circulant}).

Οι τιμές του αναπτύγματος $F_{n-1}$, θα μας υποδείξουν και τα σημεία πιθανών ριζών της γεννήτριας συνάρτησης, σε περίπτωση που το σύστημα είναι κακής κατάστασης. Αυτά επιλέγονται με τον τρόπο που προαναφέραμε στην υποενότητα \ref{Sss:411}, δηλαδή επιλέγοντας τα τοπικά ελάχιστα των $\vert F_{n-1}^1(\theta_j)\vert$ και $\vert F_{n-1}^2(\theta_j)\vert$, τα οποία λαμβάνουν τιμή πολύ κοντά στο 0. Σημειώνουμε για ακόμα μία φορά ότι σε περίπτωση που $\vert F_{n-1}^\ell(\theta_j)\vert$ και $\vert F_{n-1}^\ell(\theta_{j+1})\vert$, $\ell=1,2$ έχουν την ίδια τιμή για κάποιο $j$, η οποία είναι κοντά στο 0, υποθέτουμε ότι υπάρχει ρίζα στο $\frac{\theta_j+\theta_{j+1}}{2}$.

\begin{rem}
Αν το πραγματικό μέρος του αναπτύγματος {\en Fourier}, $F_{n-1}^1$ δεν έχει ρίζες, δεν προχωρούμε στην εκτίμηση των πιθανών ριζών του $F_{n-1}^2$ στο $G_n$, διότι ο προτεινόμενος προρρυθμιστής θα έχει τη μορφή κυκλοειδούς πίνακα, που προκύπτει από τις τιμές της $F_{n-1}$. Αυτό σημαίνει ότι δε γίνεται άρση των ριζών, όπως και στην περίπτωση όπου η γεννήτρια συνάρτηση είναι γνωστή εκ των προτέρων (βλ. υπενότητα \ref{3Ss:2.1}).
\end{rem}

\begin{rem}
Γνωρίζουμε ότι το φανταστικό μέρος του αναπτύγματος {\en Fourier}, $F_{n-1}^2$, έχει ρίζα στο $\pm\pi$, ως περιττό τριγωνομετρικό πολυώνυμο. Προκειμένου να εξετάσουμε αν η $f_2$ έχει όντως ρίζα στο $\pm\pi$, θα ελέγξουμε τις τιμές της $F_{n-1}^2$ σε μια περιοχή του $\pi$ (στο πλέγμα $G_n$). Αν αυτές είναι κοντά στο 0, θεωρούμε ότι υπάρχει ρίζα στο $\pm\pi$, διαφορετικά το σημείο $\pi$ είναι σημείο ασυνέχειας (όπως προφανώς και το $-\pi$).
\end{rem}

Για την άρση των ριζών της γεννήτριας συνάρτησης χρειαζόμαστε, όπως και στην προηγούμενη ενότητα, τις πολλαπλότητες $m_i^1$ και $m_i^2$, των ριζών της $f_1$ και $f_2$, αντίστοιχα, για το σημείο $x_i$, $i=1,2,\dots,\rho$. Υπενθυμίζουμε ότι $m_0^1$ και $m_0^2$ δηλώνουν τις αντίστοιχες πολλαπλότητες στο 0. Η εκτίμηση των πολλαπλοτήτων περιγράφηκε στην υποενότητα \ref{Sss:411}. Θα θέλαμε να αναφέρουμε ότι στη βιβλιογραφία μπορούν να βρεθούν κι άλλες ενδιαφέρουσες τεχνικές για την προσέγγιση των ιδιοτιμών, όπως για παράδειγμα οι \cite{bogoya, ekstrom, ekstrom3, ekstrom2}.

Για την περίπτωση όπου η $f_1$ έχει ρίζες οι οποίες τέμνουν τον άξονα, το αντίστοιχο ολοκλήρωμα της (\ref{eq:Simpson_integral}) είναι το:
\begin{equation*}
\frac{1}{2\pi}\int\limits_{0}^{2\pi}\left\vert F_{n-1}^1(x)\right\vert\mathrm{e}^{-\mathrm{i}(r-q)x}\mathrm{d}x,
\end{equation*}
το οποίο μπορεί να υπολογιστεί από τον σύνθετο κανόνα του {\en Simpson} στο $G_n\cup\lbrace 2\pi\rbrace$.


Το βήμα που έπεται της εκτίμησης των πολλαπλοτήτων είναι η εύρεση κατάλληλου τριγωνομετρικού πολυωνύμου, το οποίο αίρει τις ρίζες της $f$. Προφανώς, λόγω του ότι η $f_1$ είναι άρτια συνάρτηση και η $f_2$ περιττή, αν $x_i$ είναι ρίζα στο $(0,\pi]$, $-x_i$ θα είναι ρίζα στο $(-\pi,0)$, με την ίδια πολλαπλότητα. Παρατηρούμε ότι τα σημεία του πλέγματος, της προτεινόμενης τεχνικής προρρύθμισης, ανήκουν στο $[0,2\pi)$. Ωστόσο χρησιμοποιούμε το διάστημα $(-\pi,\pi]$ για την κατασκευή του τριγωνομετρικού πολυωνύμου, λόγω της προφανούς αντιστοιχίας μεταξύ των σημείων του $(-\pi,0)$ και $(\pi,2\pi)$ (μετατόπιση κατά $2\pi$). Το τριγωνομετρικό πολυώνυμο δίνεται όπως περιγράφεται στην υποενότητα \ref{Sss:411} (βλ. επίσης υποενότητα \ref{Ss:32}).

\begin{rem}
Προφανώς υπάρχει περίπτωση οι $F_{n-1}^1$ και $F_{n-1}^2$ να έχουν ρίζες σε διαφορετικά σημεία. Τότε, η γεννήτρια συνάρτηση δεν έχει καμία ρίζα και δεν προχωρούμε με την άρση των ριζών. `Ετσι, ο προτεινόμενος προρρυθμιστής είναι ο κυκλοειδής πίνακας $C_n(F_{n-1})$. Αυτό αποτελεί μια διαφορά ανάμεσα στην προτεινόμενη τεχνική προρρύθμισης και σε αυτήν της προηγούμενης ενότητας, που είχε να κάνει με ταινιωτούς προρρυθμιστές. Εκεί, υπενθυμίζουμε ότι η άρση των ριζών είναι απαραίτητη σε περίπτωση που η $F_{n-1}^1$ έχει κάποια ρίζα.
\end{rem}

Ο Αλγόριθμος \ref{2algo} περιγράφει την κατασκευή του προτεινόμενου προρρυθμιστή, σε μορφή ψευδοκώδικα.

\begin{breakablealgorithm}
\caption{Κατασκευή του Προρρυθμιστή.}
\label{2algo}
\textbf{Είσοδος:} $n\in\mathbb{N}$, $T_n$: $n\times n$ μη-συμμετρικός, πραγματικός πίνακας {\en Toeplitz}.
\begin{algorithmic}[1]
\STATE Κατασκευάστε το ισοκατανεμημένο πλέγμα $G_n$, με σημεία:\\ $\theta_j=\frac{2(j-1)\pi}{n}$, $j=1,2,\dots,n$.
\STATE\textbf{για }{$j=1,2,\dots,n$}
\STATE\quad Υπολογίστε το ανάπτυγμα {\en Fourier}: $F_{n-1}(\theta_j)=\sum\limits_{k=-n+1}^{n-1}t_k\mathrm{e}^{\mathrm{i}k\theta_j}$, $\theta_j\in G_n$.
\STATE\quad Εκτιμήστε τις $f_1$ και $f_2$ ως τις τιμές των $F_{n-1}^1(\theta_j)=\operatorname{Re}(F_{n-1}(\theta_j))$ και\\\hspace{.75pc} $F_{n-1}^2(\theta_j)=\operatorname{Im}(F_{n-1}(\theta_j))$, αντίστοιχα.
\STATE\textbf{τέλος για}
\STATE Επιλέξτε σημεία $\theta_i\in G_n$, κοντά στα τοπικά ελάχιστα της $\vert F_{n-1}^1\vert$, τέτοια ώστε $\vert F_{n-1}^1(\theta_i)\vert\simeq0$ και θεωρήστε τα ως πιθανές ρίζες της $\vert F_{n-1}^1\vert$. 
\STATE\textbf{αν} δεν έχει επιλεχθεί κανένα σημείο ως ρίζα \textbf{τότε}
	\STATE\quad Θέστε $g_n=1$ και \textbf{πηγαίνετε στο} 33.
\STATE\textbf{αλλιώς}
	\STATE\quad Επιλέξτε σημεία $\theta_i\in G_n$, κοντά στα τοπικά ελάχιστα της $\vert F_{n-1}^2\vert$, τέτοια \\\quad ώστε $\vert F_{n-1}^2(\theta_i)\vert\simeq0$ και θέστε τα ως πιθανές ρίζες της $\vert F_{n-1}^2\vert$.
\STATE\textbf{τέλος αν}
\STATE Σχηματίστε το σύνολο πιθανών ριζών $\lbrace x_i,~i=1,\dots,\rho\rbrace$ ως την ένωση των επιλεγμένων ριζών των $F_{n-1}^1$ και $F_{n-1}^2$.
\STATE\textbf{αν} δεν υπάρχει κοινή ρίζα για τις $F_{n-1}^1$ και $F_{n-1}^2$ \textbf{τότε}
	\STATE\quad Θέστε $g_n=1$ και \textbf{πηγαίνετε στο} 33.
\STATE\textbf{τέλος αν}
\STATE Υπολογίστε το συμμετρικό και αντι-συμμετρικό μέρος (του $T_n$), $S_n^1=\frac{T_n+T_n^T}{2}$ και $S_n^2=\frac{T_n-T_n^T}{2}$, αντίστοιχα.
\STATE\textbf{για }{$\ell=1,2$}
	\STATE\quad\textbf{αν} το $F_{n-1}^\ell$ παίρνει θετικές και αρνητικές τιμές \textbf{τότε}
	\STATE\hspace{1.85pc} Υπολογίστε την $\vert F_{n-1}^\ell\vert$ στο $G_n$.
	\STATE\hspace{1.85pc} Υπολογίστε το $\widehat{S}_{4k}^\ell\simeq T_{4k}(\vert F_{n-1}^\ell\vert)$ με χρήση του σύνθετου κανόνα του \\\hspace{1.86pc} {\en Simpson}, $k<\!\!<n$.
	\STATE\quad\textbf{αλλιώς}
	\STATE\hspace{1.85pc} Θέστε $\widehat{S}_{4k}^\ell=S_{4k}^\ell$.
	\STATE\quad\textbf{τέλος αν}
	\STATE\quad\textbf{για }{$j=1,2,\dots,\rho$}
		\STATE\hspace{1.85pc}\textbf{αν} $x_i$ είναι μια ρίζα του $F_{n-1}^\ell$ \textbf{τότε}
		\STATE\quad\hspace{1.85pc} Προσεγγίστε την ιδιοτιμή $\lambda_{i,k}^\ell$ του $\widehat{S}_k^\ell$ ως $\widetilde{\lambda}^\ell_{i,k}$ με λίγες επαναλήψεις\\\quad\hspace{1.85pc} της μεθόδου Αντιστρόφων Δυνάμεων, με αρχικό διάνυσμα:\\\quad\hspace{1.85pc} $\Theta_{i,k}=\frac{1}{\sqrt{k}}\left(1,\mathrm{e}^{\mathrm{i}x_i},\mathrm{e}^{2\mathrm{i}x_i},\dots,\mathrm{e}^{(k-1)\mathrm{i}x_i}\right)^T$.
		\STATE\quad\hspace{1.85pc} Επαναλάβετε το 26 για $\widehat{S}_{2k}^\ell$ και $\widehat{S}_{4k}^\ell$ ώστε να λάβετε τις $\widetilde{\lambda}^\ell_{i,2k}$ και $\widetilde{\lambda}^\ell_{i,4k}$,\\\quad\hspace{1.85pc} αντίστοιχα.
		\STATE\quad\hspace{1.85pc} Υπολογίστε το $\widetilde{s_i}^\ell=\frac{\widetilde{\lambda}^\ell_{i,k}-\widetilde{\lambda}^\ell_{i,2k}}{\widetilde{\lambda}^\ell_{i,2k}-\widetilde{\lambda}^\ell_{i,4k}}$ για να εκτιμήσετε την πολλαπλότητα\\\quad\hspace{1.85pc} $m_i^\ell$ ως τον πλησιέστερο ακέραιο στον $\log_2{\widetilde{s}_i^\ell}$.
		\STATE\hspace{1.85pc}\textbf{τέλος αν}	
	\STATE\quad\textbf{τέλος για}
\STATE\textbf{τέλος για}
\STATE Επιλέξτε το τριγωνομετρικό πολυώνυμο $g_n$, ώστε $\operatorname{Re}\left(\frac{F_{n-1}(\theta_j)}{g_n(\theta_j)}\right)>0$.
\STATE Κατασκευάστε τον κυκλοειδή προρρυθμιστή $C_n\left(\frac{F_{n-1}}{g_n}\right)$, ο οποίος έχει ως ιδιοτιμές, τις τιμές της $\frac{F_{n-1}}{g_n}$ στα σημεία του $G_n$.
\STATE Κατασκευάστε τον ταινιωτό-επί-κυκλοειδή προρρυθμιστή $T_n(g_n)C_n\left(\frac{F_{n-1}}{g_n}\right)$.
\end{algorithmic}
\end{breakablealgorithm}

Στην κατασκευή του ταινιωτού προρρυθμιστή, για μη-συμμετρικά συστήματος {\en Toeplitz} με άγνωστη γεννήτρια συνάρτηση, που περιγράψαμε παραπάνω, δώσαμε έναν τρόπο εκτίμησης πιθανών σημείων ασυνέχειας, διότι μας ενδιέφερε να μειώσουμε το σφάλμα προσέγγισης του αλγορίθμου {\en Remez}. Αυτή η εκτίμηση δεν είναι απαραίτητη για τον προτεινόμενο προρρυθμιστή, αφού ο κυκλοειδής προρρυθμιστής κατασκευάζεται από τις τιμές της $\frac{F_{n-1}}{g_n}$, στα σημεία του πλέγματος $G_n$. Σημειώνουμε ότι ο λόγος $\frac{F_{n-1}}{g_n}$ έχει πόλο στα σημεία του $G_n$, τα οποία είναι πιθανές ρίζες. Για την κατασκευή του $C_n\left(\frac{F_{n-1}}{g_n}\right)$, αντικαθιστούμε την τιμή $\frac{F_{n-1}(\theta_i)}{g_n(\theta_i)}$ με $\frac{F_{n-1}(\theta_{i+1})}{g_n(\theta_{i+1})}$ ή $\frac{F_{n-1}(\theta_{i-1})}{g_n(\theta_{i-1})}$. Μπορούμε επίσης να αντικαταστήσουμε με τον μέσο όρο των δύο τελευταίων όρων. Αυτή η τεχνική μετατόπισης χρησιμοποιήθηκε για συμμετρικά συστήματα {\en Toeplitz} στη \cite{chan_potts_steidl}.

\subsection{Θεωρητικά αποτελέσματα}\label{sec3}
Αρχικά, θα μελετήσουμε τη συσσώρευση των ιδιοτιμών για το προρρυθμισμένο σύστημα, όταν η $f$ δεν έχει ρίζες και είτε είναι επαρκώς ομαλή, είτε απλώς συνεχής. Η περίπτωση όπου η $f$ έχει σημεία ασυνέχειας θα καλυφθεί στη συνέχεια. Αναφέρουμε ότι στην \cite{Serra_1999}, ο συγγραφέας μελέτησε την προρρύθμιση συμμετρικών συστημάτων {\en Toeplitz} με άγνωστη γεννήτρια συνάρτηση και απέδειξε ιδιότητες για διάφορες περιπτώσεις όπου η γεννήτρια συνάρτηση είναι άγνωστη και συνεχώς παραγωγίσιμη ή συνεχής.

\setcounter{thm}{1}
\begin{thm}\label{thm:well_cond}
`Εστω $T_n$ ένας πραγματικός πίνακας {\en Toeplitz}, η γεννήτρια συνάρτηση του οποίου υπάρχει και είναι άγνωστη. Υποθέτουμε επίσης ότι δεν έχουν εντοπιστεί ρίζες, μέσω της προτεινόμενης τεχνικής. Τότε, οι ιδιοτιμές του προρρυθμισμένου πίνακα $C_n^{-1}(F_{n-1})T_n$ συσσωρεύονται γύρω από το σημείο $(1,0)$ του μιγαδικού επιπέδου και η συσσώρευση χαρακτηρίζεται ως:
\begin{enumerate}
\item Κύρια σε μια περιοχή του $(1,0)$, με ακτίνα της τάξεως $\mathcal{O}\left(\frac{1}{n}\right)$, αν η $f$ είναι επαρκώς ομαλή (συνεχώς παραγωγίσιμη).
\item Κύρια σε μια περιοχή του $(1,0)$, με ακτίνα της τάξεως $\mathcal{O}\left(\frac{\log{n}}{n}\right)$, αν η $f$ είναι συνεχής.
\end{enumerate} 
\end{thm}

\begin{proof}
Ο προρρυθμισμένος πίνακας γράφεται ως:
\begin{equation}\label{eq:cont_thm}
\begin{split}
C_n^{-1}(F_{n-1})T_n&=C_n^{-1}(F_{n-1})C_n(f)C_n^{-1}(f)T_n=C_n\left(\frac{f}{F_{n-1}}\right)C_n^{-1}(f)T_n\\
&=\left(I_n+C_n\left(\frac{f-F_{n-1}}{F_{n-1}}\right)\right)C_n^{-1}(f)T_n\\
&=C_n^{-1}(f)T_n+C_n\left(\frac{f-F_{n-1}}{fF_{n-1}}\right)T_n.
\end{split}
\end{equation}

Από το Θεώρημα \ref{Thm:5} έχουμε ότι οι ιδιοτιμές του πρώτου όρου, του παραπάνω αθροίσματος, έχουν κύρια συσσώρευση στο σημείο $(1,0)$ του μιγαδικού επιπέδου αν η $f$ είναι συνεχής. 

Στη συνέχεια μελετάμε τον δεύτερο όρο του αθροίσματος της (\ref{eq:cont_thm}), παίρνοντας τη νόρμα $\Vert\cdot\Vert_2$ αυτού και χρησιμοποιώντας το Λήμμα 1 της \cite{chan1993circulant} και το Λήμμα \ref{Lem:1}:
\begin{equation*}
\begin{split}
\left\Vert C_n\left(\frac{f-F_{n-1}}{fF_{n-1}}\right)T_n\right\Vert_2&\leq\left\Vert C_n\left(\frac{f-F_{n-1}}{fF_{n-1}}\right)\right\Vert_2\Vert T_n\Vert_2\\
&\leq2\left\Vert \frac{f-F_{n-1}}{fF_{n-1}}\right\Vert_\infty 2\left\Vert f\right\Vert_\infty =4\max{\left\vert \frac{f-F_{n-1}}{fF_{n-1}}\right\vert}\max{\left\vert f\right\vert}\\
&\leq4\frac{\max{\vert f\vert}}{\min{\vert fF_{n-1}\vert}}\max{\vert f-F_{n-1}\vert}\leq c\max{\vert f-F_{n-1}\vert}.
\end{split}
\end{equation*}
Στην περίπτωση \textit{1} όπου η $f$ θεωρείται συνεχώς παραγωγίσιμη, από την \cite{Serra_1999} έχουμε ότι $\max{\vert f-F_{n-1}\vert}=\mathcal{O}\left(\frac{1}{n}\right)$. Για τις υποθέσεις της περίπτωσης \textit{2}, έχουμε επίσης από την \cite{Serra_1999} ότι $\max{\vert f-F_{n-1}\vert}=\mathcal{O}\left(\frac{\log{n}}{n}\right)$.

Για να μελετήσουμε τη συσσώρευση των ιδιοτιμών του προρρυθμισμένου πίνακα $C_n^{-1}(F_{n-1})T_n$, χρησιμοποιούμε το {\en min-max} θεώρημα των {\en Courant-Fischer} για το συμμετρικό και αντι-συμμετρικό του μέρος. Προχωρούμε με την ανάλυση του συμμετρικού μέρους. `Οσον αφορά στο αντι-συμμετρικό, αυτή είναι ανάλογη.

`Εστω $A_n=\frac{C_n^{-1}(f)T_n+T_n^TC_n^{-1}(\bar{f})}{2}$ το συμμετρικό μέρος του πρώτου όρου του αθροίσματος στην (\ref{eq:cont_thm}) και $B_n=\frac{C_n\left(\frac{f-F_{n-1}}{fF_{n-1}}\right)T_n+T_n^TC_n\left(\frac{\bar{f}-\bar{F}_{n-1}}{\bar{f}\bar{F}_{n-1}}\right)}{2}$ το συμμετρικό μέρος του δευτέρου όρου της ίδιας σχέσης. Τότε, το συμμετρικό μέρος του προρρυθμισμένου πίνακα γράφεται ως $S_n=A_n+B_n$. `Εστω ότι με $\lambda_k$ και $\widetilde{\lambda}_k$, $k=1,2,\dots,n$ συμβολίζουμε της ιδιοτιμές του $A_n$ και $S_n$, αντίστοιχα, ταξινομημένες σε μη-αύξουσα σειρά: $\lambda_1\geq\lambda_2\geq\dots\geq\lambda_n$ και $\widetilde{\lambda}_1\geq\widetilde{\lambda}_2\geq\dots\geq\widetilde{\lambda}_n$. Τότε, από το {\en min-max} θεώρημα των {\en Courant-Fischer}, έχουμε:
\begin{equation*}
\begin{split}
\widetilde{\lambda}_k&=\min\limits_{V\in\mathbb{R}^{n+1-k}}\max\limits_{x\in V}{\frac{x^H(A_n+B_n)x}{x^Hx}}\leq\max\limits_{x\in W}{\frac{x^H(A_n+B_n)x}{x^Hx}}\\
&\leq\max\limits_{x\in W}{\frac{x^HA_nx}{x^Hx}}+\max\limits_{x\in W}{\frac{x^HB_nx}{x^Hx}}\leq\lambda_k+\Vert B_n\Vert_2,
\end{split}
\end{equation*}
όπου $W$ είναι ο χώρος που επιτυγχάνεται το μέγιστο του $\frac{x^HA_nx}{x^Hx}$. Από την άλλη:
\begin{equation*}
\begin{split}
\widetilde{\lambda}_k&=\max\limits_{V\in\mathbb{R}^{k}}\min\limits_{x\in V}{\frac{x^H(A_n+B_n)x}{x^Hx}}\geq\min\limits_{x\in\widetilde{W}}{\frac{x^H(A_n+B_n)x}{x^Hx}}\\
&\geq\min\limits_{x\in\widetilde{W}}{\frac{x^HA_nx}{x^Hx}}+\min\limits_{x\in\widetilde{W}}{\frac{x^HB_nx}{x^Hx}}\geq\lambda_k-\max\limits_{x\in\widetilde{W}}{\left|\frac{x^HB_nx}{x^Hx}\right|}\\
&\geq\lambda_k-\max\limits_{x\in\mathbb{R}^n}{\left|\frac{x^HB_nx}{x^Hx}\right|}=\lambda_k-\Vert B_n\Vert_2,
\end{split}
\end{equation*}
όπου $\widetilde{W}$ είναι ο χώρος στον οποίο επιτυγχάνεται το ελάχιστο του $\frac{x^HA_nx}{x^Hx}$.
Επομένως, για κάθε ιδιοτιμή $\lambda_k$ του $A_n$, οι αντίστοιχη ιδιοτιμή $\widetilde{\lambda}_k$ του $S_n$ κυμαίνεται στο διάστημα $[\lambda_k-\Vert B_n\Vert_2,\lambda_k+\Vert B_n\Vert_2]$. Αυτό σημαίνει ότι η συσσώρευση της ακολουθίας πινάκων $\{S_n\}$ είναι μια $\Vert B_n\Vert_2$-επέκταση της συσσώρευσης του $\{A_n\}$.

Η ίδια ανάλυση για τα αντι-συμμετρικά μέρη $A_n^\prime=\frac{C_n^{-1}(f)T_n-T_n^TC_n^{-1}(\bar{f})}{2}$ και $B_n^\prime=\frac{C_n\left(\frac{f-F_{n-1}}{fF_{n-1}}\right)T_n-T_n^TC_n\left(\frac{\bar{f}-\bar{F}_{n-1}}{\bar{f}\bar{F}_{n-1}}\right)}{2}$, μας δίνει ότι για κάθε ιδιοτιμή $\mu_k$ του $A_n^\prime$, η αντίστοιχη ιδιοτιμή $\widetilde{\mu}_k$ του $S_n^\prime=A_n^\prime+B_n^\prime$, κυμαίνεται στο διάστημα $[\mu_k-\Vert B_n^\prime\Vert_2,\mu_k+\Vert B_n^\prime\Vert_2]$. Η μελέτη του αντι-συμμετρικού μέρους μας οδηγεί στον ίδιο τύπο συσσώρευσης γύρω από το 0.

Λόγω της κύριας συσσώρευσης του πρώτου όρου, οι Ερμιτιανές και αντι-Ερμιτιανές ακολουθίες πινάκων $\{A_n\}$ και $\{A_n^{\prime}\}$, αντίστοιχα, παρουσιάζουν κύρια συσσώρευση στο 1 και στο 0, αντίστοιχα. Η ιδιότητα που αναφέραμε ισχύει διότι $C_n^{-1}(f)T_n=I_n+S_n+R_n$, όπου $S_n$ είναι πίνακας με μικρή νόρμα και $R_n$ είναι πίνακας χαμηλής βαθμίδας (βλ. Πόρισμα \ref{cor:S_and_L}). Επειδή το πραγματικό/φανταστικό μέρος των ιδιοτιμών ενός πίνακα βρίσκεται εντός του εύρους του Ερμιτιανού/αντι-Ερμιτιανού μέρους του πίνακα \cite{bendixson,hirsch}, το αποτέλεσμα προκύπτει από το Λήμμα \ref{lem:tyrt_zamar} (Λήμμα 2.1 της \cite{Tyrtyshnikov1999}), όπου ως $\mathcal{A}_n$ επιλέξαμε το $C_n^{-1}(f)T_n-I_n$ και ως $\mathcal{B}_n$ το $S_n$. Επομένως, για κάθε $\varepsilon>0$, υπάρχουν ακέραιοι $M,M^{\prime}>0$, τέτοιοι ώστε $M$ ιδιοτιμές του $A_n$ να κυμαίνονται εκτός του διαστήματος $(1-\varepsilon,1+\varepsilon)$ και $M^{\prime}$ ιδιοτιμές του $A_n^{\prime}$, εκτός του $(-\varepsilon,\varepsilon)$. Υποθέτουμε ότι οι πρώτες $M_1$ ($\lambda_1,\lambda_2,\dots\lambda_{M_1}$) και οι τελευταίες $M_2$ ($\lambda_{n-M_2+1},\dots,\lambda_n$) ιδιοτιμές, όπου $M_1+M_2=M$, κυμαίνονται εκτός της συσσώρευσης του $A_n$, ενώ οι πρώτες $M_1^{\prime}$ ($\mu_1,\mu_2,\dots\mu_{M_1^\prime}$) και οι τελευταίες $M_2^\prime$ ($\mu_{n-M_2^\prime+1},\dots,\mu_n$), με $M_1^\prime+M_2^\prime=M^\prime$, κυμαίνονται εκτός της συσσώρευσης του $A_n^\prime$. Οι $A_n$ και $A_n^\prime$ γράφονται ως:
\begin{equation*}
A_n=\sum\limits_{k=1}^n\lambda_k x_k x_k^H,\text{ και }A_n^\prime=\sum\limits_{k=1}^n\mu_k y_k y_k^H,
\end{equation*}
όπου $x_k$, $k=1,2,\dots,n$ είναι τα κανονικοποιημένα ιδιοδιανύσματα του $A_n$, τα οποία σχηματίζουν ορθοκανονική βάση και $y_k$, $k=1,2,\dots,n$ είναι τα αντίστοιχα ιδιοδιανύσματα του $A_n^\prime$. Χωρίζουμε τους $A_n$ και $A_n^\prime$ ως $A_n=\widetilde{A}_n+\widehat{A}_n$ και $A_n^\prime=\widetilde{A}_n^\prime+\widehat{A}_n^\prime$, αντίστοιχα, όπου:
\begin{equation*}
\begin{split}
&\widetilde{A}_n=\sum\limits_{k=M_1+1}^{n-M_2}\lambda_k x_k x_k^H,~\widehat{A}_n=\sum\limits_{k=1}^{M_1}\lambda_k x_k x_k^H+\sum\limits_{k=n-M_2+1}^{n}\lambda_k x_k x_k^H\text{ και}\\
&\widetilde{A}_n^\prime=\sum\limits_{k=M_1^\prime+1}^{n-M_2^\prime}\mu_k y_k y_k^H,~\widehat{A}_n^\prime=\sum\limits_{k=1}^{M_1^\prime}\mu_k y_k y_k^H+\sum\limits_{k=n-M_2^\prime+1}^{n}\mu_k y_k y_k^H.
\end{split}
\end{equation*}
Είναι προφανές ότι $\Vert\widetilde{A}_n-I_n\Vert_2\leq\varepsilon$ και $\Vert\widetilde{A}_n^\prime\Vert_2\leq\varepsilon$, ενώ οι $\widehat{A}_n$ και $\widehat{A}_n^\prime$ είναι πίνακες με χαμηλή βαθμίδα.

Χρησιμοποιούμε το Λήμμα \ref{lem:tyrt_zamar}, επιλέγοντας $\mathcal{A}_n=S_n+S_n^\prime-I_n=A_n+A_n^\prime+B_n+B_n^\prime-I_n$ και $\mathcal{B}_n=\widetilde{A}_n+\widetilde{A}_n^\prime+B_n+B_n^\prime-I_n$. `Εχουμε ότι $\Vert\mathcal{B}_n\Vert_2\leq\Vert\widetilde{A}_n-I_n\Vert_2+\Vert\widetilde{A}_n^\prime\Vert_2+\Vert B_n\Vert_2+\Vert B_n^\prime\Vert_2$ και $\Vert\mathcal{A}_n-\mathcal{B}_n\Vert_F^2=\Vert\widehat{A}_n+\widehat{A}_n^\prime\Vert_F^2\leq c$, αφού $\widehat{A}_n+\widehat{A}_n^\prime$ είναι πίνακας χαμηλής βαθμίδας, σταθερής και ανεξάρτητης της διάστασης $n$. Επομένως, οι ιδιοτιμές του $\mathcal{A}_n$ συσσωρεύονται με την έννοια της κύριας συσσώρευσης στον δίσκο $\lbrace z:\Vert z\Vert\leq2\varepsilon+\Vert B_n\Vert_2+\Vert B_n^\prime\Vert_2\rbrace$, για κάθε $\varepsilon>0$.

Στην περίπτωση \textit{1} (επαρκώς ομαλή συνάρτηση) έχουμε ότι $\Vert B_n\Vert_2+\Vert B_n^\prime\Vert_2=\mathcal{O}\left(\frac{1}{n}\right)$, ενώ στην περίπτωση \textit{2} (συνεχής συνάρτηση), $\Vert B_n\Vert_2+\Vert B_n^\prime\Vert_2=\mathcal{O}\left(\frac{\log{n}}{n}\right)$. Επομένως, η προρρυθμισμένη ακολουθία πινάκων $\{\mathcal{A}_n+I_n\}$ παρουσιάζει κύρια συσσώρευση των ιδιοτιμών σε μια περιοχή του $(1,0)$, με ακτίνα $\mathcal{O}\left(\frac{1}{n}\right)$ και $\mathcal{O}\left(\frac{\log{n}}{n}\right)$, στην περίπτωση \textit{1} και \textit{2}, αντίστοιχα κι έτσι η απόδειξη ολοκληρώνεται.
\end{proof}

Για να μελετήσουμε την περίπτωση όπου η $f$ έχει σημεία ασυνέχειας, ή την περίπτωση όπου η $f$ έχει ρίζες, αρχικά θα πρέπει να αποδείξουμε κάποιες ιδιότητες που αφορούν στις ιδιοτιμές και ιδιάζουσες τιμές, γινομένων ακολουθιών πινάκων, οι οποίες θα μας φανούν χρήσιμες.

\begin{lem}\label{lem:1}
`Εστω $\lbrace \mathcal{A}_n\rbrace$ και $\lbrace \mathcal{B}_n\rbrace$ ακολουθίες πινάκων, των οποίων τα Ερμιτιανά μέρη παρουσιάζουν κύρια συσσώρευση των ιδιοτιμών στο 1 και τα αντι-Ερμιτιανά μέρη αυτών παρουσιάζουν κύρια συσσώρευση των ιδιοτιμών στο 0. Τότε, η ακολουθία $\lbrace \mathcal{C}_n\rbrace$, όπου $\mathcal{C}_n=\mathcal{A}_n\mathcal{B}_n$, έχει κύρια συσσώρευση των ιδιοτιμών στο σημείο $(1,0)$, του μιγαδικού επιπέδου.
\end{lem}

\begin{proof}
Το Ερμιτιανό μέρος του $\mathcal{A}_n$, για κάθε $\varepsilon_A>0$ γράφεται ως $\frac{\mathcal{A}_n+\mathcal{A}_n^H}{2}=I_n+S_n+R_n$, όπου $S_n$ είναι ένας πίνακας με μικρή νόρμα $\Vert S_n\Vert_2\leq\varepsilon_A$ και $R_n$ είναι ένας πίνακας χαμηλής βαθμίδας $\operatorname{rank}(R_n)=k$. Το αντι-Ερμιτιανό μέρος του $\mathcal{A}_n$, για κάθε $\varepsilon_A^\prime>0$ γράφεται ως $\frac{\mathcal{A}_n-\mathcal{A}_n^H}{2}=S_n^\prime+R_n^\prime$, όπου $\Vert S_n^\prime\Vert_2\leq\varepsilon_A^\prime$ και $\operatorname{rank}(R_n^\prime)=k^\prime$. Ανάλογες ιδιότητες ισχύουν επίσης και για τον $\mathcal{B}_n$. Αυτό σημαίνει ότι, για κάθε $\varepsilon_B>0$, $\frac{\mathcal{B}_n+\mathcal{B}_n^H}{2}=I_n+Q_n+P_n$, όπου $\Vert Q_n\Vert_2\leq\varepsilon_B$ και $\operatorname{rank}(P_n)=l$, καθώς επίσης και για κάθε $\varepsilon_B^\prime>0$, $\frac{\mathcal{B}_n-\mathcal{B}_n^H}{2}=Q_n^\prime+P_n^\prime$, όπου $\Vert Q_n^\prime\Vert_2\leq\varepsilon_B^\prime$ και $\operatorname{rank}(P_n^\prime)=l^\prime$. Καταλήγουμε στο ότι $\mathcal{A}_n=I_n+S_n+S_n^\prime+R_n+R_n^\prime$ και $\mathcal{B}_n=I_n+Q_n+Q_n^\prime+P_n+P_n^\prime$. Επομένως:
\begin{equation*}
\mathcal{C}_n=\mathcal{A}_n\mathcal{B}_n=(I_n+S_n+S_n^\prime+R_n+R_n^\prime)(I_n+Q_n+Q_n^\prime+P_n+P_n^\prime)=I_n+\mathcal{S}_n+\mathcal{R}_n,
\end{equation*}
όπου:
\begin{equation*}
\mathcal{S}_n=S_n+S_n^\prime+Q_n+Q_n^\prime+S_nQ_n+S_nQ_n^\prime+S_n^\prime Q_n+S_n^\prime Q_n^\prime\text{ και}
\end{equation*}
\begin{equation*}
\mathcal{R}_n=(R_n+R_n^\prime)(I_n+Q_n+Q_n^\prime+P_n+P_n^\prime)+(I_n+S_n+S_n^\prime)(P_n+P_n^\prime).
\end{equation*}

Είναι προφανές ότι ο $\mathcal{S}_n$ είναι πίνακας με μικρή νόρμα $\Vert\mathcal{S}_n\Vert_2\leq\varepsilon$, για κάθε $\varepsilon>0$ ($\varepsilon$ είναι επιλεγμένο ως $\varepsilon_A+\varepsilon_A^\prime+\varepsilon_B+\varepsilon_B^\prime+\varepsilon_A\varepsilon_B+\varepsilon_A\varepsilon_B^\prime+\varepsilon_A^\prime\varepsilon_B+\varepsilon_A^\prime\varepsilon_B^\prime$). Επίσης ισχύει ότι $\operatorname{rank}(\mathcal{R}_n)\leq(k+k^\prime+l+l^\prime)$, το οποίο σημαίνει ότι ο $\mathcal{R}_n$ είναι πίνακας χαμηλής βαθμίδας. Χρησιμοποιώντας το Λήμμα \ref{lem:tyrt_zamar}, θέτοντας ως τον $\mathcal{A}_n$ του Λήμματος \ref{lem:tyrt_zamar} τον $\mathcal{S}_n+\mathcal{R}_n$ και ως τον $\mathcal{B}_n$ του ίδιου λήμματος, τον $\mathcal{S}_n$, προκύπτει το ζητούμενο αποτέλεσμα, επειδή έχουμε ότι $\Vert\mathcal{R}_n\Vert_F^2=\mathcal{O}(1)$ κι επειδή $\Vert\mathcal{S}_n\Vert_2\leq\varepsilon$, για κάθε $\varepsilon>0$ αρκούντως μικρό, οι ιδιοτιμές του $\mathcal{S}_n+\mathcal{R}_n$ συσσωρεύονται με την έννοια της κύριας συσσώρευσης στον κλειστό δίσκο με κέντρο το $(0,0)$ και ακτίνα $\varepsilon$. Συνεπώς, οι ιδιοτιμές του $\mathcal{C}_n=I_n+\mathcal{S}_n+\mathcal{R}_n$ έχουν κύρια συσσώρευση στο $(1,0)$.
\end{proof}

\begin{lem}\label{lem:2}
`Εστω $\lbrace \mathcal{A}_n\rbrace$ και $\lbrace \mathcal{B}_n\rbrace$ ακολουθίες πινάκων, των οποίων τα Ερμιτιανά μέρη παρουσιάζουν γενική συσσώρευση των ιδιοτιμών στα διαστήματα $(1-r_A,1+r_A)$ και $(1-r_B,1+r_B)$, με $s_A(n)$ και $s_B(n)$ ιδιοτιμές εκτός των διαστημάτων, αντίστοιχα, ενώ τα αντι-Ερμιτιανά μέρη αυτών παρουσιάζουν γενική συσσώρευση των ιδιοτιμών στα διαστήματα $(-r_A,r_A)$ και $(-r_B,r_B)$, με $s_A(n)$ και $s_B(n)$ ιδιοτιμές εκτός των διαστημάτων, αντίστοιχα. Τότε, η ακολουθία $\lbrace\mathcal{C}_n\rbrace$, όπου $\mathcal{C}_n=\mathcal{A}_n\mathcal{B}_n$, έχει γενική συσσώρευση των ιδιοτιμών σε μια περιοχή του $(1,0)$, του μιγαδικού επιπέδου, με ακτίνα $2r_A+2r_B+4r_Ar_B$ και το πολύ $2s_A(n)+2s_B(n)$ ιδιοτιμές εκτός της συσσώρευσης.
\end{lem}

\begin{proof}
`Οπως και στην απόδειξη του Λήμματος \ref{lem:1}, το Ερμιτιανό μέρος του $\mathcal{A}_n$ μπορεί να γραφεί ως $\frac{\mathcal{A}_n+\mathcal{A}_n^H}{2}=I_n+S_n+R_n$, όπου $\Vert S_n\Vert_2\leq r_A$ και $\operatorname{rank}(R_n)=s_A(n)$. Το αντι-Ερμιτιανό μέρος μπορεί να γραφεί ως $\frac{\mathcal{A}_n-\mathcal{A}_n^H}{2}=S_n^\prime+R_n^\prime$, όπου $\Vert S_n^\prime\Vert_2\leq r_A$ και $\operatorname{rank}(R_n^\prime)=s_A(n)$. Το ίδιο ισχύει και για τον $\mathcal{B}_n$, δηλαδή $\frac{\mathcal{B}_n+\mathcal{B}_n^H}{2}=I_n+Q_n+P_n$, όπου $\Vert Q_n\Vert_2\leq r_B$ και $\operatorname{rank}(P_n)=s_B(n)$, καθώς επίσης και $\frac{\mathcal{B}_n-\mathcal{B}_n^H}{2}=Q_n^\prime+P_n^\prime$, όπου $\Vert Q_n^\prime\Vert_2\leq r_B$ και $\operatorname{rank}(P_n^\prime)=s_B(n)$. Τότε, $\mathcal{A}_n=I_n+S_n+S_n^\prime+R_n+R_n^\prime$, ενώ $\mathcal{B}_n=I_n+Q_n+Q_n^\prime+P_n+P_n^\prime$. Επομένως:
\begin{equation*}
\begin{split}
\mathcal{C}_n&=\mathcal{A}_n\mathcal{B}_n=(I_n+S_n+S_n^\prime+R_n+R_n^\prime)(I_n+Q_n+Q_n^\prime+P_n+P_n^\prime)\\
&=I_n+S_n+S_n^\prime+(I_n+S_n+S_n^\prime)(Q_n+Q_n^\prime)\\
&\phantom{\hspace{40pt}}+(R_n+R_n^\prime)(I_n+Q_n+Q_n^\prime+P_n+P_n^\prime)+(I_n+S_n+S_n^\prime)(P_n+P_n^\prime)\\
&=I_n+\mathcal{S}_n+\mathcal{R}_n.
\end{split}
\end{equation*}
Είναι προφανές ότι ο $\mathcal{S}_n$ έχει φραγμένη νόρμα:
\begin{equation*}
\begin{split}
\Vert\mathcal{S}_n\Vert_2&=\Vert S_n+S_n^\prime+(I_n+S_n+S_n^\prime)(Q_n+Q_n^\prime)\Vert_2\\
&\leq\Vert S_n\Vert_2+\Vert S_n^\prime\Vert_2+(1+\Vert S_n\Vert_2+\Vert S_n^\prime\Vert_2)(\Vert Q_n\Vert_2+\Vert Q_n^\prime\Vert_2)\\
&\leq2r_A+2r_B+4r_Ar_B,
\end{split}
\end{equation*}
ενώ ο $\mathcal{R}_n$ είναι πίνακας χαμηλής βαθμίδας $\operatorname{rank}(\mathcal{R}_n)\leq2s_A(n)+2s_B(n)$.

Χρησιμοποιούμε το Λήμμα \ref{lem:tyrt_zamar}, θέτοντας ως $\mathcal{A}_n$ του Λήμματος \ref{lem:tyrt_zamar}, τον $\mathcal{S}_n+\mathcal{R}_n$ και ως $\mathcal{B}_n$ του ίδιου λήμματος, τον $\mathcal{S}_n$. Επειδή $\operatorname{rank}(\mathcal{R}_n)=\mathcal{O}(2s_A(n)+2s_B(n))$, λαμβάνουμε ότι $\Vert \mathcal{R}_n\Vert_F^2=\mathcal{O}(2s_A(n)+2s_B(n))$ κι επίσης ισχύει ότι $\Vert\mathcal{S}_n\Vert_2\leq2r_A+2r_B+4r_Ar_B$. Επομένως, οι ιδιοτιμές του $\mathcal{S}_n+\mathcal{R}_n$ συσσωρεύονται στον κλειστό δίσκο με κέντρο το $(0,0)$ και ακτίνα $2r_A+2r_B+4r_Ar_B$, με την έννοια της γενικής συσσώρευσης, με $\mathcal{O}(2s_A(n)+2s_B(n))$, ιδιοτιμές εκτός της συσσώρευσης. Συμπεραίνουμε ότι οι ιδιοτιμές του $\mathcal{C}_n=I_n+\mathcal{S}_n+\mathcal{R}_n$ έχουν γενική συσσώρευση στον κλειστό δίσκο με κέντρο το $(1,0)$ και ακτίνα $2r_A+2r_B+4r_Ar_B$, καθώς επίσης και $\mathcal{O}(2s_A(n)+2s_B(n))$ ιδιοτιμές εκτός της συσσώρευσης.
\end{proof}

Στο Λήμμα \ref{lem:2}, θεωρήσαμε τις ίδιες ακτίνες $r_A$ και $r_B$, για τα Ερμιτιανά και αντι-Ερμιτιανά μέρη των $\mathcal{A}_n$ και $\mathcal{B}_n$, αντίστοιχα, καθώς επίσης και τον ίδιο αριθμό ιδιοτιμών εκτός της συσσώρευσης, $s_A(n)$ και $s_B(n)$. Η απόδειξη είναι ανάλογη αν θεωρήσουμε διαφορετικές τιμές, αλλά η τάξη μεγέθους παραμένει η ίδια.

\begin{rem}
Πρέπει να σχολιάσουμε ότι αν οι ακτίνες $r_A$ και $r_B$ διαφέρουν κατά τάξη μεγέθους, τότε η ακτίνα του κύκλου θα πρέπει να είναι αυτή με τη μεγαλύτερη τάξη. Αν επιπλέον οι $s_A(n)$ και $s_B(n)$ διαφέρουν επίσης κατά τάξη μεγέθους, ο αριθμός των ιδιοτιμών που βρίσκονται εκτός της συσσώρευσης, θα εξαρτάται ομοίως από τη μεγαλύτερη τάξη. Αξίζει να σημειώσουμε επίσης ότι το Λήμμα \ref{lem:2} περιέχει το αποτέλεσμα του Λήμματος \ref{lem:1}, θέτοντας $r_A=r_B=0$ και $s_A(n)=c_1$, $s_B(n)=c_2$, όπου $c_1$ και $c_2$ είναι σταθερές ανεξάρτητες της διάστασης $n$.
\end{rem}

Προχωρούμε με τη μελέτη των ιδιοτιμών, όταν η $f$ είναι ασυνεχής και δεν έχει ρίζες.
\begin{thm}\label{thm:well_cond_dis}
`Εστω $T_n$ ένας πραγματικός πίνακας {\en Toeplitz}, η γεννήτρια συνάρτηση του οποίου υπάρχει και είναι άγνωστη. Υποθέτουμε επίσης ότι δεν έχουν βρεθεί ρίζες μέσω της προτεινόμενης μεθόδου. Αν η $f$ έχει πεπερασμένα σημεία ασυνέχειας, οι ιδιοτιμές του προρρυθμισμένου πίνακα $C_n^{-1}(F_{n-1})T_n$ συσσωρεύονται, με την έννοια της γενικής συσσώρευσης, σε μια περιοχή του $(1,0)$, με σταθερή ακτίνα, το πολύ ίση με $0.179/(1-0.179)$ και με $\mathcal{O}(\log{n})$ ιδιοτιμές εκτός της συσσώρευσης.
\end{thm}

\begin{proof}
`Οπως έχει ήδη αποδειχθεί στη σχέση (\ref{eq:cont_thm}), ο προρρυθμισμένος πίνακας γράφεται ως:
\begin{equation*}
C_n^{-1}(F_{n-1})T_n=C_n\left(\frac{f}{F_{n-1}}\right)C_n^{-1}(f)T_n.
\end{equation*}

Από το Θεώρημα \ref{thm:T-C_piecewise} λαμβάνουμε τη γενική συσσώρευση των ιδιοτιμών του $C_n^{-1}(f)T_n$, στο $(1,0)$ με $\mathcal{O}(\log{n})$ ιδιοτιμές εκτός της συσσώρευσης. Το Ερμιτιανό μέρος του πίνακα μπορεί να γραφεί ως $I_n+H(S_n)+H(R_n)$, όπου $H(S_n)=\frac{S_n+S_n^H}{2}$ και $H(R_n)=\frac{R_n+R_n^H}{2}$. Προφανώς, $\Vert H(S_n)\Vert_2\leq\varepsilon$ και $\Vert H(R_n)\Vert_F^2=\mathcal{O}(\log{n})$. Ανάλογες ιδιότητες ισχύουν και για το αντι-Ερμιτιανό του μέρος κι έτσι καταλήγουμε στην ίδια κατά τάξη μεγέθους γενική συσσώρευση του Ερμιτιανού και αντι-Ερμιτιανού μέρους του πίνακα $C_n^{-1}(f)T_n$.

Οι ιδιοτιμές του πρώτου όρου $C_n\left(\frac{f}{F_{n-1}}\right)$ είναι οι τιμές της συνάρτησης $\frac{f}{F_{n-1}}$ στα σημεία του πλέγματος $G_n$. Γνωρίζουμε ότι το ανάπτυγμα {\en Fourier} συγκλίνει ομοιόμορφα, όσο το $n$ τείνει στο άπειρο, με ταχύτητα σύγκλισης $\mathcal{O}\left(\frac{\log{n}}{n}\right)$, σε οποιοδήποτε συμπαγές υποσύνολο του $\mathbb{R}$ που δεν περιέχει σημεία ασυνέχειας. Ωστόσο, σε μικρές περιοχές των σημείων ασυνέχειας, εμφανίζεται το φαινόμενο {\en Gibbs} \cite{hewitt}, όπου το σχετικό σφάλμα συγκλίνει σε κάποιον σταθερό αριθμό, σχεδόν ίσο με $17.9\%$ της τιμής της $f$ \cite{bocher}. `Εχουμε:
\begin{equation*}
C_n\left(\frac{f}{F_{n-1}}\right)=C_n\left(1+\frac{f-F_{n-1}}{F_{n-1}}\right)=I_n+C_n\left(\frac{f-F_{n-1}}{F_{n-1}}\right)
\end{equation*}
και ισχύει ότι:
\begin{equation*}
\begin{split}
\left\Vert C_n\left(\frac{f-F_{n-1}}{F_{n-1}}\right)\right\Vert_2&=\left(\lambda_{\max}\left(C_n\left(\frac{f-F_{n-1}}{F_{n-1}}\right)^HC_n\left(\frac{f-F_{n-1}}{F_{n-1}}\right)\right)\right)^{\frac{1}{2}}\\
&=\left(\lambda_{\max}\left(C_n\left(\left\vert\frac{f-F_{n-1}}{F_{n-1}}\right\vert^2\right)\right)\right)^{\frac{1}{2}}\\
&\leq\left(\max\limits_x{\left\vert\frac{f(x)-F_{n-1}(x)}{F_{n-1}(x)}\right\vert^2}\right)^{\frac{1}{2}}\leq\left\Vert\frac{f-F_{n-1}}{F_{n-1}}\right\Vert_\infty.
\end{split}
\end{equation*}

Γράφουμε το ανάπτυγμα {\en Fourier} ως $F_{n-1}=F_{n-1}^1+\mathrm{i}F_{n-1}^2$. Ισχύει:
\begin{equation*}
\begin{split}
&F_{n-1}^1=f_1+\epsilon_1f_1=(1+\epsilon_1)f_1\text{ και}\\
&F_{n-1}^2=f_2+\epsilon_2f_2=(1+\epsilon_2)f_2,
\end{split}
\end{equation*}
όπου $\epsilon_1, \epsilon_2$ είναι οι αντίστοιχες συναρτήσεις σφάλματος. `Ετσι:
\begin{equation*}
\begin{split}
\left\Vert C_n\left(\frac{f-F_{n-1}}{F_{n-1}}\right)\right\Vert_2&\leq\left\Vert\frac{f-F_{n-1}}{F_{n-1}}\right\Vert_\infty=\max\limits_x{\left\vert\frac{f(x)-F_{n-1}(x)}{F_{n-1}(x)}\right\vert}\\
&=\max\limits_x{\frac{\left\vert f_1(x)-F_{n-1}^1(x)+\mathrm{i}\left(f_2(x)-F_{n-1}^2(x)\right)\right\vert}{\left\vert F_{n-1}^1(x)+\mathrm{i}F_{n-1}^2(x)\right\vert}}\\
&=\max\limits_x{\frac{\left\vert \epsilon_1(x)f_1(x)+\mathrm{i}\epsilon_2(x)f_2(x)\right\vert}{\left\vert (1+\epsilon_1(x))f_1(x)+\mathrm{i}(1+\epsilon_2(x))f_2(x)\right\vert}}.
\end{split}
\end{equation*}
Είναι προφανές ότι το παραπάνω μέγιστο εμφανίζεται σε μικρές περιοχές των σημείων ασυνέχειας όπου το φαινόμενο {\en Gibbs} λαμβάνει χώρα. `Εστω $y$ το σημείο (σε κάποια μικρή περιοχή σημείου ασυνέχειας) που δίνει τη μέγιστη τιμή. Τότε: 
\begin{equation*}
\begin{split}
\left\Vert C_n\left(\frac{f-F_{n-1}}{F_{n-1}}\right)\right\Vert_2&\leq\max\limits_x{\frac{\left\vert \epsilon_1(x)f_1(x)+\mathrm{i}\epsilon_2(x)f_2(x)\right\vert}{\left\vert (1+\epsilon_1(x))f_1(x)+\mathrm{i}(1+\epsilon_2(x))f_2(x)\right\vert}}\\
&=\frac{\left(\epsilon_1^2(y)f_1^2(y)+\epsilon_2^2(y)f_2^2(y)\right)^{\frac{1}{2}}}{\left((1+\epsilon_1(y))^2f_1^2(y)+(1+\epsilon_2(y))^2f_2^2(y)\right)^{\frac{1}{2}}}\\
&\leq\frac{0.179\left(f_1^2(y)+f_2^2(y)\right)^{\frac{1}{2}}}{(1-0.179)\left(f_1^2(y)+f_2^2(y)\right)^{\frac{1}{2}}}=\frac{0.179}{1-0.179}.
\end{split}
\end{equation*}

Υπολογίζοντας τη νόρμα του γινομένου $\mathcal{A}_n\mathcal{B}_n$, όπως στο Λήμμα \ref{lem:2} με $\mathcal{A}_n=C_n\left(\frac{f}{F_{n-1}}\right)$ και $\mathcal{B}_n=C_n^{-1}(f)T_n$, και λαμβάνοντας υπόψη ότι $r_A=\frac{0.179}{1-0.179}$ και $r_B=0$, προκύπτει ότι $\left\Vert C_n\left(\frac{f}{F_{n-1}}\right)C_n^{-1}(f)T_n\right\Vert_2\leq r_A=\frac{0.179}{1-0.179}\simeq0.218$. Επομένως, ο προρρυθμισμένος πίνακας έχει γενική συσσώρευση των ιδιοτιμών, σε μια περιοχή με κέντρο το $(1,0)$, ακτίνα το πολύ ίση με $0.218$ και $\mathcal{O}(\log{n})$ ιδιοτιμές εκτός της συσσώρευσης.
\end{proof}

Συνεχίζουμε με τη μελέτη των ιδιοτιμών και ιδιαζουσών τιμών, στην περίπτωση που η $f$ έχει ρίζες, δηλαδή για συστήματα με κακή κατάσταση.

\begin{thm}\label{thm:ill_cond}
`Εστω $T_n$ πραγματικός πίνακας {\en Toeplitz}, στον οποίο αντιστοιχεί μια άγνωστη και μιγαδική γεννήτρια συνάρτηση $f$, με ρίζες στο διάστημα $(-\pi,\pi]$. `Εστω $F_{n-1}$ το ανάπτυγμα {\en Fourier} του $T_n$ και $g_n$ το τριγωνομετρικό πολυώνυμο, το οποίο αίρει τις εκτιμώμενες ρίζες μέσω της τεχνικής που περιγράφηκε. Υποθέτουμε ότι οι εκτιμώμενες μη-μηδενικές ρίζες $\widetilde{x}_i$, $i=1,2,\dots,\rho$ έχουν ένα σφάλμα $\varepsilon_i=\widetilde{x}_i-x_i$ και ότι οι πολλαπλότητες αυτών έχουν εκτιμηθεί ακριβώς. Τότε, ο προρρυθμισμένος πίνακας $C_n^{-1}\left(\frac{F_{n-1}}{g_n}\right)T_n^{-1}(g_n)T_n$ έχει:
\begin{enumerate}
\item Γενική συσσώρευση των ιδιοτιμών σε μια περιοχή του $(1,0)$, με ακτίνα $o(1)$ και $o(n)$ ιδιοτιμές εκτός της συσσώρευσης, αν η $f$ είναι συνεχής.
\item Γενική συσσώρευση των ιδιοτιμών σε μια περιοχή του $(1,0)$, με ακτίνα σχεδόν ίση με $0.179/(1-0.179)$ και $o(n)$ ιδιοτιμές εκτός της συσσώρευσης, αν η $f$ είναι κατά τμήματα συνεχής.
\end{enumerate} 
\end{thm}

\begin{proof}
Ξεκινάμε με την απόδειξη της περίπτωσης 1. Για απλούστευση, υποθέτουμε ότι οι εκτιμώμενες ρίζες είναι οι $\pm\widetilde{x}_1\neq0$ με πολλαπλότητα $m$. Η ανάλυση για περισσότερα ζεύγη ριζών, αποτελεί μια απλή γενίκευση, όπως θα γίνει εύκολα κατανοητό από την απόδειξη παρακάτω.

Αρχικά υποθέτουμε ότι η πολλαπλότητα της ρίζας αντιστοιχεί στο πραγματικό μέρος της $f$. Επομένως, χρησιμοποιώντας την προτεινόμενη τεχνική, το τριγωνομετρικό πολυώνυμο είναι το:
\begin{equation*}
g_n(x)=\left(\cos{\widetilde{x}_1}-\cos{x}\right)^m=\left(2\sin{\frac{x+\widetilde{x}_1}{2}}\sin{\frac{x-\widetilde{x}_1}{2}}\right)^m.
\end{equation*}
`Ετσι, ο προρρυθμισμένος πίνακας γράφεται ως:
\begin{equation}\label{eq:rank_cor}
\begin{split}
P_n&=C_n^{-1}\left(\frac{F_{n-1}}{g_n}\right)T_n^{-1}(g_n)T_n(f)\\
&=C_n^{-1}\left(\frac{F_{n-1}}{g_n}\right)T_n^{-1}(g_n)\left(T_n(g)T_n\left(\frac{f}{g}\right)+L\right),
\end{split}
\end{equation}
όπου $L$ είναι πίνακας χαμηλής βαθμίδας, η οποία εξαρτάται από το εύρος ταινίας του $T_n(g)$.

Θα μελετήσουμε τον πίνακα $P_n^\prime=C_n^{-1}\left(\frac{F_{n-1}}{g_n}\right)T_n^{-1}(g_n)T_n(g)T_n\left(\frac{f}{g}\right)$, ο οποίος διαφέρει από τον $P_n$, κατά πίνακα χαμηλής βαθμίδας. Αυτός είναι όμοιος με τον:
\begin{equation}\label{eq:seq_matr}
\begin{split}
\widetilde{P}_n&=C_n\left(\frac{f}{g}\right)P_n^\prime C_n^{-1}\left(\frac{f}{g}\right)\\
&=C_n\left(\frac{fg_n}{F_{n-1}g}\right)T_n^{-1}(g_n)T_n(g)T_n\left(\frac{f}{g}\right)C_n^{-1}\left(\frac{f}{g}\right).
\end{split}
\end{equation}
Χωρίζουμε αυτό το γινόμενο πινάκων στους εξής παράγοντες:\\$C_n\left(\frac{fg_n}{F_{n-1}g}\right)$, $T_n^{-1}(g_n)T_n(g)$ και $T_n\left(\frac{f}{g}\right)C_n^{-1}\left(\frac{f}{g}\right)$. Ο τελευταίος εξ αυτών είναι όμοιως με τον προρρυθμισμένο πίνακα που αντιστοιχεί στην περίπτωση συστημάτων με καλή κατάσταση. Επομένως, αν η $f$ είναι συνεχής, έχει κύρια συσσώρευση των ιδιοτιμών του στο $(1,0)$. Σημειώνουμε επίσης ότι το Ερμιτιανό και αντι-Ερμιτιανό του μέρος έχει το ίδιο είδος συσσώρευσης, όπως εύκολα μπορεί να δει κανείς στην απόδειξη του Θεωρήματος \ref{thm:well_cond}.

Πρέπει να μελετήσουμε και τη συσσώρευση των δύο παραγόντων που απομένουν. Ο πρώτος γράφεται ως:
\begin{equation}
\begin{split}
C_n\left(\frac{fg_n}{F_{n-1}g}\right)&=C_n\left(\frac{(F_{n-1}+\varepsilon_f)(g+\varepsilon_g)}{F_{n-1}g}\right)\\
&=I_n+C_n\left(\frac{\varepsilon_f}{F_{n-1}}\right)+C_n\left(\frac{\varepsilon_g}{g}\right)+C_n\left(\frac{\varepsilon_f\varepsilon_g}{F_{n-1}g}\right).
\end{split}
\end{equation}
Το σφάλμα $\varepsilon_f=f-F_{n-1}$ του αναπτύγματος {\en Fourier} εξαρτάται από την ομαλότητα της συνάρτησης $f$, ας πούμε αν η $f$ είναι συνεχώς παραγωγίσιμη, το σφάλμα είναι τάξεως $\mathcal{O}\left(\frac{1}{n}\right)$, ενώ αν η $f$ είναι απλά συνεχής, είναι τάξεως $\mathcal{O}\left(\frac{\log{n}}{n}\right)$. Αν η $f$ είναι επαρκώς ομαλή, μπορούμε να λάβουμε, κατά τάξη μεγέθους, μικρότερο σφάλμα $\varepsilon_f$. Παρατηρούμε ότι σε οποιαδήποτε από τις παραπάνω περιπτώσεις το $\varepsilon_f$ τείνει προς το 0, όσο το $n$ τείνει στο άπειρο, δηλαδή $\varepsilon_f=o(1)$. Θεωρούμε τα διαστήματα με κέντρο τα σημεία των εκτιμώμενων ριζών $-\widetilde{x}_1$ και $\widetilde{x}_1$, $(-\widetilde{x}_1-s_n,-\widetilde{x}_1+s_n)$ και $(\widetilde{x}_1-s_n,\widetilde{x}_1+s_n)$, αντίστοιχα, με ακτίνα $s_n$ η οποία εξαρτάται μεν από το $n$, ισχύει δε ότι $s_n=o(1)$, ή ισοδύναμα $\lim\limits_{n\rightarrow\infty}\frac{s_n}{1}=0$ (ο κλασικός ασυμπτωτικός ορισμός του $o$). Η συνάρτηση $\frac{\varepsilon_f}{F_{n-1}}$ στο σημείο $\widetilde{x}_1+s_n$ μας δίνει $\frac{\varepsilon_f(\widetilde{x}_1+s_n)}{F_{n-1}(\widetilde{x}_1+s_n)}=\frac{\varepsilon_f(\widetilde{x}_1+s_n)}{f(\widetilde{x}_1+s_n)-\varepsilon_f(\widetilde{x}_1+s_n)}$. Υποθέτουμε ότι $\varepsilon_f=o(s_n^m)$, ισοδύναμα $\lim\limits_{n\rightarrow\infty}\frac{\varepsilon_f}{s_n^m}=0$. Τότε, ο παραπάνω λόγος προσεγγίζεται ως:
\begin{equation*}
\frac{\varepsilon_f(\widetilde{x}_1+s_n)}{F_{n-1}(\widetilde{x}_1+s_n)}\simeq\frac{\varepsilon_f(\widetilde{x}_1+s_n)}{f(\widetilde{x}_1+s_n)}=\frac{\varepsilon_f(\widetilde{x}_1+s_n)}{cs_n^m}=o(1).
\end{equation*}
Θα θέλαμε επίσης να σχολιάσουμε ότι έχουμε τη δυνατότητα επιλογής του $s_n$. Για παράδειγμα, αν $\varepsilon_f=\mathcal{O}\left(\frac{1}{n}\right)$ και $m=1$, μπορούμε να επιλέξουμε $s_n=\mathcal{O}\left(\frac{1}{\sqrt{n}}\right)$ ή $s_n=\mathcal{O}\left(\frac{\log{n}}{n}\right)$, υπάρχουν δηλαδή άπειρες επιλογές.

Το ίδιο ισχύει και για τα σημεία $-\widetilde{x}_1-s_n$, $-\widetilde{x}_1+s_n$ και $\widetilde{x}_1-s_n$. Είναι προφανές ότι ο προαναφερθής λόγος παραμένει τάξεως $o(1)$, για τα σημεία που ανήκουν στο $(-\pi,\pi]\backslash(-\widetilde{x}_1-s_n,-\widetilde{x}_1+s_n)\cup(\widetilde{x}_1-s_n,\widetilde{x}_1+s_n)$. Επομένως, λαμβάνουμε τη συσσώρευση των ιδιοτιμών του $C_n\left(\frac{\varepsilon_f}{F_{n-1}}\right)$ σε μια περιοχή του $(0,0)$, με ακτίνα $o(1)$. Προκειμένου να βρούμε το είδος συσσώρευσης, μένει να μετρήσουμε πόσες ιδιοτιμές κυμαίνονται εκτός αυτής και αντιστοιχούν σε σημεία των παραπάνω διαστημάτων. Η απόσταση κάθε διαστήματος είναι $2s_n$, επομένως λαμβάνουμε το πολύ $\frac{4s_n}{2\pi}n=o(n)$ ιδιοτιμές εκτός της συσσώρευσης, το οποίο σημαίνει ότι έχουμε γενική συσσώρευση των ιδιοτιμών σε μια μικρή περιοχή του $(0,0)$.

Θα μελετήσουμε τον όρο $C_n\left(\frac{\varepsilon_g}{g}\right)$. Η συνάρτηση σφάλματος $\varepsilon_g$ εξαρτάται από το μέγεθος του σφάλματος κατά την εκτίμηση της ρίζας $\varepsilon_1=\widetilde{x}_1-x_1$. Ο {\en S.~Serra-Capizzano} απέδειξε στην \cite{Serra_1999}, ότι αν $f\in C^k$, αυτό το σφάλμα φράσσεται από $\vert\varepsilon_1\vert\leq C\left(\frac{\log{n}}{n^k}\omega(f;n^{-1})\right)^{\frac{1}{m}}$, οπότε $\vert\varepsilon_1\vert=\mathcal{O}\left(\frac{\log{n}}{n^k}\right)^{\frac{1}{m}}$. Στην περίπτωση όπου $f\in C$, λαμβάνουμε το σφάλμα $\vert\varepsilon_1\vert=\mathcal{O}\left(\frac{\log{n}}{n}\right)^{\frac{1}{m}}$, ενώ αν η $f$ είναι συνεχώς παραγωγίσιμη, είναι γνωστό ότι $\vert\varepsilon_1\vert=\mathcal{O}\left(\frac{1}{n}\right)^{\frac{1}{m}}$. Στη συνέχεια, μελετάμε το λόγο $\frac{\varepsilon_g}{g}$:
\begin{equation*}
\begin{split}
\frac{\varepsilon_g(x)}{g(x)}&=\frac{g_n(x)-g(x)}{g(x)}=\frac{(\cos{\widetilde{x}_1}-\cos{x})^m-(\cos{x_1}-\cos{x})^m}{(\cos{x_1}-\cos{x})^m}\\
&=\frac{(\cos{\widetilde{x}_1}-\cos{x_1})\sum\limits_{i=0}^{m-1}\left(\cos{\widetilde{x}_1}-\cos{x}\right)^{m-i-1}\left(\cos{x_1}-\cos{x}\right)^{i}}{(\cos{x_1}-\cos{x})^m}\\
&=\frac{\cos{\widetilde{x}_1}-\cos{x_1}}{\cos{x_1}-\cos{x}}\times\sum\limits_{i=0}^{m-1}\left(\frac{\cos{\widetilde{x}_1}-\cos{x}}{\cos{x_1}-\cos{x}}\right)^i\\
&=\frac{\sin{\frac{x_1+\widetilde{x}_1}{2}}\sin{\frac{x_1-\widetilde{x}_1}{2}}}{\sin{\frac{x+x_1}{2}}\sin{\frac{x-x_1}{2}}}\sum\limits_{i=0}^{m-1}\left(\frac{\sin{\frac{x+\widetilde{x}_1}{2}}\sin{\frac{x-\widetilde{x}_1}{2}}}{\sin{\frac{x+x_1}{2}}\sin{\frac{x-x_1}{2}}}\right)^i.
\end{split}
\end{equation*}
`Οπως και στη μελέτη του $\frac{\varepsilon_1}{F_{n-1}}$, έτσι κι εδώ, θεωρούμε τα διαστήματα $(-x_1-s_n,-x_1+s_n)$ και $(x_1-s_n,x_1+s_n)$, τα οποία τώρα έχουν ως κέντρο τις ακριβείς ρίζες $-x_1$ και $x_1$, αντίστοιχα και ακτίνα $s_n$ που εξαρτάται από το $n$, αλλά $s_n=o(1)$. Υπολογίζοντας τη συνάρτηση $\frac{\varepsilon_g}{g}$ στο σημείο $x_1+s_n$, έχουμε:
\begin{equation*}
\frac{\varepsilon_g(x_1+s_n)}{g(x_1+s_n)}=\frac{\sin{\frac{x_1+\widetilde{x}_1}{2}}\sin{\frac{-\varepsilon_1}{2}}}{\sin{\left(x_1+\frac{s_n}{2}\right)}\sin{\frac{s_n}{2}}}\sum\limits_{i=0}^{m-1}\left(\frac{\sin{\frac{x_1+s_n+\widetilde{x}_1}{2}}\sin{\frac{s_n-\varepsilon_1}{2}}}{\sin{\left(x_1+\frac{s_n}{2}\right)}\sin{\frac{s_n}{2}}}\right)^i.
\end{equation*}
Η ακτίνα $s_n$ έχει επιλεχθεί έτσι ώστε $\varepsilon_1=o(s_n)$. Τότε, το τελευταίο άθροισμα είναι φραγμένο και θετικό μακριά από το 0. Επομένως, ο πρώτος λόγος είναι αυτός που χαρακτηρίζει το μέγεθος του $\frac{\varepsilon_g}{g}$, στο $x_1+s_n$.
\begin{equation*}
\frac{\varepsilon_g(x_1+s_n)}{g(x_1+s_n)}=C\frac{\sin{\frac{x_1+\widetilde{x}_1}{2}}\sin{\frac{-\varepsilon_1}{2}}}{\sin{\left(x_1+\frac{s_n}{2}\right)}\sin{\frac{s_n}{2}}}\sim\frac{\varepsilon_1}{s_n}=o(1).
\end{equation*}
Υπολογίζοντας το λόγο $\frac{\varepsilon_g}{g}$ στα σημεία $-x_1-s_n$, $-x_1+s_n$ και $x_1-s_n$, καταλήγουμε στο ίδιο αποτέλεσμα. `Οπως και στην προηγούμενη περίπτωση, ο λόγος $\frac{\varepsilon_g}{g}$ παραμένει τάξεως $o(1)$ για τα σημεία που ανήκουν στο $(-\pi,\pi]\backslash(-x_1-s_n,-x_1+s_n)\cup(x_1-s_n,x_1+s_n)$. Επομένως, έχουμε γενική συσσώρευση των ιδιοτιμών του $C_n\left(\frac{\varepsilon_g}{g}\right)$ σε μια περιοχή του $(0,0)$, με ακτίνα τάξεως $o(1)$ και το πολύ $\frac{4s_n}{2\pi}n=o(n)$ ιδιοτιμές εκτός της συσσώρευσης.

Στη συνέχεια θεωρούμε ότι η πολλαπλότητα $m$ της ρίζας, αντιστοιχεί στο φανταστικό μέρος της $f$. Τότε, ο λόγος $\frac{g_n}{g}$ δίνεται ως:
\begin{equation*}
\frac{g_n(x)}{g(x)}=\frac{(\cos{\widetilde{x}_1}-\cos{x})^m}{(\cos{x_1}-\cos{x})^m}\frac{\sin{x}-\mathrm{i}(\cos{\widetilde{x}_1}-\cos{x})^l}{\sin{x}-\mathrm{i}(\cos{x_1}-\cos{x})^l}.
\end{equation*}
Αφού $\varepsilon_1=o(1)$, το δεύτερο κλάσμα είναι κοντά στο 1, για κάθε $x\in(-\pi,\pi]$. Πιο συγκεκριμένα, $\frac{\sin{x}-\mathrm{i}(\cos{\widetilde{x}_1}-\cos{x})^l}{\sin{x}-\mathrm{i}(\cos{x_1}-\cos{x})^l}=1+o(1)$. `Αρα το μέγεθος του $\frac{g_n}{g}$ χαρακτηρίζεται από το πρώτο κλάσμα, το οποίο παραπάνω μελετήθηκε εκτενώς. Προφανώς, η ακολουθία πινάκων του όρου $C_n\left(\frac{\varepsilon_f\varepsilon_g}{F_{n-1}g}\right)$, που είναι το γινόμενο $C_n\left(\frac{\varepsilon_f}{F_{n-1}}\right)C_n\left(\frac{\varepsilon_g}{g}\right)$ έχει γενική συσσώρευση των ιδιοτιμών σε μια περιοχή του $(0,0)$, με ακτίνα $o(1)$ και $o(n)$ ιδιοτιμές εκτός της συσσώρευσης.

Συμπερασματικά, η ακολουθία πινάκων $\left\{ C_n\left(\frac{fg_n}{F_{n-1}g}\right)\right\}$ έχει γενική συσσώρευση των ιδιοτιμών, σε μια περιοχή του $(1,0)$ με ακτίνα $o(1)$ και $o(n)$ ιδιοτιμές οι οποίες κυμαίνονται εκτός της συσσώρευσης. Το μέγεθος της ακτίνας χαρακτηρίζεται από το μεγαλύτερο μέγεθος εκ των $C_n\left(\frac{\varepsilon_f}{F_{n-1}}\right)$ και $C_n\left(\frac{\varepsilon_g}{g}\right)$ και ο αριθμός ιδιοτιμών που χαρακτηρίζουν τη συσσώρευση ως γενική, από τον αντίστοιχο μεγαλύτερο αριθμό. Προφανώς, το Ερμιτιανό και αντι-Ερμιτιανό του μέρος έχουν το ίδιο είδος συσσώρευσης.

Μένει να μελετήσουμε τον όρο $T_n^{-1}(g_n)T_n(g)$. Είναι γνωστό ότι το φάσμα αυτού του πίνακα είναι ισοκατανεμημένο με αυτό του $T_n\left(\frac{g}{g_n}\right)$, που σημαίνει ότι μπορούμε να εξετάσουμε τον τελευταίο πίνακα, αντί του $T_n^{-1}(g_n)T_n(g)$. `Εχουμε:
\begin{equation*}
T_n\left(\frac{g}{g_n}\right)=T_n\left(\frac{g_n-\varepsilon_g}{g_n}\right)=I_n-T_n\left(\frac{\varepsilon_g}{g_n}\right).
\end{equation*}
Θα εξετάσουμε το φάσμα του $T_n\left(\frac{\varepsilon_g}{g_n}\right)$. Προκειμένου να αποδείξουμε τη συσσώρευση των ιδιοτιμών της ακολουθίας πινάκων $\left\{ T_n\left(\frac{\varepsilon_g}{g_n}\right)\right\}$, θεωρούμε τα διαστήματα $(-\widetilde{x}_1-s_n,-\widetilde{x}_1+s_n)$ και $(\widetilde{x}_1-s_n,\widetilde{x}_1+s_n)$, τα οποία έχουν το ίδιο μέγεθος $2s_n$, που εξαρτάται από το $\varepsilon_1$ ($\varepsilon_1=o(s_n)$, $s_n=o(1)$), όπως προηγουμένως. Η ίδια ανάλυση μας δίνει ότι $\frac{\varepsilon_g(x)}{g_n(x)}\leq r_n=o(1)$ για όλα τα $x$, εκτός των παραπάνω διαστημάτων.

Σταθεροποιούμε το $n$ και χρησιμοποιούμε το θεώρημα {\en Szeg{\"o}} με:
\begin{equation}
F(z)=\begin{cases}
1,~&\vert z\vert\geq r_n+h\\
0,~&\vert z\vert\leq r_n
\end{cases},
\end{equation}
για αρκούντως μικρό $h$. Τότε, ο αριθμός των ιδιοτιμών εκτός της περιοχής με κέντρο το $(1,0)$ και ακτίνα $r_n=o(1)$, δίνεται ως:
\begin{equation*}
\lim\limits_{N\rightarrow\infty}\frac{1}{N}\#\lbrace\lambda_i: \vert\lambda_i\vert>r_n\rbrace=\frac{1}{2\pi}\int\limits_{-\pi}^\pi F\left(\frac{\varepsilon_g(x)}{g_n(x)}\right)\mathrm{d}x\leq\frac{1}{2\pi}\int\limits_{x\in I_p} 1\mathrm{d}x=\frac{4s_n}{2\pi}=\frac{2s_n}{\pi},
\end{equation*} 
όπου $I_p=(-\widetilde{x}_1-s_n,-\widetilde{x}_1+s_n)\cup(\widetilde{x}_1-s_n,\widetilde{x}_1+s_n)$. Επομένως, $\#\lbrace\lambda_i: \vert\lambda_i\vert>r_n\rbrace\leq2\frac{s_n}{\pi}N$ για αρκούντως μεγάλη τιμή του $N$. Πηγαίνοντας πίσω στο $n$, λαμβάνουμε $\#\lbrace\lambda_i: \vert\lambda_i\vert>r_n\rbrace\leq2\frac{s_n}{\pi}n=o(n)$.

Με άλλα λόγια αποδείξαμε ότι ο $T_n\left(\frac{g}{g_n}\right)$ γράφεται ως $T_n\left(\frac{g}{g_n}\right)=I_n+S_n+R_n$, όπου $\Vert S_n\Vert_2=r_n=o(1)$ και $\Vert R_n\Vert_F^2\sim2\frac{s_n}{\pi}n=o(n)$. Εφόσον ο $T_n\left(\frac{g}{g_n}\right)$ προέκυψε από τον $T_n^{-1}(g_n)T_n(g)$ προσθέτοντας έναν πίνακα χαμηλής βαθμίδας, η οποία είναι σταθερή και ανεξάρτητη της διάστασης $n$, καταλήγουμε στο ότι $T_n^{-1}(g_n)T_n(g)=I_n+S_n+R_n^\prime$, όπου ο $R_n^\prime$ είναι επίσης πίνακας χαμηλής βαθμίδας. Επομένως, $\Vert R_n^\prime\Vert_F^2\sim2\frac{s_n}{\pi}n=o(n)$. Είναι προφανές ότι οι ιδιοτιμές του Ερμιτιανού και αντι-Ερμιτιανού μέρους του $T_n^{-1}(g_n)T_n(g)$ έχουν το ίδιο είδος γενικής συσσώρευσης.

Στη συνέχεια χρησιμοποιούμε το Λήμμα \ref{lem:2} για να αποδείξουμε τη γενική συσσώρευση. Επιλέγοντας $\mathcal{A}_n=C_n\left(\frac{fg_n}{F_{n-1}g}\right)$ και $\mathcal{B}_n=T_n^{-1}(g_n)T_n(g)$, έχουμε ότι το γινόμενο $C_n\left(\frac{fg_n}{F_{n-1}g}\right)T_n^{-1}(g_n)T_n(g)$ έχει γενική συσσώρευση των ιδιοτιμών, σε μια περιοχή του $(1,0)$ με ακτίνα $o(1)$, η οποία αντιστοιχεί στη μεγαλύτερη ακτίνα των δύο παραγόντων που το απαρτίζουν, καθώς επίσης και $o(n)$ ιδιοτιμές που κυμαίνονται εκτός της συσσώρευσης, οι οποίες επίσης εξαρτόνται από τους δύο παράγοντες του γινομένου.

Τέλος, χρησιμοποιούμε για άλλη μία φορά το ίδιο λήμμα, επιλέγοντας $\mathcal{A}_n=C_n\left(\frac{fg_n}{F_{n-1}g}\right)T_n^{-1}(g_n)T_n(g)$ και $\mathcal{B}_n=T_n\left(\frac{f}{g}\right)C_n^{-1}\left(\frac{f}{g}\right)$. `Ετσι λαμβάνουμε τη γενική συσσώρευση των ιδιοτιμών του προρρυθμισμένου πίνακα σε μια περιοχή του $(1,0)$, με ακτίνα ίση με αυτή του πρώτου όρου και $o(n)$ ιδιοτιμές εκτός της συσσώρευσης.

Αποδείξαμε ότι η ακολουθία πινάκων (\ref{eq:seq_matr}) χωρίζεται ως $I_n+S_n+R_n$, όπου $\Vert S_n\Vert_2=o(1)$ είναι η μεγαλύτερη από τις ακτίνες των τριών όρων του γινομένου και η νόρμα $\Vert R_n\Vert_F^2=o(n)$ αντιστοιχεί στον μεγαλύτερο αριθμό των ιδιοτιμών που κυμαίνονται εκτός της συσσώρευσης, βάσει των όρων του γινομένου. Προφανώς, το Ερμιτιανό και αντι-Ερμιτιανό του μέρος έχουν το ίδιο είδος συσσώρευσης των ιδιοτιμών.

Συνοψίζοντας, αποδείξαμε ότι ο πίνακας $\widetilde{P}_n$, ή ο όμοιος αυτού $P_n^\prime$ χωρίζεται ως $P_n^\prime=I_n+S_n+R_n$, όπου $\Vert S_n\Vert_2=o(1)$ είναι η τελική ακτίνα της περιοχής συσσώρευσης και $\Vert R_n\Vert_F^2=o(n)$, ο τελικός αριθμός των ιδιοτιμών που κυμαίνονται εκτός της συσσώρευσης. Επειδή ο $P_n$ διαφέρει από τον $P_n^\prime$ κατά έναν πίνακα χαμηλής και σταθερής βαθμίδας (βλ. \ref{eq:rank_cor}), ο προρρυθμισμένος πίνακας γράφεται ως $P_n=I_n+S_n+R_n^\prime$, όπου $\Vert R_n^\prime\Vert_F^2=o(n)$, το οποίο αποδεικνύει τη γενική συσσώρευση των ιδιοτιμών.

Συνεχίζουμε με την απόδειξη της περίπτωσης \textit{2}. Αυτή γίνεται ακολουθώντας τα ίδια βήματα με την απόδειξη της περίπτωσης \textit{1}, αλλά εντοπίζονται δύο διαφορές. Η πρώτη έχει να κάνει με τον τρίτο όρο του γινομένου της (\ref{eq:seq_matr}), όπου ο αριθμός των ιδιοτιμών που κυμαίνονται εκτός της συσσώρευσης είναι $\mathcal{O}(\log{n})$, λόγω της ασυνέχειας της $f$, όπως αποδείχθηκε στο Θεώρημα \ref{thm:well_cond_dis}. Η δεύτερη και πιο ουσιαστική διαφορά έχει να κάνει με τη μελέτη του πρώτου όρου του γινομένου (\ref{eq:seq_matr}), δηλαδή του $C_n\left(\frac{fg_n}{F_{n-1}g}\right)$ και ειδικότερα του πίνακα $C_n\left(\frac{\varepsilon_f}{F_{n-1}}\right)$. Επειδή εμφανίζεται το φαινόμενο {\en Gibbs}, αυτός έχει συσσώρευση των ιδιοτιμών σε μια περιοχή του $(0,0)$ με ακτίνα σχεδόν ίση με $0.179/(1-0.179)$, η οποία από τη μία είναι μικρή, αλλά από την άλλη είναι $\mathcal{O}(1)$ (σταθερή). Αυτή η ακτίνα, ακολουθώντας την παραπάνω απόδειξη, είναι η μεγαλύτερη σε μέγεθος κι έτσι καταλήγουμε στη γενική συσσώρευση των ιδιοτιμών του προρρυθμισμένου πίνακα, σε μια περιοχή του $(1,0)$ με ακτίνα το πολύ ίση με $0.179/(1-0.179)$ και $o(n)$ ιδιοτιμές να κυμαίνονται εκτός της συσσώρευσης, βάσει των τριών όρων του γινομένου.
\end{proof}

\begin{rem}
Στην περίπτωση όπου το πραγματικό και φανταστικό μέρος της γεννήτριας συνάρτησης δεν έχουν τα ίδια σημεία ασυνέχειας, η περιοχή συσσώρευσης των ιδιοτιμών, γύρω από το $(1,0)$, έχει ακτίνα αρκετά μικρότερη του $0.218$.
\end{rem}

Προχωράμε με τη μελέτη της συσσώρευσης των ιδιαζουσών τιμών του προρρυθμισμένου πίνακα.
\begin{thm}
`Εστω $T_n$ ένας πραγματικός πίνακας {\en Toeplitz}, στον οποίο αντιστοιχεί μια άγνωστη και μιγαδική γεννήτρια συνάρτηση $f$, με ρίζες στο διάστημα $(-\pi,\pi]$. `Εστω $F_{n-1}$ το ανάπτυγμα {\en Fourier} του $T_n$ και $g_n$ το τριγωνομετρικό πολυώνυμο, το οποίο αίρει τις εκτιμώμενες ρίζες μέσω της τεχνικής που περιγράφηκε. Υποθέτουμε ότι οι εκτιμώμενες μη-μηδενικές ρίζες $\widetilde{x}_i$, $i=1,2,\dots,\rho$ έχουν ένα σφάλμα $\varepsilon_i=\widetilde{x}_i-x_i$ και ότι οι πολλαπλότητες αυτών έχουν εκτιμηθεί με ακρίβεια. Τότε, ο προρρυθμισμένος πίνακας $C_n^{-1}\left(\frac{F_{n-1}}{g_n}\right)T_n^{-1}(g_n)T_n$ έχει:
\begin{enumerate}
\item Γενική συσσώρευση των ιδιαζουσών τιμών σε ένα διάστημα γύρω από το 1, με μήκος $o(1)$ και $o(n)$ ιδιάζουσες τιμές εκτός της συσσώρευσης, αν η $f$ είναι συνεχής.
\item Γενική συσσώρευση των ιδιαζουσών τιμών στο διάστημα $[0.6847,1.2374]$ και $o(n)$ ιδιάζουσες τιμές εκτός της συσσώρευσης, αν η $f$ είναι κατά τμήματα συνεχής.
\end{enumerate} 
\end{thm}

\begin{proof}
Όπως εύκολα παρατηρεί κανείς, οι υποθέσεις του θεωρήματος είναι ίδιες με αυτές του Θεωρήματος \ref{thm:ill_cond}. Επομένως, αποδεικνύεται ότι στην περίπτωση \textit{1} οι ιδιοτιμές του προρρυθμισμένου πίνακα $P_n$ έχουν γενική συσσώρευση σε μια περιοχή του $(1,0)$, με ακτίνα $r=o(1)$ και $r=0.218$ στην περίπτωση \textit{2}, καθώς επίσης και $s(n)=o(n)$ ιδιοτιμές εκτός της συσσώρευσης. Αποδείχθηκε ότι ο $P_n$ γράφεται ως $P_n=I_n+\mathcal{S}_n+\mathcal{R}_n$, όπου $\Vert\mathcal{S}_n\Vert_2\leq r$ και $\Vert\mathcal{R}_n\Vert_F^2=\mathcal{O}(s(n))$. Προφανώς, το συμμετρικό του μέρος γράφεται ως $\frac{P_n+P_n^T}{2}=I_n+S_n+R_n$, όπου $S_n=\frac{\mathcal{S}_n+\mathcal{S}_n^T}{2}$ και $R_n=\frac{\mathcal{R}_n+\mathcal{R}_n^T}{2}$ και το αντι-συμμετρικό του μέρος ως $\frac{P_n-P_n^T}{2}=S_n^\prime+R_n^\prime$, όπου $S_n^\prime=\frac{\mathcal{S}_n-\mathcal{S}_n^T}{2}$ και $R_n^\prime=\frac{\mathcal{R}_n-\mathcal{R}_n^T}{2}$. Ισχύει ότι $\Vert S_n\Vert_2\leq r$, $\Vert S_n^\prime\Vert_2\leq r$, $\Vert R_n\Vert_F^2=\mathcal{O}(s(n))$ και $\Vert R_n^\prime\Vert_F^2=\mathcal{O}(s(n))$.

Για να μελετήσουμε τη συσσώρευση των ιδιαζουσών τιμών, μελετάμε τη συσσώρευση των ιδιοτιμών της ακολουθίας πινάκων που αντιστοιχεί στο σύστημα των κανονικών εξισώσεων $\{Q_n\}=\{P_n^T\}\{P_n\}$. Επειδή η ακολουθία πινάκων $\lbrace Q_n\rbrace$ αποτελεί ένα γινόμενο ακολουθιών πινάκων, μπορούμε να χρησιμοποιήσουμε το Λήμμα \ref{lem:2} με:
\begin{equation*}
\begin{split}
&\mathcal{A}_n=P_n^T=I_n+S_n-S_n^\prime+R_n-R_n^\prime\text{ και}\\
&\mathcal{B}_n=P_n=I_n+S_n+S_n^\prime+R_n+R_n^\prime.
\end{split}
\end{equation*}
Τότε:
\begin{equation*}
\begin{split}
\mathcal{C}_n&\equiv Q_n=P_n^TP_n=(I_n+S_n-S_n^\prime+R_n-R_n^\prime)(I_n+S_n+S_n^\prime+R_n+R_n^\prime)\\
&=I_n+2S_n+S_n^2-{S_n^\prime}^2+(R_n-R_n^\prime)(I_n+S_n+S_n^\prime+R_n+R_n^\prime)\\
&\hspace{2pc}+(I_n+S_n-S_n^\prime)(R_n+R_n^\prime)\\
&=I_n+\widehat{\mathcal{S}}_n+\widehat{\mathcal{R}}_n,
\end{split}
\end{equation*}
όπου $\widehat{\mathcal{S}}_n=2S_n+S_n^2-{S_n^\prime}^2$ και $\widehat{\mathcal{R}}_n=(R_n-R_n^\prime)(I_n+S_n+S_n^\prime+R_n+R_n^\prime)+(I_n+S_n-S_n^\prime)(R_n+R_n^\prime)$. Είναι προφανές ότι $\Vert\widehat{\mathcal{S}}_n\Vert_2\leq2\Vert S_n\Vert_2+\Vert S_n^2\Vert_2+\Vert{S_n^\prime}^2\Vert_2\leq2\Vert S_n\Vert_2+\Vert S_n\Vert_2^2+\Vert S_n^\prime\Vert_2^2\leq2r+2r^2$, ενώ $\operatorname{rank}(\widehat{R}_n)\leq2\operatorname{rank}(R_n)+2\operatorname{rank}(R_n^\prime)=4\mathcal{O}(s(n))$.

Επομένως, η ακολουθία πινάκων των κανονικών εξισώσεων $\{Q_n\}$ έχει γενική συσσώρευση των ιδιοτιμών στο $\left(1-2r-2r^2,1+2r+2r^2\right)$ με $\mathcal{O}(s(n))$ ιδιοτιμές εκτός της περιοχής συσσώρευσης. Επειδή οι ιδιάζουσες τιμές του προρρυθμισμένου πίνακα είναι οι τετραγωνικές ρίζες των ιδιοτιμών του $Q_n$, λαμβάνουμε ότι, ο προρρυθμισμένος πίνακας έχει γενική συσσώρευση των ιδιαζουσών τιμών στο $\left(\sqrt{1-2r-2r^2},\sqrt{1+2r+2r^2}\right)$ με $\mathcal{O}(s(n))$ ιδιάζουσες τιμές εκτός του διαστήματος.

Στην περίπτωση \textit{1} από το Θεώρημα \ref{thm:ill_cond}, έχουμε ότι $r=o(1)$ και συνεπώς το παραπάνω διάστημα είναι σχεδόν το ίδιο με το $(1-r,1+r)$. Συνεπώς, το μήκος του είναι σχεδόν ίσο με $2r=o(1)$. 

Στην περίπτωση \textit{2} εμφανίζεται το φαινόμενο {\en Gibbs} και $r=0.218$. Επομένως, το παραπάνω διάστημα είναι ίσο με $[0.6847,1.2374]$.

Αυτή η παρατήρηση μας οδηγεί στο συμπέρασμα ότι η συσσώρευση των ιδιοτιμών συμβαδίζει με αυτή των ιδιαζουσών τιμών, βάσει της ακτίνας $r$ και των αριθμό των $\mathcal{O}(s(n))$ ιδιοτιμών που κυμαίνονται εκτός της συσσώρευσης.
\end{proof}

\begin{rem}
Πρακτικά, όταν η $f$ είναι κατά τμήματα συνεχής η περιοχή συσσώρευσης των ιδιοτιμών, καθώς επίσης και το διάστημα συσσώρευσης των ιδιαζουσών τιμών, είναι μικρότερη/ο όπως θα δούμε και θα εξηγήσουμε παρακάτω, στα αριθμητικά πειράματα (βλ. Παραδείγματα \ref{exp:ex3} και \ref{exp:ex4}).
\end{rem}


\subsection{Αριθμητικά αποτελέσματα}\label{sec4}

Εδώ θα δείξουμε την αποτελεσματικότητα της προτεινόμενης τεχνικής προρρύθμισης και την ισχύ των θεωρητικών αποτελεσμάτων που αποδείξαμε. Οι επιλογές υλοποίησης είναι ίδιες με αυτές των προηγούμενων κεφαλαίων, εκτός από τα δύο τελευταία παραδείγματα όπου μεταβάλαμε το κριτήριο τερματισμού των μεθόδων σε $\frac{\Vert r^{(k)}\Vert_2}{\Vert r^{(0)}\Vert_2}\leq 10^{-7}$, προκειμένου να λάβουμε μεγαλύτερη ακρίβεια. Στους πίνακες επαναλήψεων με $n$ δηλώνουμε προφανώς τη διάσταση του συστήματος, $\mathcal{C}_n(f)$ είναι ο κυκλοειδής προρρυθμιστής του προηγούμενου κεφαλαίου, όπου η γεννήτρια συνάρτηση είναι γνωστή εκ των προτέρων. Αναφέρουμε ότι αυτός προτάθηκε στην \cite{noutsos2021band}. $\mathcal{C}_n$ είναι ο προρρυθμιστής που κατασκευάστηκε με τον τρόπο που περιγράψαμε. Για τα συστήματα με κακή κατάσταση, όπου χρησιμοποιούμε ταινιωτού{ς}-επί-κυκλοειδείς προρρυθμιστές, χρησιμοποιούμε αντιστοίχως τον συμβολισμό $\mathcal{BC}_n(f)$ και $\mathcal{BC}_n$. Σε κάποια παραδείγματα γίνεται σύγκριση και με την τεχνική προρρύθμισης της προηγούμενης ενότητας, ενώ σε όλα εξ αυτών δίνονται οι χρόνοι {\en CPU}, από την κατασκευή του προρρυθμιστή έως ότου επιτευχθεί η σύγκλιση στη λύση του συστήματος. Ο κώδικας υλοποιήθηκε σε ένα σύστημα με τετραπύρηνο {\en Intel i5-3470} στα 3.2 {\en GHz} και 4 {\en GB DDR3 RAM} στα 1600 {\en MHz}.

\begin{exmp}\normalfont
Στο πρώτο παράδειγμα αυτής της ενότητας, φαίνεται η ισχύς του Θεωρήματος \ref{thm:well_cond}. Θεωρούμε τις συνεχείς συναρτήσεις, οι οποίες δεν έχουν ρίζες στο $(-\pi,\pi]$, $\mathfrak{f}_{11}(x)=x^2\sin^2{(x)}+1+\mathrm{i}\sin{(x)}x^2$ και $\mathfrak{f}_{12}(x)=x^2+1+\mathrm{i}\mathfrak{h}_1^\prime(x)$, όπου: $$\mathfrak{h}_1^\prime(x)=\begin{cases}
-\pi-x,~&x\in\big(-\pi,-\frac{\pi}{2}\big]\\
x,~&x\in\big(-\frac{\pi}{2},\frac{\pi}{2}\big]\\
\pi-x,~&x\in\big(\frac{\pi}{2},\pi\big]
\end{cases}.$$

Στους Πίνακες \ref{tab:ex1} και \ref{tab:ex1i} δίνουμε τους αριθμούς επαναλήψεων της {\en PGMRES} και τους χρόνους {\en CPU}, έως ότου επιτευχθεί η σύγκλιση στη λύση των συστημάτων που έχουν ως γεννήτρια συνάρτηση την $\mathfrak{f}_{11}$ και $\mathfrak{f}_{12}$, αντίστοιχα. `Οπως φαίνεται, δίνεται και ο αριθμός επαναλήψεων με χρήση του προρρυθμιστή $\mathcal{C}_n(f)$, όταν δηλαδή η γεννήτρια συνάρτηση είναι γνωστή εκ των προτέρων. Η σύγκλιση επιτυγχάνεται στις ίδιες επαναλήψεις και σε πολύ κοντινό χρόνο. Παρατηρούμε επίσης ότι για τις δύο μεγαλύτερες διαστάσεις, η ταχύτερη σύγκλιση επιτυγχάνεται με προρρύθμιση και όχι χωρίς καμία τεχνική προρρύθμισης. Είναι προφανές ότι οι πίνακες του συγκεκριμένου παραδείγματος έχουν καλή κατάσταση κι επομένως τα πλεονεκτήματα της προρρύθμισης δεν είναι ξεκάθαρα.

Η συσσώρευση των ιδιοτιμών για τα προρρυθμισμένα συστήματα δίνεται στο Σχήμα \ref{fig:ex1}. Το Σχήμα \ref{fig:ex1_fig2} δείχνει τη συσσώρευση των ιδιοτιμών, κοντά στο $(1,0)$, όταν $n=2048$, με χρήση του προτεινόμενου προρρυθμιστή (μπλε κύκλοι), καθώς επίσης και με χρήση του βέλτιστου κυκλοειδή προρρυθμιστή \cite{chan1988optimal} (πράσινα διαμάντια). Παρατηρούμε μια πολύ πιο συμπαγή συσσώρευση με χρήση του πρώτου, όπως άλλωστε συνέβαινε και για γνωστές εκ των προτέρων γεννήτριες συναρτήσεις.

\begin{table}[H]
\centering
\begin{tabular}{c|cc|cc|cc}
\toprule
$n$ & $I_n$ & {\en CPU} & $\mathcal{C}_n(\mathfrak{f}_{11})$ & {\en CPU} & $\mathcal{C}_n$ & {\en CPU} \\\midrule
1024 & 40 & 0.0929 & 6 & 0.0832 & 6 & 0.0959 \\
2048 & 39 & 0.1262 & 6 & 0.0842 & 6 & 0.0990 \\
4096 & 37 & 0.1708 & 6 & 0.0972 & 6 & 0.1080 \\
8192 & 36 & 0.2297 & 6 & 0.1259 & 6 & 0.1493 \\\bottomrule
\end{tabular}
\caption{{\en PGMRES:} Επαναλήψεις και χρόνοι {\en CPU} ($\mathfrak{f}_{11}$).}
\label{tab:ex1}
\end{table}

\begin{table}[H]
\centering
\begin{tabular}{c|cc|cc|cc}
\toprule
$n$ & $I_n$ & {\en CPU} & $\mathcal{C}_n(\mathfrak{f}_{12})$ & {\en CPU} & $\mathcal{C}_n$ & {\en CPU} \\\midrule
1024 & 29 & 0.0968 & 5 & 0.1002 & 5 & 0.1227 \\
2048 & 29 & 0.1116 & 4 & 0.1149 & 4 & 0.1270 \\
4096 & 28 & 0.1455 & 4 & 0.1254 & 4 & 0.1297 \\
8192 & 27 & 0.1865 & 4 & 0.1604 & 4 & 0.1651 \\\bottomrule
\end{tabular}
\caption{{\en PGMRES:} Επαναλήψεις και χρόνοι {\en CPU} ($\mathfrak{f}_{12}$).}
\label{tab:ex1i}
\end{table}

Στο Σχήμα \ref{fig:42b} παρατηρούμε επίσης ότι η συσσώρευση των ιδιοτιμών του $\mathcal{C}_n^{-1}T_n(\mathfrak{f}_{12})$ γύρω από το $(1,0)$, δεν είναι τόσο συμπαγής, όσο στο Σχήμα \ref{fig:42a}, που αφορά στον $\mathcal{C}_n^{-1}T_n(\mathfrak{f}_{11})$. Αυτό ήταν αναμενόμενο από το Θεώρημα \ref{thm:well_cond}, διότι η $\mathfrak{f}_{12}$ είναι μια συνεχής συνάρτηση, αλλά όχι συνεχώς παραγωγίσιμη όπως είναι η $\mathfrak{f}_{11}$ και η περιοχή συσσώρευσης είναι της τάξεως $\mathcal{O}\left(\frac{\log{n}}{n}\right)$, αντί $\mathcal{O}\left(\frac{1}{n}\right)$.

\begin{figure}[H]
    \centering
    \subfloat[Ιδιοτιμές.]{{\includegraphics[width=0.45\linewidth]{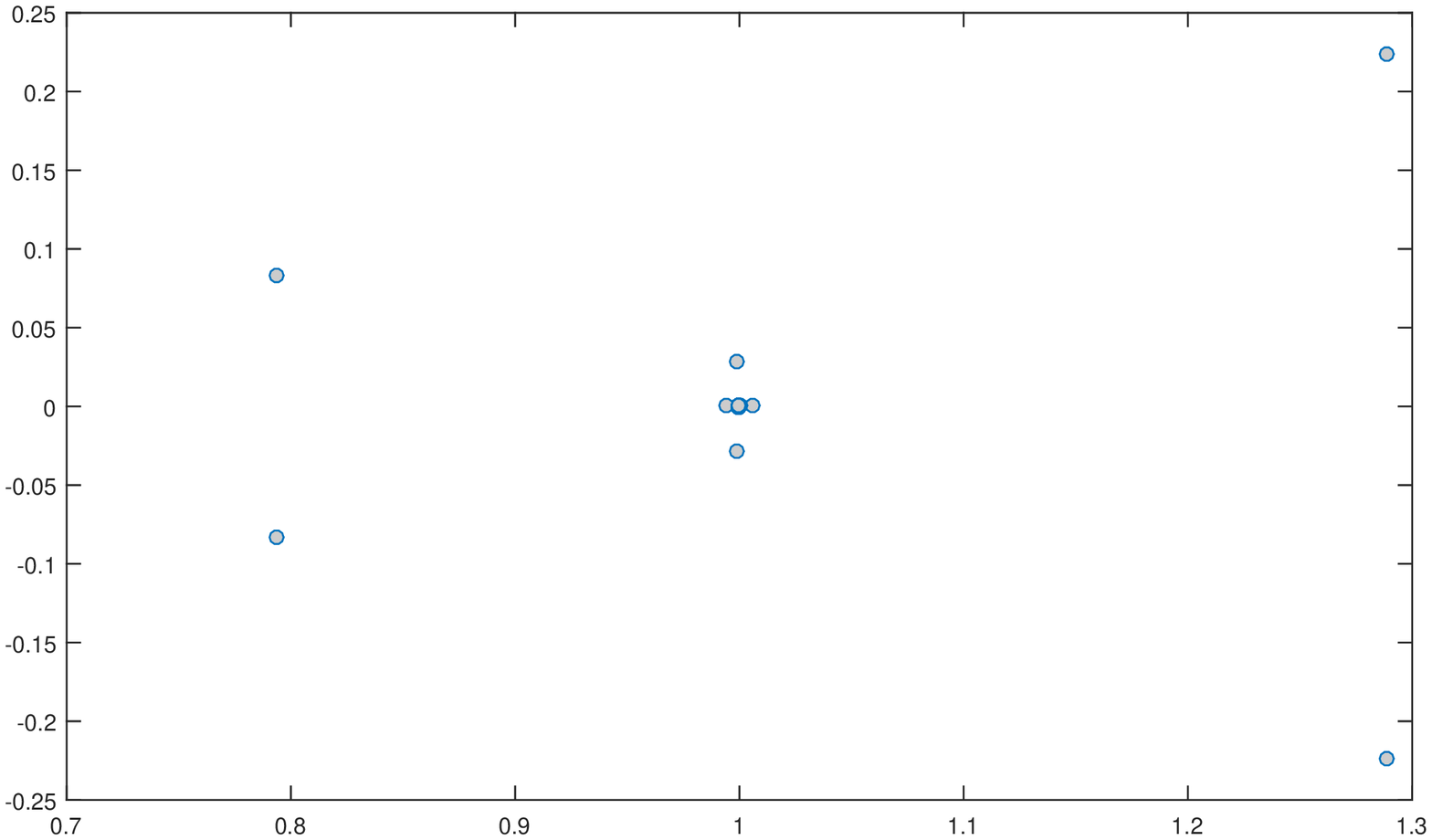}}}%
    \qquad
    \subfloat[Ιδιοτιμές.]{{\includegraphics[width=0.45\linewidth]{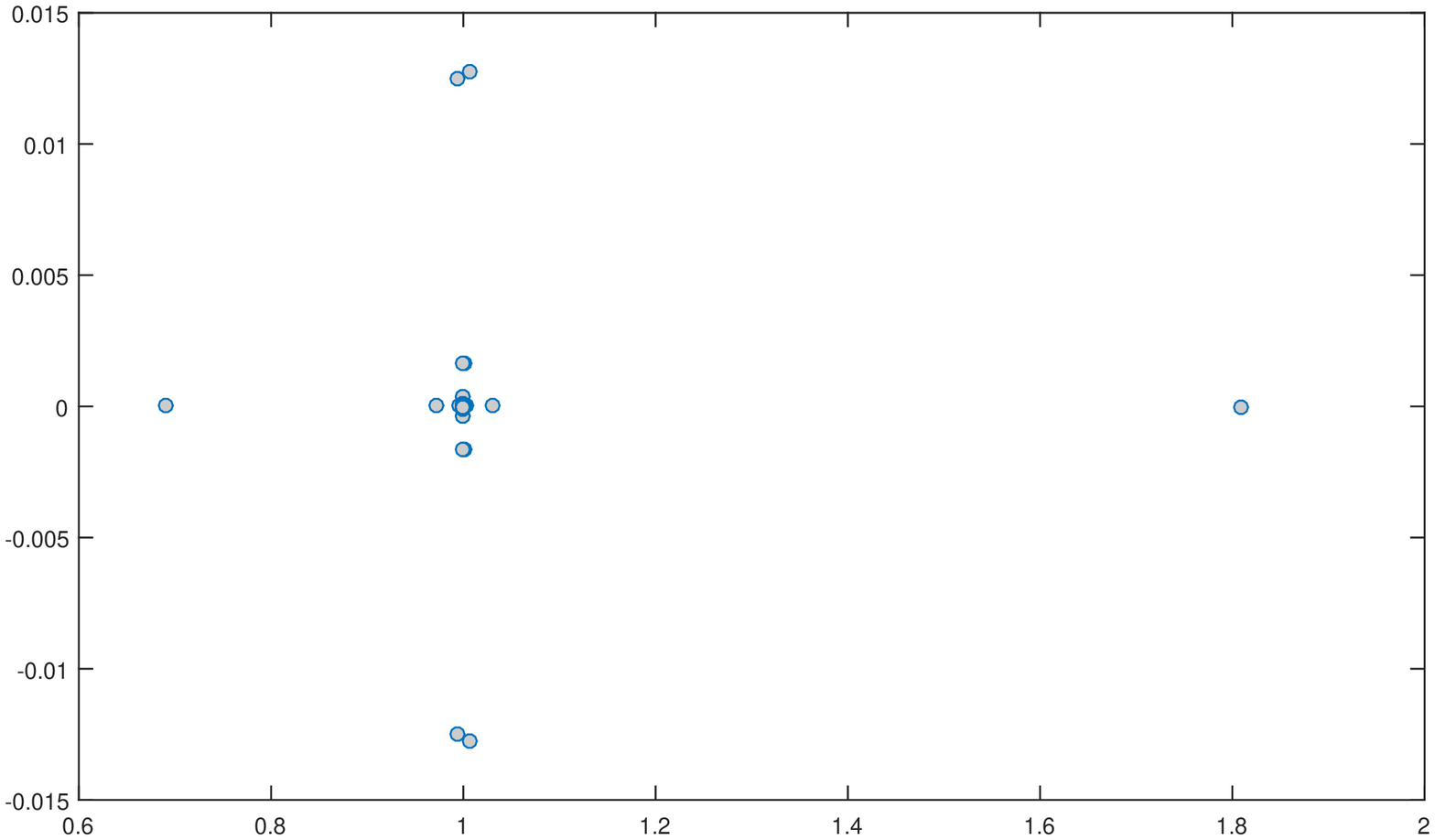}}}%
    \caption{Ιδιοτιμές (αριστερά $\mathfrak{f}_{11}$) και (δεξιά $\mathfrak{f}_{12}$).}
    \label{fig:ex1}
\end{figure}

\begin{figure}[H]
    \centering
    \subfloat[Ιδιοτιμές κοντά στο $(1,0)$.]{{\label{fig:42a}\includegraphics[width=0.45\linewidth]{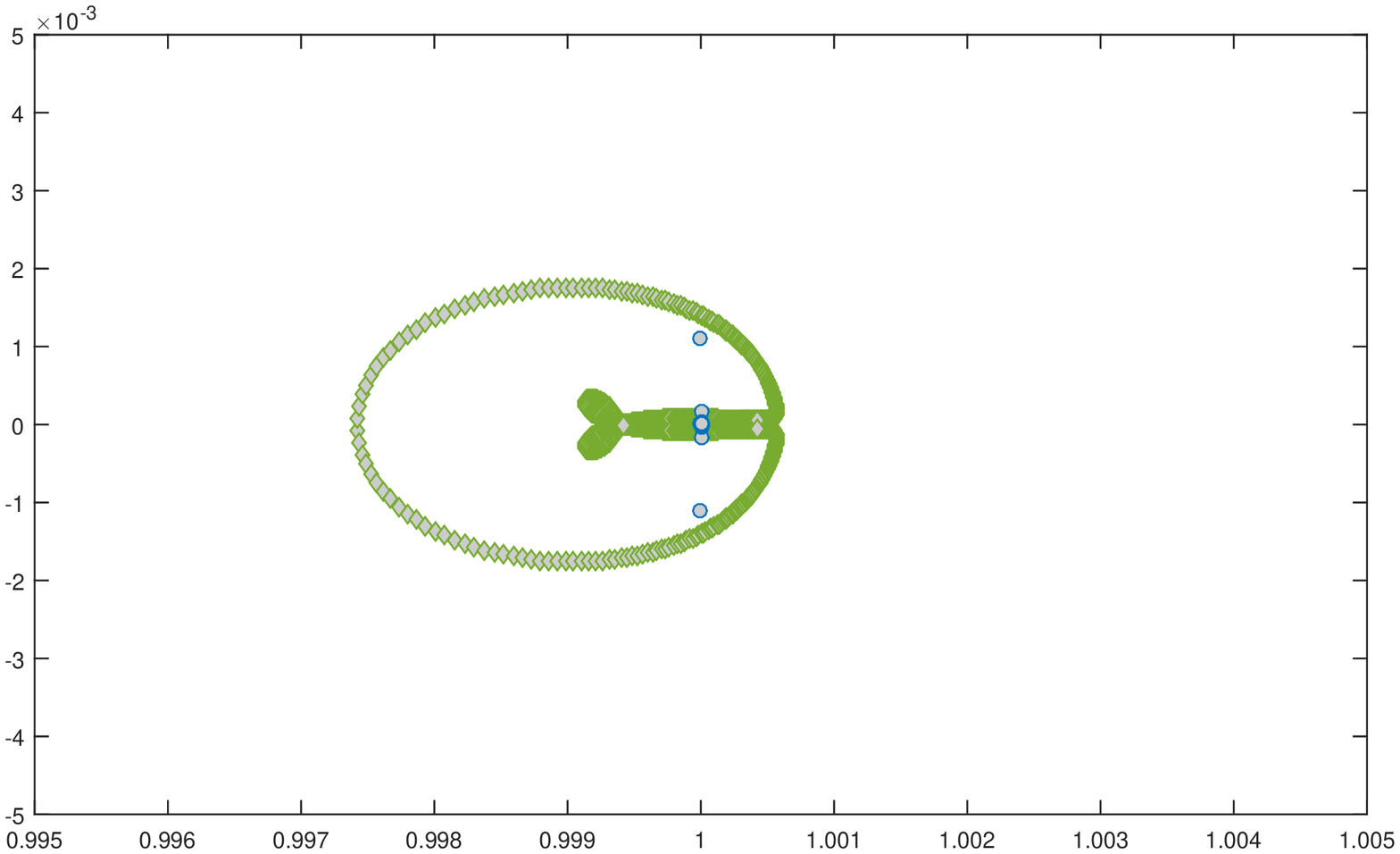}}}%
    \qquad
    \subfloat[Ιδιοτιμές κοντά στο $(1,0)$.]{{\label{fig:42b}\includegraphics[width=0.45\linewidth]{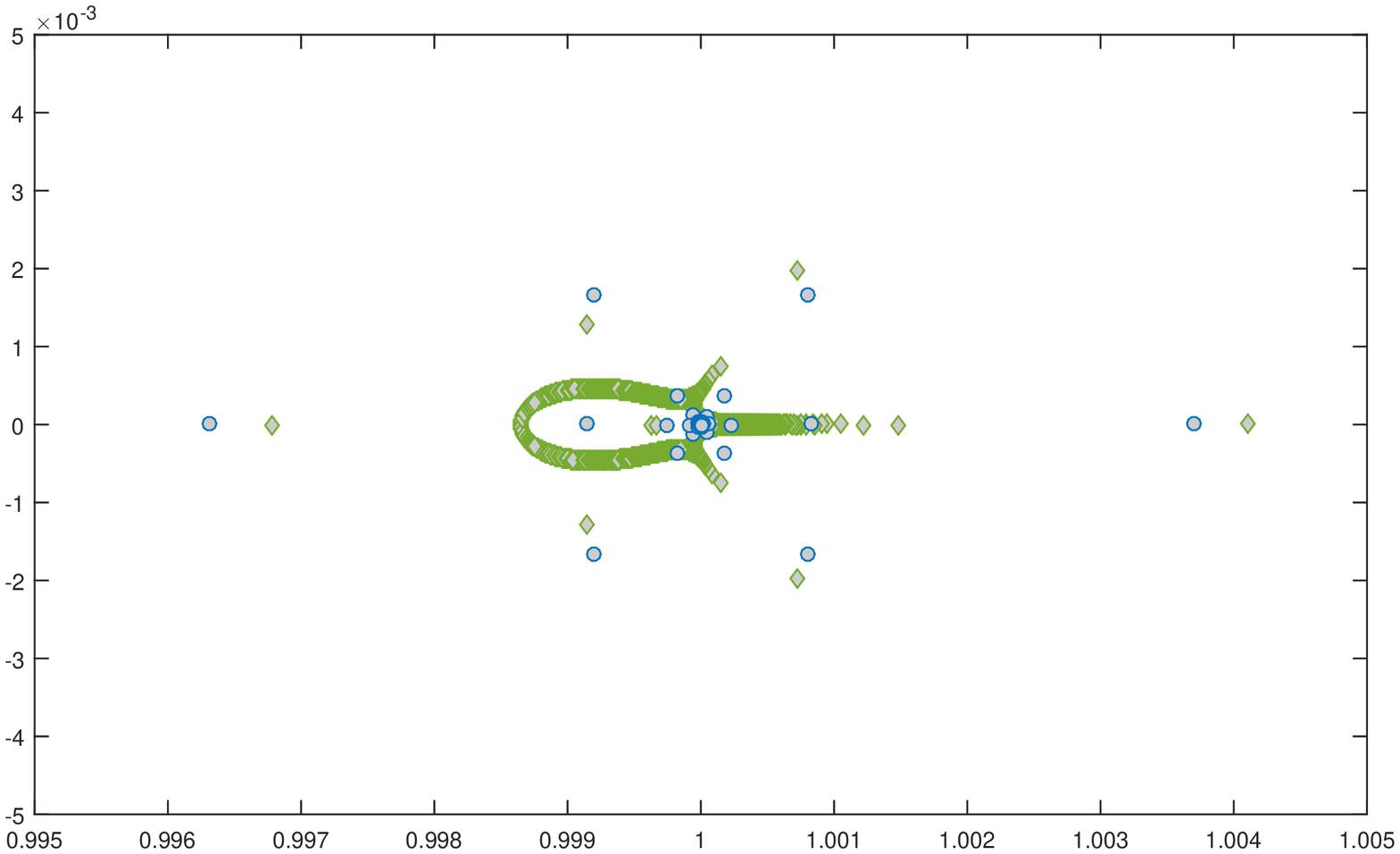}}}%
    \caption{Ιδιοτιμές (αριστερά $\mathfrak{f}_{11}$) και (δεξιά $\mathfrak{f}_{12}$).}
    \label{fig:ex1_fig2}
\end{figure}
\end{exmp}

\begin{exmp}\label{exp:ex2}\normalfont
Σε αυτό το παράδειγμα ασχολούμαστε με την προρρύθμιση του συστήματος {\en Toeplitz}, το οποίο έχει ως γεννήτρια συνάρτηση την $\mathfrak{f}_{13}(x)=x^2+1+\mathrm{i}x$, $x\in(-\pi,\pi]$. Αυτή δεν έχει ρίζες και το φανταστικό της μέρος παρουσιάζει ασυνέχεια στο $\pi$, που σημαίνει ότι αντιστοιχεί στο Θεώρημα \ref{thm:well_cond_dis}. 

\begin{table}[H]
\centering
\begin{tabular}{c|cc|cc|cc|cc}
\toprule
$n$ & $I_n$ & {\en CPU} & $\mathcal{C}_n(\mathfrak{f}_{13})$ & {\en CPU} & $\mathcal{C}_n$ & {\en CPU} & $R_{4,4}$ & {\en CPU} \\\midrule
1024 & 35 & 0.1023 & 6 & 0.0758 & 6 & 0.0855 & 7 & 0.2011 \\
2048 & 34 & 0.1183 & 6 & 0.0808 & 6 & 0.0917 & 7 & 0.2254 \\
4096 & 33 & 0.1716 & 5 & 0.0773 & 5 & 0.0816 & 7 & 0.2790 \\
8192 & 32 & 0.2142 & 5 & 0.0811 & 5 & 0.1067 & 6 & 0.3790 \\\bottomrule
\end{tabular}
\caption{{\en PGMRES:} Επαναλήψεις και χρόνοι {\en CPU} ($\mathfrak{f}_{13}$).}
\label{tab:ex2}
\end{table}

\begin{figure}[H]
    \centering
    \includegraphics[width=1\linewidth]{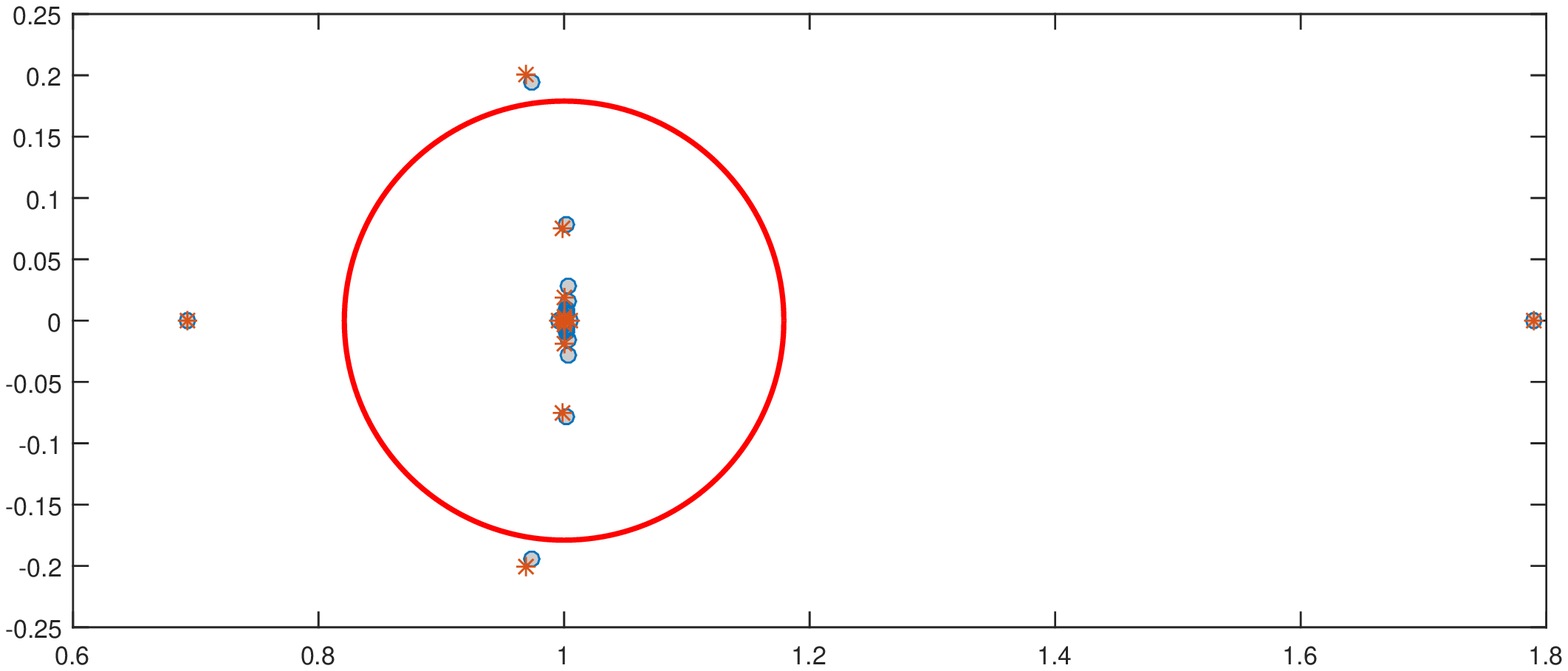}%
    \vspace{-4.5pc}\caption{Ιδιοτιμές ($\mathfrak{f}_{13}$).}
    \label{fig:ex2}
\end{figure}

Οι επαναλήψεις και ο χρόνος {\en CPU}, με χρήση των κυκλοειδών προρρυθμιστών όπως και στο προηγούμενο παράδειγμα, δίνεται στο Πίνακα \ref{tab:ex2}. Σε αυτόν δίνουμε και τον αριθμό επαναλήψεων με χρήση του ταινιωτού {\en Toeplitz} προρρυθμιστή $R_{4,4}$, της προηγούμενης ενότητας. Η υπεροχή του $\mathcal{C}_n$ επί του $R_{4,4}$ είναι προφανής. Δίνουμε επίσης τη συσσώρευση των ιδιοτιμών του $\mathcal{C}_n^{-1}T_n(\mathfrak{f}_{13})$ (μπλε κύκλοι) και του $\mathcal{C}_n^{-1}(\mathfrak{f}_{13})T_n(\mathfrak{f}_{13})$ (πορτοκαλί αστέρια) στο Σχήμα \ref{fig:ex2}. Ο κόκκινος κύκλος σε αυτό, δηλώνει την περιοχή συσσώρευσης για την ασυνεχή περίπτωση, με ακτίνα $0.179$. Παρατηρούμε ότι στην πράξη οι ιδιοτιμές έχουν μια καλύτερη συσσώρευση, απ'' ότι αναμένονταν, κοντά στο σημείο $(1,0)$. Μια εξήγηση γι αυτό το φαινόμενο θα δοθεί στο επόμενο παράδειγμα.

Το Σχήμα \ref{fig:44a} δείχνει τη συσσώρευση των ιδιοτιμών κοντά στο $(1,0)$. Σε αυτό φαίνεται η διαφορά, ανάμεσα στη συσσώρευση σε σημείο και στη συσσώρευση σε περιοχή της τάξεως $\mathcal{O}\left(\frac{\log{n}}{n}\right)$. Το Σχήμα \ref{fig:44b} δείχνει την αντίστοιχη συσσώρευση, με χρήση του προρρυθμιστή $R_{4,4}$. Η διαφορά ανάμεσα στις συσσωρεύσεις, αντιστοιχεί και στη διαφορά ανάμεσα στις επαναλήψεις. Ωστόσο, και σε αυτό το παράδειγμα μελετάμε έναν πίνακα με καλή κατάσταση και τα οφέλη της προρρύθμισης δε γίνονται ακόμη αισθητά. Η διαφορά ανάμεσα στις επαναλήψεις θα αυξηθεί στα επόμενα παραδείγματα, όπου μελετάμε την προρρύθμιση συστημάτων με κακή κατάσταση.
\begin{figure}
    \centering
    \subfloat[Ιδιοτιμές κοντά στο $(1,0)$.]{{\label{fig:44a}\includegraphics[width=0.45\linewidth]{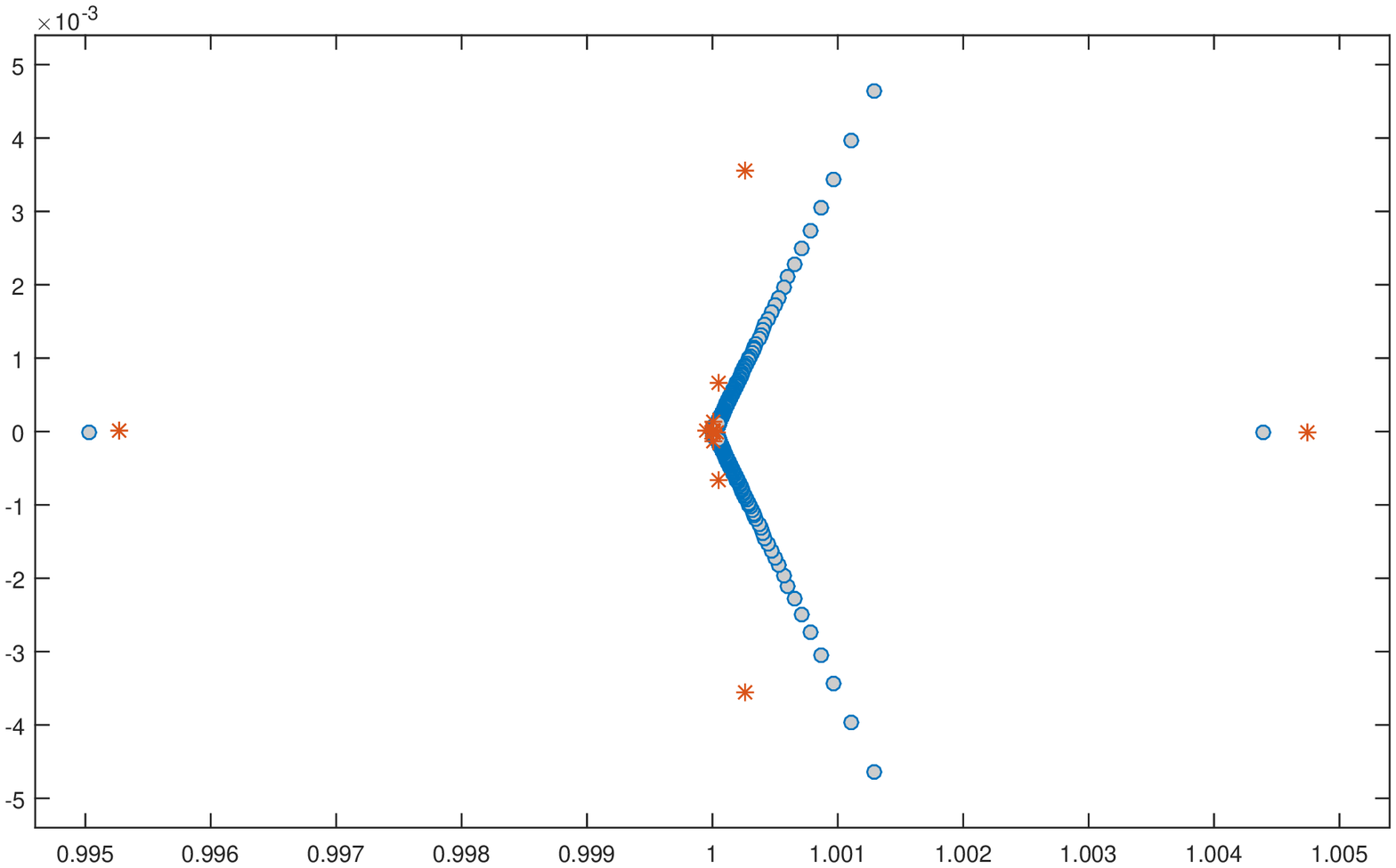}}}%
    \qquad
    \subfloat[Ιδιοτιμές.]{{\label{fig:44b}\includegraphics[width=0.45\linewidth]{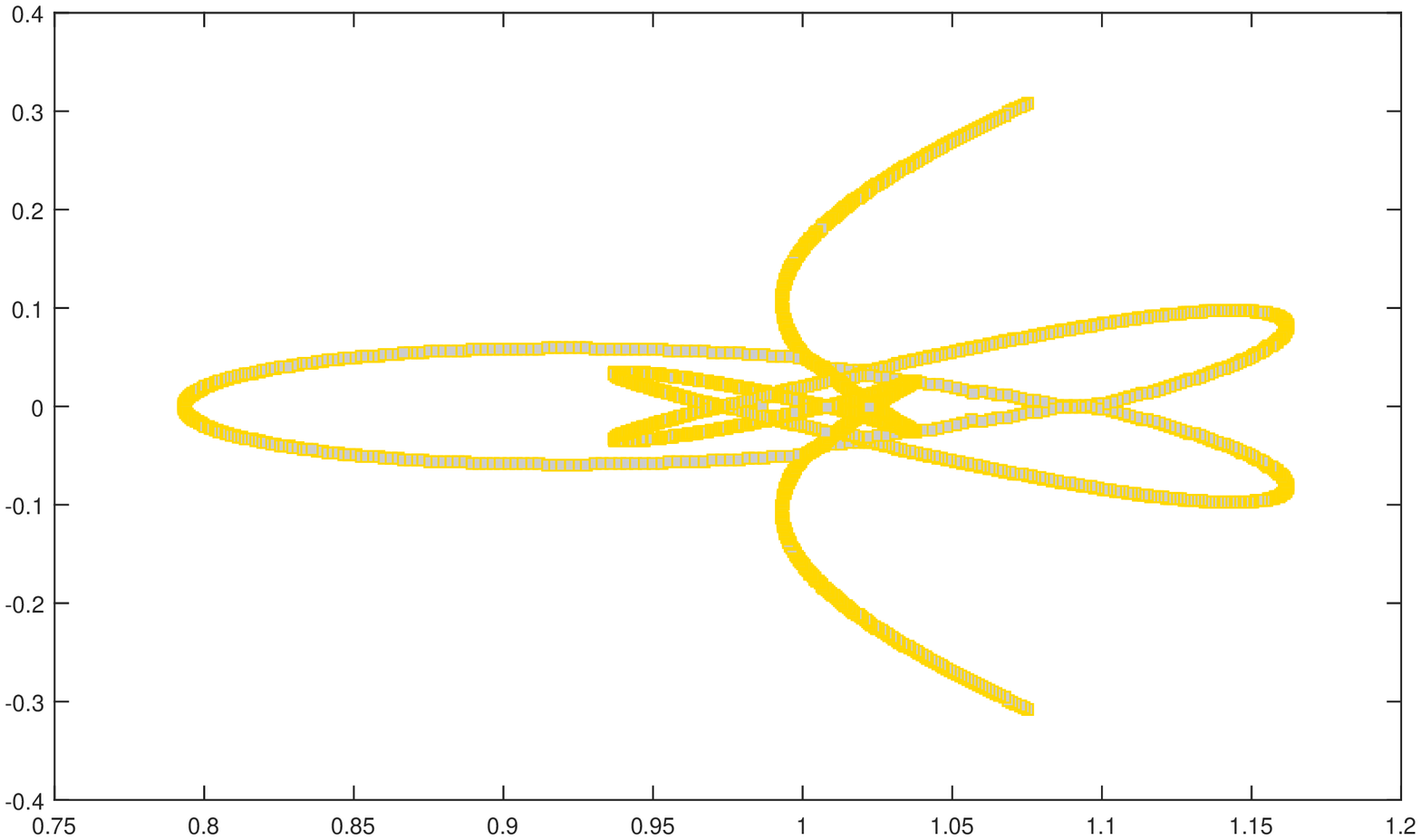}}}%
    \caption{Ιδιοτιμές ($\mathfrak{f}_{13}$).}
    \label{fig:ex2_fig2}
\end{figure}
\end{exmp}

\begin{exmp}\label{exp:ex3}\normalfont
Σε αυτό το παράδειγμα η γεννήτρια συνάρτηση του πίνακα {\en Toeplitz} είναι η $\mathfrak{f}_2(x)=x^2+\mathrm{i}x^3$, $x\in(-\pi,\pi]$, η οποία έχει ρίζα στο 0 και προσεγγίζεται με ακρίβεια χρησιμοποιώντας την προτεινόμενη τεχνική προρρύθμισης. Λεπτομέρειες για την προσέγγιση της πολλαπλότητας που έχει η ρίζα δίνονται στον Πίνακα \ref{tab:est_ex3}. Σημειώνουμε ότι αν και στο Παράδειγμα \ref{ex:1} μελετήσαμε το ίδιο σύστημα, η πολλαπλότητα των ριζών εκτιμήθηκε σωστά, αλλά ελαφρώς διαφορετικά, όπως φαίνεται και στον Πίνακα \ref{tab:1}. Υπενθυμίζουμε ότι το πλέγμα διαφέρει σε σχέση με την προηγούμενη ενότητα. Για την εκτίμηση των πολλαπλοτήτων, τρέξαμε 4 επαναλήψεις της μεθόδου Αντίστροφων Δυνάμεων, με $\Theta_{i,k}=\frac{1}{\sqrt{k}}\left(1,\mathrm{e}^{\mathrm{i}x_0},\mathrm{e}^{2\mathrm{i}x_0},\dots,\mathrm{e}^{(k-1)\mathrm{i}x_0}\right)^T$ και $x_0=0$ ως αρχικό διάνυσμα. Καταλήξαμε στο ότι η πολλαπλότητα της ρίζας του πραγματικού μέρους είναι ίση με 2 και αυτή του φανταστικού με 3. Αυτό συνεπάγεται ότι το τριγωνομετρικό πολυώνυμο που είναι κατάλληλο για την άρση των ριζών δίνεται ως $g(x)=2-2\cos{x}$, αφού η πολλαπλότητα της ρίζας του πραγματικού μέρους είναι μικρότερη από αυτή του φανταστικού.

\begin{table}
\centering
\begin{tabular}{c|cc|cc}
\toprule
$k$ & $\widetilde{\lambda}^1_{0,k}$ & $\log_2{(\widetilde{s_0}^1)}$ & $\widetilde{\lambda}^2_{0,k}$ & $\log_2{(\widetilde{s_0}^2)}$\\\midrule
16 & 0.0351 & 1.9224 & 0.0118 & 2.9337\\
32 & 0.0092 & & 0.0015 &\\
64 & 0.0024 & & 0.0002 &\\\bottomrule
\end{tabular}
\caption{Πολλαπλότητα των ριζών ($\mathfrak{f}_2$).}
\label{tab:est_ex3}
\end{table}

Όπως καταλαβαίνει κανείς, στους Πίνακες \ref{tab:ex3} και \ref{tab:ex3i}, δίνουμε τις επαναλήψεις και τους χρόνους {\en CPU}, με χρήση των {\en PGMRES} και {\en PCGN}. Εκεί, τα πλεονεκτήματα της προρρύθμισης γίνονται πλέον ξεκάθαρα. Σχολιάζουμε ότι χωρίς προρρύθμιση η λύση του $1024\times 1024$ συστήματος λαμβάνεται σε 547.8306 δευτερόλεπτα (δ), με τη μέθοδο {\en GMRES}. Σχολιάζουμε επίσης ότι ο χρόνος κατασκευής των προρρυθμιστών αυξάνεται με αργό ρυθμό, όσο η διάσταση του συστήματος μεγαλώνει, ενώ ο χρόνος εκτέλεσης της μεθόδου επίλυσης είναι της τάξεως $\mathcal{O}(n\log{n})$, όπως αναμένονταν από τη θεωρία. Για παράδειγμα, όταν $n=4096$ η κατασκευή του $\mathcal{BC}_n$ γίνεται σε 0.2003 δ και η λύση λαμβάνεται σε 0.1635 δ, ενώ οι αντίστοιχες τιμές όταν $n=8192$ είναι 0.2211 δ και 0.3533 δ. Η υπεροχή του $\mathcal{BC}_n$ επί του $R_{4,4}$ είναι προφανής, αφού η λύση λαμβάνεται σε λιγότερές από τις μισές επαναλήψεις.

\begin{table}[htbp]
\centering
\begin{tabular}{c|c|cc|cc|cc}
\toprule
$n$ & $I_n$ & $\mathcal{BC}_n(\mathfrak{f}_2)$ & {\en CPU} & $\mathcal{BC}_n$ & {\en CPU} & $R_{4,4}$ & {\en CPU}\\\midrule
1024 & $>$500 & 7 & 0.0942 & 11 & 0.2890 & 28 & 0.4834 \\
2048 & $>$500 & 7 & 0.1312 & 11 & 0.3036 & 28 & 0.6360 \\
4096 & $>$500 & 8 & 0.1660 & 12 & 0.3636 & 28 & 0.9496 \\
8192 & $>$500 & 8 & 0.2491 & 12 & 0.5744 & 27 & 1.4868 \\\bottomrule
\end{tabular}
\caption{{\en PGMRES:} Επαναλήψεις και χρόνοι {\en CPU} ($\mathfrak{f}_2$).}
\label{tab:ex3}
\end{table}

\begin{table}[htbp]
\centering
\begin{tabular}{c|c|cc|cc|cc}
\toprule
$n$ & $I_n$ & $\mathcal{BC}_n(\mathfrak{f}_2)$ & {\en CPU} & $\mathcal{BC}_n$ & {\en CPU} & $R_{4,4}$ & {\en CPU}\\\midrule
1024 & - & 15 & 0.0626 & 19 & 0.2658 & 45 & 0.5775 \\
2048 & - & 16 & 0.1022 & 22 & 0.2949 & 49 & 0.8101 \\
4096 & - & 21 & 0.1750 & 26 & 0.3737 & 54 & 1.4418 \\
8192 & - & 23 & 0.3139 & 29 & 0.5115 & 57 & 2.6980 \\\bottomrule
\end{tabular}
\caption{{\en PCGN:} Επαναλήψεις και χρόνοι {\en CPU} ($\mathfrak{f}_2$).}
\label{tab:ex3i}
\end{table}

Το Σχήμα \ref{fig:ex3} δείχνει τη συσσώρευση των ιδιοτιμών και ιδιαζουσών τιμών, όταν $n=2048$. Πιο συγκεκριμένα, το Σχήμα \ref{fig:5a} δείχνει τη συσσώρευση των ιδιοτιμών του $\mathcal{BC}_n^{-1}T_n(\mathfrak{f}_2)$ και το Σχήμα \ref{fig:5b}, αυτή των ιδιαζουσών τιμών για τον ίδιο πίνακα. Ομοίως, τα Σχήματα \ref{fig:5c} και \ref{fig:5d} δείχνουν τη συσσώρευση των ιδιοτιμών και ιδιαζουσών τιμών του $R_{4,4}^{-1}T_n(\mathfrak{f}_2)$. Οι κόκκινες γραμμές χαρακτηρίζουν τα άκρα του διαστήματος $[0.7602,1.1925]$. Η συσσώρευση των ιδιοτιμών στο Σχήμα \ref{fig:5a} μοιάζει να είναι γενική, σε μια περιοχή του $(1,0)$ με ακτίνα $0.179$ και $\mathcal{O}(\log{n})$ ιδιοτιμές εκτός του διαστήματος συσσώρευσης, όπως αποδείχθηκε στο Θεώρημα \ref{thm:ill_cond}, ενώ στο Σχήμα \ref{fig:5c} η συσσώρευση είναι σε ένα μεγαλύτερο ορθογώνιο, όπως αποδείχθηκε στο Θεώρησα \ref{thm:2} (βλ. επίσης Θεώρημα \ref{thm:eig_clustering}).

Θα θέλαμε να σχολιάσουμε τη συμπεριφορά των ιδιαζουσών τιμών μεταξύ των προαναφερθέντων πινάκων. Στο Σχήμα \ref{fig:5b} παρατηρούμε ότι, αν και η συσσώρευση είναι στο διάστημα $[0.7602,1.1925]$, μοιάζει να έχουμε συσσώρευση γύρω από το σημείο 1 και μόνο λίγες ιδιάζουσες τιμές να είναι μακριά από αυτό. Αυτό οφείλεται στο φαινόμενο {\en Gibbs}, το οποίο εμφανίζεται στις περιοχές ασυνέχειας. Από την άλλη, στο Σχήμα \ref{fig:5d}, παρατηρούμε ότι οι ιδιάζουσες τιμές δεν έχουν τόσο αυστηρή δομή κι εξαπλώνονται σε μια περιοχή του 1, αφού μέσω του αλγορίθμου {\en Remez} παρουσιάζεται διακύμανση της καμπύλης γύρω από το 1, σε όλο το διάστημα ορισμού της συνάρτησης.

\begin{figure}[htbp]%
    \centering
    \subfloat[Ιδιοτιμές.]{{\label{fig:5a}\includegraphics[width=0.45\linewidth]{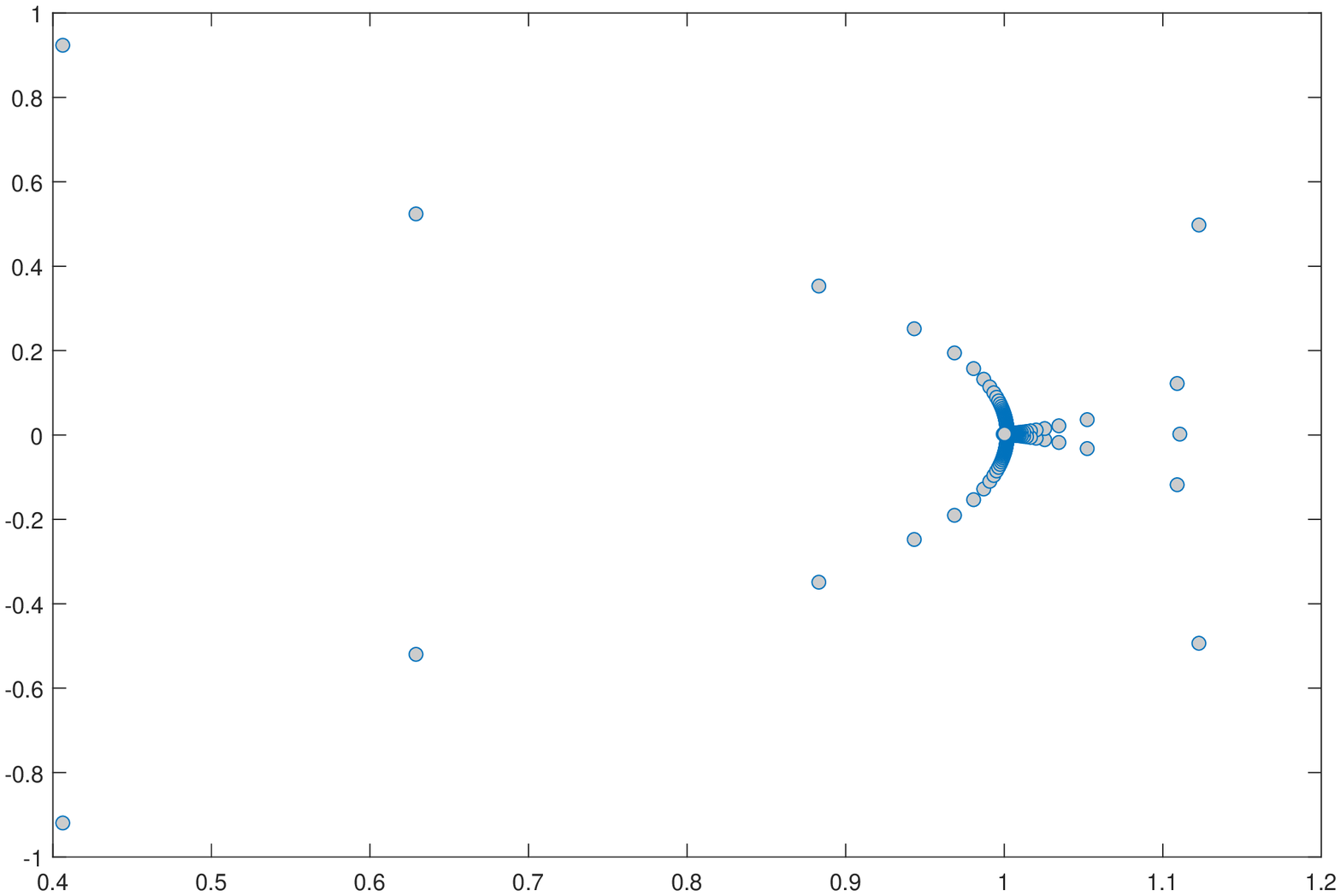}}}%
    \qquad
    \subfloat[Ιδιάζουσες τιμές.]{{\label{fig:5b}\includegraphics[width=0.45\linewidth]{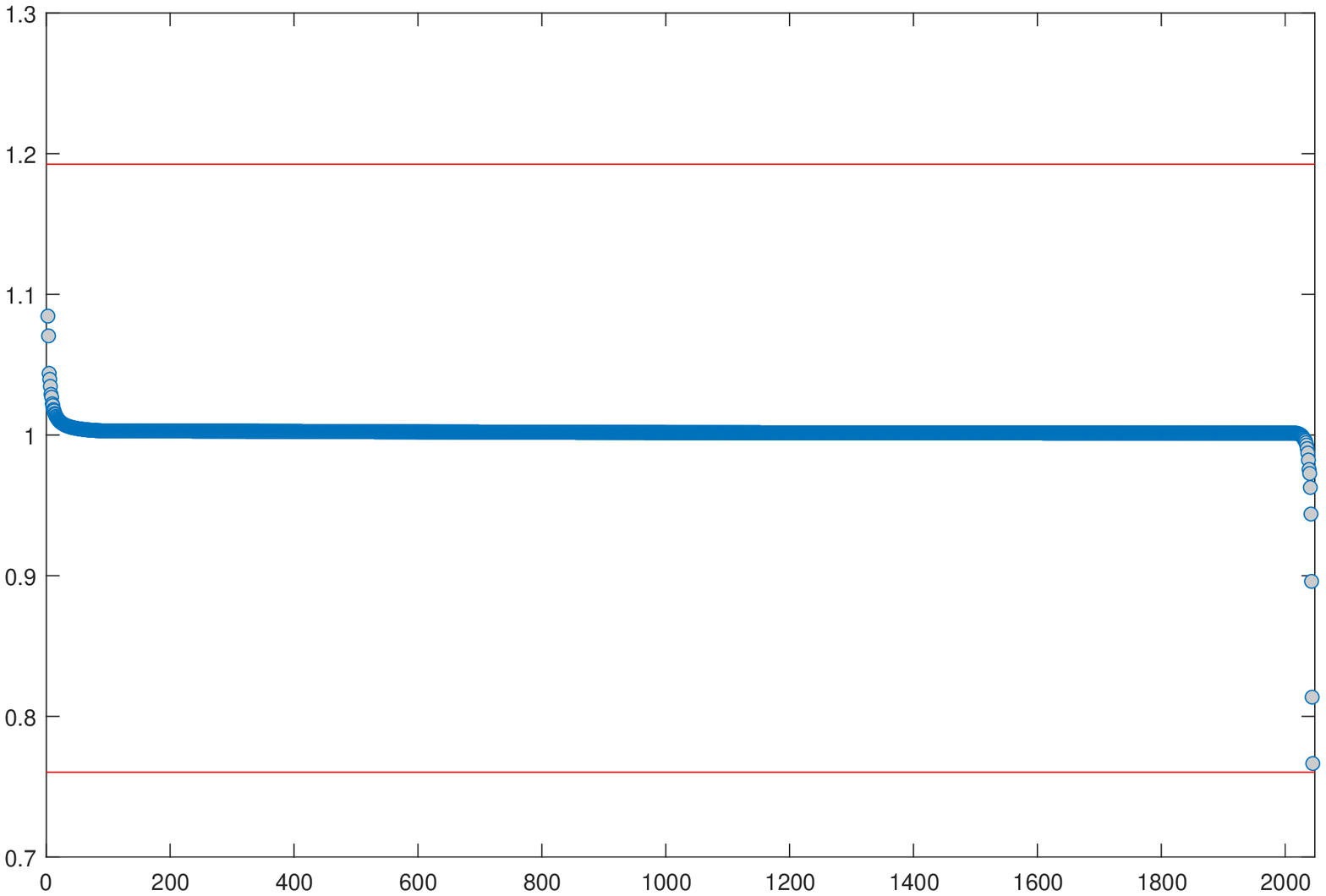}}}

    \subfloat[Ιδιοτιμές.]{{\label{fig:5c}\includegraphics[width=0.45\linewidth]{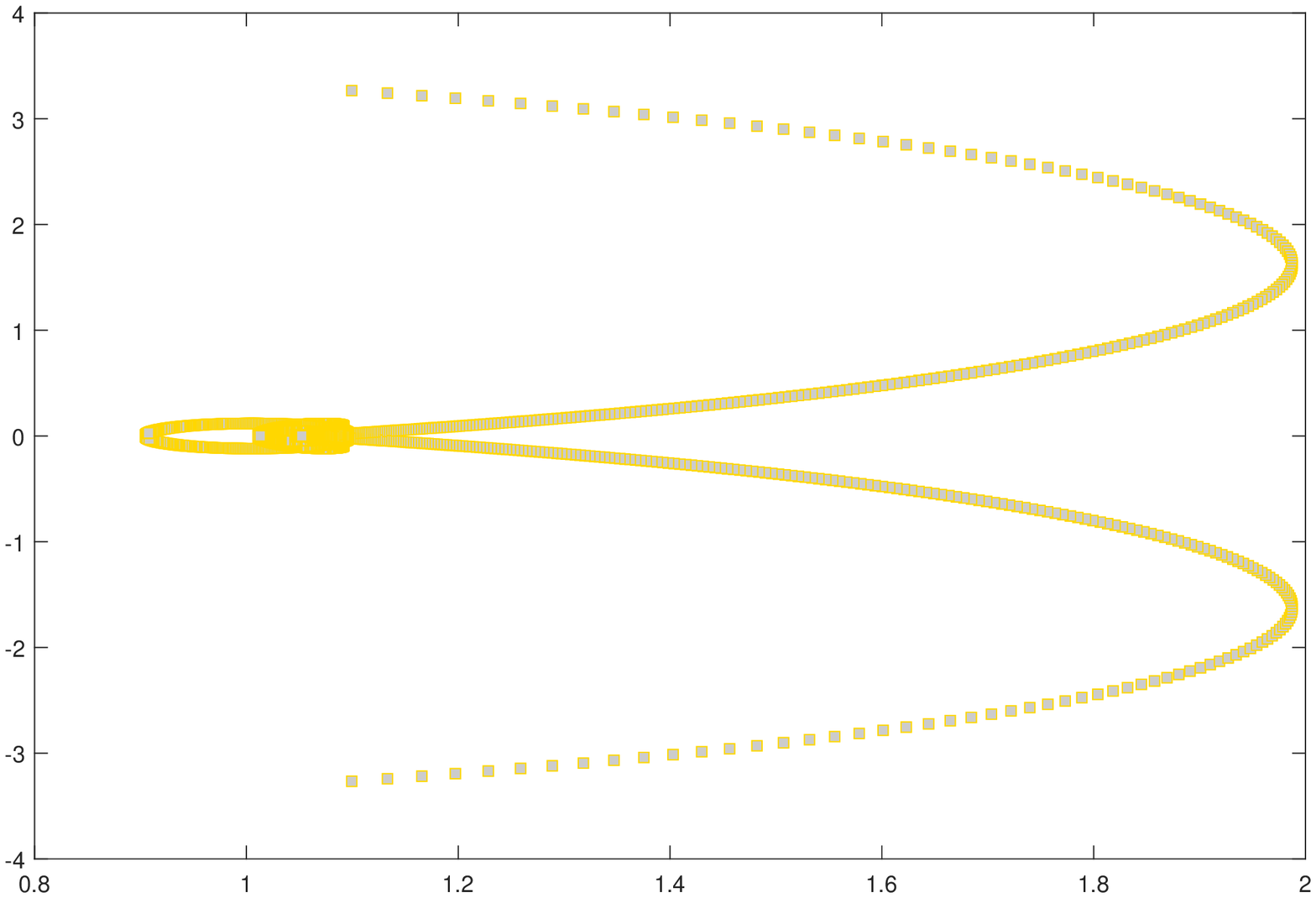}}}%
    \qquad
    \subfloat[Ιδιάζουσες τιμές.]{{\label{fig:5d}\includegraphics[width=0.45\linewidth]{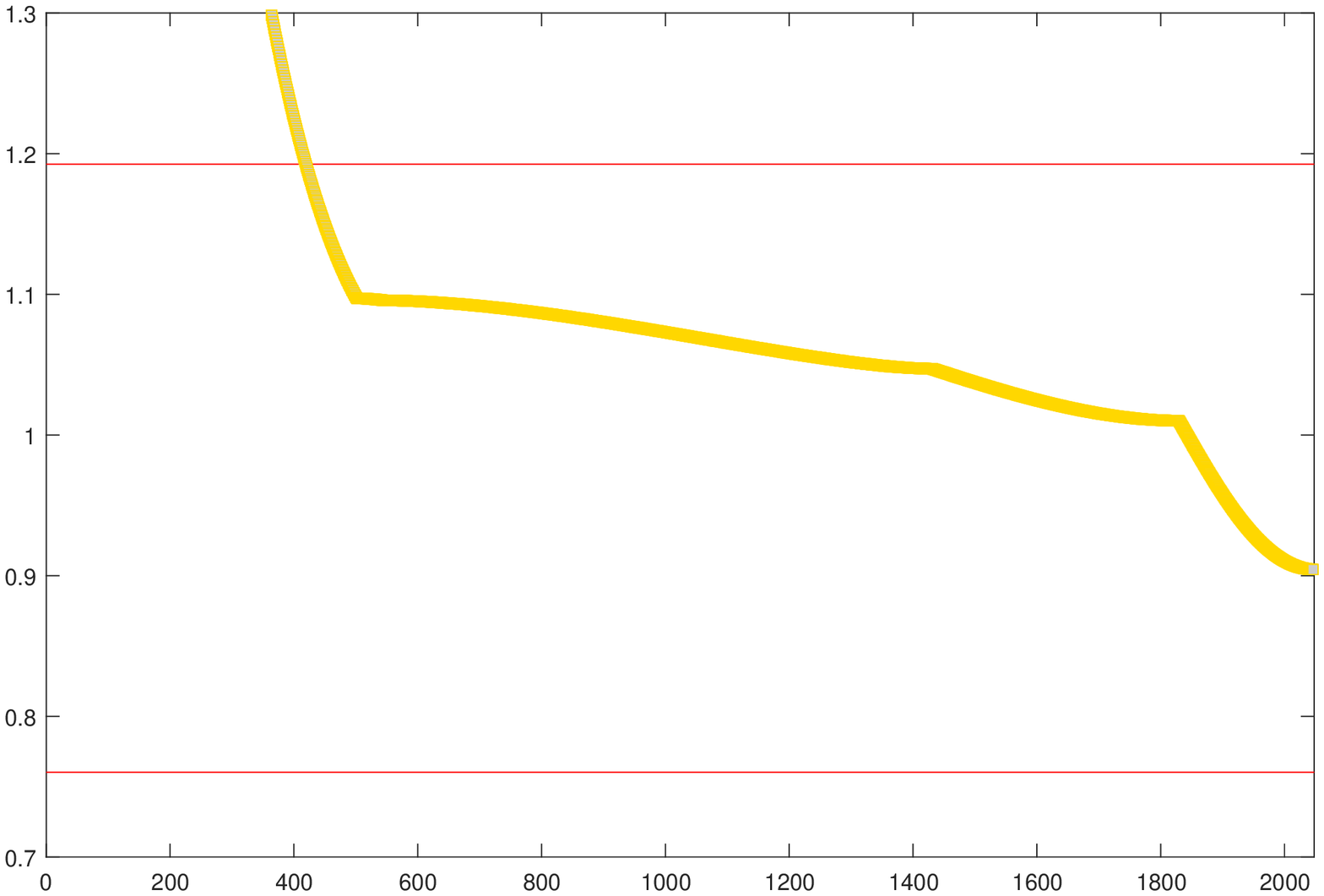}}}%
    \caption{Ιδιοτιμές και ιδιάζουσες τιμές ($\mathfrak{f}_2$).}
    \label{fig:ex3}
\end{figure}

\begin{figure}[htbp]%
    \centering
    \includegraphics[width=0.9\linewidth]{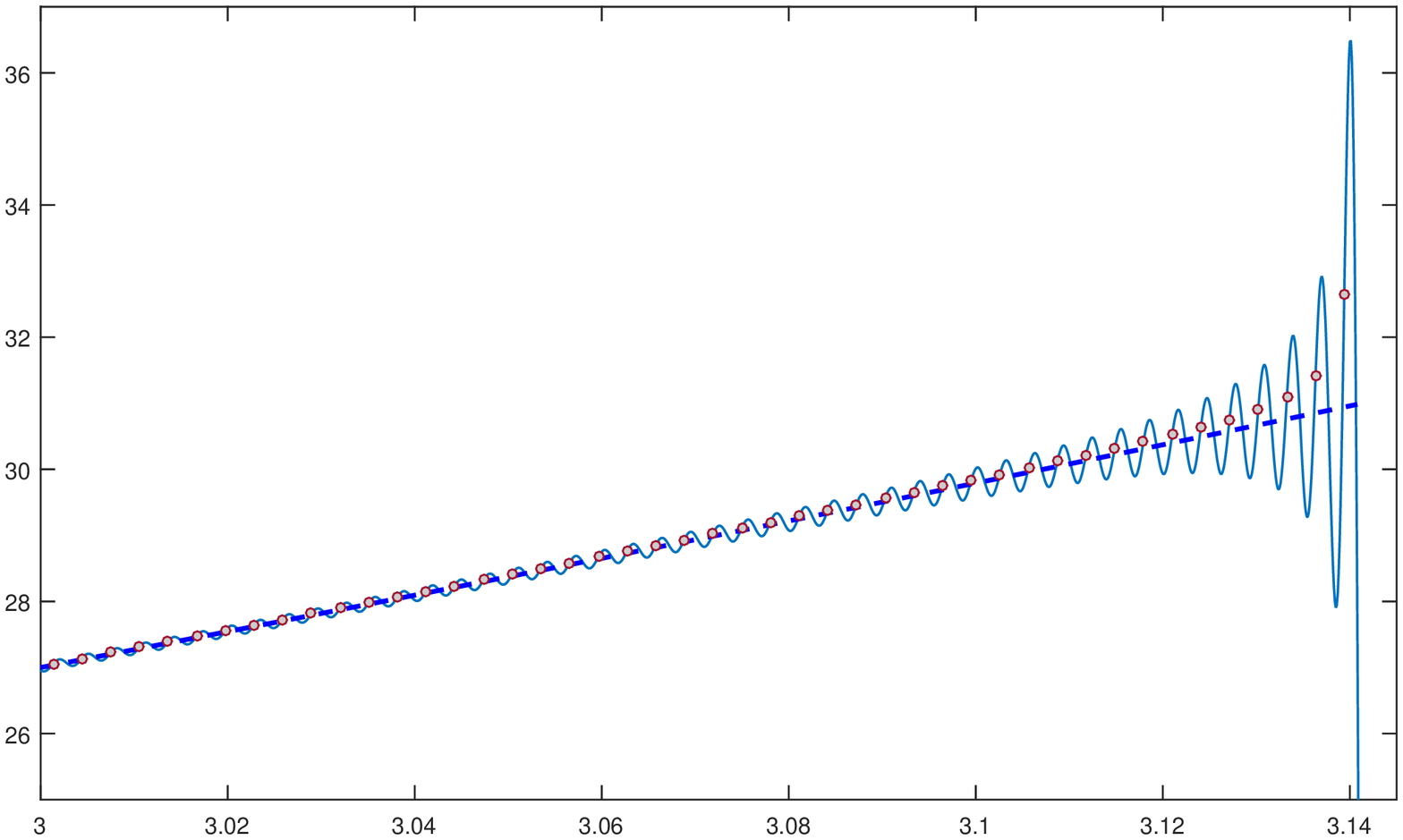}%
    \caption{Φαινόμενο {\en Gibbs} κοντά στο $\pi$ ($\operatorname{Im}(\mathfrak{f}_2)$).}
    \label{fig:ex3_fig2}
\end{figure}

Στο Σχήμα \ref{fig:ex3_fig2} δίνουμε το ανάπτυγμα {\en Fourier} $F_{n-1}^2$, του φανταστικού μέρους της $\mathfrak{f}_2$ (μπλε γραμμή) κοντά στο $\pi$, όπου εμφανίζεται το φαινόμενο {\en Gibbs}. Εκεί δίνουμε και τις τιμές του $F_{n-1}^2$ στα σημεία του πλέγματος $G_n$ (κόκκινοι κύκλοι). Η διακεκομμένη γραμμή είναι η γραφική παράσταση της $\operatorname{Im}(\mathfrak{f}_2)$. Παρατηρούμε ότι οι τιμές στο $G_n$ δε λαμβάνουν τις μέγιστες δυνατές τιμές της ταλάντωσης που σχετίζεται με το φαινόμενο {\en Gibbs}. Αυτός είναι ο λόγος που μπορεί να επιτύχουμε συσσώρευση σε μια μικρότερη περιοχή, όπως είδαμε στο Παράδειγμα \ref{exp:ex2}.
\end{exmp}

\begin{exmp}\label{exp:ex4}\normalfont
Σε αυτό το παράδειγμα, ο πίνακας του συστήματος {\en Toeplitz} έχει κακή κατάσταση, διότι η γεννήτρια συνάρτηση αυτού είναι η $\mathfrak{f}_9(x)=(x^2-1)^2+\mathrm{i}x(x^2-1)$, $x\in(-\pi,\pi]$, την οποία μελετήσαμε προγενέστερα στο Παράδειγμα \ref{exmp:414}. Βλέπουμε ότι η $\mathfrak{f}_9$ έχει ρίζες στο $\pm 1$ (1 και $2\pi-1$, στο διάστημα $[0,2\pi)$), που δεν είναι στοιχεία του $G_n$. Παρατηρούμε επίσης ότι το φανταστικό μέρος της γεννήτριας συνάρτησης έχει ασυνέχεια στο $\pi$, που σημαίνει ότι αναμένουμε μια πιο ελαστική συσσώρευση, λόγω του φαινομένου {\en Gibbs}.

Η ρίζα στο 1, εκτιμήθηκε στο πλέγμα $G_n$, χρησιμοποιώντας το ανάπτυγμα {\en Fourier}, ως $\widetilde{x}_1=1.0002$ (ακριβέστερα 1.000155), όταν $n=2048$. Περισσότερες λεπτομέρειες για την εκτίμηση της ρίζας δίνεται, για διάφορες διαστάσεις του συστήματος, στον Πίνακα \ref{tab:r_est_ex4}. Σημειώνουμε ότι σε αυτό το παράδειγμα οι εκτιμήσεις δόθηκαν στο ίδιο σημείο για διαφορετικές διαστάσεις του $G_n$. Οι πολλαπλότητες των ριζών, όπως φαίνεται στον Πίνακα \ref{tab:est_ex4}, εκτιμήθηκαν μέσω της διαδικασίας που περιγράψαμε ως $m_1^1=2$, $m_0^2=1$ και $m_1^2=1$. Επομένως, $g_n(x)=(\cos{\widetilde{x}_1}-\cos{x})^2+\mathrm{i}\sin{x}(\cos{\widetilde{x}_1}-\cos{x})$. 
 
\begin{table}[H]
\centering
\begin{tabular}{c|cc|cc}
\toprule
\multirow{2}{*}{$n$} & \multicolumn{2}{c|} {$\operatorname{Re}(F_{n-1})$} & \multicolumn{2}{c} {$\operatorname{Im}(F_{n-1})$} \\
& $\widetilde{x}_1$ & $F_{n-1}^1(\widetilde{x}_1)$ & $\widetilde{x}_1$ & $F_{n-1}^2(\widetilde{x}_1)$\\\midrule
1024 & $\theta_{164\phantom{1}}$ & $-3.378*10^{-5}$ & $\theta_{164\phantom{1}}$ & $-4.422*10^{-3}$\\
2048 & $\theta_{327\phantom{1}}$ & $-8.367*10^{-6}$ & $\theta_{327\phantom{1}}$ & $-2.055*10^{-3}$\\
4096 & $\theta_{653\phantom{1}}$ & $-2.019*10^{-6}$ & $\theta_{653\phantom{1}}$ & $-8.722*10^{-4}$\\
8192 & $\theta_{1305}$ & $-4.320*10^{-7}$ & $\theta_{1305}$ & $-2.806*10^{-4}$\\\bottomrule
\end{tabular}
\caption{Εκτίμηση της $\widetilde{x}_1$ στο $G_n$ ($\mathfrak{f}_9$).}
\label{tab:r_est_ex4}
\end{table}

Από την \cite{Serra_1999} γνωρίζουμε ότι το σφάλμα στην εκτίμηση της ρίζας είναι $\varepsilon_1=\mathcal{O}\left(\frac{\log{n}}{n}\right)$. Από τον Πίνακα \ref{tab:r_est_ex4}, έχουμε ότι $\varepsilon_1\sim\frac{1}{n}$, το οποίο έρχεται σε συμφωνία με τη θεωρία, αλλά είναι μια καλύτερη προσέγγιση από $\frac{\log{n}}{n}$. `Οπως αποδείχθηκε στο Θεώρημα \ref{thm:ill_cond}, ο τελευταίος παράγοντας $T_n\left(\frac{\mathfrak{f}_9}{g}\right)C_n^{-1}\left(\frac{\mathfrak{f}_9}{g}\right)$ έχει γενική συσσώρευση των ιδιοτιμών στο σημείο $(1,0)$, με $\mathcal{O}(\log{n})$ ιδιοτιμές εκτός του διαστήματος συσσώρευσης, εξαιτίας της ασυνέχειας της $\mathfrak{f}_9$. Για τον δεύτερο παράγοντα, $T_n^{-1}(g_n)T_n(g)$, θεωρούμε ως $s_n\sim\frac{\log{n}}{n}$, για να λάβουμε γενική συσσώρευση των ιδιοτιμών σε μια περιοχή του $(1,0)$, με ακτίνα $r_n=\mathcal{O}\left(\frac{1}{\log{n}}\right)$ και $\mathcal{O}(\log{n})$ ιδιοτιμές εκτός του διαστήματος συσσώρευσης. Λόγω του φαινομένου {\en Gibbs}, οι ιδιοτιμές του πρώτου παράγοντα, $C_n\left(\frac{\mathfrak{f}_9g_n}{F_{n-1}g}\right)$, έχουν κύρια συσσώρευση σε μια περιοχή του $(1,0)$, με ακτίνα το πολύ $0.179$. Επομένως, από το Θεώρημα $\ref{thm:ill_cond}$ καταλήγουμε σε γενική συσσώρευση των ιδιοτιμών του προρρυθμισμένου πίνακα, σε μια περιοχή του $(1,0)$, με ακτίνα το πολύ ίση με $0.179$ και $\mathcal{O}(\log{n})$ ιδιοτιμές εκτός του διαστήματος συσσώρευσης.

\begin{table}[H]
\centering
\begin{tabular}{c|cc|cccc}
\toprule
$k$ & $\widetilde{\lambda}^1_{1,k}$ & $\log_2{(\widetilde{s_0}^1)}$ & $\widetilde{\lambda}^2_{0,k}$ & $\log_2{(\widetilde{s_0}^2)}$ & $\widetilde{\lambda}^2_{1,k}$ & $\log_2{(\widetilde{s_1}^2)}$\\\midrule
16 & 0.1128 & 1.6781 & 0.1307 & 0.8825 & 0.1836 & 0.8956\\
32 & 0.0338 & & 0.0689 &  & 0.1083 &\\
64 & 0.0091 & & 0.0355 &  & 0.0678 &\\\bottomrule
\end{tabular}
\caption{Πολλαπλότητα των ριζών ($\mathfrak{f}_9$).}
\label{tab:est_ex4}
\end{table}

Στο Σχήμα \ref{fig:ex4} δίνουμε τη συσσώρευση των ιδιοτιμών. Στο Σχήμα \ref{fig:7a}, ο κόκκινος κύκλος είναι η περιοχή με ακτίνα $0.179$, που αναφέραμε παραπάνω. Θα θέλαμε να σημειώσουμε ότι για $n=2048$, υπάρχουν 8 επιπλέον ιδιοτιμές, εκτός του διαστήματος συσσώρευσης, που δε φαίνονται στο Σχήμα \ref{fig:7a}, το οποίο κάνει τις ιδιοτιμές που κυμαίνονται εκτός του διαστήματος συσσώρευσης να είναι 10. Στο Σχήμα \ref{fig:7b} δίνεται μια μεγέθυνση πολύ κοντά στο $(1,0)$. 

\begin{rem}
Στο Σχήμα \ref{fig:7a} η συσσώρευση παρατηρείται σε μια μικρότερη περιοχή και φαίνεται ότι το ίδιο φαινόμενο, που περιγράψαμε στο Σχήμα \ref{fig:ex3_fig2}, του Παραδείγματος \ref{exp:ex3}, εμφανίζεται και σε αυτό το παράδειγμα.
\end{rem}

\begin{figure}[htbp]%
    \centering
    \subfloat[Ιδιοτιμές.]{{\label{fig:7a}\includegraphics[width=0.45\linewidth]{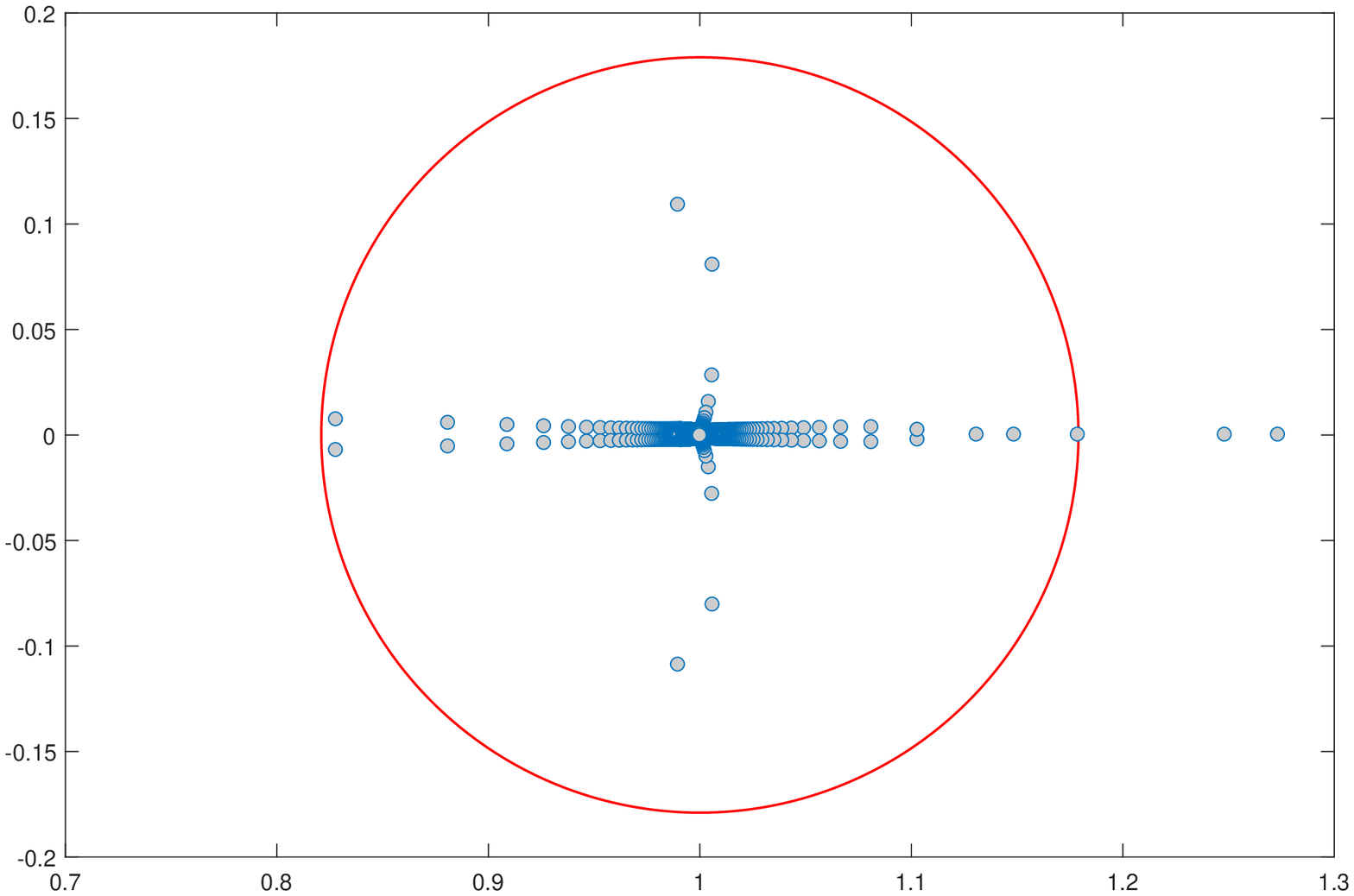}}}%
    \qquad
    \subfloat[Ιδιοτιμές κοντά στο $(1,0)$.]{{\label{fig:7b}\includegraphics[width=0.45\linewidth]{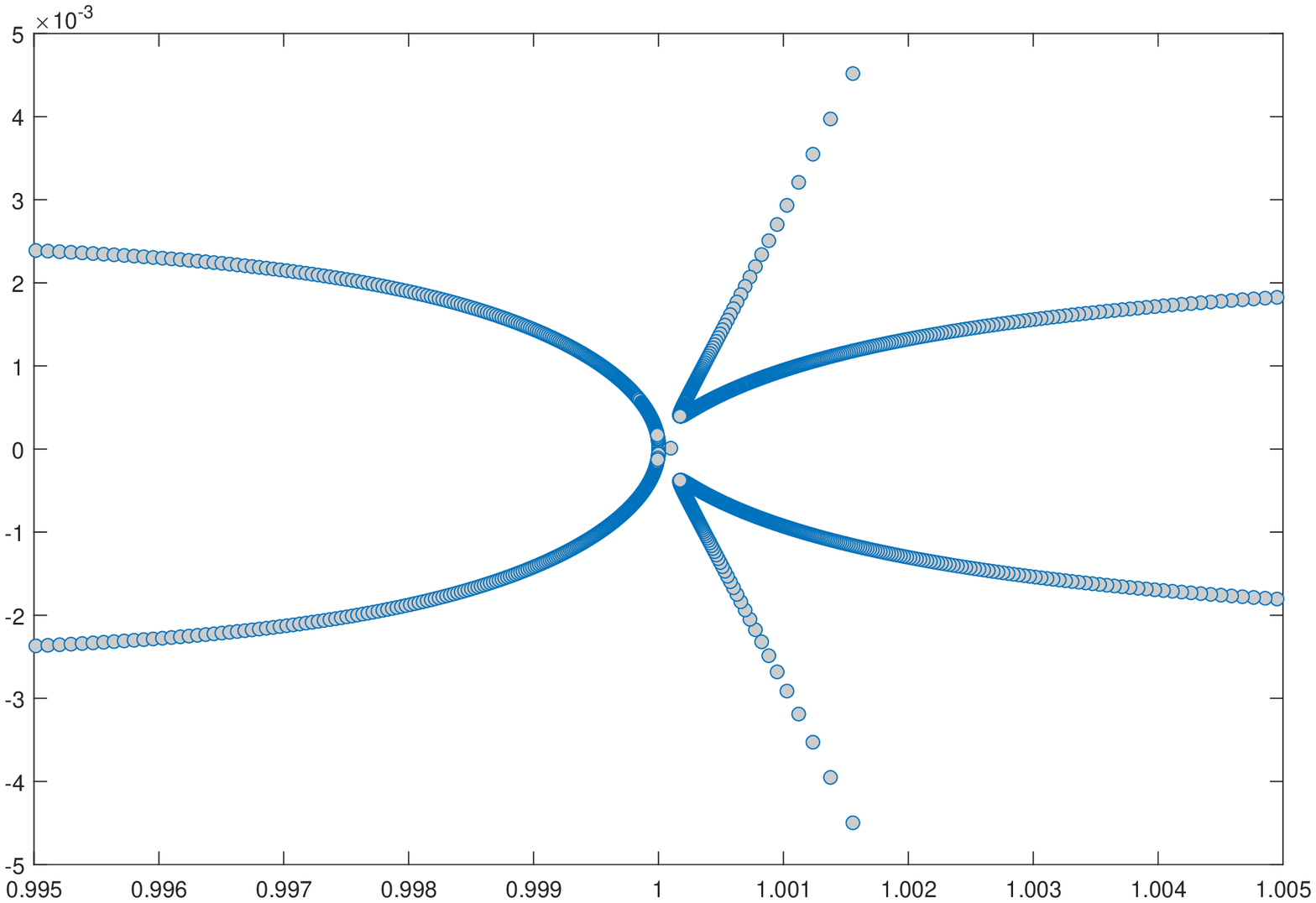}}}%
    \caption{Ιδιοτιμές ($\mathfrak{f}_9$).}
    \label{fig:ex4}
\end{figure}

Στους Πίνακες \ref{tab:ex4} και \ref{tab:ex4i} παρουσιάζουμε τις επαναλήψεις των μεθόδων {\en PGMRES} και {\en PCGN}, αντίστοιχα, καθώς επίσης και τους αντίστοιχους χρόνους. Υπενθυμίζουμε ότι το κριτήριο τερματισμού άλλαξε σε $\frac{\Vert r^{(k)}\Vert_2}{\Vert r^{(0)}\Vert_2}\leq 10^{-7}$.

\begin{table}[htbp]
\centering
\begin{tabular}{c|c|cc|cc|cc}
\toprule
$n$ & $I_n$ & $\mathcal{BC}_n(\mathfrak{f}_9)$ & {\en CPU} & $\mathcal{BC}_n$ & {\en CPU} & $R_{4,4}$ & {\en CPU}\\\midrule
1024 & $>$500 & 7 & 0.1227 & 11 & 0.3080 & 21 & 1.3286 \\
2048 & $>$500 & 6 & 0.1373 & 11 & 0.3252 & 22 & 2.4072 \\
4096 & $>$500 & 6 & 0.1780 & 12 & 0.3837 & 22 & 4.5923 \\
8192 & $>$500 & 6 & 0.2287 & 13 & 0.5177 & 22 & 9.6205 \\\bottomrule
\end{tabular}
\caption{{\en PGMRES:} Επαναλήψεις και χρόνοι {\en CPU} ($\mathfrak{f}_9$).}
\label{tab:ex4}
\end{table}

\begin{table}[htbp]
\centering
\begin{tabular}{c|c|cc|cc|cc}
\toprule
$n$ & $I_n$ & $\mathcal{BC}_n(\mathfrak{f}_9)$ & {\en CPU} & $\mathcal{BC}_n$ & {\en CPU} & $R_{4,4}$ & {\en CPU}\\\midrule
1024 & - & 11 & 0.1398 & 19 & 0.3922 & 45 & 2.3847 \\
2048 & - & 11 & 0.1928 & 20 & 0.4966 & 49 & 4.8036 \\
4096 & - & 11 & 0.3162 & 19 & 0.7050 & 53 & 10.277 \\
8192 & - & 11 & 0.5281 & 19 & 1.1097 & 57 & 23.194 \\\bottomrule
\end{tabular}
\caption{{\en PCGN:} Επαναλήψεις και χρόνοι {\en CPU} ($\mathfrak{f}_9$).}
\label{tab:ex4i}
\end{table}
\end{exmp}

\begin{exmp}\normalfont
Ως τελευταίο παράδειγμα, χρησιμοποιούμε την προτεινόμενη τεχνική προρρύθμισης για ένα σύστημα, του οποίου ο πίνακας {\en Toeplitz} έχει μια συνεχή γεννήτρια συνάρτηση με ρίζες στο 0 και στο $\pm 1$, δηλαδή στο 1 και στο $2\pi-1$, (στο διάστημα που μελετάμε). Το πραγματικό μέρος της συνάρτησης που επιλέξαμε, παίρνει τιμές κοντά στο 0 για $x\in[0,1.2]$, όπως φαίνεται στο Σχήμα \ref{fig:ex5_real_part}, όπου δίνουμε τις πρώτες 450 τιμές του $F_{n-1}^1$ στο $G_n$, $n=2048$. Μετά από αυτή τη μικρή εισαγωγή, σημειώνουμε ότι η γεννήτρια συνάρτηση δίνεται ως $\mathfrak{f}_{14}(x)=x^2(x^2-1)^2+\mathrm{i}\mathfrak{h}_3^\prime(x)$, όπου:

$$\mathfrak{h}_3^\prime(x)=\begin{cases}
x+\pi,~&x\in\big(-\pi,-\pi+\frac{1}{2}\big]\\
\frac{x}{-2\pi+3}+\frac{1}{-2\pi+3},~&x\in\big(-\pi+\frac{1}{2},-\frac{1}{2}\big]\\
\frac{x}{2\pi-3},~&x\in\big(-\frac{1}{2},\frac{1}{2}\big]\\
\frac{x}{-2\pi+3}-\frac{1}{-2\pi+3},~&x\in\big(\frac{1}{2},\pi-\frac{1}{2}\big]\\
x-\pi,~&x\in\big(\pi-\frac{1}{2},\pi\big]
\end{cases}.$$

\begin{figure}[htbp]%
    \centering
    \includegraphics[width=0.9\linewidth]{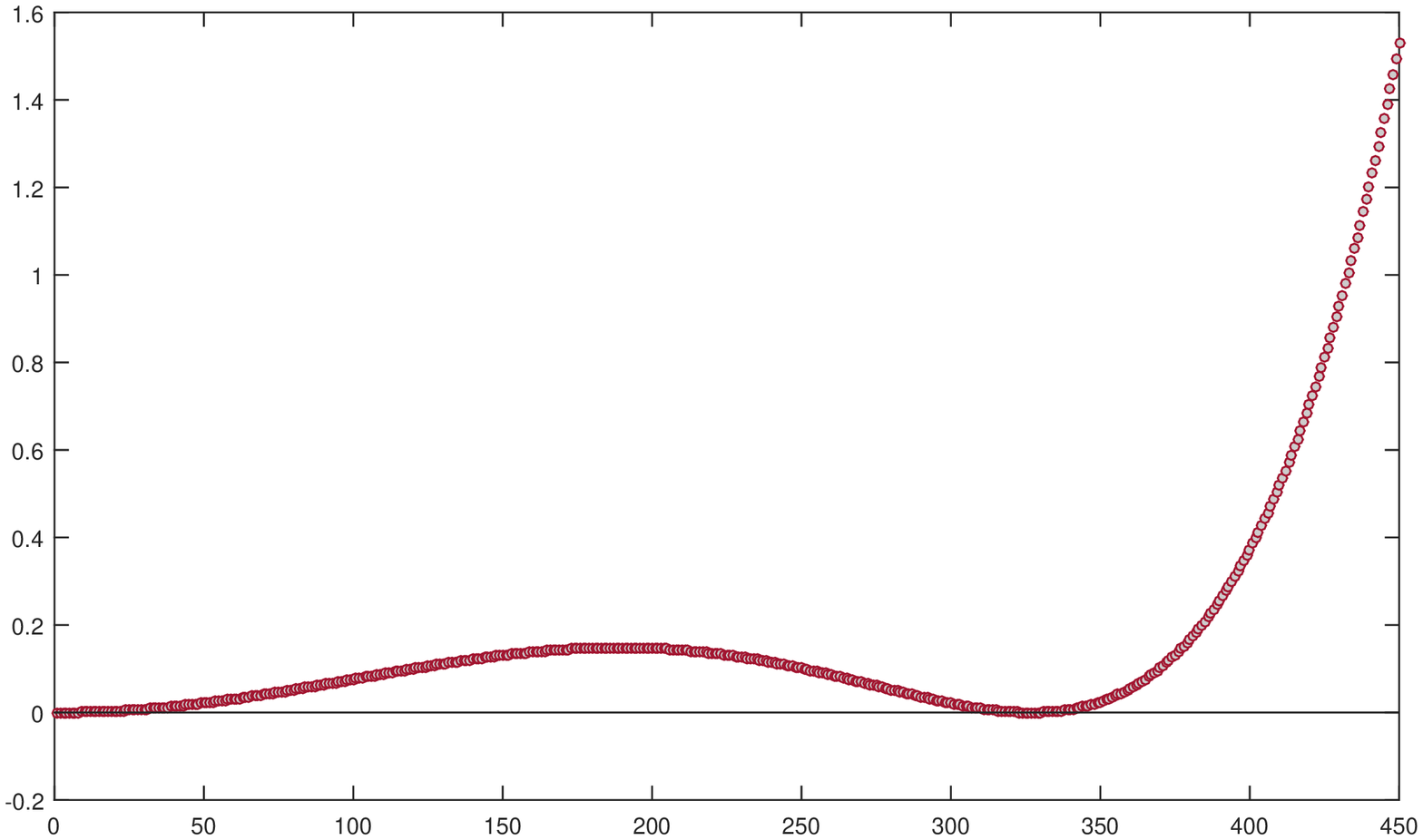}%
    \caption{Πρώτες 450 τιμές του $F_{n-1}^1$ στο $G_n$ ($\mathfrak{f}_{14}$).}
    \label{fig:ex5_real_part}
\end{figure}

`Οπως και στο προηγούμενο παράδειγμα, η ρίζα στο 1 εκτιμήθηκε ως $\widetilde{x}_1=1.000155$, για διάφορες διαστάσεις του προρρυθμισμένου συστήματος. Σχολιάζουμε μια λεπτομέρεια στην εκτίμηση της $\widetilde{x}_1$. Η απόλυτη τιμή του $F_{n-1}^1$ έχει δύο τοπικά ελάχιστα στην περιοχή της ρίζας, για παράδειγμα, όταν $n=2048$ λαμβάνουμε τα τοπικά ελάχιστα στα σημεία $\theta_{325}$ και $\theta_{329}$. Ωστόσο, θεωρούμε ότι η $f_1$ έχει μια ρίζα στο $\theta_{327}$, που είναι ο μέσος όρος των παραπάνω ποσοτήτων, διότι τα $\theta_{325}$ και $\theta_{329}$ διαφέρουν κατά $\mathcal{O}\left(\frac{1}{n}\right)$. Ακριβώς η ίδια συμπεριφορά παρατηρήθηκε σε όλες τις διαστάσεις που εξετάσαμε ($n:1024$, 2048, 4096, 8192).

Εκτιμήσαμε τις πολλαπλότητες των ριζών, τρέχοντας μόνο 2 επαναλήψεις της μεθόδου Αντίστροφων Δυνάμεων. `Οπως φαίνεται στους Πίνακες \ref{tab:est_ex5} και \ref{tab:est_ex5i}, οι πολλαπλότητες εκτιμήθηκαν ορθά ως $m_0^1=2$, $m_1^1=2$, $m_0^2=1$ και $m_1^2=1$. Το τριγωνομετρικό πολυώνυμο που χρησιμοποιήσαμε για την άρση των ριζών είναι το $g_n(x)=(2-2\cos{x})(\cos{\widetilde{x}_1}-\cos{x})^2+\mathrm{i}\sin{x}(\cos{x}-\cos{\widetilde{x}_1})$.

\begin{table}[H]
\centering
\begin{tabular}{c|cccc}
\toprule
$k$ & $\widetilde{\lambda}^1_{0,k}$ & $\log_2{(\widetilde{s_0}^1)}$ & $\widetilde{\lambda}^1_{1,k}$ & $\log_2{(\widetilde{s_1}^1)}$\\\midrule
16 & 0.0286 & 1.7869 & 0.0782 & 1.8854\\
32 & 0.0081 & & 0.2279 &\\
64 & 0.0021 & & 0.0078 &\\\bottomrule
\end{tabular}
\caption{Πολλαπλότητα των ριζών ($\operatorname{Re}(\mathfrak{f}_{14})$).}
\label{tab:est_ex5}
\end{table}

\begin{table}[H]
\centering
\begin{tabular}{c|cccc}
\toprule
$k$ & $\widetilde{\lambda}^2_{0,k}$ & $\log_2{(\widetilde{s_0}^2)}$ & $\widetilde{\lambda}^2_{1,k}$ & $\log_2{(\widetilde{s_1}^2)}$\\\midrule
16 & 0.0362 & 0.6224 & 0.0381 & 0.7394\\
32 & 0.0207 &  & 0.0210 &\\
64 & 0.0107 &  & 0.0108 &\\\bottomrule
\end{tabular}
\caption{Πολλαπλότητα των ριζών ($\operatorname{Im}(\mathfrak{f}_{14})$).}
\label{tab:est_ex5i}
\end{table}

Οι αριθμοί επαναλήψεων και οι χρόνοι εκτέλεσης, με χρήση της {\en PGMRES} και του προτεινόμενου προρρυθμιστή, καθώς επίσης και του ταινιωτού-επί-κυκλοειδή προρρυθμιστή του προηγούμενου κεφαλαίου, όπου η γεννήτρια συνάρτηση θεωρείται γνωστή εκ των προτέρων, δίνεται στον Πίνακα \ref{tab:ex5}. Σε αυτόν παρατηρούμε την καλή συμπεριφορά και αποτελεσματικότητα του $\mathcal{BC}_n$, αφού οι επαναλήψεις που λαμβάνουμε είναι κοντά με τις επαναλήψεις κατόπιν χρήσεως του $\mathcal{BC}_n(\mathfrak{f}_{14})$.

\begin{table}[htbp]
\centering
\begin{tabular}{c|c|cc|cc}
\toprule
$n$ & $I_n$ & $\mathcal{BC}_n(\mathfrak{f}_{14})$ & {\en CPU} & $\mathcal{BC}_n$ & {\en CPU} \\\midrule
1024 & $>$500 & 7 & 0.2286 & 9 & 0.6108 \\
2048 & $>$500 & 7 & 0.3982 & 8 & 0.7936 \\
4096 & $>$500 & 6 & 0.7771 & 8 & 1.2602 \\
8192 & $>$500 & 6 & 1.4390 & 8 & 2.7490 \\\bottomrule
\end{tabular}
\caption{{\en PGMRES:} Επαναλήψεις και χρόνοι {\en CPU} ($\mathfrak{f}_{14}$).}
\label{tab:ex5}
\end{table}

\begin{figure}[htbp]%
    \centering
    \subfloat[Ιδιοτιμές όταν $n=1024$ και $n=8192$.]{{\label{fig:10a}\includegraphics[width=0.45\linewidth]{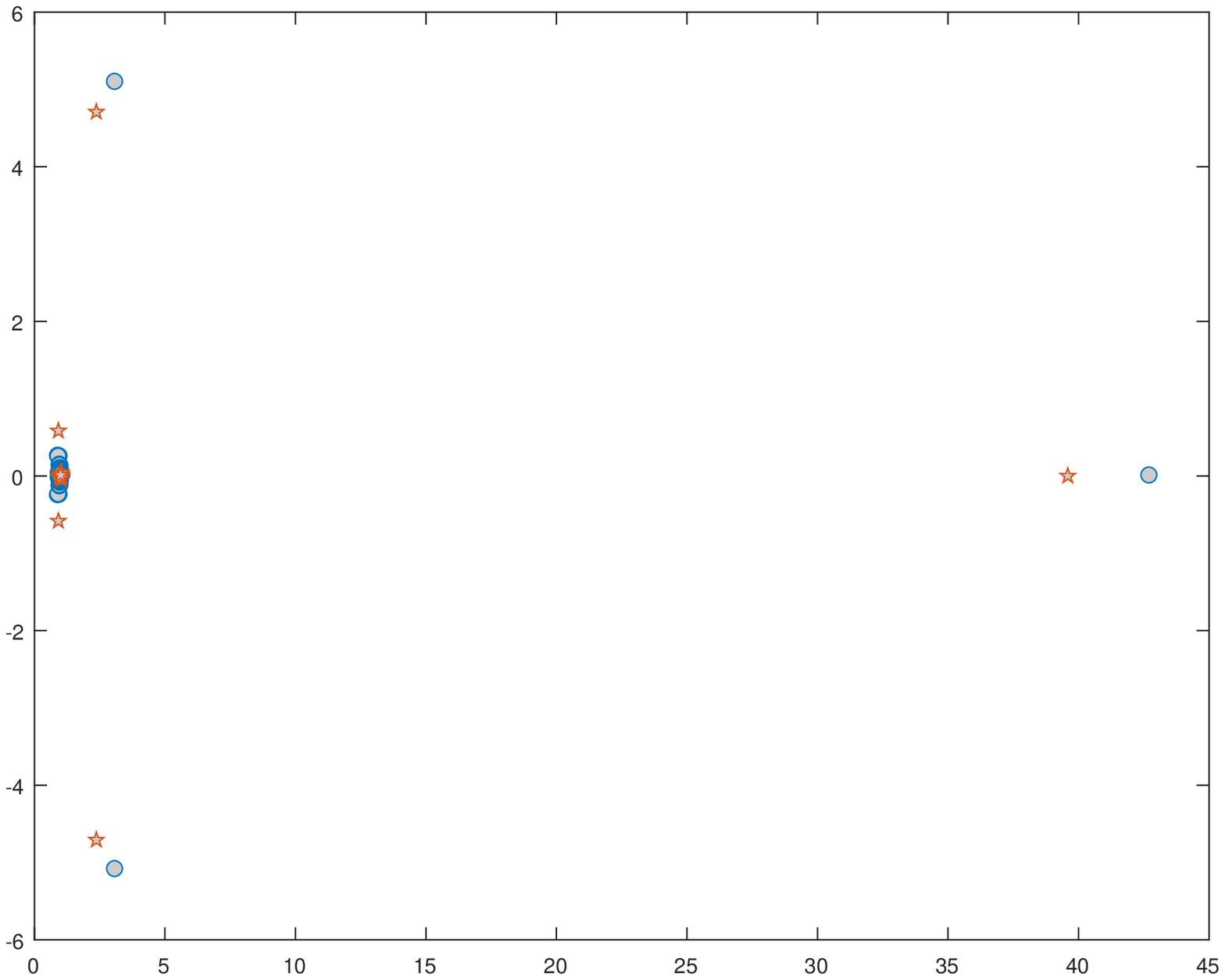}}}%
    \qquad
    \subfloat[Ιδιοτιμές κοντά στο $(1,0)$.]{{\label{fig:10b}\includegraphics[width=0.45\linewidth]{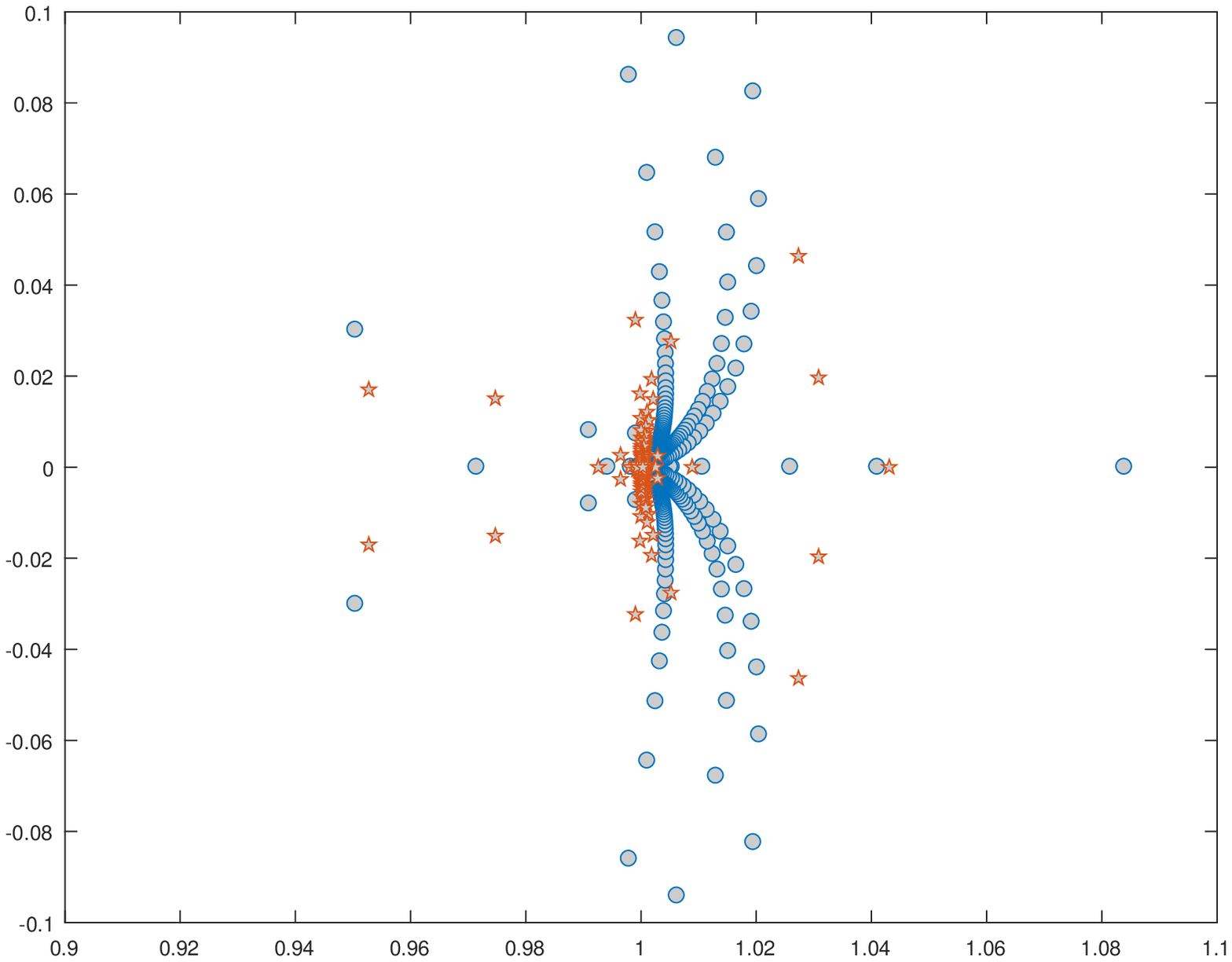}}}%
    \caption{Ιδιοτιμές ($\mathfrak{f}_{14}$).}
    \label{fig:ex5}
\end{figure}

`Οπως και στο προηγούμενο παράδειγμα, αναμένουμε ένα σφάλμα στην εκτίμηση των ριζών $\varepsilon_1=\mathcal{O}\left(\frac{\log{n}}{n}\right)$. Ωστόσο, αυτό φαίνεται να είναι $\varepsilon_1\sim\frac{1}{n}$. `Οπως αποδείχθηκε στο Θεώρημα \ref{thm:ill_cond}, οι ιδιοτιμές του τελευταίου παράγοντα, $T_n\left(\frac{\mathfrak{f}_{14}}{g}\right)C_n^{-1}\left(\frac{\mathfrak{f}_{14}}{g}\right)$, παρουσιάζουν κύρια συσσώρευση στο σημείο $(1,0)$, αφού αυτή η συνάρτηση είναι συνεχής. Για τον $T_n^{-1}(g_n)T_n(g)$, θεωρήσαμε ως $s_n\sim\frac{\log{n}}{n}$ για να λάβουμε τη γενική συσσώρευση των ιδιοτιμών, σε μια περιοχή του $(1,0)$, με ακτίνα $r_n=\mathcal{O}\left(\frac{1}{\log{n}}\right)$ και $\mathcal{O}(\log{n})$ ιδιοτιμές εκτός της συσσώρευσης. Για τον πρώτο παράγοντα $C_n\left(\frac{\mathfrak{f}_{14}g_n}{F_{n-1}g}\right)$, επειδή η $\mathfrak{f}_{14}$ είναι συνεχής, αλλά όχι επαρκώς ομαλή, μέσω του Θεωρήματος \ref{thm:ill_cond} καταλήγουμε σε κύρια συσσώρευση των ιδιοτιμών σε μια περιοχή του $(1,0)$, με ακτίνα $r_n=\mathcal{O}\left(\frac{\log{n}}{n}\right)$. Επομένως, από το ίδιο θεώρημα καταλήγουμε στη γενική συσσώρευση των ιδιοτιμών του προρρυθμισμένου πίνακα σε μια περιοχή του $(1,0)$, με ακτίνα $\mathcal{O}\left(\frac{1}{\log{n}}\right)$ και $\mathcal{O}(\log{n})$ ιδιοτιμές εκτός της συσσώρευσης. Στο Σχήμα \ref{fig:ex5} παρουσιάζουμε τη συσσώρευση των ιδιοτιμών όταν $n=1024$ (μπλε κύκλοι) και $n=8192$ (πορτοκαλί πεντάγραμμα). Στο Σχήμα \ref{fig:10a} μπορούμε εύκολα να ξεχωρίσουμε τις περιοχές συσσώρευσης και βλέπουμε ότι η περιοχή που αντιστοιχεί στη διάσταση $n=8192$ είναι πολύ μικρότερη από αυτή για $n=1024$. Με άλλα λόγια, όσο μεγαλύτερη είναι η διάσταση $n$, τόσο μικρότερη γίνεται η περιοχή συσσώρευσης. Το Σχήμα \ref{fig:10b} είναι μια μεγέθυνση της συσσώρευσης των ιδιοτιμών κοντά στο $(1,0)$, με ακτίνα $\frac{1}{\log{1024}}=0.1$.
\end{exmp}

\newpage

\pagestyle{syn}
\addcontentsline{toc}{chapter}{Σύνοψη}

\chapter*{Σύνοψη}

Στην παρούσα διατριβή μελετήσαμε και προτείναμε προρρυθμιστές για την ταχεία επίλυση μη-συμμετρικών και πραγματικών συστημάτων {\en Toeplitz}, με μεθόδους υποχώρων {\en Krylov} και πιο συγκεκριμένα με την Προρρυθμισμένη Γενικευμένη μέθοδο Ελαχίστων Υπολοίπων {\en (PGMRES)} και την Προρρυθμισμένη μέθοδο Συζυγών Κλίσεων για το σύστημα των Κανονικών Εξισώσεων {\en (PCGN)}. Οι προρρυθμιστές που προτάθηκαν και των οποίων αποδείχθηκε η αποτελεσματικότητα ανήκουν στις κατηγορίες των ταινιωτών πινάκων {\en Toeplitz} και των κυκλοειδών πινάκων.

Πιο συγκεκριμένα, μετά από μια σύντομη ιστορική αναδρομή στην προρρύθμιση συστημάτων {\en Toeplitz}, στο πρώτο κεφάλαιο δόθηκαν οι βασικοί ορισμοί και κάποια χρήσιμα θεωρητικά αποτελέσματα σχετικά με τη συσσώρευση του φάσματος. Ακολούθως, στο δεύτερο κεφάλαιο προτάθηκε ως προρρυθμιστής ένας ταινιωτός πίνακας {\en Toeplitz}, ο οποίος προκύπτει κατόπιν άρσης των ριζών της, γνωστής εκ των προτέρων, γεννήτριας συνάρτησης και βέλτιστης ομοιόμορφης προσέγγισης ή παρεμβολής, με τριγωνομετρικά πολυώνυμα. Τα αποτελέσματα που παρουσιάστηκαν σε αυτό μπορούν να βρεθούν στη \cite{NT_2019}. Στο επόμενο κεφάλαιο μελετήθηκε ένα είδος κυκλοειδή προρρυθμιστή για μη-συμμετρικά συστήματα {\en Toeplitz} με καλή κατάσταση, δηλαδή συστήματα των οποίων η γεννήτρια συνάρτηση δεν έχει ρίζες. Σε αυτό έγινε επίσης μελέτη ενός ταινιωτού-επί-κυκλοειδή προρρυθμιστή, για συστήματα με κακή κατάσταση, δηλαδή συστήματα που παράγονται από κάποια γεννήτρια συνάρτηση με ρίζες. Τα αποτελέσματα αυτού του κεφαλαίου μπορούν να βρεθούν στη \cite{noutsos2021band}. Αναφέρουμε ότι μελετήθηκε τόσο η συνεχής, όσο και η κατά τμήματα συνεχής περίπτωση και ότι η γεννήτρια συνάρτηση σε αυτό το κεφάλαιο θεωρήθηκε επίσης γνωστή εκ των προτέρων. Αυτή η θεώρηση δεν έγινε στο τελευταίο κεφάλαιο της διατριβής, όπου μελετήθηκαν κατάλληλα προσαρμοσμένοι προρρυθμιστές των προηγούμενων κεφαλαίων για συστήματα με άγνωστη γεννήτρια συνάρτηση. `Εγινε εκτίμηση των πιθανών ριζών και των πολλαπλοτήτων αυτών, από τις τιμές του πίνακα συντελεστών με χρήση του αναπτύγματος {\en Fourier}. Η προσαρμογή των ταινιωτών {\en Toeplitz} προρρυθμιστών δημοσιεύθηκε στα πρακτικά διεθνούς επιστημονικού συνεδρίου \cite{chay} και αυτή των ταινιωτών-επί-κυκλοειδών προρρυθμιστών μπορεί να βρεθεί στην εργασία \cite{eajam}. Στο τέλος του κάθε κεφαλαίου δόθηκαν διάφορα αριθμητικά παραδείγματα, στα οποία φανερώνεται η αποτελεσματικότητα των προτεινόμενων προρρυθμιστών.

Η αναγκαιότητα της παραπάνω έρευνας προέκυψε από το γεγονός ότι σε πολλές εφαρμογές, όπως στην επεξεργασία εικόνας, στην επεξεργασία σήματος και σε εφαρμογές που προκύπτουν από τη διακριτοποίηση διαφορικών εξισώσεων, εμφανίζονται μη-συμμετρικά και πραγματικά συστήματα {\en Toeplitz}. Σε ακόμη περισσότερες εφαρμογές εμφανίζονται συστήματα των οποίων ο πίνακας συντελεστών είναι σχεδόν {\en Toeplitz (quasi-Toeplitz)}. Τέτοιοι πίνακες προκύπτουν κυρίως από τη διακριτοποίηση διαφορικών και ολοκληρωτικών εξισώσεων και είναι {\en Toeplitz} πίνακες με μια διαταραχή της {\en (Toeplitz)} δομής στο άνω αριστερά και κάτω δεξιά μέρος του πίνακα ή ````{\en Toeplitz} συν διαγώνιος'''' πίνακας ή ````{\en Toeplitz} συν ταινιωτός'''' πίνακας με σχετικά μικρό πλάτος ταινίας. Εκτιμούμε ότι οι προτεινόμενοι προρρυθμιστές με κατάλληλη προσαρμογή θα είναι αποτελεσματικοί για τέτοιου είδους συστήματα και αυτό αποτελεί σκέψεις για μελλοντική έρευνα.

Ιδιαίτερο ενδιαφέρον, κυρίως στις εφαρμογές, παρουσιάζουν τα συστήματα που έχουν ως πίνακα αγνώστων δι-διάστατους (ή γενικότερα $d$-διάστατους) πίνακες {\en Toeplitz} που προκύπτουν από διακριτοποίηση μερικών διαφορικών εξισώσεων κυρίως συνοριακών προβλημάτων, καθώς και ρητών διαφορικών εξισώσεων {\en (fractional differential equations)}. Εδώ προκύπτουν ιδιαίτερες δυσκολίες ως προς την υπολογιστική πολυπλοκότητα των αλγορίθμων, κυρίως των ταινιωτών προρρυθμιστών, καθώς και στην ανάπτυξη θεωρίας σχετιζόμενης με τη συσσώρευση του φάσματος, κυρίως στους δι-διάστατους κυκλοειδείς προρρυθμιστές για τους οποίους έχουν αποδειχθεί αρνητικά αποτελέσματα \cite{NSV_2003, NSV_2004, ST_2000}. Ωστόσο ο συνδυασμός δι-διάστατων ταινιωτών {\en Toeplitz} και κυκλοειδών πινάκων φαίνεται να αντιμετωπίζει τα προβλήματα με αποτελεσματικότητα, όπως στη συμμετρική περίπτωση \cite{NV_2011}. Ιδιαίτερη δυσκολία παρουσιάζουν δι-διάστατα προβλήματα, όπου η γεννήτρια συνάρτηση είναι άγνωστη, στην εκτίμηση των ριζών και των αντίστοιχων πολλαπλοτήτων. Ωστόσο, θετικά αποτελέσματα που αφορούν στη συμμετρική περίπτωση \cite{NSV_2005, NSV_2006, NSV_2006i} είναι ενθαρρυντικά. `Ολα τα παραπάνω αποτελούν σκέψεις για περαιτέρω έρευνα.

\newpage



\bibliographystyle{abbrv}
\pagestyle{bib}
\addcontentsline{toc}{chapter}{Βιβλιογραφία}
\bibliography{biblio}
\end{document}